\documentclass[preprint,3p]{elsarticle1}
\usepackage{graphicx}
\usepackage{amssymb}
\usepackage{comment}

\usepackage{lineno}

\usepackage{amsmath}
\usepackage{amsthm}
\usepackage{epsfig}

\usepackage{mathrsfs}

\usepackage{psfrag}
\usepackage{color}

\journal{Linear Algebra and its Applications}

\newtheorem{theorem}{Theorem}
\newtheorem{problem}[theorem]{Problem}
\newtheorem{lemma}[theorem]{Lemma}
\newtheorem{example}[theorem]{Example}
\newtheorem{definition}[theorem]{Definition}

\newtheorem{corollary}[theorem]{Corollary}
\newtheorem{remark}{Remark}

\newtheoremstyle{algstyle}%
  {10mm}       % measure of space to leave above the theorem. E.g.: 3pt
  {10mm}       % measure of space to leave below the theorem. E.g.: 3pt
  {\tt}   % name of font to use in the body of the theorem
  {0pt}        % measure of space to indent
  {\bfseries}  % name of head font
  {\newline}   % punctuation between head and body
  {10mm}       % space after theorem head
  {\thmname{#1}\thmnumber{ #2}\thmnote{ (#3)}}          
\theoremstyle{algstyle}
\newtheorem{algorithm}{Algorithm}

\newtheoremstyle{algdashstyle}%
  {10mm}       % measure of space to leave above the theorem. E.g.: 3pt
  {10mm}       % measure of space to leave below the theorem. E.g.: 3pt
  {\tt}   % name of font to use in the body of the theorem
  {0pt}        % measure of space to indent
  {\bfseries}  % name of head font
  {\newline}   % punctuation between head and body
  {10mm}       % space after theorem head
  {\thmname{#1}\thmnumber{ #2}$'$\thmnote{ (#3)}}          % Manually specify head
\theoremstyle{algdashstyle}

\newcommand{\nw}[1]{%
\textbf{#1}%
}

\newcommand{\mnw}[1]{%
\boldsymbol{#1}%
}

\newcommand{\ppmatrix}[1]{%
\begin{pmatrix} #1 \end{pmatrix}%
}

\newcommand{\lrar}{\leftrightarrow}

\newcommand{\equaln}{\hspace{0.1cm} = \hspace{0.1cm}}

\newcommand{\equivd}{:\equiv }

\newcommand{\V}{\mbox{$\mathcal V$}} 

\newcommand{\F}{\mbox{$\cal F$}} 
    \newcommand{\0}{{\mathbf 0}}        
\newcommand{\Vs}{{\cal V}_{S}}             
             
\newcommand{\Vsp}{{\cal V}_{SP}}           			%new  \cal V_S 
\newcommand{\Msp}{{\cal M}_{SP}}           			%new  \cal V_S 
\newcommand{\Gsp}{{\cal G}_{SP}}           			%new  \cal V_S 
\newcommand{\Gsq}{{\cal G}_{SQ}}           			%new  \cal V_S 
\newcommand{\Ksp}{{\cal K}_{SP}}           			%new  \cal V_S 
\newcommand{\Asp}{{\cal A}_{SP}}           			%new  \cal V_S 
\newcommand{\A}{{\cal A}}           			%new  \cal V_S 
\newcommand{\Kpq}{{\cal K}_{PQ}}           			%new  \cal V_S 
\newcommand{\Apq}{{\cal A}_{PQ}}           			%new  \cal V_S 
\newcommand{\Ksq}{{\cal K}_{SQ}}           			%new  \cal V_S 
\newcommand{\Q}{{\mathbb{Q}}}
\newcommand{\Vsq}{{\cal V}_{SQ}}

\newcommand{\Vpq}{{\cal V}_{PQ}}            			%new  \cal V_PQ 
\newcommand{\Mpq}{{\cal M}_{PQ}}            			%new  \cal V_PQ 
\newcommand{\D}[0]{{\cal D}}                    	%         \cal D
\newcommand{\G}[0]{{\cal G}}                       %        \cal G
  
\newcommand{\K}[0]{{\cal K}}                       %        \cal K
\newcommand{\M}{\mbox{$\cal M$}}

\newcommand{\N}[0]{{\cal N}}    							%new     \cal N
\newcommand{\It}[0]{{\cal I}_{t}}                   %        \cal I_t

\newcommand{\x}{\mbox{${\bf x}$}}             		%new  \bf x
\newcommand{\s}[1]{ short  \ #1}                 % new  \short
\begin{document}

\begin{frontmatter}

%% Title, authors and addresses

%% use the tnoteref command within \title for footnotes;
%% use the tnotetext command for the associated footnote;
%% use the fnref command within \author or \address for footnotes;
%% use the fntext command for the associated footnote;
%% use the corref command within \author for corresponding author footnotes;
%% use the cortext command for the associated footnote;
%% use the ead command for the email address,
%% and the form \ead[url] for the home page:
%%
%% \title{Title\tnoteref{label1}}
%% \tnotetext[label1]{}
%% \author{Name\corref{cor1}\fnref{label2}}
%% \ead{email address}
%% \ead[url]{home page}
%% \fntext[label2]{}
%% \cortext[cor1]{}
%% \address{Address\fnref{label3}}
%% \fntext[label3]{}

%\title{Analysis of Linear Dynamical Systems without using State Space Representation}
\title{Implicit Linear Algebra and Basic Circuit Theory II:
 port behaviour of  rigid  multiports
}
%% use optional labels to link authors explicitly to addresses:
%% \author[label1,label2]{<author name>}
%% \address[label1]{<address>}
%% \address[label2]{<address>}

\author[]{H. Narayanan}
%\corref{cor1}}
\ead{hn@ee.iitb.ac.in}
%\cortext[cor1]{Corresponding author}
%\author[hari]{Hariharan Narayanan}
%\ead{hariharan.narayanan@tifr.res.in}
\address{Department of Electrical Engineering, Indian Institute of Technology Bombay}
%\address[hari]{School of Technology and Computer Science, Tata Institute of Fundamental Research}

\begin{abstract}
In this paper, we define the notion of rigidity for linear electrical multiports
and for matroid pairs. We show the parallel  between the 
two and study the consequences of this parallel. We  
 present  applications  to 
testing, using purely matroidal methods, whether a 
connection of rigid multiports yields a 
 linear network with unique solution.
We also indicate that rigidity can be regarded as 
the closest notion to duality that can be hoped for, 
  when the spaces correspond to different physical constraints, such as topological and device characteristic.

%=========================================================
%
%In this paper we derive some basic results of circuit theory 
%using `Implicit Linear Algebra' (ILA). This approach has the advantage of 
%simplicity and generality. Implicit linear algebra is outlined in
%(see \cite{HNarayanan1986a}, Section 8 of \cite{HNarayanan1997,HNarayanan2009}), 
%=============================================================

%In this paper we study ideas which have proved useful in topological 
%network theory \cite{HNarayanan1986a,HNarayanan,narayanan1987topological,HNarayanan1997} in the context of lattices of numbers.
A vector space  $\Vs$ is a collection of row vectors, over a field $\mathbb{F}$ on a finite column set $S,$ closed under addition and multiplication by elements 
of $\mathbb{F}.$ We denote the space of all vectors on $S$ by $\F_S$ 
and the space containing only the zero vector on $S$ by $\0_S.$
An affine space $\A_S$ is the sum  of a vector $\alpha_S$ with $\Vs.$
The dual $\Vs^{\perp}$ of a vector space $\Vs$ is the collection
of all vectors whose dot product with vectors in $\V_S$ is zero.

A matroid $\M_S$  on $S,$
is a family of `independent' sets with the property that maximal 
independent sets contained in any given subset of $S$ have the same cardinality.
%A base of $\M_S$ is a maximal independent subset of $\M_S$ contained in $S.$
%The dual $\M_S^*$ of $\M_S$ has bases which are complements of bases 
%of $\M_S.$
%Let $\M_{AB},\M_B,$ be matroids on $A\uplus B,B,$ respectively.
%Then $\M_{AB}\vee \M_B$ denotes the matroid on $A\uplus B$ whose independent 
%sets are unions of independent sets of $\M_{AB},\M_B.$

A multiport is an ordered pair $(\V^1_{AB},\A^2_{B}),$
where $\V^1_{AB}$ is the solution space on $A\uplus B$ of the Kirchhoff current and voltage 
equations of the graph of the multiport and $\A^2_{B}\equivd \alpha_B+\V^2_B$ is 
the device characteristic of the multiport, with $A$ corresponding to port voltages
and currents and $B$ corresponding to internal voltages and currents.
A solution of $(\V^1_{AB},\A^2_{B}),$ is a vector $(x_A,x_B)$ such that 
$(x_A,x_B)\in \V^1_{AB}, x_B\in\A^2_{B}.$
A network is  a multiport $(\V^1_{B},\A^2_{B}).$

The pair $\{\V^1_{AB},\alpha_B+\V^2_{B}\}$ is said to be  rigid iff 
it has a solution $(x_A,x_B)$ for every vector $\alpha_B$ and 
given a restriction $x_A$ of the solution, $x_B$ is unique.
%$r(\V^1_{AB}+(\V^2_{B}+\0_A))= r(\V^1_{AB})+r(\V^2_{B})$ and\\ 
%$r((\V^1_{AB})^{\perp}+((\V^2_B)^{\perp}+\0_A))= 
%r((\V^1_{AB})^{\perp})+r((\V^2_B)^{\perp}).$
A rigid network has a unique solution.\\
The pair $\{\M^1_{AB},\M^2_{B}\}$ is said to be  rigid iff 
the two matroids have disjoint bases which cover $B.$
We show that the properties of rigid pairs of matroids closely parallel
those of rigid multiports.
%$r(\M^1_{AB}\vee\M^2_{B})= r(\M^1_{AB})+r(\M^2_{B})$ and 
%$r((\M^1_{AB})^{*}\vee(\M^2_B)^{*})= 
%r((\M^1_{AB})^{*})+r((\M^2_B)^{*}).$

New multiports can be constructed from multiports $(\V^1_{AB},\A^2_{B}),
(\V^1_{CD},\A^2_{D}),$ $A,B,C,D,$ pairwise disjoint, by forcing the variables  on 
$A_1\subseteq A$ and $D_1\subseteq D$ to lie in an affine space $\A_{A_1C_1}.$
This new multiport $((\V^1_{AB}\oplus \V^1_{CD}), (\A^2_{B}\oplus \A^2_{D}\oplus \A_{A_1C_1}))$
is said to be obtained from the original ones by connection through $\A_{A_1C_1}.$ We give necessary and sufficient conditions for this multiport to be 
rigid given that the original multiports are rigid.
Networks are usually constructed in this manner from smaller multiports,
so that these conditions are necessary and sufficient for the 
final network to have a unique solution.

We use the methods developed in the paper to show that a multiport 
with independent and controlled sources and positive or negative resistors, whose parameters can be taken to be algebraically independent over $\Q,$
is rigid, if certain simple topological conditions are satisfied by the device edges.
%the voltage sources, controlled as well as independent, 
%do not form a loop with controlling current branches and 
%if the current sources, controlled as well as independent,
%do not form a cutset with controlling voltage branches.
\end{abstract}
\begin{keyword}
%% keywords here, in the form: keyword \sep keyword
Rigidity, multiports, matroids, implicit duality.
%% MSC codes here, in the form: \MSC code \sep code
%% or \MSC[2008] code \sep code (2000 is the default)
\MSC   15A03, 15A04, 94C05, 94C15 

\end{keyword}

\end{frontmatter}

%%
%% Start line numbering here if you want
%%
%\linenumbers

%% main text
\section{Note}
This paper is a sequel to \\`Implicit Linear Algebra and Basic Circuit Theory (authors H.Narayanan and Hariharan Narayanan)\\
arXiv:2005.00838v1,May 2020' \cite{narayanan2020} hereinafter denoted by `ILABCI'.
\subsection{Errata  for ILABCI}
1. Line 14 of page 5  should read \\
`... $Ax=b.x^TA^T=b^T$..'
\\in place of 
\\`... $Ax=b.x^TA=b^T$..'
\\
2. Line 27 of page 7 should read `$\cdots$ where
$e_i, i=0, \cdots ,k-1, $' 
\\in place of
\\ `$\cdots$ where
$e_i, i=1, \cdots ,k-1, $'
\\
3. Line 26 of page 9 should read 
\\`..\nw{skewed composition} $\Ksp\rightleftharpoons\Kpq$..'
\\in place of 
\\`..\nw{skewed composition} $\Ksp\lrar\Kpq$..'
\\4. Theorem 26  part 3b should read \\
 `If $\N_P$ is regular $r(\V_{S'S"}^{\perp})+r(\breve{\V}_{P'P"})
=|S|+|P|.$'
\\in place of 
\\`If $\N_P$ is regular $r(\V_{S'S"})+r(\breve{\V}_{P'P"})
=|S|+|P|.$'
\\5. In the proof of Theorem 26  part 3b 
\\`$r(\V_{S'S"}^{\perp})$' should replace `$r(\V_{S'S"})$'
wherever the latter occurs.
\subsection{Notational changes from ILABCI}
`Regular multiports' are now replaced by `Rigid Multiports'.
This was felt necessary since `regular' 
is already overused in the literature.
Further, a network with unique solution is analogous to 
rigid structures in statics.
The methods of this paper are equally applicable to static structures.

\section{Introduction}
\label{sec:intro}
The notion of multiports is fundamental to, and characteristic of,  electrical
 network theory \cite{belevitch68, desoerkuh}. Decomposition into multiports is a useful technique 
for theoretical and computational analysis of the network \cite{HNarayanan1997}.
It is also the natural way of synthesizing a complex network in terms of 
simpler ones \cite{belevitch68}. In recent times the importance of ports has  also been emphasized 
for system theory \cite{vanderscaftport}. 

An electrical network can be regarded as a theoretical
model of a practical circuit. A test for validity of this model 
is whether it has a unique solution. A natural question is, `what is the 
corresponding test for validity of a multiport as a model for a practical 
multiport?'. We claim that the test should be for `rigidity' of the multiport
 when it is linear. This paper is an attempt to justify this claim.

We say that a linear multiport is rigid if it has a solution for arbitrary 
source values and has a unique solution for a given valid port condition.
To see the relevance of the notion, we note that
when  the multiport is a network without ports,  rigidity translates to the network having a unique solution.
%If we connect rigid multiports at some of their ports according to rules 
%which again can be stated in terms of rigidity we get a rigid multiport and, as above,  if 
%the latter has no ports it will have a unique solution.
Further, if multiports are to be connected to form a network with unique solution, it is necessary that they be rigid.
Therefore, it is only rigid multiports which, after some preprocessing,
 can be handled by conventional circuit simulators.
The notion is seen to be natural, when we show in this paper that 
if we connect rigid multiports, to retain rigidity, we only have to 
verify certain interface properties, which again can be interpreted in terms 
of rigidity.
%
%the questions posed in terms of rigidity 
%  are useful for understanding electrical multiport connections. 

Rigid multiports behave the way linear multiports are expected to behave
by a working engineer.
For instance, we show in this paper, if the device characteristic is `proper' (dimension equal to half 
the number of current and voltage variables), the multiport behaviour (the affine space of voltage current pairs that can occur at the ports) will be proper.
Further, under simple topological conditions, assuming algebraic independence 
of parameters of the device characteristic, the port behaviour will actually 
be `hybrid' (in the equations of the port behaviour, exactly one of the current-voltage variables of a port will occur on the left side).

We show that there is a corresponding notion of rigidity for pairs of matroids and its properties 
exactly mimic those of multiport rigidity. We exploit this similarity between 
multiports and matroid pairs to derive matroidal tests for rigidity 
of multiports which are made up of nonzero resistors, independent and controlled sources,
 assuming that the resistances, gains, transresistances, transconductances, are algebraically independent over $\Q.$

We now give a brief summary of the  paper.\\
A vector space  $\Vs$ is a collection of row vectors, over a field $\mathbb{F}$ on a finite column set $S,$ closed under addition and multiplication by elements
of $\mathbb{F}.$ 
(The vector space may be thought of as composed of row vectors and $S$ may be regarded as the column set.)
A linear multiport can be treated as  an ordered pair $(\V^1_{AB},\A^2_{B}),$
where $\V^1_{AB}$ is the solution space on $A\uplus B$ of the Kirchhoff current and voltage 
equations of the graph of the multiport and $\A^2_{B}\equivd \alpha_B+\V^2_B$ is 
the device characteristic of the multiport, with $A$ corresponding to port voltages
and currents and $B$ corresponding to internal voltages and currents.
%A solution of $(\V^1_{AB},\A^2_{B}),$ is a vector $(x_A,x_B)$ such that 
%$(x_A,x_B)\in \V^1_{AB}, x_B\in\A^2_{B}.$
%A network is  a multiport $(\V^1_{B},\A^2_{B}).$
%A linear multiport can be treated as  a pair $\{\V_{AB},\A_B\}, $ where $\V_{AB}$ is the solution space of the topological constraints (Kirchhoff's current 
%and voltage laws) and $\A_B\equivd \alpha_B+\V_B$ is the device characteristic. 
(We refer to $\V_B$ as the vector space translate of $\A_B.$) 
%A typical vector in $\V_{AB}$ will have the form
A solution of the multiport is a vector of the form 
 $(v_{P'},i_{P"},v_{S'},i_{S"}),$ with $P',P",$ being disjoint copies of the port branches $P$ and $S',S",$ being disjoint copies of the internal
 branches $S,$  $A\equivd P'\uplus P"$ and $ B\equivd S'\uplus S".$ 
This vector is a solution of the multiport provided it belongs to $\V_{AB}$
and its restriction $(v_{S'},i_{S"}),$ belongs to $\A_B.$ 

We say that the linear multiport is rigid iff
\\(a) it has a solution for arbitrary values of sources, i.e., for arbitrary 
values of $\alpha_B,$
and\\
(b) for a given valid port condition (a vector $(v_{P'},i_{P"})$ that is a restriction of a solution of the multiport) 
there is a unique multiport solution. 
%of which it is a restriction.
\\If either of the above conditions is violated, freely available 
circuit simulators (such as SPICE) cannot handle networks made up of the multiports.

Rigidity of $\{\V_{AB},\A_B\}, $ can be shown to be equivalent to the 
condition 
$$r(\V_{AB}+(\V_{B}+\0_A))= r(\V_{AB})+r(\V_{B})\ \ \mbox{and}\ \ 
r(\V_{AB}^{\perp}+(\V_B^{\perp}+\0_A))= 
r(\V_{AB}^{\perp})+r(\V_B^{\perp}).$$
($\V_X^{\perp}$ is the collection of all vectors on $X$ whose dot product
with every vector in $\V_X$ is zero.)

Multiports are usually connected together to form larger multiports, 
through additional device characteristic at some of the ports.
% (see Figure \ref{fig:networkmultiportconnection}).
% (see Figure \ref{fig:networkmultiportconnection}).
So the natural question is, under what conditions is such a connection 
of rigid multiports, rigid.
To answer this and a related question in brief, we introduce some notation.

For any vector space $\V_{AB},$ where $A,B$ are disjoint, we define
$\V_{(-A)B}\equivd \{(-f_A,f_B):(f_A,f_B)\in \V_{AB}\}.$ 
Let $\Vsp,\Vpq$ be vector spaces on $S\uplus P,P\uplus Q,$ where $S,P,Q$ 
are pairwise disjoint. 
We denote the space of all vectors on $S$ by $\F_S$
and the space containing only the zero vector on $S$ by $\0_S.$
We define the 
 {restriction}  of ${\mathcal{V}_{SP}}$ to $S$  to be 
${\mathcal{V}_{SP}\circ S}\equivd \{f_S:(f_S,f_P)\in \mathcal{V}_{SP}\},$
the {contraction}  of ${\mathcal{V}_{SP}}$ to $S$ to be
${\mathcal{V}_{SP}\times S}\equivd \{f_S:(f_S,0_P)\in \mathcal{V}_{SP}\},$
a {minor} of $\V_X$ to be a vector space $(\V_X\circ X_1)\times X_2, X_2\subseteq X_1\subseteq X.$ We define the {sum}
$\mathcal{V}_{SP} + \mathcal{V}_{PQ}$ to be $ (\mathcal{V}_{SP} \oplus \0_{Q}) + (\0_{S} \oplus \mathcal{V}_{PQ}),$ the {intersection} 
${\mathcal{V}_{SP} \cap \mathcal{V}_{PQ}}$ of $\mathcal{V}_{SP}$, $\mathcal{V}_{PQ}$ to be
$(\mathcal{V}_{SP} \oplus  \F_{Q}) \cap (\F_{S} \oplus \mathcal{V}_{PQ}),$
the {matched composition} 
$\Vsp\lrar \Vpq$  to be $(\Vsp+\V_{(-P)Q})\times (S\uplus Q)$
and the {skewed composition}    
$\Vsp\rightleftharpoons \Vpq$  to be $(\Vsp+\V_{PQ})\times (S\uplus Q).$
\\
The implicit duality theorem for vector spaces (Theorem \ref{thm:idt0})
 states that $(\Vsp\lrar \Vpq)^{\perp}\ =\ \Vsp^{\perp}\rightleftharpoons  \Vpq^{\perp}.$

In terms of the above notation, 
we have two questions:\\
{\it Question 1}. Let $(\V_{ABC},\A_C)$ be a rigid multiport, where $A\uplus B$ correspond to port variables and $C,$ to internal variables. Suppose we impose 
additional device characteristic $\A_B$ on the port variables $B.$
When would $(\V_{ABC},\A_B\oplus \A_C)$ be rigid? The answer is `it is rigid iff $(\V_{ABC}\lrar \V_C, \V_B)$ is rigid' (Theorem \ref {thm:derivedregularity}),
$ \V_C, \V_B,$ being vector space translates of $\A_C, \A_B,$ respectively.\\
{\it Question 2}. 
`given two rigid multiports $(\V_{ABC},\A_C), (\V_{D\tilde{B}E},\A_E),$\\$
A,B,C,D,\tilde{B},E,$ being pairwise disjoint, when is the connection 
$(\V_{ABC}\oplus \V_{D\tilde{B}E}, \A_C\oplus\A_E\oplus \A_{B\tilde{B}}),$ 
 rigid?' and the answer is `it is rigid iff $([(\V_{ABC}\lrar \V_C)\oplus (\V_{D\tilde{B}E}\lrar \V_E)], \V_{B\tilde{B}})$ is rigid' (Theorem \ref {thm:regularmultiportrecursive}).

Next, we show that there is an analogous notion of rigidity for pairs of matroids which leads to a parallel theory to that of pairs of vector spaces.
%to rigidity of pairs of vector spaces is ana

A matroid $\M_S$  on $S,$
is a family of `independent' sets with the property that maximal
independent sets contained in any given subset of $S$ have the same cardinality.
A base of $\M_S$ is a maximal independent subset of $\M_S$ contained in $S.$
%It can be shown that all bases of a matroid have the same cardinality.
The rank $r(\M_S)$ of $\M_S$ is the cardinality of any of its bases.
The dual $\M_S^{*}$ of $\M_S$ is the matroid whose bases are complements 
of bases of $\M_S.$ We denote the matroid on $S$ for which $S$ is independent, by
 $\F_S$
and the matroid on $S$  which has no independent sets, by
$\0_S.$
\\
Let $\Msp,\Mpq$ be matroids on $S\uplus P,P\uplus Q,$ where $S,P,Q$
are pairwise disjoint. 
We define the
 {restriction}  of ${\mathcal{M}_{SP}}$ to $S$  to be
the matroid on $S,$ whose independent sets are the same as those of ${\mathcal{M}_{SP}}$ that are contained in $S$,
%${\mathcal{M}_{SP}\circ S}\equivd \{f_S:(f_S,f_P)\in \mathcal{M}_{SP}\},$
the {contraction}  of ${\mathcal{M}_{SP}}$ to $S$ to be
${\mathcal{M}_{SP}\times S}\equivd  (\mathcal{M}^*_{SP}\circ S)^*,$
a {minor} of $\M_X$ to be a matroid $(\M_X\circ X_1)\times X_2, X_2\subseteq X_1\subseteq X.$ We define the
 union $\mathcal{M}_{SP} \vee \mathcal{M}_{PQ}$ to be the matroid  on $S\uplus P\uplus Q,$ 
whose independent sets are unions of independent sets of the individual 
matroids $\mathcal{M}_{SP}$ and $ \mathcal{M}_{PQ}$
%$\mathcal{M}_{SP} + \mathcal{M}_{PQ}$ to be $ (\mathcal{M}_{SP} \oplus \0_{Q}) + (\0_{S} \oplus \mathcal{M}_{PQ}),$ the intersection
%${\mathcal{M}_{SP} \cap \mathcal{M}_{PQ}}$ of $\mathcal{M}_{SP}$, $\mathcal{M}_{PQ}$ to be
%$(\mathcal{M}_{SP} \oplus  \F_{Q}) \cap (\F_{S} \oplus \mathcal{M}_{PQ}),$
and the matched composition 
$\Msp\lrar \Mpq$  to be $(\Msp\vee\Mpq)\times (S\uplus Q).$
The implicit duality theorem for matroids (Theorem \ref{thm:idt0m})
 states that $(\Msp\lrar \Mpq)^{*}=\Msp^{*}\lrar \Mpq^{*}.$\\
We say the pair $\{\M_{AB},\M_B\}$ is rigid iff 
$$r(\M_{AB}\vee \M_{B})= r(\M_{AB})+r(\M_{B}) \ \ \mbox{and}
\ \ r((\M_{AB})^{*}\vee(\M_B)^{*})= 
r(\M_{AB}^{*})+r(\M_B^{*}).$$

There is a natural way of associating a matroid with a vector space.
The matroid $\M(\Vs)$ is defined to be the matroid on $S$ whose 
independent sets are the independent columns of $\Vs.$
It can be shown that $\{\M(\V_{AB}),\M(\V_B)\}$ is rigid when $\{\V_{AB},\V_B\}$ is rigid (Theorem \ref{thm:rigidmatroidvector}).

Analogous to the answers to the two questions on rigidity of vector spaces, 
we have that \\
when $(\M_{ABC}, \M_C)$ is given to be rigid, $(\M_{ABC},\M_B\oplus \M_C)$ is rigid iff  $(\M_{ABC}\lrar \M_C, \M_B)$ is rigid
 and that 
 when $(\M_{ABC}, \M_C), (\M_{D\tilde{B}E},\M_E)$ are given to be rigid,\\ $(\M_{ABC}\vee \M_{D\tilde{B}E}, \M_C\vee\M_E\vee \M_{B\tilde{B}})$
 is rigid 
iff $([(\M_{ABC}\lrar \M_C)\vee (\M_{D\tilde{B}E}\lrar \M_E)], \M_{B\tilde{B}})$ is rigid.
These follow from Theorem \ref{thm:derivedregularitym}.

The link between the results on vector space  pair rigidity and matroid pair rigidity
 arises when we have $\M(\V_{AB}+\V_B)=\M(\V_{AB})\vee \M(\V_B)$ and
$\M(\V_{AB}^{\perp}+\V_B^{\perp})=\M(\V_{AB}^{\perp})\vee \M(\V_B^{\perp}).$
This happens, for instance, when the device characteristic $\V_B$ of the multiport $(\V_{AB},\V_B)$ is characterized by parameters 
 which are algebraically independent over $\Q,$
 with the vector space $\V_{AB}$ being over $\Q.$
For such pairs $\{\V_{AB},\V_B\}, $ we have that 
$\{\M(\V_{AB}),\M(\V_B)\}$ is rigid iff $\{\V_{AB},\V_B\}$ is rigid 
 and that $\M(\V_{AB}\lrar \V_B) $\\$= \M(\V_{AB})\lrar \M(\V_B)$ (Theorem \ref{thm:rigidmatroidvector}).
Further when $\V_{AB},\V_B$ are minors of $\V_{AX},\V_X,$ respectively  
and satisfy certain additional conditions, we can show that rigidity of the former pair is equivalent to the rigidity of the latter pair.
There is a parallel result for matroid pairs where we can show that 
when $\M_{AB},\M_B$ are minors of $\M_{AX},\M_X,$ respectively
and satisfy certain additional conditions, the rigidity of the former pair is equivalent to the rigidity of the latter pair.
Because of these results, starting with the fact that $\{\M(\V_{AB}),\M(\V_B)\}$ is rigid iff $\{\V_{AB},\V_B\}$ is rigid
we can conclude that $\{\M(\V_{AX}),\M(\V_X)\}$ is rigid iff $\{\V_{AX},\V_X\}$ is rigid (Theorem \ref{thm:maxdistancenew2}).

We use these ideas
to show that a multiport
with independent and controlled sources and positive or negative resistors, whose parameters can be taken to be algebraically independent over $\Q,$
is rigid if \\
(a) the voltage sources, controlled as well as independent,
do not form a loop with controlling current branches and
(b) the current sources, controlled as well as independent,
do not form a cutset with controlling voltage branches (Theorem \ref{thm:purslowgraphnew}). In this case the test for rigidity is linear time on the 
size of the graph of the multiport and additionally the multiport 
behaviour $\V_{AX}\lrar \V_X$ has a hybrid representation and can be built in a convenient way.

The work reported here has its roots in the work of Purslow,
who gave necessary conditions for a circuit with resistors and
controlled sources to have a unique solution \cite{purslow},
and Iri and Tomizawa,  who introduced the approach using matroids 
exploiting algebraic
independence of parameters \cite{iritomizawa,iritomizawa2}.
Extensive work of this kind  was carried out in the 70's and 80's studying the solvability of 
systems 
(specifically using the matroid union and matroid intersection theorems
\cite{edm65a, iritomizawa,iritomizawa2,irisurvey,irireview,irifujishige,iriapplications,iriprogress,murotairi1,murotabook0, murotabook, recski78, recski79}).
A summary of the work relevant to uniqueness of solution for linear networks according to this approach is available
in \cite{recski89}. In our case, the building blocks are rigid multiports and we need rigidity to be preserved even after connection
of these building blocks. This requires some additional techniques,
 such as the use of the implicit duality theorem for vector spaces
and matroids (Theorems \ref{thm:idt0}, \ref{thm:idt0m}).

We now present a brief outline of the paper.

Section \ref{sec:Preliminaries} is on vector space and network preliminaries.
Among other things, it contains an extended definition of sum and intersection
of (row) vector spaces with different column sets, the definition 
of matched and skewed composition and basic
definitions on multiports.

Section \ref{sec:iitidt} is on two basic theorems of `implicit linear algebra'
\cite{narayanan2016} - implicit inversion theorem, which is used to discuss
 `port transformation' (\ref{subsec:porttransformation}) and implicit duality theorem which is used extensively in the paper.

Section \ref{sec:rigid} is on rigid multiports which are studied in terms of  
the properties of  rigid pairs of affine spaces. 
Theorem \ref{thm:regularrecursive} contains the basic properties. The main theorem 
from the point of view of applications is Theorem \ref{lem:derivedregularity}, which states that $(\V_{WTV}, \V_T\oplus \V_V)$ is rigid iff 
$(\V_{WTV}, \V_V)$ and $(\V_{WTV}\lrar \V_V,\V_T)$ are rigid.
Subsection \ref{subsec:rigidfamiliesaffine} generalizes these ideas to 
 families of affine spaces on sets which have the property that no element 
 belongs to more than two of the underlying sets. Theorem \ref{thm:associativerigidrecursive} is 
the general version of Theorem \ref{lem:derivedregularity}.
It reduces the rigidity of a family of affine spaces of the above kind, recursively, to 
rigidity of two complementary subfamilies and an interface pair of affine spaces.
Theorem \ref{thm:associativerigidrecursiven} is a version which is in terms of 
a partition of the original family and a corresponding interface 
set of affine spaces.
%gives 
%conditions for rigidity of $(\V_{WTV}, \V_T\oplus \V_V)$ 
%given the rigidity of $(\V_{WTV}, \V_V).$

Section \ref{sec:matroidrigidity} is on matroidal rigidity. It begins with 
a subsection on matroid preliminaries which is deliberately written in a 
manner parallel to vector space preliminaries in Section \ref{sec:Preliminaries}. This is followed by a subsection on rigidity of matroid pairs, which 
parallels the development of rigid pairs of affine spaces of Section \ref{sec:rigid}. The main theorem
from the point of view of applications is Theorem \ref{thm:derivedregularitym},
which states that $(\M_{WTV}, \M_T\oplus \M_V)$ is rigid iff 
$(\M_{WTV}, \M_V)$ and $(\M_{WTV}\lrar \M_V,\M_T)$ are rigid.
The proof of this theorem is a line by line `matroidal' translation of that of the corresponding vector space result (Theorem \ref{lem:derivedregularity}). It is relegated to 
 \ref{subsec:rigidpairsmp}.
Subsection \ref{subsec:rigidfamiliesmatroid} generalizes these ideas to
 families of affine spaces on sets which have the property that no element
 belongs to more than two of the underlying sets. Theorem \ref{thm:associativerigidrecursivem} is
the general version of Theorem \ref{thm:derivedregularitym}.
It reduces the rigidity of a family of matroids of the above kind, recursively, to
rigidity of subfamilies and an interface pair of matroids.
Theorem \ref{thm:associativerigidrecursivenm} is a version which is in terms of
a partition of the original family and a corresponding interface
set of matroids.

Section \ref{subsec:connect} is on connecting rigid multiports to yield a 
rigid multiport. The main theme is on matroidal tests for rigidity of multiports. This section uses the ideas in Sections \ref{sec:rigid} and \ref{sec:matroidrigidity}
to give a simple topological sufficient condition for rigidity of multiports containing 
non zero resistors and independent and controlled sources (Theorem \ref{thm:purslowgraphnew}). The proof requires the development of a theory of the relation between 
multiport and matroid pair rigidity. This development is given in  \ref{sec:matroidalrigidity}. 
Theorem \ref{thm:purslow2} is a generalization of Theorem \ref{thm:purslowgraphnew} that includes gyrators and ideal transformers in the set of devices.
This requires the testing of independence of certain sets of columns of a Dirac
structure. The proof is given in \ref{sec:purslow2}.
The Dirac structure itself arises as the port behaviour of a multiport 
which has  only devices with a Dirac structure. 
It is more convenient to work with a corresponding multiport with the same Dirac devices but with 
minimal ports. This requires the idea of port minimization which is 
discussed in \ref{subsec:porttransformation}.
%which contains a detailed discussion of 
%matroidal conditions for rigidity of multiports.

Section \ref{sec:conclusions} is on conclusions.

%The link between vector space pair rigidity and matroid pair rigidity 
%is the fact that if $\M(\V_{AB}+\V_B)=\M(\V_{AB})\vee \M(\V_B)$ and 
%$\M(\V_{AB}^{\perp}+\V_B^{\perp})=\M(\V_{AB}^{\perp})\vee \M(\V_B^{\perp})$ 
%then $\{\M(\V_{AB}),\M(\V_B)\}$ is rigid iff $\{\V_{AB},\V_B\}$ is rigid (Theorem \ref{thm:rigidmatroidvector}).
%is provided in 
%Theorem \ref{thm:rigidmatroidvector}. 

\section{Preliminaries}
\label{sec:Preliminaries}
The preliminary results and the notation used are from \cite{HNarayanan1997}.

A \nw{vector} $\mnw{f}$ on a finite set $X$ over $\mathbb{F}$ is a function $f:X\rightarrow \mathbb{F}$ where $\mathbb{F}$ is a field. 
%In this paper, we work only  with the rational field $\mathbb{Q}.$

%The \nw{length} of a vector $x$ is the Euclidean norm $||x||$ of $x.$

%

%The size of a set $X$ is denoted by $\mnw{|X|}.$
%and would often be written as $(f_X,f_Y)$ during operations dealing with such vectors. 
The {\bf sets} on which vectors are defined are  always {\bf finite}. When a vector $x$ figures in an equation, we use the 
convention that $x$ denotes a column vector and $x^T$ denotes a row vector such as
in `$Ax=b,x^TA=b^T$'. Let $f_Y$ be a vector on $Y$ and let $X \subseteq Y$. The \textbf{restriction $f_Y|_X$} of $f_Y$ to $X$ is defined as follows:
%\begin{align*}
$f_Y|_X \equivd g_X, \textrm{ where } g_X(e) = f_Y(e), e\in X.$

%\end{align*}

When $f$ is on $X$ over $\mathbb{F}$, $\lambda \in \mathbb{F},$   the \nw{scalar multiplication} $\mnw{\lambda f}$ of $f$ is on $X$ and is defined by $(\lambda f)(e) \equivd \lambda [f(e)]$, $e\in X$. When $f$ is on $X$ and $g$ on $Y$ and both are over $\mathbb{F}$, we define $\mnw{f+g}$ on $X\cup Y$ by \\
%\begin{align*}
% (f+g)(e) &\equiv \left\{ \begin{matrix}
%                          f(e) + g(e),& e\in X \cap Y\\
%			    f(e),& e\in X \setminus Y\\
%			    g(e),& e\in Y \setminus X.
%                         \end{matrix}
%\right.
%\end{align*}
$(f+g)(e)\equivd f(e) + g(e),e\in X \cap Y,\ (f+g)(e)\equivd  f(e), e\in X \setminus Y,
\ (f+g)(e)\equivd g(e), e\in Y \setminus X.
$
(For ease in readability, we will henceforth use $X-Y$ in place of $X \setminus Y.$)

%The size of a set $X$ is denoted by $\mnw{|X|}.$
When $X, Y, $ are disjoint,  $f_X+g_Y$ is written as  $\mnw{(f_X, g_Y)}.$ When $f,g$ are on $X$ over $\mathbb{F},$ the \textbf{dot product} $\langle f, g \rangle$ of $f$ and $g$ is defined by 
%\begin{align*}
$ \langle f,g \rangle \equivd \sum_{e\in X} f(e)g(e).$
%\end{align*}
When $X$, $Y$ are disjoint, $\mnw{X\uplus Y}$ denotes the disjoint
union of $X$ and $Y.$ A vector $f_{X\uplus  Y}$ on $X\uplus Y$ is  written as $\mnw{f_{XY}}.$

We say $f$, $g$ are \textbf{orthogonal} (orthogonal) iff $\langle f,g \rangle$ is zero.
%We say $f$, $g$ are \textbf{$\mathbb{Z}$-orthogonal}  iff $\langle f,g \rangle$ is an integer.

An \nw{arbitrary  collection} of vectors on $X$ 
%with $0_X$ as a member 
is denoted by $\mnw{\mathcal{K}_X}$. 
When $X$, $Y$ are disjoint we usually write $\mathcal{K}_{XY}$ in place of $\mathcal{K}_{X\uplus Y}$.
We write $\K_{XY}$ as $\K_X\oplus \K_Y$ iff
$\K_{XY}\equivd\{f_{XY}:f_{XY}=(f_X,g_Y), f_X\in \K_X, g_Y\in \K_Y\}.$
We refer to $\K_X\oplus \K_Y$ as the \nw{direct sum} of $\K_X, \K_Y.$ 

A collection $\K_X$ is a \nw{vector space} on $X$ iff it is closed under 
addition and scalar multiplication. 
The notation $\mnw{\V_X}$ always denotes
a vector space on $X.$
For any collection $\K_X,$  $\mnw{span(\K_X)}$ is the vector space of all
linear combinations of vectors in it.
We say $\K_X$ is an \nw{affine space} on $X,$ iff it can be expressed as
$x_X+\V_X,$ where $x_X$ is a vector and $\V_X,$ a vector space on $X.$
The latter is unique for the affine space $\K_X$ and is said to be its \nw{vector space translate}.
The notation $\mnw{\A_X}$ always denotes
an \nw{affine space} on $X.$
When $\A_X$ and $\A'_X$ have the same vector space translate, we say that 
they are \nw{translates} of each other.

For a vector space  $\V_X,$ since we take $X$ to be finite,
any maximal independent subset of $\V_X$ has size less than or equal to $|X|$ and this 
size can be shown
to be unique. A maximal independent subset of a vector
space $\V_X$ is called its \nw{basis} and its  size 
is called the  {\bf dimension} or \nw{rank} of $\V_X$ and denoted by ${\mnw{dim}(\V_X)}$
 or by ${\mnw{r}(\V_X)}.$
For any collection of vectors $\K_X,$
the rank $\mnw{r}(\K_X)$
is defined to be $dim(span(\K_X)).$
The collection of all linear combinations of the rows of a matrix $A$ is a vector space 
that is denoted by $row(A).$

For any collection of vectors
$\mathcal{K}_X,$   the collection $\mnw{\mathcal{K}_X^{\perp}}$ is defined by
%\begin{align*}
$ {\mathcal{K}_X^{\perp}} \equivd \{ g_X: \langle f_X, g_X \rangle =0\},$
%f_X\in \mathcal{K}_X 
%\}.
%\end{align*}
It is clear that $\mathcal{K}_X^{\perp}$ is a vector space for 
any $\mathcal{K}_X.$ When $\mathcal{K}_X$ is a vector space $\V_X,$
 and the underlying set $X$ is finite, it can be shown that $({\mathcal{V}_X^{\perp}})^{\perp}= \mathcal{V}_X$ 
and  $\mathcal{V}_X,{\mathcal{V}_X^{\perp}}$ are said to be \nw{complementary orthogonal}. 
The symbol $0_X$ refers to the \nw{zero vector} on $X$ and $\mnw{0_X}$  refers to the \nw{zero vector space} on $X.$ The symbol $\mnw{\F_X}$  refers  to the collection of all vectors on $X$ over the field in question.
It is easily seen, when $X,Y$ are disjoint, and $\K_X, \K_Y$           
contain zero vectors, that $(\K_X\oplus \K_Y)^{\perp}=
\K_X^{\perp}\oplus\K_Y^{\perp}.$

A matrix of full row rank, whose rows generate a vector space $\V_X,$
is called a \nw{representative matrix} for $\V_X.$
We will refer to the index set $X$ as the \nw{set of columns} of $\V_X.$

If $B^1_X,B^2_X$ are representative matrices of $\V_X,$ it is clear
that the solution spaces of $B^1_Xy=0$ and $B^2_Xy=0$ are identical, being equal to $\V^{\perp}_X.$ A consequence is that a subset of columns of $B^1_X$ is independent  iff it is independent in $B^2_X.$
We say, in such a case, that the corresponding {columns of $\V_X$ are independent}.
A \nw{column base of $\V_X$} is a maximally independent set of columns of $\V_X.$  A \nw{column cobase of $\V_X$} is the complement $X-T$ of a column base $T$
 of  $\V_X.$

A representative matrix which can be put in the form $(I\ |\ K)$ after column
permutation, is called a \nw{standard representative matrix}.
It is clear that every vector space has a  standard representative matrix.
If $(I\ |\ K)$ is  a standard representative matrix of $\V_X,$
it is easy to see  that $(-K^T|I)$ is a standard representative matrix of $\V^{\perp}_X.$
Therefore we have the following result.
\begin{theorem}
\label{thm:perperp}
Let $\V_X$ be a vector space on $X.$ Then
\begin{enumerate}
\item 
$r(\V_X)+r(\V^{\perp}_X)=|X|$ and $((\V_X)^{\perp})^{\perp}=\V_X;$
\item $T$ is a column base of $\V_X$ iff $T$ is a column cobase of $\V^{\perp}_X.$
\end{enumerate}
\end{theorem}
\begin{remark}
\label{rem:realtocomplex}
When the field $\mathbb{F}\equivd \mathbb{C},$ it is usual to interpret
the dot product $\langle f_X, g_X \rangle$ of $f_X,g_X$ to be
the inner product $\Sigma f(e)\overline{g(e)}, e\in X,$
where $\overline{g(e)}$ is the complex conjugate of $g(e).$
In place of $\V^{\perp}_X$ we must use $\V^{*}_X,$
where $\V^{*}_X\equivd \{ g_X: \langle f_X, g_X \rangle =0\},$
 $\langle f_X, g_X \rangle$ being taken to be as above.
The definition of adjoint, which is introduced later, must be in terms of
$\V^{*}_X$ instead of in terms of $\V^{\perp}_X.$
\\
In this case, in place of the transpose of a matrix $Z$ we use $Z^*,$ the conjugate
transpose of $Z.$
\\
The above interpretation of dot product is {\it essential} if we wish to
define the notion of power for networks defined over the complex field.

The
Implicit Duality Theorem (Theorem \ref{thm:idt0}), would go through
with either definition of dot product and the corresponding
definition of orthogonality.
\end{remark}

The collection
$\{ (f_{X},\lambda f_Y) : (f_{X},f_Y)\in \mathcal{K}_{XY} \}$
is denoted by
$ \mnw{\mathcal{K}_{X(\lambda Y)}}.
$
When $\lambda = -1$ we would write $ {\mathcal{K}_{X(\lambda Y)}}$  more simply as $\mnw{\mathcal{K}_{X(-Y)}}.$
Observe that $(\mathcal{K}_{X(-Y)})_{X(-Y)}=\mathcal{K}_{XY}.$
%$\Ipp $ is defined to be  the vector space $ \{ (f_P,f_P'):f_P\in \mathscr{F}_P\}.$ 
%\end{align*}
%\subsection{Building copies}

We say sets $X$, $X'$ are \nw{copies of each other} iff they are disjoint and there is a bijection, usually clear from the context, mapping  $e\in X$ to $e'\in X'$.
When $X,X'$ are copies of each other, the vectors $f_X$ and $f_{X'}$ are said to be copies of each other with  $f_{X'}(e') \equivd  f_X(e), e \in X.$ 
The copy $(\K_X)_{X'}$ of $\K_X$ is defined by
 $(\K_X)_{X'}\equivd\{f_{X'}:f_X\in \K_X\}.$ 
When $X$ and $X'$ are copies of each other, the notation for interchanging the positions of variables with index sets $X$ and $X'$ in a collection $\mathcal{K}_{XX'Y}$ is given by $\mnw{(\mathcal{K}_{XX'Y})_{X'XY}}$, that is\\
$(\mathcal{K}_{XX'Y})_{X'XY}
 \equivd \{(g_X,f_{X'},h_Y)\ :\ (f_X,g_{X'},h_Y) \in \mathcal{K}_{XX'Y},\ g_X\textrm{ being  copy of }g_{X'},\ f_{X'}\textrm{ being  copy of }f_X  \}.$
An affine space $\mathcal{A}_{XX'}$ is said to be {\bf proper}
iff the rank of its vector space translate is $|X|=|X'|.$

%\section{Basic operations}
%\label{sec:basic}
%The basic operations we use in this paper are as follows:

\subsection{Restriction and contraction}
%We remind the reader that
%$\mathcal{K}_X$ denotes an arbitrary collection of vectors on $X$ with a zero vector as a member.

The \nw{restriction}  of $\mnw{\mathcal{K}_{SP}}$ to $S$ is defined by
$\mnw{\mathcal{K}_{SP}\circ S}\equivd \{f_S:(f_S,f_P)\in \mathcal{K}_{SP}\}.$
The \nw{contraction}  of $\mnw{\mathcal{K}_{SP}}$ to $S$ is defined by
$\mnw{\mathcal{K}_{SP}\times S}\equivd \{f_S:(f_S,0_P)\in \mathcal{K}_{SP}\}.$
The sets on which we perform the contraction operation would always 
have the zero vector as a member so that the resulting set would be nonvoid.
It is clear that restriction and contraction of vector spaces are also
vector spaces.
Here again, $\mnw{\mathcal{K}_{SPZ}\circ SP}$, $\mnw{\mathcal{K}_{SPZ} \times SP}$, respectively
when $S,P,Z,$ are pairwise disjoint,  denote ${\mathcal{K}_{S\uplus P\uplus Z}\circ (S\uplus P)}$, ${\mathcal{K}_{S\uplus P\uplus Z} \times (S \uplus P)}.$

We usually write $(\K_S\circ A)\times B , (\K_S\times A)\circ B $ respectively, more simply as $\K_S\circ A\times B, \K_S\times A\circ B .$ 
We refer to $\K_S\circ A\times B, B\subseteq A \subseteq S, $ as a \nw{minor} of  $\K_S.$

We have the following set of basic results  regarding minors and complementary orthogonal vector spaces.
\begin{theorem}
\label{thm:dotcrossidentity}
Let $S\cap P=\emptyset, T_2\subseteq T_1\subseteq S,$
$\V_S$ be a vector space  on $S,$ $\Vsp$ be a vector space  on $S\uplus P.$
We then have the following.
\begin{enumerate}
\item $(\V_S \times T_1 ) \cdot T_2 = (\V_S \cdot (S - (T_1 - T_2))\times T_2.$
\item $r(\Vsp)=r(\Vsp\circ S)+r(\Vsp\times P).$
\item $\V_{SP}^{\perp}\circ P= (\V_{SP}\times P)^{\perp}.$
\item $\V_{SP}^{\perp}\times S= (\V_{SP}\circ S)^{\perp}.$
\item 
$P_1\subseteq P$ is a column base of $\Vsp\times P$ iff 
there exists $S_1\subseteq S$ such that  $S_1$ is a column base of $\Vsp\circ S$
and $S_1\cup P_1$ is a column base of $\Vsp.$
\item If $P_1\subseteq P$ is a column base of $\Vsp\times P,$
then there exists a column base $P_2$ of $\Vsp\circ P,$
that contains it.
\item 
Let $A,B\subseteq S\uplus P, A\cap B=\emptyset,$ and let there exist a column base $D_1$  of $\Vsp$ that contains $A$ and a column cobase  $E_2$  of $\Vsp$ that
contains $B.$ Then there exists a column base $D$ of $\Vsp$  that contains $A$ but does not intersect $B.$
\end{enumerate}
\end{theorem}

\subsection{Graphs}
%A \nw{directed graph} $\G$ is a triplet $(V(\G),E(\G),f^d)$, where  the sets $V(\G)$, $E(\G),$ define the \nw{vertices} (or \nw{nodes}),   \nw{edges}, respectively of the graph, $f^d$ is the \nw{incidence function} which associates  with each edge an ``ordered pair'' of vertices,
%called respectively its $+$ve and $-$ve endpoint and usually represented 
%by an arrow from positive to negative end point.
%We will refer to a {\bf directed graph in brief as a graph}. 
For a \nw{directed graph $\G$} (graph in brief), the vertex and edge sets are denoted by $V(\G),E(\G),$ 
respectively.
%\begin{comment}
An edge is usually diagrammatically represented with an arrow going from its 
positive to its negative endpoint. The \nw{degree} of a node is the number 
of edges incident at it, with edges with single endpoints (\nw{self loops}) counted twice. An \nw{isolated} vertex has degree zero. An \nw{undirected path} between vertices $v_0,v_k$ of $\G$ is a 
sequence $(v_0,e_0,v_1, e_1, \cdots , e_{k-1},v_k),$ where 
$e_i, i=0, \cdots ,k-1, $ is incident at $v_i.$
A graph is said to be \nw{connected}, if there exists an undirected path between every pair of nodes. Otherwise it is said to be \nw{disconnected}.
A subgraph $\G_1$  of a graph $\G,$  is a graph  where $V(\G_1)\subseteq V(\G),E(\G_1)\subseteq E(\G),$ and edges of $\G_1$ have the same positive and negative endpoints as in $\G.$ 
%The graph $\mathcal{G}'$ in the example is a connected graph. 
A disconnected graph has  \nw{connected components} which are individually connected  with no edges between the components.

A \nw{loop} is a connected subgraph with the degree of each node equal to $2.$
An \nw{orientation} of a loop is a sequence of all its edges such that each
edge has a common end point with the edge succeeding it, the first edge
being treated as succeeding the last. Two orientations in which the succeeding edge to a given
edge agree are treated as the same so that there are only two possible 
orientations for a loop. %It is clear that the orientation of the edges 
%also assigns an orientation of the vertices. 
The relative orientation of an edge 
with respect to that of the loop is positive, if the orientation 
of the loop agrees with the direction (positive node to negative node)
of the edge and negative if opposite.
%\end{comment}

A \nw{tree} subgraph of a graph is a connected sub-graph of the original graph with no loops. The set of edges of a tree subgraph is called a \nw{tree} and its edges are called \nw{branches}. A \nw{spanning tree} is a maximal tree with respect to the edges of a connected graph.
A \nw{cotree} of a graph is the (edge set) complement of a spanning tree of the connected graph.
%A \nw{forest} of a disconnected graph is a disjoint union of the spanning trees of its connected components. The complement of a forest is called \nw{coforest}.
%For simplicity, {\bf we refer to a forest (coforest) as a tree (cotree)}
% even when it is not clear that the graph is connected.

A \nw{cutset} is a minimal 
 subset of edges which when deleted from the graph increases the count of connected components by one.
Deletion of the edges of a cutset breaks exactly one of the components 
of the graph, say $\G_1,$ into two, say $\G_{11},\G_{12}.$   A cutset can be oriented in one 
of two ways corresponding to the ordered pair $(\G_{11},\G_{12})$ or the ordered pair $(\G_{12},\G_{11}).$
The relative orientation of an edge 
with respect to that of the cutset is positive if the orientation, say $(\G_{11},\G_{12}),$
of the cutset agrees with the direction (positive node to negative node)
of the edge and negative if opposite.

Let $\mathcal{G}$ be a graph with $S\equivd E(\mathcal{G})$  and let $T\subseteq S.$ Then 
$\mnw{\mathcal{G} sub (S-T)}$ denotes the graph obtained by removing the edges $T$ from $\mathcal{G}.$ 
This operation is referred to also as {\bf deletion} or
open circuiting of the edges $T.$
The graph $\mnw{\mathcal{G} \circ (S-T)}$ is obtained by removing  the isolated vertices from $\mathcal{G} sub (S-T).$
The graph $\mnw{\mathcal{G} \times (S-T)}$ is obtained by removing the edges $T$ from $\mathcal{G}$ and fusing the end vertices of the removed edges. If any isolated vertices (i.e., vertices not incident on any edges) result, they are deleted. 
Equivalently, one may first build $\mathcal{G} sub T,$
and treat each of its connected components, including the isolated nodes 
as a `supernode' of another graph with edge set $S-T.$
If any of the supernodes is isolated, it is removed.
This would result in ${\mathcal{G} \times (S-T)}.$
This operation is referred to also as {\bf contraction} or 
short circuiting of the edges $T.$
We refer to $(\G\times T)\circ W, (\G\circ T)\times W$ respectively, more simply by $\G\times T\circ W, \G\circ T\times W. $
\\
If disjoint edge sets $A,B$ 
are respectively deleted and contracted, 
the order in which these operations are performed 
can be seen to be irrelevant.
Therefore, $\G\circ (S-A)\times (S-(A\uplus B))=\G\times (S-B)\circ (S-(A\uplus B)).$
(Note that $\times, \circ $ are also used to denote vector space operations. However, the context would make clear
whether the objects involved are graphs or vector spaces.)

\nw{Kirchhoff's Voltage Law (KVL)} for a graph states that the sum of the signed voltages of
edges 
around an oriented loop is zero - the sign of the voltage of an edge 
being positive if the edge orientation agrees with the orientation of the loop
and negative if it opposes.\\
\nw{Kirchhoff's Current Law (KCL)} for a graph states that the sum of the signed currents leaving 
a node is zero, the sign of the current of an edge being positive if 
 its positive endpoint is the node in question.\\
We refer to the space of vectors $v_{S},$ which satisfy Kirchhoff's Voltage Law (KVL) of the graph $\mathcal{G},$
by $\mnw{\V^v(\mathcal{G})}$ and to the space of vectors $i_{S},$ which satisfy Kirchhoff's Current Law (KCL) of the graph $\mathcal{G},$
by $\mnw{\V^i(\mathcal{G})}.$
{\it These vector spaces will, unless otherwise stated, be taken as  
over $\Re.$}

The following are useful results on vector spaces associated with graphs.
\begin{theorem}
\label{thm:tellegen}
{\bf Tellegen's Theorem} \cite{tellegen},\cite{penfield}  Let $\G$ be a graph on edge set $S.$ Then $\V^i(\mathcal{G})= (\V^v(\mathcal{G})^{\perp}).$
\end{theorem}
\begin{lemma}
\label{lem:minorgraphvectorspace}
\cite{tutte}
Let $\G$ be a graph on edge set $S.$
Let $W\subseteq T\subseteq S.$
\begin{enumerate}
\item $ \V^v(\mathcal{G}\circ T)= (\V^v(\mathcal{G}))\circ T, \ \ \  \V^v(\mathcal{G}\times T)= (\V^v(\mathcal{G}))\times T,\ \ \V^v(\mathcal{G}\circ T\times W)= (\V^v(\mathcal{G}))\circ T \times W;$
\item $ \V^i(\mathcal{G}\circ T)= (\V^i(\mathcal{G}))\times T, \ \ \  \V^i(\mathcal{G}\times T)= (\V^i(\mathcal{G}))\circ T,
\V^i(\mathcal{G}\times T\circ W)= (\V^i(\mathcal{G}))\circ T \times W.$
\item A subset of columns is a column base of $\V^v(\G)$ iff it is a tree of
$\G.$
\item A subset of columns is a column base of $\V^i(\G)$ iff it is a cotree of
$\G.$

\end{enumerate}
\end{lemma}
(Proof also available at \cite{HNarayanan1997}).
\subsection{Sum and Intersection}
Let $\mathcal{K}_{SP}$, $\mathcal{K}_{PQ}$ be collections of vectors on sets $S\uplus P,$ $P\uplus Q,$ respectively, where $S,P,Q,$ are pairwise disjoint. The \nw{sum} $\mnw{\mathcal{K}_{SP}+\mathcal{K}_{PQ}}$ of $\mathcal{K}_{SP}$, $\mathcal{K}_{PQ}$ is defined over $S\uplus P\uplus Q,$ as follows:\\
 $\mathcal{K}_{SP} + \mathcal{K}_{PQ} \equivd  \{  (f_S,f_P,0_{Q}) + (0_{S},g_P,g_Q), \textrm{ where } (f_S,f_P)\in \mathcal{K}_{SP}, (g_P,g_Q)\in \mathcal{K}_{PQ} \}.$\\
%When $S$, $Q,$ are disjoint, $\mathcal{K}_S + \mathcal{K}_{Q}$ is usually written in this paper as $\mnw{\mathcal{K}_S \oplus \mathcal{K}_Q}$ and is called the \nw{direct sum}.
Thus,
$\mathcal{K}_{SP} + \mathcal{K}_{PQ} \equivd (\mathcal{K}_{SP} \oplus \0_{Q}) + (\0_{S} \oplus \mathcal{K}_{PQ}).$\\
%\subsubsection{Intersection}
The \nw{intersection} $\mnw{\mathcal{K}_{SP} \cap \mathcal{K}_{PQ}}$ of $\mathcal{K}_{SP}$, $\mathcal{K}_{PQ}$ is defined over $S\uplus P\uplus Q,$ where $S,P,Q,$ are pairwise disjoint, as follows:
$\mathcal{K}_{SP} \cap \mathcal{K}_{PQ} \equivd \{ f_{SPQ} : f_{S P Q} = (f_S,h_P,g_{Q}),$
%f_{(S\cup Y)} = (y_{(S\setminus Y)},f_Y), 
 $\textrm{ where } (f_S,h_P)\in\mathcal{K}_{SP}, (h_P,g_Q)\in\mathcal{K}_{PQ}.%  x_{Q}\in \F_{Q},y_{S} \in  \F_{ S}
\}.$\\
Thus,
$\mathcal{K}_{SP} \cap \mathcal{K}_{PQ}\equivd (\mathcal{K}_{SP} \oplus  \F_{Q}) \cap (\F_{S} \oplus \mathcal{K}_{PQ}).$\\

It is immediate from the definition of the operations that sum and intersection of
vector spaces remain vector spaces.

The following set of identities is useful.
\begin{theorem}
\label{thm:sumintersection}
Let $\V^1_A, \V^2_B, \V_S,\V'_S $ be vector spaces. Then
\begin{enumerate}
\item $r(\V_S)+r(\V'_S)=r(\V_S+\V'_S)+r(\V_S\cap \V'_S);$
\item $r(\V_A)+r(\V_B)=r(\V_A+\V_B)+r[(\V_A\times (A\cap B))\cap (\V_B\times (A\cap B))];$

\item $(\V^1_A+\V^2_B)^{\perp}=(\V^1_A)^{\perp}\cap (\V^2_B)^{\perp};$
\item $(\V^1_A\cap \V^2_B)^{\perp}=(\V^1_A)^{\perp}+ (\V^2_B)^{\perp}.$
\item (a) $(\V^1_A+\V^2_B)\circ X= \V^1_A\circ X + \V^2_B\circ X ,X\subseteq A\cap B;$\\
(b) $(\V^1_A\cap\V^2_B)\times X= \V^1_A\times X \cap \V^2_B\times X ,X\subseteq A\cap B.$
\end{enumerate}
\end{theorem}

\subsection{Networks and multiports}
\label{subsec:networks}
A (static) {\bf electrical network $\mathcal{N},$} or a `network' in short, is a pair $(\mathcal{G},\mathcal{K}),$ where $\mathcal{G}$
%\equivd (V(\G),E(\G),f^d)$ i
is a directed graph with vertex set $V(\G)$ and edge set $E(\G)$
and $\mathcal{K}, $ called the \nw{device characteristic} of the network, is a collection of pairs of vectors  $(v_{S'},i_{S"}),S\equivd E(\G)$ where  $S',S"$  are disjoint copies of $S,\ $  $v_{S'},i_{S"}$ are real or complex vectors on the edge set of the graph.
%We call $\mathcal{D},$  the `device characteristic' of the network.
In this paper, for the most part, we deal with affine device characteristic
and with real vectors, unless otherwise stated. When the device characteristic $\K_{S'S"}$ is affine,
we say the network is \nw{linear}. If $\V_{S'S"}$ is the vector space
translate of $\K_{S'S"},$ we say that $\K_{S'S"}$ is the \nw{source accompanied}
 form of $\V_{S'S"}.$ An affine space $\A_{S'S"}$ is said to be \nw{proper}
iff its vector space
translate $\V_{S'S"}$ has dimension $|S'|=|S"|.$

A {\bf solution} of $\mathcal{N}\equivd (\G,\K)$ on graph $\G$
is a pair
 $(v_{S'},i_{S"}),S\equivd E(\G)$ satisfying\\
$v_{S'}\in \V^v(\mathcal{G}),\ \ i_{S"} \in \V^i(\mathcal{G})$
  (KVL,KCL)  and $(v_{S'},i_{S"})\in \mathcal{K}.$
The KVL,KCL conditions are also called \nw{topological} constraints.
Let   $S',S"$ be disjoint copies of $S,$
let $\V_{S'}\equivd \V^v(\mathcal{G}).$ Then by Theorem \ref{thm:tellegen}, we have $(\V^{\perp}_{S'})_{S"}= \V^i(\mathcal{G}).$ Let $
\ \K_{S'S"}$ be the device characteristic of $\N.$
The set of solutions of  $\mathcal{N}$ may be written, using the extended
definition of intersection as
$$\V_{S'}\cap (\V^{\perp}_{S'})_{S"}\cap \K_{S'S"}=[\V_{S'}\oplus (\V^{\perp}_{S'})_{S"}]\cap \K_{S'S"}.$$
This has the form `[Solution set of topological constraints]$\cap$ [Device characteristic]'.\\

A \nw{multiport $\mathcal{N}_P$ on graph $\G_{SP}$ and device characteristic $\K_{S'S"}$} is an ordered pair $ (\G_{SP},\K_{S'S"}),$ i.e.,  a network with some subset $P$ of its
edges which have no device characteristic constraints, specified as \nw{ports}. 
%These latter edges are referred to as  as \nw{ports} of the multiport.
The multiport is said to be \nw{linear} iff its device characteristic is affine.
Let $\N_P$ be on graph $\G_{SP}$ with device characteristic $\K.$
Let $\V_{S'P'}\equivd (\V^v(\G_{SP}))_{S'P'},$ so that $ (\V^{\perp}_{S'P'})_{S"P"}= (\V^i(\G_{SP}))_{S"P"},$ and let $\K_{S'S"}, $ be the affine device characteristic 
on the edge set $S.$ The device characteristic of $\mathcal{N}_P,$
when the latter is regarded as a network, 
would be $\K\equivd \K_{S'S"} \oplus \F_{P'P"}.$
However, we would refer to $\K_{S'S"}$ as the device characteristic 
of the multiport $\N_P.$
\\
The set of solutions of  $\mathcal{N}_P$ may be writen, using the extended
definition of intersection as
$$\V_{S'P'}\cap (\V^{\perp}_{S'P'})_{S"P"}\cap \K_{S'S"}=[\V_{S'P'}\oplus (\V^{\perp}_{S'P'})_{S"P"}]\cap \K_{S'S"}.$$
We say the multiport is \nw{consistent} iff its set of solutions 
is nonvoid.

The multiport $\mathcal{N}_P$ would impose a relationship 
%$([\V_{S'P'}\oplus (\V^{\perp}_{S'P'})_{S"P"}]\cap \K_{S'S"})\circ P'P",$ 
between
$v_{P'},i_{P"}.$ We call this the  \nw{multiport behaviour} (port behaviour for short) ${\K}_{P'P"}$ at $P,$  of $\N_P,$ defined 
by 
${\K}_{P'P"}\equivd [(\V_{S'P'}\oplus (\V^{\perp}_{S'P'})_{S"P"})\cap \K_{S'S"}]\circ P'P".$
%$$\K_{P'P"}\equiv ((\V_{S'P'}\cap (\V^{\perp}_{S'P'})_{S"P"}\cap \K_{S'S"})\circ P'P")_{P'-P"}.$$
%$\breve{\K}_{P'P"}\equivd [([\V_{S'P'}\oplus (\V^{\perp}_{S'P'})_{S"P"}]\cap \K_{S'S"})\circ P'P")]_{P'(-P")}= ([\V_{S'P'}\oplus (\V^{\perp}_{S'P'})_{S"(-P")}]\cap \K_{S'S"})\circ P'P".$ 
%$= ((\V_{S'P'}\oplus (\V^{\perp}_{S'P'})_{S"(-P")})\lrar  \K_{S'S"}.$
When the device characteristic of $\mathcal{N}_P$ is affine, its {multiport behaviour} ${\K}_{P'P"}$ at $P$ would be 
affine if it were not void.
\begin{remark}
The usual definition of multiport behaviour would involve a change of sign
of the current variables. In the interest of readability, we have avoided this, since it  is not relevant for this paper.
\end{remark}
Let the multiports $\N_{RP},{\N}_{\tilde{P}Q}$ be on graphs $\G_{RSP},\G_{\tilde{P}MQ}$ respectively, with the primed and double primed sets obtained from
$R,S,P,\tilde{P},M,Q,$ being pairwise disjoint,
%$S',T',P',\tilde{P}',Q',W',S",T",P",\tilde{P}",Q",W"$ pairwise disjoint,
and let them have device characteristics $\K^{S},{\K}^{M}$ respectively.
Let $\K^{P\tilde{P}}$ denote a collection of vectors $\K^{P\tilde{P}}_{P'\tilde{P}'P"\tilde{P}"}.$
The multiport $\mnw{[\N_{RP}\oplus {\N}_{\tilde{P}Q}]\cap \K^{P\tilde{P}}},$
with ports $R\uplus Q$ obtained by \nw{connecting $\N_{RP},{\N}_{\tilde{P}Q}$
through $\K^{P\tilde{P}}$},
is on graph $\G_{RSP}\oplus\G_{\tilde{P}MQ}$ with device characteristic
$\K^{S}\oplus {\K}^{M}\oplus  \K^{P\tilde{P}}.$ 

\subsection{Matched and Skewed Composition}
\label{sec:matched}
In this section we introduce an operation between collections of vectors
motivated by the connection of multiport behaviours across ports.

Let $\Ksp,\Kpq,$ be collections of vectors respectively on $S\uplus P,P\uplus Q,$ with $S,P,Q,$ being pairwise disjoint.
%Further, let $\0_{SP}\in \Ksp, \0_{PQ}\in \K_{PQ}.$

The \nw{matched composition} $\mnw{\mathcal{K}_{SP} \leftrightarrow \mathcal{K}_{PQ}}$ is on $S\uplus Q$ and is defined as follows:
\begin{align*}
 \mathcal{K}_{SP} \leftrightarrow \mathcal{K}_{PQ} 
  &\equivd \{
                 (f_S,g_Q): (f_S,h_P)\in \Ksp, 
(h_P,g_Q)\in \Kpq\}.
%(g|_{(S\setminus Y)} , h|_{(Y \setminus S)}), \textrm{ where } g\in \mathcal{K}_S, h\in \mathcal{K}_Y \textrm{ \& } g|_{(S\cap Y)} = h|_{(S \cap Y)}
%\}.
\end{align*}
Matched composition is referred to as matched sum in \cite{HNarayanan1997}.

The \nw{skewed composition} $\mnw{\mathcal{K}_{SP} \rightleftharpoons \mathcal{K}_{PQ}}$ is on $S\uplus Q$ and is defined as follows:
\begin{align*}
 \mathcal{K}_{SP} \rightleftharpoons \mathcal{K}_{PQ} 
  &\equivd \{
                 (f_S,g_Q): (f_S,h_P)\in \Ksp, 
(-h_P,g_Q)\in \Kpq\}.
%(g|_{(S\setminus Y)} , h|_{(Y \setminus S)}), \textrm{ where } g\in \mathcal{K}_S, h\in \mathcal{K}_Y \textrm{ \& } g|_{(S\cap Y)} = h|_{(S \cap Y)}
%\}.
%\ \mbox{Note that}
\end{align*}

We note that the port behaviour of the multiport $\N_P\equivd (\Gsp,\K_{S'S"})
$ can be written as\\ $[\V_{S'P'}\oplus (\V^{\perp}_{S'P'})_{S"P"}]\lrar \K_{S'S"}.$ 

The following result is from \cite{HNarayanan1997}. We give a new proof.
\begin{theorem}
\label{cor:ranklrar}
Let $\Vsp,\Vpq,$ be vector spaces on $S\uplus P, P\uplus Q,$ respectively,
with $S,P,Q,$ being pairwise disjoint.
Then
$r(\Vsp\lrar \Vpq) = r(\Vsp\times S)+r(\Vpq\times Q)+r(\Vsp\circ P\cap \Vpq\circ P) - r(\Vsp\times P\cap \Vpq\times P) .$
\end{theorem}
\begin{proof}
Here we have used part 2 of Theorem \ref{thm:dotcrossidentity}
 and part 2 of Theorem \ref{thm:sumintersection}.
\\$r(\Vsp\lrar \Vpq) \equivd r((\Vsp+ \V_{(-P)Q}) \times SQ)=
r(\Vsp+ \V_{(-P)Q})-r((\Vsp+ \V_{(-P)Q})\circ P)$\\$=r(\Vsp)+ r(\V_{(-P)Q}) - r(\Vsp\times P\cap 
 \V_{(-P)Q}\times P)-r(\Vsp\circ P+\V_{(-P)Q}\circ P)$\\$
= r(\Vsp)+ r(\V_{(-P)Q}) - r(\Vsp\times P\cap 
 \V_{(-P)Q}\times P)- r(\Vsp\circ P)-r(\V_{(-P)Q}\circ P)+ r((\Vsp\circ P)\cap (\V_{(-P)Q}\circ P))  $
\\$  = [r(\Vsp)- r(\Vsp\circ P)]+ [r(\V_{(-P)Q})-r(\V_{(-P)Q}\circ P)] - r(\Vsp\times P\cap 
 \V_{(-P)Q}\times P)+ r((\Vsp\circ P)\cap (\V_{(-P)Q}\circ P))  $\\$
=r(\Vsp\times S)+r(\V_{(-P)Q}\times Q)+r(\Vsp\circ P\cap \V_{(-P)Q}\circ P) - r(\Vsp\times P\cap \V_{(-P)Q}\times P) .$
\\The result follows noting that 
$\V_{(-P)Q}\times Q=\V_{PQ}\times Q, \V_{(-P)Q}\circ P=\V_{PQ}\circ P, \V_{(-P)Q}\times P=\V_{PQ}\times P.$
\end{proof}
The following result is immediate from the definition of matched
and skewed composition.
% and from the fact that 
%if $\Vpq$ is a vector space on $P\uplus Q,$ so is $(\Vpq)_{-PQ}.$ 
\begin{theorem}
\label{thm:matchedprop}
Let $\Ksp,\Kpq$ be collections of vectors on $S\uplus P,  P\uplus Q,$ respectively.
Then,
$$\Ksp\lrar \Kpq\ \ =\ \  (\Ksp+\K_{(-P)Q})\times SQ=(\Ksp\cap \Kpq)\circ SQ;$$
$$\Ksp\rightleftharpoons \Kpq\ \ =\ \  (\Ksp\cap \K_{(-P)Q})\circ SQ=(\Ksp+ \Kpq)\times SQ.$$
%$$\Ksp\lrar \Kpq = (\Ksp+(\Kpq)_{(-P)Q})\times SQ;\ \ 
%\Ksp\rightleftharpoons \Kpq\ \ =\ \  (\Ksp+ \Kpq)\times SQ.$$
\end{theorem}

\section{Implicit Inversion and Implicit Duality}
\label{sec:iitidt}
In this section we present two results which can be regarded as lying at the foundation 
of Implicit Linear Algebra. They are essentially available in \cite{HNarayanan1986a}. More general versions are available in \cite{HNarayanan1997}, \cite{narayanan2020}. Proof of the version of implicit inversion theorem given 
below is available in \cite{narayanan2020}. 

\subsection{ Implicit Inversion Theorem}
\begin{theorem}
\label{thm:IIT}
%{\bf Implicit Inversion Theorem%}
Let  $\Vsp$ be a vector space on $S\uplus P$ and let  $ \Kpq,\Ksq$
be collections of vectors  with $S,P,Q,$ being pairwise disjoint.
Consider the equation
\begin{align}
\label{eqn:IITA}
\Vsp \lrar \Kpq =\Ksq, 
\end{align}
with $\Kpq$ treated as unknown.
We have the following.
\begin{enumerate}
\item given $\Vsp, \Ksq,$  there exists $ \hat{\K}_{PQ}, $ satisfying  Equation \ref{eqn:IITA}, only if   $\Vsp\circ S\supseteq \Ksq \circ S$\\ and $\Ksq+\Vsp\times S \subseteq \Ksq.$
\item given $\Vsp, \Ksq,$ if  $\Vsp\circ S\supseteq \Ksq \circ S$ and 
$\Ksq+\Vsp\times S \subseteq \Ksq,$
then $\hat{\K}_{PQ}\equivd \Vsp \lrar \Ksq $
satisfies the equation.

%\item  given $\Ksp, \Kpq,\Ksq$ satisfying the equation, we have $\Ksp \lrar \Ksq =\Kpq $ iff
%$\Ksp\circ P\supseteq \Kpq \circ P$ and $\Kpq+ \Ksp\times P\subseteq \Kpq.$
\item given $\Vsp, \Ksq,$ assuming that  Equation \ref{eqn:IITA}, with $\Kpq$ treated as unknown, is satisfied by some $\hat{\K}_{PQ} ,$ it is unique if  the additional conditions $\Vsp\circ P\supseteq \Kpq \circ P$ and
$\Kpq+ \Vsp\times P\subseteq \Kpq$ are imposed.
\end{enumerate}
\end{theorem}
%The proof of  a general version of Theorem \ref{thm:IIT}, is given in the appendix.
The case where $\Kpq,\Ksq$ are vector spaces is particularly important 
for us. Here we observe that the condition $\Ksq+\Vsp\times S \subseteq \Ksq,$
reduces to $\Vsp\times S \subseteq \Vsq \times S.$

\begin{theorem}
\label{thm:IITlinear}
Consider the equation
$$\Vsp \lrar \Vpq =\Vsq, $$
where $\Vsp, \Vpq, \Vsq$ are vector spaces respectively on $S\uplus P,P\uplus Q,S\uplus Q,$ with $S,P,Q,$ being pairwise disjoint.% and $\Lsq$
%is a number lattice on $S\uplus P,P\uplus Q,S\uplus Q,$
%respectively. Let $\Vsp\times S=\0_S.$
We then have the following.
\begin{enumerate}
\item given $\Vsp, \Vsq,$  there exists $ \Vpq, $ satisfying the equation only if  $\Vsp\circ S\supseteq \Vsq \circ S$ and $\Vsp\times S \subseteq \Vsq \times S.$
\item given $\Vsp, \Vsq,$ if  $\Vsp\circ S\supseteq \Vsq \circ S$ and $\Vsp\times S \subseteq \Vsq \times S, $ then $\hat{\V}_{PQ}\equivd \Vsp \lrar \Vsq $
satisfies the equation.
%and $\Vsp\times S\subseteq \Lsq \times S;$
%\item  given $\Vsp, \Vpq,\Vsq$ satisfying the equation, we have $\Vsp \lrar \Vsq =\Vpq $ iff
%$\Vsp\circ P\supseteq \Vpq \circ P$ and $\Vsp\times P\subseteq \Vpq \times P.$
\item given $\Vsp, \Vsq,$ assuming that the equation  $\Vsp \lrar \Vpq =\Vsq $
is satisfied by some $\hat{\V}_{PQ}, $ it is unique if  the additional conditions
 $\Vsp\circ P\supseteq \Vpq \circ P$ and $\Vsp\times P\subseteq \Vpq \times P$
are imposed.
\end{enumerate}
\end{theorem}

The following is  a useful affine space version of Theorem \ref{thm:IIT}
\cite{narayanan2020}.
\begin{theorem}
\label{thm:IIT2}
Let $\Asp,\Apq$ be affine spaces on $S\uplus P,P\uplus Q,$ where $S,P,Q$ 
are pairwise disjoint sets. Let $\Vsp,\Vpq$ respectively, be the vector 
space translates of $\Asp,\Apq.$ Let $\Asp\lrar \Apq$ be \nw{nonvoid} and let\\
$\alpha_{SQ}\in \Asp\lrar \Apq.$ Then,\\
1. $\Asp\lrar \Apq = \alpha_{SQ}+(\Vsp\lrar \Vpq).$\\ 
2. $\Apq = \Asp\lrar (\Asp\lrar \Apq)$ iff
$\Vsp\circ P\supseteq \Vpq \circ P$ and $\Vsp\times P\subseteq \Vpq \times P.$
\end{theorem}

\subsection{Implicit Duality Theorem }
Implicit Duality Theorem is a part of network theory folklore.
It also has versions in many different areas of mathematics, such as linear 
inequalities and  linear equations with integral solutions \cite{HNarayanan1997}, linear and partial  differential equations \cite{HN2000,HN2000a,HNPS2013}, matroids and submodular functions \cite{STHN2014}.
A version in the context of Pontryagin Duality is available in \cite{forney2004}.
However, its applications are insufficiently emphasized in the literature.

\begin{theorem}
\label{thm:idt0}
{\bf Implicit Duality Theorem}
Let $\Vsp, \Vpq$ be vector spaces respectively on $S\uplus P,P\uplus Q,$ with $S,P,Q,$ being pairwise disjoint.%
 We then have,
$(\mathcal{V}_{SP}\leftrightarrow \mathcal{V}_{PQ})^\perp 
\ \equaln\ \mathcal{V}_{SP}^\perp \rightleftharpoons \mathcal{V}_{PQ}^\perp 
.$ In particular,
$(\mathcal{V}_{SP}\leftrightarrow \mathcal{V}_{P})^\perp \ \equaln\ \mathcal{V}_{SP}^\perp \leftrightarrow \mathcal{V}_{P}^\perp
.$
\end{theorem}

\begin{proof}
From Theorem \ref{thm:sumintersection}, we have that $(\V_X+\V_Y)^{\perp}=(\V^{\perp}_X\cap \V_Y^{\perp})$
and, using Theorem \ref{thm:perperp}, that $(\V_X\cap\V_Y)^{\perp}=(\V^{\perp}_X+\V_Y^{\perp}).$
We have, using Theorems \ref{thm:dotcrossidentity},$\ $\ref{thm:matchedprop},\\
$(\Vsp\lrar \Vpq)^{\perp}= [(\Vsp\cap \Vpq)\circ SQ]^{\perp}=
[(\Vsp^{\perp}+ \Vpq^{\perp})\times SQ]= (\Vsp^{\perp}\rightleftharpoons\Vpq^{\perp}).$
\\For any vector space $\V_X,$ we have $\V_X= \V_{(-X)}.$
Therefore $(\mathcal{V}_{SP}\leftrightarrow \mathcal{V}_{P})^{\perp} = (\mathcal{V}_{SP}^{\perp} \rightleftharpoons \mathcal{V}_{P}^{\perp})
= \mathcal{V}_{SP}^{\perp} \lrar \mathcal{V}_{P}^{\perp}
.$

\end{proof}
Theorem \ref{thm:idt0} is useful to derive results of the kind 
`if the device characteristic
of a multiport has a certain property, so does the port behaviour'.
To illustrate, consider the case of the device characteristic $\V_{S'S"}$ being  \nw{Dirac}, i.e.,
 satisfying the condition $(\V_{S'S"})_{S"S'}= \V_{S'S"}^{\perp}.$ 
Note that $[(\V^v(\Gsp))_{S'P'}\oplus (\V^i(\Gsp))_{S"P"}]^{\perp}
= [(\V^i(\Gsp))_{S'P'}\oplus (\V^v(\Gsp))_{S"P"}]= [(\V^v(\Gsp))_{S'P'}\oplus (\V^i(\Gsp))_{S"P"}]_{S"P"S'P'}$ (using Theorem \ref{thm:tellegen}),
so that $(\V^v(\Gsp))_{S'P'}\oplus (\V^i(\Gsp))_{S"P"}$ is Dirac.

We have the following well known result.
\begin{corollary}
\label{cor:reciprocalDirac}
%\begin{enumerate}
%\item Let $\V_{S'P'S"P"}, \hat{\V}_{S'P'S"P"}$ be adjoints of each other
%and so also $\V_{S'S"}, \hat{\V}_{S'S"}.$
%Then $\V_{S'P'S"P"}\lrar \V_{S'S"}, $ and $\hat{\V}_{S'P'S"P"}\lrar \hat{\V}_{S'S"}, $ are also adjoints of each other.
%\item Let $\V_{S'P'S"P"}, \hat{\V}_{S'P'S"P"}$ be Dirac duals of each other
%and so also $\V_{S'S"}, \hat{\V}_{S'S"}.$\\
%Then $\V_{S'P'S"P"}\lrar \V_{S'S"}, $ and $\hat{\V}_{S'P'S"P"}\lrar \hat{\V}_{S'S"}, $ are also Dirac duals  of each other.
%\item Let $\V_{S'P'S"P"},\V_{S'S"},$ be reciprocal. Then so is
%$\V_{S'P'S"P"}\lrar \V_{S'S"}. $
Let $\V_{S'P'S"P"},\V_{S'S"},$ be Dirac. Then so is
$\V_{S'P'S"P"}\lrar \V_{S'S"}. $
%\end{enumerate}
\end{corollary}
\begin{proof}
Let $\V_{P'P"}\equivd (\V_{S'P'S"P"}\lrar \V_{S'S"}).$
We have  $ (\V_{P'P"}^{\perp})_{P"P'}= ((\V_{S'P'S"P"}\lrar \V_{S'S"})^{\perp})_{P"P'}=(\V_{S'P'S"P"}^{\perp}\lrar \V_{S'S"}^{\perp})_{P"P'}=  (\V_{S'P'S"P"}^{\perp})_{S"P"S'P'}\lrar (\V_{S'S"}^{\perp})_{S"S'}
={\V}_{S'P'S"P"}\lrar {\V}_{S'S"}={\V}_{P'P"} .$
\end{proof}

\section{Rigid  Multiports}
\label{sec:rigid}
%================================================
%
%Emphasize early that rigid connection of rigid multiports,
%if they yield a network then it is rigid and has a unique solution.
%
%
%======================================================
\subsection{Introduction}
\label{subsec:regular}
We have seen before that a linear multiport is basically a pair $\{\V_{AB},\A_B\}, $ where $\V_{AB}$ is the solution space of the topological constraints and $\A_B$ is the device characteristic. 
\\We are particularly interested in multiports which are well behaved 
in the following sense: \\
(a) they are consistent for arbitrary values of sources, 
and
(b) for a given port condition (a voltage-current vector in the  
multiport behaviour) there is a unique solution for the  
multiport. Practically speaking, it is only such multiports that can be handled by freely
available circuit simulators after some preprocessing.
We call such multiports rigid. In this section we study their connections 
which result again in rigid multiports.
% by concentrating on rigid pairs of affine spaces.
%We do this by first formally defining rigid multiports 
%and then simplifying their study by.
% We do this by concentrating on rigid pairs of affine spaces.
Testing for rigidity
is essentially linear algebraic,
 although of course, because rigid multiports are well behaved, this labour can be outsourced to circuit simulators. However, in some important special cases the computation can be made 
entirely matroidal, the matroids being very simple.
Fortunately, rigidity of matroid pairs can be studied in a manner 
entirely parallel to the study of rigid affine pairs. This we do in Subsection \ref{subsec:rigidpairsm}
 and the two studies are brought together
in Subsection \ref{subsec:connect} and \ref{sec:matroidalrigidity},
 to yield simple sufficient conditions for 
multiports involving controlled sources to be rigid and their port behaviour to have a hybrid representation.

%===============================================
%
%They can be regarded as possible building blocks for constructing linear networks 
%with unique solution.
%Therefore, a natural question that arises is whether connection of rigid
%multiports results in a rigid
% multiport. In general, if the
%conditions at the common ports do not match, the resulting multiport
%will not even be consistent. However, if at the ports where the connection is made, certain conditions,
%which could be regarded as  a kind of local rigidity, are satisfied,
%rigidity is preserved. To enable our discussion, 
%we first define rigid pairs of affine spaces. These are adequate 
%for some important cases of connection of multiports.
%For instance, if two rigid multiports, whose port behaviours are rigid 
% with respect to each other, are connected to make up a network with no ports, then that network will 
%have a unique solution. 
%In the subsequent subsection, we examine rigidity of families of affine 
%spaces.
%
%===================================================
\begin{definition}
\label{def:regular}
Let multiport $\N_P\equivd (\G_{SP}, \A_{S'S"}), $ 
where 
 $\A_{S'S"}
=\alpha_{S'S"}+\V_{S'S"}.$\\
The multiport $\N_P$ is said to be \nw{rigid
} iff every  multiport $\hat{\N}_P\equivd (\G_{SP}, \hat{\A}_{S'S"}), $
 where $\hat{\A}_{S'S"}=\hat{\alpha}_{S'S"}+\V_{S'S"},$
has a non void set of
solutions  
and has a unique solution corresponding to every vector in its multiport behaviour. 
A
 multiport $\N_P\equivd (\G_{SP}, \A_{S'S"})$ is  said to be \nw{proper} iff 
it is rigid and $\A_{S'S"}$ is proper.
\\Let network  $\N\equivd (\G_{S}, \A_{S'S"}), $ 
where 
 $\A_{S'S"}
=\alpha_{S'S"}+\V_{S'S"}.$\\
We say $\N$ is \nw{rigid},  equivalently \nw{proper}, iff every network $\hat{\N}\equivd (\G_{S}, \hat{\A}_{S'S"}), $
 where $\hat{\A}_{S'S"}=\hat{\alpha}_{S'S"}+\V_{S'S"},$
has a unique solution.
\end{definition}
\begin{example}
A multiport  $\N_P$ is rigid  in the following cases.
There are no topological conditions on the port edges $P.$
They may contain loops and cutsets.
\begin{enumerate}
\item $\N_P$ has  only positive resistors, voltage and current 
sources and the voltage sources do not form loops and current sources 
do not form cutsets. In this case the dimension of the port behaviour 
$\A_{P'P"}$ is $|P|.$
\item  $\N_P$ has  only  resistors and norators which form neither loops nor cutsets. In this case the dimension of the port behaviour
$\V_{P'P"}$ can be as large as $2|P|.$
\item $\N_P$ has  only resistors and nullators.
In this case the dimension of the port behaviour
$\V_{P'P"}$ can be as small as $0.$
\item $\N_P$ has  only nonzero resistors, independent and controlled sources,
 and the values of resistors and the parameters of the controlled 
sources are algebraically independent over $\Q.$ The independent 
and controlled sources satisfy topological conditions
of Theorem \ref{thm:purslowgraphnew}.
In this case the dimension of the port behaviour
$\A_{P'P"}$ is $|P|.$
\end{enumerate}
$\N_P$ is not rigid  in the following cases.
\begin{enumerate}
\item $\N_P$ has voltage sources that form loops or current sources 
that form cutsets. In this case the multiport will have no solution
 for some choice of source values.
\item $\N_P$ has norators which contain loops or cutsets.
In this case the port conditions cannot determine the internal conditions uniquely.
\end{enumerate}
It can be shown that any given affine space ${\A}_{P'P"},$
can be realized as the multiport behaviour of a rigid
 multiport.
\end{example}

\subsection{Rigid pairs of affine spaces}
\label{subsec:rigidpairs}
\begin{definition}
\label{def:affinerigid}
Let $\A_{AB},\A_{BC}$ be affine spaces on sets $A\uplus B, B\uplus C,$ respectively, A,B,C being pairwise disjoint.
Further, let $\A_{AB},\A_{BC}$ have vector  space translates $\V_{AB},\V_{BC},$
respectively.
\\
We say the pair $\{\A_{AB},\A_{BC}\}$ has the \nw{full sum property} iff
$\V_{AB}\circ B+\V_{BC}\circ B=\F_B.$\\
We say the pair $\{\A_{AB},\A_{BC}\}$ has the \nw{zero intersection property} iff
$\V_{AB}\times B\cap \V_{BC}\times B=\0_B.$\\
We say $\A_{AB}$ is \nw{rigid with respect to} $\A_{BC},$ or that the pair $\{\A_{AB},\A_{BC}\}$ is \nw{rigid,
} iff  it has the full sum property and the zero intersection property.
\\When $B=\emptyset,$ we take $\{\A_{A},\A_{C}\}$ to be  {rigid
}. 
\end{definition}

When affine spaces are direct sums of other spaces, it is clear 
from the definition that rigidity involving the former can be inferred from 
that involving the latter.
\begin{lemma}
\label{lem:directsum}
The pair $\{\A_{AB}\oplus\A_{CD},\A^1_B\oplus\A^2_C\}$ is rigid
iff the pairs $\{\A_{AB},\A^1_B\},$
$\{\A_{CD},\A^2_C\}$ are rigid.
\end{lemma}

The next result states
the basic facts about rigid
 pairs.
\begin{theorem}
\label{thm:regularrecursive}
Let $\{\A_{AB},\A_{BC}\}$ be a rigid
 pair and let $\V_{AB},\V_{BC}$ be the vector space translates of $\A_{AB},\A_{BC},$ respectively. Then
\begin{enumerate}
\item The full sum property  of $\{\A_{AB},\A_{BC}\}$ is equivalent to 
 $\hat{\A}_{AB}\lrar \hat{\A}_{BC}$ being nonvoid, whenever 
$\V_{AB},\V_{BC}$ are the vector space translates of $\hat{\A}_{AB}, \hat{\A}_{BC},$ respectively.
Further, 
 when the full sum property holds for $\{\A_{AB},\A_{BC}\},$  $\hat{\A}_{AB}\lrar \hat{\A}_{BC}$  has  vector space translate $\V_{AB}\lrar \V_{BC}.$
\item The zero intersection  property of $\{\A_{AB},\A_{BC}\}$ is equivalent to 
%$(\V_{AB}\cap \V_{BC})\times B=\0_B.$
the statement that , \\if $(f_A, f_C)\in \A_{AB}\lrar \A_{BC}$ and $(f_A,f_B,f_C), (f_A,f'_B,f_C)\in \A_{AB}\cap \A_{BC},$
then $f_B=f'_B.$
\item  The pair  $\{\A_{AB},\A_{BC}\}$ is rigid
 \\iff
$r(\V_{AB}+\V_{BC})= r(\V_{AB})+r(\V_{BC})$ and $  r(\V_{AB}^{\perp}+\V_{BC}^{\perp})= r(\V_{AB}^{\perp})+r(\V_{BC}^{\perp}).$
\item The pair $\{\A_{AB},\A_{BC}\}$ has the zero intersection (full sum) property iff $\{\V^{\perp}_{AB},\V^{\perp}_{BC}\}$ has the full sum (zero intersection) property.
Therefore
 $\{\A_{AB},\A_{BC}\}$ is {rigid
} iff
$\{\V^{\perp}_{AB},\V^{\perp}_{BC}\}$ is rigid.
%\item The pair $\{\A_{AB}\oplus\A_{CD},\A^1_B\oplus\A^2_C\}$ is rigid
%iff the pairs $\{\A_{AB},\A^1_B\},$
%$\{\A_{CD},\A^2_C\}$ are rigid.
%\item Let $\A^1_B,\A^2_B$ have $\V^1_B,\V^2_B,$
%respectively as their vector space translates.
%\begin{enumerate}
%\item 
%If $\{\V^1_B,\V^2_B\}$ 
%is rigid, $\A^1_B\cap\A^2_B$ will contain exactly one vector.
%\item 
%If $\{\V^1_B,\V^2_B\}$ is not rigid, then there exist $\hat{\A}_B^1, \hat{\A}_B^2,$  with $\V^1_B,\V^2_B,$ respectively, as their vector space translates, 
%such that $\hat{\A}_B^1\cap \hat{\A}_B^2$  is not a singleton.
%\item If $r(\V^1_B)+r(\V^2_B)=|B|,$ then
%the full sum property and zero intersection property are equivalent
%and ${\A}_B^1\cap {\A}_B^2 $ is a singleton iff $\{\A^1_B, \A^2_B\}$
%is rigid.
%\end{enumerate}
%\item The pair  $\{\A_{AB},\A_{BC}\}$ is rigid
% only if 
%$r(\V_{AB}\lrar \V_{BC})= r(\V_{AB})+r(\V_{BC})-|B|.$
\end{enumerate}
\end{theorem}
\begin{proof} 
%{\it Proof of Theorem \ref{thm:regularrecursive}}
1.
%It is clear that $(\V_{AB}\circ B) +  (\V_{BC}\circ B)=(\V_{AB}+ \V_{BC})\circ B,$ so that $(\V_{AB}\circ B) +  (\V_{BC}\circ B)=\F_B$  is equivalent to
%$(\V_{AB}+ \V_{BC})\circ B=\F_B.$\\
Let $\hat{\A}_{AB}=(\alpha_A,\alpha_B)+\V_{AB}, \hat{\A}_{BC}=\beta_{BC}+\V_{BC}.$
%\\By the definition of matched composition, $\hat{\A}_{AB}\lrar \hat{\A}_{BC}$ is nonvoid iff $\hat{\A}_{AB}\cap \hat{\A}_{BC}$ is nonvoid,\\ i.e., iff $\hat{\A}_{AB}\circ B\ \cap\  \hat{\A}_{BC}\circ B$ is nonvoid.\\
%Let $\hat{\A}_{AB}=(\alpha_A,\alpha_B)+\V_{AB}, \hat{\A}_{BC}=\beta_{BC}+\V_{BC}.$
\\By the definition of matched composition, $\hat{\A}_{AB}\lrar \hat{\A}_{BC}$ is nonvoid iff $\hat{\A}_{AB}\cap \hat{\A}_{BC}$ is nonvoid,\\ i.e., iff $\hat{\A}_{AB}\circ B\ \cap\  \hat{\A}_{BC}\circ B$ is nonvoid,\\
i.e., iff there exist $\lambda_B\in \V_{AB}\circ B, \sigma_B\in \V_{BC}\circ B,$
such that $\alpha_B+\lambda_B=\beta_B+\sigma_B,$\\
i.e., iff there exist $\lambda_B\in \V_{AB}\circ B, \sigma_B\in \V_{BC}\circ B,$
such that $\alpha_B-\beta_B=\sigma_B-\lambda_B.$\\
Clearly when $\V_{AB}\circ B+\V_{BC}\circ B=\F_B,$
there exist $\lambda_B\in \V_{AB}\circ B, \sigma_B\in \V_{BC}\circ B,$
such that $\alpha_B-\beta_B=\sigma_B-\lambda_B$
 so that $\hat{\A}_{AB}\lrar \hat{\A}_{BC}$ is nonvoid.
\\
On the other hand, if $\V_{AB}\circ B+\V_{BC}\circ B\ne \F_B,$
there exist $\alpha_B,\beta_B$ such that $\alpha_B-\beta_B \notin\V_{AB}\circ B+\V_{BC}\circ B,$ so that $\hat{\A}_{AB}\cap \hat{\A}_{BC}$ and therefore
$\hat{\A}_{AB}\lrar \hat{\A}_{BC}$ is void.\\
By Theorem \ref{thm:IIT2}, if $\hat{\A}_{AB}\lrar \hat{\A}_{BC}$ is nonvoid, its vector space translate is $\V_{AB}\lrar \V_{BC}.$

2. 
%It is clear that $(\V_{AB}\times B) \cap  (\V_{BC}\times B)=(\V_{AB}\cap \V_{BC})\times B,$ so that $(\V_{AB}\times B) \cap  (\V_{BC}\times B)=\0_B$  is equivalent to
%$(\V_{AB}\cap \V_{BC})\times B=\0_B.$\\
Let $(\V_{AB}\times B)\cap (\V_{BC}\times B)=\0_B.$
If $(f_A,f_C)\in \A_{AB}\lrar \A_{BC}$ and \\$(f_A,f_B,f_C), (f_A,f'_B,f_C)\in \A_{AB}\cap \A_{BC},$
then $(f_A,f_B)- (f_A,f'_B)=(0_A,(f_B-f'_B))\in \V_{AB},$ and  similarly $((f_B-f'_B),0_C)\in \V_{BC},$ so that $(f_B-f'_B)\in (\V_{AB}\times B)\cap (\V_{BC}\times B)=\0_B.$
\\
Next suppose whenever $(f_A,f_C)\in \A_{AB}\lrar \A_{BC}$ and $(f_A,f_B,f_C), (f_A,f'_B,f_C)\in \A_{AB}\cap \A_{BC},$       we have $f_B=f'_B.$
Suppose $(\V_{AB}\times B) \cap (\V_{BC}\times B)\ne \0_B.$
Let $g_B\in (\V_{AB}\times B) \cap (\V_{BC}\times B)$ and let $g_B\ne 0_B.$
\\If $(f_A,f_C)\in \A_{AB}\lrar \A_{BC},$ then there exists $f_B$ such that  $(f_A,f_B) \in \A_{AB}$ and
$(f_B,f_C) \in  \A_{BC}.$
Clearly $(f_A,f_B+g_B) \in \A_{AB}$ and
 $(f_B + g_B,f_C) \in  \A_{BC},$
so that we must have $(f_A,f_B+g_B,f_C) \in \A_{AB}\cap \A_{BC}.$
But this means $f_B=f_B+g_B,$
a contradiction.
We conclude that $(\V_{AB}\times B)\cap (\V_{BC}\times B)=\0_B.$

3. Let $\{\A_{AB},\A_{BC}\}$ be rigid
.
Then $(\V_{AB}\circ B+\V_{BC}\circ B)=\F_B$ and $(\V_{AB}\times B\cap \V_{BC}\times B)=\0_B.$\\
Now, using part 2 of Theorem \ref{thm:sumintersection} and the zero intersection property, \\ $r(\V_{AB}+\V_{BC})
%= r((\V_{AB}\oplus \0_C)+(\0_A\oplus \V_{BC}))=
%r(\V_{AB}\oplus \0_C) + r(\0_A\oplus \V_{BC})- r((\V_{AB}\oplus \0_C)\cap (\0_A\oplus \V_{BC}))$\\ 
= r(\V_{AB})+r(\V_{BC})-r(\V_{AB}\times B\cap \V_{BC}\times B)=  r(\V_{AB})+r(\V_{BC}).$
%(using the zero intersection property).
\\
%$r(\V_{AB})+r(\0_A+\V_B)-r(\V_{AB}\cap (\0_A+\V_B)) $\\ (using Theorem \ref{thm:sumintersection})\\
%$= r(\V_{AB})+r(\V_B)-r(\V_{AB}\times B\cap \V_B)=  r(\V_{AB})+r(\V_B).$
%\\
Next $(\V_{AB}\circ B+\V_{BC}\circ B)=\F_B$ iff $(\V_{AB}\circ B+\V_{BC}\circ B)^{\perp}=\F_B^{\perp},$ i.e.,
$  (\V_{AB}^{\perp}\times B\cap\V_{BC}^{\perp}\times B)=\0_B.$
Using the above argument, we therefore have
$r(\V_{AB}^{\perp}+\V_{BC}^{\perp})=r(\V_{AB}^{\perp})+r(\V_{BC}^{\perp}).$
\\
On the other hand, suppose $r(\V_{AB}+\V_{BC})= r(\V_{AB})+r(\V_{BC})$ and $  r(\V_{AB}^{\perp}+\V_{BC}^{\perp})= r(\V_{AB}^{\perp})+r(\V_{BC}^{\perp}).$
\\
The former condition implies $r(\V_{AB}\times B\cap \V_{BC}\times B)=0,$ i.e., $(\V_{AB}\times B\cap \V_{BC}\times B)=\0_B.$
 and the latter implies $r(\V_{AB}^{\perp}\times B\cap \V_{BC}^{\perp}\times B)=0,$ i.e.,
$(\V_{AB}^{\perp}\times B\cap \V_{BC}^{\perp}\times B)=\0_B,$
i.e., $(\V_{AB}^{\perp}\times B\cap \V_{BC}^{\perp}\times B)^{\perp}=(\V_{AB}\circ B+ \V_{BC}\circ B)=\F_B$ (using Theorems \ref{thm:sumintersection},\ref{thm:dotcrossidentity}).

4. We have, as before,
$(\V_{AB}\circ B+\V_{BC}\circ B)^{\perp}=\V_{AB}^{\perp}\times B\ \cap \ \V_{BC}^{\perp}\times B=\F^{\perp}_B= \0_B,$
and\\
$(\V_{AB}\times B\ \cap\ \V_{BC}\times B)^{\perp}=\V_{AB}^{\perp}\circ B+ \V_{BC}^{\perp}\circ B=\0^{\perp}_B= \F_B.$
\end{proof}
\begin{remark}
1. Given a pair $\{\V_{AB},\V_B\}$ we can always construct a rigid
 pair
$\{\V_{AB},\tilde{\V}_B\}$ such that\\ $\V_{AB}\lrar \tilde{\V}_B=\V_{AB}\lrar \V_B.$
We give a construction for $\tilde{\V}_B$ below.\\
Since $\V_{AB}\lrar \V_B= \V_{AB}\lrar (\V_B\cap \V_{AB}\circ B)$
$= \V_{AB}\lrar (\V_B+\V_{AB}\times B),$
we can, without loss of generality, take $\V_{AB}\times B\subseteq \V_B\subseteq \V_{AB}\circ B.$
Split $\F_B$ as $\F_B=\V_{AB}\circ B + \V^1_B,$ where
the subspaces $\V_{AB}\circ B ,\V^1_B,$ have zero intersection.
Next split $ \V_B$ as $\V_{AB}\times B +\V^2_B,$
taking the subspaces $\V_{AB}\times B,\V^2_B,$ to have zero intersection.
Take $\tilde{\V}_B\equivd \V^1_B+\V^2_B.$\\
Clearly $\tilde{\V}_B+\V_{AB}\circ B =\F_B.$
Next,
$\tilde{\V}_B \cap \V_{AB}\times B =\0_B.$
Thus $\{\V_{AB},\tilde{\V}_B \} $ is rigid.

We will first show that $\V_{AB}\lrar \tilde{\V}_B\subseteq \V_{AB}\lrar \V_B.$
Let $f_A\in \V_{AB}\lrar \tilde{\V}_B.$
Then there exist $(f_{A},f_B)\in \V_{AB}$ and $f_B\in\tilde{\V}_B.$
We can write $f_B$ as $f_B=f^1_B+f^2_B,$ where $f^1_B\in \V^1_B,
f^2_B\in \V^2_B\subseteq \V_B\subseteq \V_{AB}\circ B .$ Now $f_B,f^2_B\in \V_{AB}\circ B,$ so that
$f^1_B\in \V_{AB}\circ B$ and therefore in $\V_{AB}\circ B \cap\V^1_B.$
This means that $f^1_B=0_B$ so that $f_B= f^2_B\in\V^2_B\subseteq \V_B.$
It follows that $\V_{AB}\lrar \tilde{\V}_B\subseteq \V_{AB}\lrar \V_B.$

To see the reverse containment, let $f_A\in \V_{AB}\lrar \V_B.$
Then there exist $(f_{A},f_B)\in \V_{AB}$ and $f_B\in\V_B.$
We can write $f_B$ as $f_B=f^3_B+f^2_B,$ where $f^3_B\in \V_{AB}\times B,
f^2_B\in \V^2_B.$ Clearly, $(f_{A},f_B)-(0_A,f^3_B)=(f_{A},f^2_B)\in \V_{AB}$
and by definition of $\tilde{\V}_B,$ we have $f^2_B\in \tilde{\V}_B.$Therefore $f_A\in \V_{AB}\lrar \tilde{\V}_B.$
\\
We thus see that  $\{\V_{AB},\tilde{\V}_B \} $ is rigid
and $\V_{AB}\lrar \tilde{\V}_B= \V_{AB}\lrar \V_B.$

Thus rigidity is a property of the pair $\{\V_{AB},\V_B\}$
 that cannot be inferred from $\V_{AB}\lrar {\V}_B$
 alone.

2. A consequence of part 1 is that whenever a multiport $\N_P$ has
a port behaviour ${\V}_{P'P"},$ we can alter the
multiport by changing the device characteristic
so that the resulting multiport $\tilde{\N_P}$ also has multiport behaviour ${\V}_{P'P"},$ but is rigid
.
Therefore every multiport behaviour is realizable as that of a rigid
 multiport. Use of parts 1 and 2 of Theorem \ref{thm:regularrecursive} will prove this fact for
 affine multiport behaviours also.

3. The above ideas do not appear to be true for matroids in general
unless we can somehow bring in the implicit inversion theorem,
which we know to be false for matroids in general.
\end{remark}

The next lemma is of use in characterizing rigid networks without ports
and, in the case of rigid multiports, the dimension of the multiport
of the multiport behaviour.
\begin{lemma}
\label{lem:rankfactsvectors}
\begin{enumerate}
\item Let $\A^1_B,\A^2_B$ have $\V^1_B,\V^2_B,$
respectively as their vector space translates.
\begin{enumerate}
\item
If $\{\V^1_B,\V^2_B\}$
is rigid, $\A^1_B\cap\A^2_B$ will contain exactly one vector.
\item
If $\{\V^1_B,\V^2_B\}$ is not rigid, then there exist $\hat{\A}_B^1, \hat{\A}_B^2,$  with $\V^1_B,\V^2_B,$ respectively, as their vector space translates,
such that $\hat{\A}_B^1\cap \hat{\A}_B^2$  is not a singleton.
\item If $r(\V^1_B)+r(\V^2_B)=|B|,$ then
the full sum property and zero intersection property are equivalent
and ${\A}_B^1\cap {\A}_B^2 $ is a singleton iff $\{\A^1_B, \A^2_B\}$
is rigid.
\end{enumerate}
\item If  $\{\A_{AB},\A_{BC}\}$ is rigid, 
then
$r(\V_{AB}\lrar \V_{BC})= r(\V_{AB})+r(\V_{BC})-|B|.$
\end{enumerate}
\end{lemma}
\begin{proof}
1(a). Let $\A^1_B\equivd \alpha_B+\V^1_B$ and let $\A^2_B\equivd \beta_B+\V^2_B.$ Suppose $\{\V^1_B,\V^2_B\}$ is rigid.
Let $M,N$ be representative matrices of $\V^1_B,\V^2_B,$ respectively.
Consider the equation
$\alpha^T+\lambda ^TM= \beta ^T+\sigma ^TN,$
which characterizes the affine space $\A^1_B\cap\A^2_B.$
This equation is equivalent to
\begin{align}
\label{eqn:6a}
\alpha^T-\beta ^T=(-\lambda ^T| \sigma ^T)\ppmatrix{M\\N}.
\end{align}
This equation has a unique solution since $M,N$ have independent rows and
their ranks add up to $|B|.$

1(b). On the other hand, suppose $\{\V^1_B,\V^2_B\}$ is not rigid. Then the
ranks of $M,N$ either add up to a value below $|B|$ or above it.
In the former case,
we can choose $\alpha^T,\beta ^T$ in such a way that that the equation 
has no solution and in the latter, the equation will not have a unique solution. Thus, in either case, $\hat{\A}_B^1\cap \hat{\A}_B^2$
will be a non singleton.

1(c). We have $|B|=r(\V^1_B)+r(\V^2_B)= r(\V^1_B+\V^2_B) + r(\V^1_B\cap \V^2_B),$ so that $|B|=r(\V^1_B)+r(\V^2_B)= r(\V^1_B+\V^2_B) $ iff
$r(\V^1_B\cap \V^2_B)=0.$
Next, if $\{\A^1_B,\A^2_B\}$ is not rigid, Equation \ref{eqn:6a}
does not have a unique solution since it does not have a non singular coefficient matrix.

2. By Theorem \ref{cor:ranklrar}, we have\\
$r(\V_{AB}\lrar \V_{BC})=r(\V_{AB}\times A)+ r(\V_{BC}\times C) +r(\V_{AB}\circ B\cap \V_{BC}\circ B)-r(\V_{AB}\times B\cap \V_{BC}\times B).$
\\
Now $r(\V_{AB}\circ B\cap \V_{BC}\circ B)=r(\V_{AB}\circ B)+r(\V_{BC}\circ B)-
r(\V_{AB}\circ B+ \V_{BC}\circ B).$\\
Therefore, using Theorems \ref{thm:sumintersection},\ref{thm:dotcrossidentity}, we have
 $r(\V_{AB}\lrar \V_{BC})$\\$=r(\V_{AB}\times A)+ r(\V_{BC}\times C) +
r(\V_{AB}\circ B)+r(\V_{BC}\circ B)-
r(\V_{AB}\circ B+ \V_{BC}\circ B)-r(\V_{AB}\times B\cap \V_{BC}\times B)$
\\
$=r(\V_{AB})+r(\V_{BC})-r(\V_{AB}\circ B+\V_{BC}\circ B)-r(\V_{AB}\times B\cap \V_{BC}\times B).$ Therefore,
\\
%Since $r(\V_{AB}\circ B+\V_{BC}\circ B)\leq |B|$ and $r(\V_{AB}\times B\cap \V_{BC}\times B)\geq 0,$
we have $r(\V_{AB}\lrar \V_{BC})=r(\V_{AB})+r(\V_{BC})-|B|,$ if
$\V_{AB}\circ B+\V_{BC}\circ B=\F_B$ and $\V_{AB}\times B\cap \V_{BC}\times B=\0_B,$
i.e., if $(\A_{AB}, \A_{BC})$ is rigid.

\end{proof}
\begin{remark}
Let $r(\V^1_B)+r(\V^2_B)=|B|,$
let $M,N$ be representative matrices of $\V^1_B,\V^2_B,$
and let $Q,P$ be representative matrices of
$(\V^1_B)^{\perp},(\V^2_B)^{\perp},$ respectively.
The rigidity  of $\{\V^1_B,\V^2_B\}$ is equivalent to 
$(M^T|N^T)^T$ being invertible and the rigidity of $\{(\V^1_B)^{\perp},(\V^2_B)^{\perp}))\}$
is equivalent to $((P^T|Q^T)^T) $ being invertible. 
However, if $(\hat{P}^T|\hat{Q}^T)$
 is the inverse of $(M^T|N^T)^T,$ then $\hat{Q},\hat{P}$ are representative matrices of
$(\V^1_B)^{\perp},(\V^2_B)^{\perp},$
 respectively. It can be seen that $(\hat{P}^T|\hat{Q}^T)$ is invertible 
iff $(P^T|Q^T)^T $  is. Therefore, the equivalence of rigidity of the pairs $\{\V^1_B,\V^2_B\}$ and
$\{(\V^1_B)^{\perp},(\V^2_B)^{\perp}\}$ is clear.
\end{remark}

\begin{corollary}
\label{cor:adjointrigidity}
Let $\N_P\equivd (\G_{SP}, \A_{S'S"})$ be a multiport with   $\A_{S'S"}=\alpha_{S'S"}+\V_{S'S"}.$ We then have the following.
\begin{enumerate}
\item $\N_P$ is rigid iff $\{\V^v(\G_{SP})_{S'P'}\oplus (\V^i(\G_{SP})_{S"P"}, \V_{S'S"}\}$ is rigid, i.e., iff $\N^{hom}_P\equivd (\G_{SP}, \V_{S'S"})$ is rigid.
\item Let $\N_P, \N^{hom}_P$ have the multiport behaviours ${\A}_{P'P"},
{\V}_{P'P"}.$ Then, if $\N_P$ is rigid, ${\A}_{P'P"}$ is nonvoid and has ${\V}_{P'P"}$ as its vector space translate.
\item If $\N_P$ is rigid, its multiport behaviour is  proper iff $\A_{S'S"}$ is proper.
%$\N_P$ is proper iff $\A_{S'S"}$ is proper.
%\item $\N_P$ is rigid iff $\N^{adj}_P$ is rigid.
\end{enumerate}
\end{corollary}
\begin{proof}
For parts 1 and 2 we apply Theorem \ref{thm:regularrecursive}, taking\\
$\A_{AB}= \V^v(\G_{SP})_{S'P'}\oplus \V^i(\G_{SP})_{S"P"},$
$\A_B= \A_{S'S"}, C=\emptyset.$\\
Part 1 follows from part 3 of Theorem \ref{thm:regularrecursive}.
\\
Part 2 follows from parts 1 and 2 of Theorem \ref{thm:regularrecursive}.
Parts 1 and 2.   

3. This follows from part 2 of Lemma \ref{lem:rankfactsvectors},
taking $C=\emptyset,$ $ A=P'\uplus P", B= S'\uplus S",$\\$\V_{AB}=\V^v(\G_{SP})_{S'P'}\oplus \V^i(\G_{SP})_{S"P"},$ and $\V_B= \V_{S'S"}.$ It can be seen that,
 in this case, since $r(\V_{AB})=|S|+|P|,$ we must have $ r(\V_B)=|S|,$  iff  
$r(\V_{AB}\lrar\V_B)=|P|.$ 
%Since $({\V}_{P'P"})_{P'(-P")}=\V_{AB}\lrar\V_B
%,$ the result follows.
\end{proof}
We note that when the multiport is not rigid,  
 the dimension of the port behaviour can range from zero to twice the number of ports, even if the device characteristic
is proper. (For an interesting discussion of this situation see \cite{recski19}.)

When we set $P=\emptyset,$ in Corollary \ref{cor:adjointrigidity}, we get the following result.
\begin{corollary}
\label{cor:adjointrigidity2}
Let $\N\equivd (\G_{S}, \A_{S'S"})$ be a network with   $\A_{S'S"}=\alpha_{S'S"}+\V_{S'S"}.$ Then
\begin{enumerate}
\item $\N$ is rigid iff $\N^{hom}\equivd (\G_{S}, \V_{S'S"})$
is rigid;
\item if $\A_{S'S"}$ is proper, then $\N$ has a unique solution iff 
$\N^{hom}$ has a unique solution.
\end{enumerate}
\end{corollary}
\begin{proof}
1. This is immediate from part 1 of Corollary \ref{cor:adjointrigidity}.

2. Let $\A_{S'S"}$ be proper.
Then $r((\V^v(\G_S))_{S'}\oplus (\V^i(\G_S))_{S"})+r(\V_{S'S"})=2|S|=|S'\uplus S"|.$ Therefore, by part 1(c)  of Lemma \ref{lem:rankfactsvectors}, 
$\N$ has a unique solution, i.e.,
$[(\V^v(\G_S))_{S'}\oplus (\V^i(\G_S))_{S"}]\cap  \A_{S'S"}$
is a singleton iff\\ $\{[(\V^v(\G_S))_{S'}\oplus (\V^i(\G_S))_{S"}], \A_{S'S"}\}$ is rigid, i.e., iff $\{[(\V^v(\G_S))_{S'}\oplus (\V^i(\G_S))_{S"}], \V_{S'S"}\}$ is rigid, \\
i.e., iff $[(\V^v(\G_S))_{S'}\oplus (\V^i(\G_S))_{S"}]\cap  \V_{S'S"}$ 
is a singleton, 
\end{proof}
The next theorem is useful for examining the rigidity
 of multiports, 
obtained from other such, by augmenting the device characteristic or by connection through an affine characteristic. 

\begin{theorem}
\label{lem:derivedregularity}
Let $\V_{WTV},\V_V, \V_T,$ be vector spaces on $W\uplus T\uplus V, V,T,$ 
respectively.
% and let $(\V_{WTV},\V_R)$ be regular.
Then $\{\V_{WTV},\V_T\oplus\V_V\}$ is rigid 
\begin{enumerate}
\item
iff 
$\{\V_{WTV},\V_V\},$  $\{\V_{WTV}\lrar\V_V,\V_T\}$  are rigid;
\item iff
$\{\V_{WTV},\V_V\}$ is rigid and $\{\V_{WTV}\cap \V_V,\V_T\},$ 
$\{\V_{WTV}+ \V_V,\V_T\}$ satisfy  the full sum and zero intersection 
properties respectively;
\item iff $\{\V_{WTV},\V_V\},$ $\{\V_{WTV}\cap \V_V,\V_T\},$
$\{\V_{WTV}+ \V_V,\V_T\}$ are rigid.
\end{enumerate}
\end{theorem}
%The proof is relegated to the Appendix.\\
\begin{proof}
 We need the following simple lemma, whose routine proof we omit.
\begin{lemma}
\label{lem:factsvectorspaces}
(a) $(\V_{WTV}\circ TV+\V_T\oplus\V_V)\circ V=\V_{WTV}\circ V+\V_V;$
\\(b) $(\V_{WTV}\circ TV+\V_T\oplus\V_V)\times T=(\V_{WTV}\circ TV+\V_V)\times T+\V_T;$
\\(c) $\V_{WTV}\circ TV\lrar \V_V=(\V_{WTV}\lrar \V_V)\circ T.$
\end{lemma}

1. We have $\V_{WTV}\circ TV+\V_T\oplus\V_V=\F_{TV}$ iff \\
$r(\V_{WTV}\circ TV+\V_T\oplus\V_V)=r((\V_{WTV}\circ TV+\V_T\oplus\V_V)\circ V)+r((\V_{WTV}\circ TV+\V_T\oplus\V_V)\times T)=|{T\uplus V}|,$
using Theorem \ref{thm:dotcrossidentity}.
Since $r(\V_X)\leq |X|,$
we have $\V_{WTV}\circ TV+\V_T\oplus\V_V=\F_{TV}$ iff \\
$(\V_{WTV}\circ TV+\V_T\oplus\V_V)\circ V=\F_{V},$
and
$(\V_{WTV}\circ TV+\V_T\oplus\V_V)\times T=\F_{T},$
\\i.e., iff
$\V_{WTV}\circ V+\V_V=\F_{V},$
and
$(\V_{WTV}\circ TV+\V_V)\times T+\V_T=\F_{T},$
(using Lemma \ref{lem:factsvectorspaces} parts  (a) and (b) above),
i.e., iff
$\V_{WTV}\circ V+\V_V=\F_{V},$
and
$(\V_{WTV}\circ TV\lrar \V_V)+\V_T= (\V_{WTV}\lrar \V_V)\circ T+\V_T=\F_{T},$
using Theorem \ref{thm:matchedprop}, Lemma \ref{lem:factsvectorspaces} part  (c) above  and the fact that $(\V_V)_{(-V)}=\V_V.$
\\Thus,
$\{\V_{WTV},\V_V\oplus \V_T\},$  satisfies the full sum property iff 
$\{\V_{WTV},\V_V\},$ $\{\V_{WTV}\lrar\V_V,\V_T\}$ satisfy the full sum property.

We will now use Theorems \ref{thm:perperp}, \ref{thm:sumintersection}, \ref{thm:dotcrossidentity} and 
Theorem \ref{thm:idt0} and get the second half
of the property of rigidity that is required to complete the proof.

From the above argument, we have
$\V_{WTV}^{\perp}\circ TV+\V_T^{\perp}\oplus\V_V^{\perp}=\F_{TV},$\\ iff
$\V_{WTV}^{\perp}\circ V+\V_V^{\perp}=\F_{V}$ and
$ (\V_{WTV}^{\perp}\lrar \V_V^{\perp})\circ T+\V_T^{\perp}=\F_{T},$
\\i.e., $(\V_{WTV}^{\perp}\circ TV+\V_T^{\perp}\oplus\V_V^{\perp})^{\perp}=\F^{\perp}_{TV}$\\ iff
$(\V_{WTV}^{\perp}\circ V+\V_V^{\perp})^{\perp}=\F^{\perp}_{V}$ and
$ ((\V_{WTV}^{\perp}\lrar \V_V^{\perp})\circ T+\V_T^{\perp})^{\perp}=\F^{\perp}_{T},$
\\i.e., (using Theorems \ref{thm:perperp}, \ref{thm:sumintersection})
$(\V_{WTV}\times TV)\ \cap\ (\V_T\oplus\V_V)=\0_{TV},$
\\iff
$(\V_{WTV}^{\perp}\circ V)^{\perp}\ \cap \ \V_V=\0_{V}$ and
$ (\V_{WTV}^{\perp}\lrar \V_V^{\perp})^{\perp}\times T\ \cap \ \V_T=\0_{T},$
\\i.e., (using Theorems \ref{thm:perperp}, \ref{thm:dotcrossidentity}, \ref{thm:idt0}) $(\V_{WTV}\times TV)\ \cap\ (\V_T\oplus\V_V)=\0_{TV},$
\\iff
$(\V_{WTV}\times V)\ \cap \ \V_V=\0_{V}$ and
$ (\V_{WTV}\lrar \V_V)\times T\ \cap \ \V_T=\0_{T}.$

Thus $\{\V_{WTV},\V_T\oplus\V_V\}$ satisfies the full sum and zero intersection 
properties iff $\{\V_{WTV},\V_V\},$\\  $\{\V_{WTV}\lrar\V_V,\V_T\}$
 satisfy the full sum and zero intersection
properties.

2. We have $\V_{WTV}\lrar \V_V\equivd (\V_{WTV}\cap \V_V)\circ WT=
(\V_{WTV}+ \V_V)\times WT.$
Rigidity of $\{\V_{WTV}\lrar \V_V, \V_T\}$ 
is equivalent to the full sum and zero intersection conditions\\
$(\V_{WTV}\lrar \V_V)\circ T+ \V_T=\F_T,$
and $(\V_{WTV}\lrar \V_V)\times T\cap \V_T=\0_T.$\\
The full sum condition is equivalent to $(\V_{WTV}\cap \V_V)\circ WT\circ T+\V_T=(\V_{WTV}\cap \V_V)\circ T+\V_T=\F_T$
and the zero intersection condition is equivalent to $(\V_{WTV}+ \V_V)\times WT\times T\cap \V_T=(\V_{WTV}+ \V_V)\times T\cap \V_T=\0_T.$

3. We observe that 
$(\V_{WTV}\cap \V_V)\subseteq (\V_{WTV}+ \V_V),$ so that 
$(\V_{WTV}\cap \V_V)\circ T\subseteq (\V_{WTV}+ \V_V)\circ T$ and $(\V_{WTV}+ \V_V)\times T\supseteq (\V_{WTV}\cap \V_V)\times T.$
\\ Therefore, if $\{\V_{WTV}\cap \V_V,\V_T\}$ satisfies the full
sum condition, so does  $\{\V_{WTV}+ \V_V,\V_T\}$
and if\\ $\{\V_{WTV}+ \V_V,\V_T\}$ satisfies the zero
intersection condition, so does  $\{\V_{WTV}\cap \V_V,\V_T\}.$
\\ 
The result now follows from the previous part.
\end{proof}
We will first use Theorem \ref{lem:derivedregularity} to consider the rigidity of a multiport when we assign a device characteristic to some of its ports. The statement of this result needs an extension of the definition of a multiport.

\begin{definition}
A \nw{generalized multiport} $\N^g_A$ is an ordered pair $(\V_{AB},\K_B),A\cap B=\emptyset.$ 
 where $\K_{B}$ is an arbitrary collection of real vectors on $B.$
We say $\N^g_A$ is \nw{linear} iff $\K_{B}=\A_{B},$
 that it is \nw{rigid} iff $\{\V_{AB},\A_{B}\}$
 is rigid.
\end{definition}

%======================================================================
%
%\begin{definition}
%A \nw{generalized multiport} $\N^g_P$ is an ordered pair $(\V_{S'S"P'P"},\K_{S'S"}),$ where $\K_{S'S"}$ is an arbitrary collection of real vectors on 
%$S'\uplus S".$ We say $\N^g_P$ is \nw{linear} iff $\K_{S'S"}=\A_{S'S"},$
% that it is \nw{rigid} iff $\{\V_{S'S"P'P"},\A_{S'S"}\}$
% is rigid and that it is \nw{proper} iff it is rigid and $\V_{S'S"P'P"},\A_{S'S"},$ are proper. 
%\end{definition}
We note that a multiport $\N_P\equivd (\G_{SP},\K_{S'S"})$ can be regarded
as a 
generalized multiport \\$\N^g_{P'P"}\equivd ((\V^v(\G_{SP}))_{S'P'}\oplus (\V^i(\G_{SP}))_{S"P"},\K_{S'S"}).$ We will say, in this case (abusing notation),
 $\N_P=\N^g_{P'P"}.$
\\
Similarly, a network $\N\equivd (\G_{SP},\K_{S'S"})$ can be regarded
as a
generalized network \\$\N^g\equivd ((\V^v(\G_{S}))_{S'}\oplus (\V^i(\G_{S}))_{S"},\K_{S'S"}).$\\
Further, a generalized network can be regarded as a generalized multiport $\N^g_{P'P"},$ with $P=\emptyset.$
\begin{theorem}
\label{thm:derivedregularity}
Let $\N^g_{W}\equivd (\V_{WTV},\A_{T}\oplus \A_{V})$ be a generalized multiport 
and let $\V_{T},\V_{V}$ be vector space translates of $\A_{T},\A_{V},$
 repectively.
Then $\N^g_{W}$ is rigid iff the generalized multiports $\N^{1g}_{WT}\equivd (\V_{WTV},\V_{V})$
and  $\N^{2g}_W\equivd ((\V_{WTV}\lrar \V_{V}), \V_{T}),
$ are rigid.
%$ are rigid.
\end{theorem}
\begin{proof}
We note that, by the definition of rigidity, a  pair $\{\V_{AB},\A_B\},$ is rigid iff $\{\V_{AB},\V_B\},$ is rigid, where $\V_B$ is the vector space translate of 
$\A_B.$
Therefore $\N^g_{W}\equivd (\V_{WTV},\A_{T}\oplus \A_{V})$ is rigid iff
$\{\V_{WTV},\V_{T}\oplus \V_{V}\}$ is rigid.
The result now follows  from Theorem \ref{lem:derivedregularity}.
%taking $W$ to be $P'\uplus P",$ $V$ to be $M'\uplus M",$ and $T$ to be $S'\uplus S".$
\end{proof}

We next present conditions for connections of rigid multiports 
to result in a rigid multiport.
The result essentially states that the connection of two rigid multiports 
through an affine space  results in a rigid multiport iff the multiport behaviours are rigid with respect to the affine space.
\begin{theorem}
\label{thm:regularmultiportrecursive}
Let $\N^g_{A}\equivd (\V_{AB},\A_B),{\N}^g_{C}\equivd (\V_{CD},\A_D)$
be generalized multiports with $A,B,C,D,$ pairwise disjoint.
Let $\V_B,\V_D$ be vector space translates of $\A_B,\A_D$ respectively.
\\Let $[\N^g_{A}\oplus {\N}^g_{C}]\cap \A_{A_1C_1}\equivd ((\V_{AB}\oplus \V_{CD}),(\A_B\oplus \A_D\oplus \A_{A_1C_1})), A_1\subseteq A, C_1\subseteq C,$
%be obtained by connecting $\N^g_{A},{\N}^g_{C}$
%through $\A_{A_1C_1}.$
and let \\ $\N^g_{(A-A_1)(C-C_1)}\equivd ((\V_{AB}\lrar\V_B)\oplus (\V_{CD}\lrar\V_D),\V_{A_1C_1}).$ 
Then the generalized multiport 
$[\N^g_{A}\oplus {\N}^g_{C}]\cap \A_{A_1C_1}$ is rigid iff the generalized 
multiports 
$\N^g_{A},{\N}^g_{C}$ and $\N^g_{(A-A_1)(C-C_1)}$ are rigid.
%
%
%${[\N_{RP}\oplus {\N}_{\tilde{P}Q}]\cap \A^{P\tilde{P}}}$
% is rigid iff $\N_{RP},{\N}_{\tilde{P}Q}$ and $\N^g_{RQ}$ are rigid.
%
%
%===============================================================
%Let $\N_{RP},{\N}_{\tilde{P}Q}$ be multiports on graphs 
%$\G_{RSP},\G_{\tilde{P}MQ}$ respectively and device characteristics
%$\A_{S'S"},{\A}_{M'M"}$ respectively, with $R,S,P,\tilde{P},M,Q$ being pairwise disjoint.
%\\Let the multiport ${[\N_{RP}\oplus {\N}_{\tilde{P}Q}]\cap \A^{P\tilde{P}}},$
%with ports $R\uplus Q$ be obtained by connecting $\N_{RP},{\N}_{\tilde{P}Q}$
%through $\A^{P\tilde{P}}$.
%\\
%Let $\breve{\A}_{R'P'R"P"},\breve{\A}_{\tilde{P}'Q'\tilde{P}"Q"}$
%be the multiport behaviours of $\N_{RP},{\N}_{\tilde{P}Q}$ respectively.\\
%Let $\A_{R'P'R"P"}\equivd (\breve{\A}_{R'P'R"P"})_{R'P'(-R")(-P")},$
%let $\A_{\tilde{P}'Q'\tilde{P}"Q"}\equivd (\breve{\A}_{\tilde{P}'Q'\tilde{P}"Q"})_{\tilde{P}'Q'(-\tilde{P}")(-Q")},$ and let\\ $\V_{R'P'R"P"},\V_{\tilde{P}'Q'\tilde{P}"Q"}$ respectively be their vector space translates.\\
%Let $\N^g_{RQ}\equivd (\V_{R'P'R"P"}\oplus \V_{\tilde{P}'Q'\tilde{P}"Q"},\A^{P\tilde{P}}_{P'P"\tilde{P}'\tilde{P}"}).$
%
%Then ${[\N_{RP}\oplus {\N}_{\tilde{P}Q}]\cap \A^{P\tilde{P}}}$
% is rigid iff $\N_{RP},{\N}_{\tilde{P}Q}$ and $\N^g_{RQ}$ are rigid.
% and
%$[\A_{R'P'R"P"}\oplus \A_{\tilde{P}'Q'\tilde{P}"Q"},\A^{P\tilde{P}}_{P'P"\tilde{P}'\tilde{P}"}]$ is regular.
\end{theorem}
\begin{proof}
By the definition of rigidity,  the generalized multiport 
$((\V_{AB}\oplus \V_{CD}),(\A_B\oplus \A_D\oplus \A_{A_1C_1}))$ is rigid iff
$\{(\V_{AB}\oplus \V_{CD}),(\V_B\oplus \V_D\oplus \V_{A_1C_1})\}$ is rigid.
We now apply Theorem \ref{lem:derivedregularity}, taking
\\$W=(A-A_1)\uplus (C-C_1),V=B\uplus D, T=A_1\uplus C_1,
\V_{WTV}= \V_{AB}\oplus \V_{CD}, \V_V= \V_B\oplus \V_D,\V_T= \V_{A_1C_1}.$
We have, $\{(\V_{AB}\oplus \V_{CD}),(\V_B\oplus \V_D\oplus \V_{A_1C_1})\}$ is rigid \\iff $\{(\V_{AB}\oplus \V_{CD}), (\V_B\oplus \V_D)\},\{(\V_{AB}\oplus \V_{CD})\lrar (\V_B\oplus \V_D), \V_{A_1C_1}\},$ are rigid.
\\Now 
$\{(\V_{AB}\oplus \V_{CD}), (\V_B\oplus \V_D)\}$ is rigid iff 
$\{\V_{AB},\V_B\},\{\V_{CD},\V_D)\}$ are rigid and \\
$((\V_{AB}\oplus \V_{CD})\lrar (\V_B\oplus \V_D), \V_{A_1C_1})
=
%$ is rigid iff 
((\V_{AB}\lrar \V_B)\oplus (\V_{CD}\lrar \V_D),\V_{A_1C_1}).$ 
\\ Next $(\V_{AB},\V_B),(\V_{CD},\V_D)$ are rigid  iff
$(\V_{AB},\A_B),(\V_{CD},\A_D)$ are rigid. 
\\The result follows.
\end{proof}

\subsection{Rigidity and complementary orthogonality}
\label{sec:rigiddual}
We show in this subsection that
when we have complementary orthogonality, we automatically have rigidity.
This suggests 
 rigidity can be regarded as  a weak form of complementary 
orthogonality.

\begin{lemma}
\label{lem:gyrator_r}
Let $\V_{AB},\V_{BC}$ be vector spaces on $A\uplus B,B\uplus C,$ with
$A,B,C,$ being pairwise disjoint. Let $\tilde{B}$ be a copy  of $B,$ disjoint
from $A,B,C$ and let $ \V_{\tilde{B}C}\equiv (\V_{BC})_{\tilde{B}C}.$
Let $\V_{B\tilde{B}}$
have the representative matrix $(K|I),$ where $K$ corresponds to columns $B$ and is invertible.
Let $\hat{\V}_{BC}\equivd \V_{\tilde{B}C} \lrar \V_{B\tilde{B}}.$
We then have the following.
\begin{enumerate}
\item
(a)$\V_{\tilde{B}C}= \hat{\V}_{BC}\lrar \V_{B\tilde{B}}.$
\\
(b)$\hat{\V}_{BC}^{\perp}=(\V^{\perp}_{\tilde{B}C} \rightleftharpoons \V^{\perp}_{B\tilde{B}}).$
\\
(c) $\V^{\perp}_{\tilde{B}C}= (\hat{\V}_{BC}^{\perp}\rightleftharpoons \V^{\perp}_{B\tilde{B}}).$
\item $\{\V_{AB}, \hat{\V}_{{B}C}\}$ is rigid iff $\{(\V_{AB}\oplus \V_{\tilde{B}C}), \V_{B\tilde{B}}\}$ is rigid.
\item
Let $r(\V_{AB}\circ B)+r(\V_{BC}\times B)= r(\V_{AB}\circ B+\V_{BC}\times B)=|B|,$
$r(\V^{\perp}_{AB}\times B)+r(\V^{\perp}_{BC}\circ B)= 
r(\V^{\perp}_{AB}\times B+\V^{\perp}_{BC}\circ B)
=|B|.$
Then,
  $\{\V_{AB}, \V_{BC}\}$ is rigid.
%(b) $\{\V_{AB}, (\V^{\perp}_{AB})_{\tilde{A}B}\},$ is rigid, where $\tilde{A}$ is 
%a copy of $A$ disjoint from $A,B;$ 
%(c)  if $\V_{AB},\V_{BC}$ are real vector spaces, $\{\V_{AB}, \hat{\V}_{\tilde{B}C}\}$ is rigid if  $K$ is symmetric positive 
%or negative definite.
\item
Let $\V_{AB},\V_{BC}$ be real vector spaces and let $(\V_{AB}\circ B)^{\perp}= \V_{BC}\times B.$
Then,
(a)  $\{\V_{AB}, \V_{BC}\}$ is rigid;
(b) $\{\V_{AB}, (\V^{\perp}_{AB})_{\tilde{A}B}\},$ is rigid, where $\tilde{A}$ is
a copy of $A$ disjoint from $A,B;$
(c)  $\{\V_{AB}, \hat{\V}_{{B}C}\}$ is rigid if  $K$ is symmetric positive
or negative definite.
\end{enumerate}
\end{lemma}
\begin{proof}
1(a). It is clear that   $\V_{B\tilde{B}}\circ \tilde{B}=\F_{\tilde{B}}.$
Further, since $K$ is invertible, it is clear that $\V_{B\tilde{B}}\times \tilde{B}=\0_{\tilde{B}}.$
By Theorem \ref{thm:IITlinear}, since $\V_{B\tilde{B}}\circ \tilde{B}\supseteq \V_{\tilde{B}C}\circ \tilde{B}$ and $ \V_{B\tilde{B}}\times \tilde{B}\subseteq \V_{\tilde{B}C}\times \tilde{B},$ from $\hat{\V}_{BC}=\V_{\tilde{B}C} \lrar \V_{B\tilde{B}}, $ it follows that 
$\V_{\tilde{B}C}=(\hat{\V}_{BC}\lrar \V_{B\tilde{B}}).$
%\\We have 
%$\hat{\V}_{BC}\equivd \V_{\tilde{B}C} \lrar \V_{B\tilde{B}}.$
%\\Suppose $(h_{\tilde{B}},g_C)\in \V_{\tilde{B}C}.$ Since $\V_{B\tilde{B}}\circ \tilde{B}=\F_{\tilde{B}},$ there exists $f_B$ such that $(f_B,h_{\tilde{B}})\in 
%\V_{B\tilde{B}}.$ Therefore $(f_B,g_C)\in (\V_{\tilde{B}C} \lrar \V_{B\tilde{B}})= \hat{\V}_{BC}.$
%Now $(f_B,g_C)\in \hat{\V}_{BC}, (f_B,h_{\tilde{B}})\in 
%\V_{B\tilde{B}}.$ \\ Therefore, $(h_{\tilde{B}},g_C)\in (\hat{\V}_{BC}\lrar \V_{B\tilde{B}}).$ Thus $\V_{\tilde{B}C}\subseteq (\hat{\V}_{BC}\lrar \V_{B\tilde{B}}).$
%
%Next, let $(h_{\tilde{B}},g_C)\in (\hat{\V}_{BC}\lrar \V_{B\tilde{B}}).$
%Then there exists $f_B$ such that $(f_B, g_C)\in \hat{\V}_{BC},
%(f_B,h_{\tilde{B}})\in \V_{B\tilde{B}}.$
%Since $\hat{\V}_{BC}= (\V_{\tilde{B}C} \lrar \V_{B\tilde{B}}),$ it follows
%that there exists $h_{\tilde{B}}'$ such that $(h_{\tilde{B}}',g_C)\in \V_{\tilde{B}C}, (f_B,h_{\tilde{B}}')\in \V_{B\tilde{B}}.$
%Therefore $(0_B,h_{\tilde{B}}-h_{\tilde{B}}')\in \V_{B\tilde{B}}\times \tilde{B}=\0_{\tilde{B}}$
%and $h_{\tilde{B}}=h_{\tilde{B}}'.$
%Therefore, $(h_{\tilde{B}},g_C)\in\V_{\tilde{B}C}.$\\
%Thus $\V_{\tilde{B}C}\supseteq (\hat{\V}_{BC}\lrar \V_{B\tilde{B}}).$
%We conclude that $\V_{\tilde{B}C}=(\hat{\V}_{BC}\lrar \V_{B\tilde{B}}).$

We are given that $\hat{\V}_{BC}\equivd \V_{\tilde{B}C} \lrar \V_{B\tilde{B}}.$
Therefore, by Theorem \ref{thm:idt0}, and (a) above, we have
1(b) and 1(c).

2. We have $\hat{\V}_{BC}\equivd \V_{\tilde{B}C} \lrar \V_{B\tilde{B}}.$ 
We note that, when $P,Q,R,S,T$ are pairwise disjoint,\\
$(\V_{PQ}\lrar\V_{RS})\lrar\V_{ST}=(\V_{PQ}\lrar\V_{RS})\lrar\V_{ST}.$
Therefore $\hat{\V}_{BC}\times B= \hat{\V}_{BC}\lrar \0_C$\\$=  \V_{\tilde{B}C} \lrar \V_{B\tilde{B}}\lrar \0_C= \V_{\tilde{B}C} \lrar\0_C\lrar \V_{B\tilde{B}}=\V_{\tilde{B}C}\times \tilde{B}\lrar\V_{B\tilde{B}}.$
We are given that $(K|I)$ is a representative matrix of $\V_{B\tilde{B}}.$
Therefore,
$f_B \in (\V_{AB}\times B)\cap (\hat{\V}_{BC}\times B),$
iff $f_B\in (\V_{AB}\times B)$ and $f_B^T= g_B^TK,$ with $g_B\in ({\V}_{BC}\times B),$ i.e., iff $f_B\in (\V_{AB}\times B), g_{\tilde{B}}\in ({\V}_{BC}\times B)_{\tilde{B}},$ and $f_B^T= g_{\tilde{B}}^TK,$ i.e., iff \\$(f_B,g_{\tilde{B}})\in 
(\V_{AB}\times B \oplus ({\V}_{BC}\times B)_{\tilde{B}})\cap \V_{B\tilde{B}},$
i.e., iff $(f_B,g_{\tilde{B}})\in 
((\V_{AB}\oplus {\V}_{\tilde{B}C})\times B{\tilde{B}})\cap \V_{B\tilde{B}}.$
Thus the zero intersection property holds for $\{\V_{AB},\hat{\V}_{BC}\}$
iff it holds for $\{(\V_{AB}\oplus {\V}_{\tilde{B}C}),\V_{B\tilde{B}}\}.$

Next, we have $\hat{\V}_{BC}^{\perp}= (\V^{\perp}_{\tilde{B}C} \rightleftharpoons \V^{\perp}_{B\tilde{B}})= (\V^{\perp}_{\tilde{B}C} \lrar \V^{\perp}_{(-B)\tilde{B}})$ (by part 1 above), $((K^T)^{-1}|I)$ is a representative matrix 
for $\V^{\perp}_{(-B)\tilde{B}},$
and $(\V_{AB}\oplus {\V}_{\tilde{B}C})^{\perp}=\V_{AB}^{\perp}\oplus {\V}_{\tilde{B}C}^{\perp}.$ Therefore, by the same argument as above, it follows that
zero intersection property holds for $\{\V^{\perp}_{AB},\hat{\V}^{\perp}_{BC}\}$
iff it holds for $\{(\V^{\perp}_{AB}\oplus {\V}^{\perp}_{\tilde{B}C}),\V^{\perp}_{(-B)\tilde{B}}\},$ i.e., for $\{(\V^{\perp}_{AB}\oplus {\V}^{\perp}_{\tilde{B}C}),\V^{\perp}_{B\tilde{B}}\}.$
By part 4 of Theorem \ref{thm:regularrecursive}, we therefore have that the full sum property holds for $\{\V_{AB},\hat{\V}_{BC}\}$
iff it holds for $\{(\V_{AB}\oplus {\V}_{\tilde{B}C}),\V_{B\tilde{B}}\}.$

Thus $\{\V_{AB},\hat{\V}_{BC}\}$ is rigid iff $\{(\V_{AB}\oplus {\V}_{\tilde{B}C}),\V_{B\tilde{B}}\}$ is rigid.

3. Since $\V_{XY}\circ Y\supseteq \V_{XY}\times Y,$ for any vector space  $\V_{XY},$ it follows from the hypothesis, that the full sum condition holds for
$\{\V_{AB},\V_{BC}\}$ and for $\{\V^{\perp}_{AB},\V^{\perp}_{BC}\}.$
By part 4 of Theorem \ref{thm:regularrecursive}, it follows that
 both the full sum and the zero intersection properties hold for
$\{\V_{AB},\V_{BC}\}.$
We conclude that $\{\V_{AB},\V_{BC}\}$ is rigid.

4(a).
Since $\V_{AB}\times B\subseteq \V_{AB}\circ B,$ and $(\V_{AB}\circ B)^{\perp}=\V_{BC}\times B,$ we must have that $(\V_{AB}\times B)^{\perp}\supseteq \V_{BC}\times B.$ 
For any real vector space $\V_X,$ we have $\V_X\cap \V_X^{\perp}=\0_X$
 (no nonzero real vector can be orthogonal to itself).
 Since the vector spaces under consideration are real, it is therefore clear that
$\V_{AB}\times B\cap \V_{BC}\times B=\0_B.$ \\
Next $(\V_{AB}\circ B)^{\perp}=\V_{BC}\times B$
implies that $(\V_{AB}^{\perp}\times B)=\V^{\perp}_{BC}\circ B.$
Arguing as before, it is clear that \\$\V_{AB}^{\perp}\times B\cap \V_{BC}^{\perp}\times B=\0_B.$
Therefore by part 4 of Theorem \ref{thm:regularrecursive},
the full sum property holds for $\{\V_{AB}, \V_{BC}\}.$
We conclude that $\{\V_{AB}, \V_{BC}\}$ is rigid.

4(b). We note that $(\V_{AB}\circ B)^{\perp}=\V_{AB}^{\perp}\times B=
(\V^{\perp}_{AB})_{\tilde{A}B}\times B.$
The result now follows from 4(a) above.

4(c). 
%We note that, when $P,Q,R,S,T$ are pairwise disjoint,
%$(\V_{PQ}\lrar\V_{RS})\lrar\V_{ST}$\\$=(\V_{PQ}\lrar\V_{RS})\lrar\V_{ST}.$
We have,
$\hat{\V}_{BC}\times B= \hat{\V}_{BC}\lrar \0_C=
(\V_{\tilde{B}C} \lrar \V_{B\tilde{B}})\lrar \0_C=(\V_{\tilde{B}C}\lrar \0_C)\lrar \V_{B\tilde{B}}$\\$= \V_{\tilde{B}C}\times \tilde{B}\lrar \V_{B\tilde{B}}.$
%Let $\hat{\V}_{BC}\times B\equivd (\V_{\tilde{B}C} \times \tilde{B}\lrar \V_{B\tilde{B}}) .$ 
Since $(K|I)$ is a representative matrix of $\V_{B\tilde{B}},$ it is clear that $f_B\in \hat{\V}_{BC}\times B$ iff there exists $g_{\tilde{B}}\in \V_{\tilde{B}C}\times \tilde{B},$ such that $f^T_B=g_{\tilde{B}}^TK.$
If $M$ is a  representative matrix for $\V_{\tilde{B}C}\times \tilde{B},$
i.e., for $({\V}_{BC}\times B)_{\tilde{B}},$
%Therefore $M$ is also a
%a representative matrix for ${\V}_{BC}\times B.$
%Since 
%$K$ is invertible, it follows from the fact that  
%$\hat{\V}_{BC}\times B= \V_{\tilde{B}C}\times \tilde{B}\lrar \V_{B\tilde{B}},$ 
it follows that
$ MK$ is a representative matrix for $\hat{\V}_{BC}\times B.$
Let $N$ be a representative matrix of $\V_{AB}\times B.$
Since $\V_{AB}\times B\subseteq \V_{AB}\circ B,$ and $(\V_{AB}\circ B)^{\perp}=\V_{BC}\times B,$ it follows that $NM^T=0.$
Suppose $f_B\in (\V_{AB}\times B)\cap (\hat{\V}_{BC}\times B).$
Then $f_B^T=\lambda ^T N= \sigma^T MK,$ for some $\lambda , \sigma.$  Therefore
$f_B^TM^T\sigma = \lambda ^T NM^T\sigma = \sigma^T MKM^T\sigma =0.$
Since $K$ is positive or negative definite, so must $MKM^T$ be, and we must have $\sigma =0.$
 Therefore $f_B=\sigma^T MK=0.$ Thus $(\V_{AB}\times B)\cap (\hat{\V}_{BC}\times B)=\0_B.$
\\
To prove  that  the zero intersection property holds
 for $\{\V_{AB}^{\perp},\hat{\V}_{BC}^{\perp}\},$
we note that $\hat{\V}_{BC}^{\perp}=
{\V}_{{B}C}^{\perp}\lrar \V_{(-B)\tilde{B}}^{\perp},$  by part 1 above,
and since $(\V_{AB}\circ B)^{\perp}= \V_{BC}\times B,$ we must have
$((\V_{AB}\circ B)^{\perp})^{\perp}= (\V_{BC}\times B)^{\perp},$
i.e., $(\V_{AB}^{\perp}\times B)^{\perp}= \V_{BC}^{\perp}\circ B\supseteq \V_{BC}^{\perp}\times B.$
%Therefore, $(\V_{AB}^{\perp}\times B)^{\perp}\supseteq \V_{BC}^{\perp}\times B.$
%$(\V_{BC}^{\perp}\circ B)\subseteq (\V_{AB}^{\perp}\times B)^{\perp}.$
%where
%$\hat{\V}_{BC}\equivd (\V_{\tilde{B}C} \lrar \V_{B\tilde{B}}).$
%Further, since $\hat{\V}_{BC}={\V}_{BC}\lrar  \V_{B\tilde{B}},$ we must have
%$\hat{\V}_{BC}^{\perp}=({\V}_{BC}\lrar \V_{B\tilde{B}})^{\perp}=
%{\V}_{BC}^{\perp}\lrar\V_{B\tilde{B}}^{\perp}.$
Next since $(K|I)$ is a representative matrix of $\V_{B\tilde{B}},$
and $K$ is symmetric positive or negative definite,
$(K^{-1}|I)$ is a representative matrix of $\V_{(-B)\tilde{B}}^{\perp}$
and $K^{-1}$ is symmetric positive or negative definite.
Therefore, using the same argument as the one for proving  the zero intersection property
of $\{\V_{AB},\hat{\V}_{BC}\}$
we can show that  the zero intersection property holds
 $\{\V_{AB}^{\perp},\hat{\V}_{BC}^{\perp}\}.$
By part 4 of Theorem \ref{thm:regularrecursive}, this means that  the full sum property holds for  $\{\V_{AB},\hat{\V}_{BC}\}.$
Thus $\{\V_{AB},\hat{\V}_{BC}\}$
is rigid.
\end{proof}
Lemma \ref{lem:gyrator_r} implies the following result about networks.

We need a few preliminary definitions.
%\begin{definition}
Let $\N^g_{A}\equivd (\V_{AB},\A_B).$
We say $\N^g_{\tilde{A}}$ is a \nw{disjoint copy} of $\N^g_{A},$ iff
$\N^g_{\tilde{A}}\equivd ((\V_{AB})_{\tilde{A}\tilde{B}},\A_{\tilde{B}}),$
where $\tilde{A},\tilde{B} $ are copies respectively of $A,B,$ with
$A,B, \tilde{A},\tilde{B} $ being pairwise disjoint.
\\The \nw{dual} of $\N^g_{A},$ denoted $\N^{gdual}_{A},$
is defined by $\N^{gdual}_{A}\equivd (\V^{\perp}_{AB},\V^{\perp}_B).$
\\It is clear that the dual of $\N^{gdual}_{A}$ is $(\N^g_A)^{hom}\equivd (\V_{AB},\V_B),
$ where $\V_B$ is the vector space associate of $\A_B.$
%\end{definition}
\begin{theorem}
Let $A,B,D,\tilde{A}, \tilde{B}, \tilde{D},$ be pairwise disjoint with $\tilde{A},\tilde{B},\tilde{D},$ being copies respectively of $A,B,D.$
Let $\N^g_{AB}\equivd (\V_{ABD},\A_D)$ and let
$\N^{gdual}_{\tilde{A}\tilde{B}}$ be a disjoint copy of $\N^{gdual}_{AB}.$
\begin{enumerate}
\item If $\N^g_{AB}$ is rigid,   then  the generalized
multiport
$[\N^g_{AB}\oplus \N^{gdual}_{\tilde{A}\tilde{B}}]\cap I_{B\tilde{B}}$ is also rigid,
where
$I_{B\tilde{B}}$ has the representative matrix $(I|I)$ with the
identity submatrices corresponding to columns $B,\tilde{B}.$
\item If $\N^g_{AB}$ is rigid,   then  the generalized
multiport
$[\N^g_{AB}\oplus \N^{gdual}_{\tilde{A}\tilde{B}}]\cap \V_{B\tilde{B}}$ is also rigid,
where
$\V_{B\tilde{B}}$ has the representative matrix $(K|I),$ 
$K$ symmetric positive or negative definite, with the
identity submatrix corresponding to columns $\tilde{B}.$
\end{enumerate}
\end{theorem}
\begin{proof}
1. By  part 4 of Theorem \ref{thm:regularrecursive},
%part 2 of Lemma \ref{lem:gyrator_r2}, 
if  $\N^g_{AB}\equivd (\V_{ABD},\A_D)$ is rigid, then so is
$ \N^{gdual}_{A{B}}\equivd (\V^{\perp}_{ABD},\V^{\perp}_D).$
%(taking $C$ of the lemma to be $\emptyset,B$ to be $D, $ $A$ to be $AB$).
\\Therefore the disjoint copy $\N^{gdual}_{\tilde{A}\tilde{B}}$ of
$ \N^{gdual}_{A{B}}$ is rigid. Since $\N^g_{AB}, \N^{gdual}_{\tilde{A}\tilde{B}}$ are rigid, 
 by Theorem \ref{thm:regularmultiportrecursive},\\ $[\N^g_{AB}\oplus \N^{gdual}_{\tilde{A}\tilde{B}}]\cap I_{B\tilde{B}}$ is rigid iff $\{(\V_{ABD}\lrar \V_D)\oplus 
(\V^{\perp}_{\tilde{A}\tilde{B}\tilde{D}}\lrar \V^{\perp}_{\tilde{D}}), I_{B\tilde{B}}\}$ is rigid, where $\V_D$ is the vector space translate of $\A_D.$
By Theorem \ref{thm:idt0}, $(\V_{ABD}\lrar \V_D)^{\perp}=
\V^{\perp}_{ABD}\lrar \V^{\perp}_D.$
\\Let $\V_{AB}\equivd \V_{ABD}\lrar \V_D, {\V}'_{AB}\equivd (\V^{\perp}_{ABD}\lrar \V^{\perp}_D)$ and let $ {\V}'_{\tilde{A}{B}}, {\V}'_{\tilde{A}\tilde{B}}$ be copies of ${\V}'_{AB}.$
\\By part 4(b) of Lemma \ref{lem:gyrator_r}, $\{\V_{AB}, {\V}'_{\tilde{A}B}\}$ is rigid.
By part 2 of Lemma \ref{lem:gyrator_r},
%It is easy to see that 
$\{\V_{AB}, {\V}'_{\tilde{A}B}\}$ is rigid iff\\ 
$\{\V_{AB}\oplus {\V}'_{\tilde{A}\tilde{B}},I_{B\tilde{B}}\}$ 
is rigid.
Thus $\{\V_{AB}\oplus {\V}'_{\tilde{A}\tilde{B}},I_{B\tilde{B}}\}\equivd\{(\V_{ABD}\lrar \V_D)\oplus 
(\V^{\perp}_{\tilde{A}\tilde{B}\tilde{D}}\lrar \V^{\perp}_{\tilde{D}}), I_{B\tilde{B}}\}$ is rigid and, as we saw before, this is equivalent
to  $[\N^g_{AB}\oplus \N^{gdual}_{\tilde{A}\tilde{B}}]\cap I_{B\tilde{B}}$ being rigid.
\\2. By part 2 of Lemma \ref{lem:gyrator_r},
$\{\V_{AB}\oplus {\V}'_{\tilde{A}\tilde{B}},I_{B\tilde{B}}\}$
is rigid iff $\{\V_{AB}, {\V}'_{\tilde{A}{B}}\}$ is rigid.
By part 4(c) of Lemma \ref{lem:gyrator_r}, 
$\{\V_{AB}, {\V}'_{\tilde{A}{B}}\}$ is rigid iff 
$\{\V_{AB}, \hat{\V}_{\tilde{A}{B}}\}$ is rigid
where $\hat{\V}_{\tilde{A}{B}}=({\V}'_{\tilde{A}{B}})_{\tilde{A}\tilde{B}}\lrar \V_{B\tilde{B}}.$ By part 2 of  Lemma \ref{lem:gyrator_r},
$\{\V_{AB}, \hat{\V}_{\tilde{A}{B}}\}$ is rigid iff 
 $\{\V_{AB}\oplus {\V}'_{\tilde{A}\tilde{B}},\V_{B\tilde{B}}\} $ is rigid. But this is the same as $[\N^g_{AB}\oplus \N^{gdual}_{\tilde{A}\tilde{B}}]\cap \V_{B\tilde{B}}$ being rigid.
\end{proof}

\begin{remark}
We have that $\{\V_X,\V_X^{\perp}\}$ is rigid.  By Theorem \ref{thm:idt0}, $\{\V_A\lrar \V_B, \V^{\perp}_A\lrar \V^{\perp}_B\},B\subseteq A,$ is also rigid.
We have that $\{\V^1_A+ \V^2_A, (\V^1_A)^{\perp}\cap  (\V^2_A)^{\perp}\}$ is rigid since $(\V^1_A+ \V^2_A)^{\perp}=(\V^1_A)^{\perp}\cap  (\V^2_A)^{\perp}.$
\\There appear to be no counterparts to these facts in terms of rigidity
without using complementary orthogonality.
\end{remark}

\subsection{Rigid families of affine spaces}
\label{subsec:rigidfamiliesaffine}

The motivation for the notion of rigidity is to capture the properties of existence
and uniqueness of solutions for linear networks treating them as
a connection of multiports. The rigidity of pairs of affine spaces appears to be the
right starting point. We need to
carry this idea over to whole families of affine spaces in order to
speak of many multiports going to make up a larger network or multiport.
For this purpose  we first give a definition of rigidity
for families relevant to our problem (`associative families') and then justify the definition by describing
the properties of rigid associative families through subsequent lemmas.
Finally in Theorem \ref{thm:associativerigidrecursiven}, we have a recursive result which expresses the rigidity
of an associative  family of affine spaces in terms of a partition 
into rigid subfamilies and the rigidity of another family where the blocks of the 
partition are treated as appropriate  derived affine spaces.
\begin{definition}
\label{def:gen_lrar1}
An  \nw{associative family}  $\mathcal{H}\equivd \{\K_{Y_i},i=1, \cdots , m\}$ is a family of collections of vectors where no element in $Y\equivd \bigcup_{i=1}^mY_i$ belongs to more than two of the $Y_i.$ 
When $Y_i=\emptyset,$ we will take $\K_{Y_i}=\K_{\emptyset}\equivd \emptyset .$
Let $\mathcal{H}_{nonvoid}\equivd \{\K_{Y_i},i=1, \cdots , t\}$ be the subfamily 
of $\mathcal{H}$ whose elements are non void collections that belong to 
$\mathcal{H}.$ 
%We will permit $\K_{\emptyset}\equivd \emptyset $ to be one of the 
%elements of $\mathcal{H}.$ 
For $i=1, \cdots , t,$ let $Z_i$ be the subset of elements of $Y_i$ 
which belong to none of the $Y_j, j\ne i$ and let $Z\equivd \bigcup_{i=1}^tZ_i.$\\
We define $\mnw{\lrar(\mathcal{H})} \equivd$\\$\{(f_{Z_1}, \cdots , f_{Z_t}):(f_{Z_i},f_{Y_i-Z_i})\in \K_{Y_i},i=1, \cdots , t \ \mbox{and} \ f_{Y_i-Z_i}(e)=f_{Y_j-Z_j}(e) \ \mbox{whenever}\ e \in Y_i\cap Y_j, i\ne j  \}.$
We define
 $\mnw{\rightleftharpoons(\mathcal{H})}\equivd$\\$\{(f_{Z_1}, \cdots , f_{Z_t}):(f_{Z_i},f_{Y_i-Z_i})\in \K_{Y_i},i=1, \cdots , t \ \mbox{and} \ f_{Y_i-Z_i}(e)=-f_{Y_j-Z_j}(e) \ \mbox{whenever}\ e \in Y_i\cap Y_j, i\ne j   \}.$
\end{definition}
\begin{example}
\label{eg:associativefamily}
Let $\mathcal{H}$ be a $0,1$ matrix such that every column has atmost two nonzero elements
 as shown below
\begin{align}
\mathcal{H}=\ppmatrix{
1&0&1&0&1&0&1&0&0\\
0&0&1&0&1&0&1&0&0\\
0&0&0&0&0&0&0&0&1\\
0&1&0&1&0&1&0&1&0\\
0&1&0&1&0&0&0&1&0\\
0&0&0&0&0&0&0&0&0}.
\end{align}
Let us name the columns as $e_1, \cdots, e_9$ and 
 treat the rows of $\mathcal{H}$ as characteristic vectors of sets of columns. 
Then the sets corresponding to rows $1$ to $9$ would be 
$Y_1\equivd\{e_1,e_3,e_5,e_7\}, Y_2\equivd\{e_3,e_5,e_7\},Y_3\equivd\{e_9\},Y_4\equivd\{e_2,e_4,e_6,e_8\},Y_5\equivd\{e_2,e_4,e_8\},Y_6\equivd\emptyset .$
Let $\mathcal{H}\equivd \{\K_{Y_1}, \K_{Y_2},\K_{Y_3},\K_{Y_4},\K_{Y_5},\K_{Y_6}\}.$
Then $\mathcal{H}$ is an associative family.
\end{example}

One could equivalently use a graph (with selfloops permitted)
to construct an associative family, using the sets in the family to be the 
edges incident at the different vertices.
\begin{definition}
For an associative family $\mathcal{H}\equivd \{\K_{Y_1}, \cdots , \K_{Y_m}\},$
the graph $\mnw{\G_{\mathcal{H}}}$ is defined to have  vertices $\K_{Y_1}, \cdots , \K_{Y_m},$
edges $Y_1\cup \cdots \cup Y_m$ with an edge $e$ lying between 
vertices $\K_{Y_i},\K_{Y_j}$ iff $e\in Y_i\cap Y_j.$
We will refer to $\G_{\mathcal{H}}$ as the graph of $\mathcal{H}.$
\end{definition}
\begin{remark}
We note the following.
\begin{enumerate}
\item Both the operations $\lrar(\cdot)$ and $\rightleftharpoons(\cdot ),$ are well defined over associative families and they agree with the corresponding 
operations over pairs of collections of vectors.
If $\mathcal{H}\equivd \{\K_{Y_1},\K_{Y_2},\K_{Y_3}\},$ is an associative family, it is clear that
\\$\lrar(\mathcal{H})= (\K_{Y_1}\lrar\K_{Y_2})\lrar \K_{Y_3}= \K_{Y_1}\lrar(\K_{Y_2}\lrar \K_{Y_3})$ and\\ 
$\rightleftharpoons(\mathcal{H})= [(\K_{Y_1}\rightleftharpoons\K_{Y_2})\rightleftharpoons \K_{Y_3}]= [\K_{Y_1}\rightleftharpoons(\K_{Y_2}\rightleftharpoons \K_{Y_3})].$
By induction it follows that, treated as binary operations, both $\lrar(\cdot)$ and $\rightleftharpoons(\cdot ),$ are associative over an associative family 
$\mathcal{H}$ and independent of the way the brackets are constructed these operations
reduce to $\lrar(\mathcal{H})$ and $\rightleftharpoons(\mathcal{H} ),$ respectively.
\item Suppose $\mathcal{H}\equivd \{\K_{Y_i},i=1, \cdots , m\}$ is an associative 
family with every element of $\bigcup _{i=1}^{m}\K_{Y_i}$ belonging to exactly 
two of the $Y_i.$ Then $(\lrar(\mathcal{H}))=(\rightleftharpoons(\mathcal{H}))=\emptyset.$ For consistency of notation,
 in this case we denote $\lrar(\mathcal{H})$ and $\rightleftharpoons(\mathcal{H})$ by $\K_{\emptyset}.$
We note that, by the definition of operations $\lrar(\cdot)$ and $\rightleftharpoons(\cdot ),$ we must have  $(\K_{\emptyset}\lrar \K_X)= \K_X,$ and $(\K_{\emptyset}\rightleftharpoons \K_X)= \K_X,$ 
whenever $\K_X= (\lrar(\mathcal{H})),$ or $\K_X =  (\rightleftharpoons(\mathcal{H})).$
If the graph $\G_{\mathcal{H}}$ has edges and is connected, then 
it does not have isolated vertices. This means $\mathcal{H}$ has no element 
which is an empty collection of vectors. Further, if it is connected and  has at least one selfloop, no subfamily $\mathcal{H}_1$ of $\mathcal{H}$ would be such that $\lrar(\mathcal{H}_1)=\K_{\emptyset}.$ 
\item Let $\mathcal{H}=\mathcal{H}_1\uplus \mathcal{H}_2.$ Then $\lrar(\mathcal{H})\ =\ (\lrar(\mathcal{H}_1)) \lrar  (\lrar(\mathcal{H}_2)),$
and\\ $\rightleftharpoons(\mathcal{H})\ =\ (\rightleftharpoons(\mathcal{H}_1)) \rightleftharpoons(\rightleftharpoons(\mathcal{H}_2)) .$
\item By induction on the size of $\mathcal{H},$ one can prove the following generalization of 
Theorem \ref{thm:idt0}:
\begin{theorem}
\label{thm:idtassociative}
Let $\mathcal{H}\equivd \{\V_{Y_i},i=1, \cdots , m\}$ be an associative family of vector 
spaces
and let $\mathcal{H}^{\perp}\equivd \{\V^{\perp}_{Y_i},i=1, \cdots , m\}.$
Then
$(\lrar(\mathcal{H}))^{\perp}=(\rightleftharpoons(\mathcal{H}^{\perp})).$
\end{theorem}
\end{enumerate}
\end{remark}
We know that, by definition,  $(\K_{Y_1}\rightleftharpoons\K_{Y_2})=(\K_{Y_1}\lrar\K_{(-Y_1\cap Y_2)(Y_2-Y_1)}).$ It is convenient to extend
this identity to associative families. We need the following definition
for this purpose.
\begin{definition}
\label{def:skewedpair}
Let $\mathcal{H}= \{\K_{Y_i}, i=1, \cdots , m\}, $
 be an associative family.
% with the $\K_{Y_i}$ non empty.
We construct  an associative family $\hat{\mathcal{H}}_{1,\cdots,m}$  as follows.
%Let $Z$ be the set of all elements of $\bigcup_i Y_i$ which belong to
%only one of the $Y_i.$
Let $Y_{i1}\equivd\{e, e\in Y_i\cap Y_j, j>i\}$
%.(Y_i-Z) \ \mbox{but}\ e\not\in Y_j,j<i\}$
 and let $Y_{i2}\equivd Y_i-Y_{i1}.$
Define $\hat{\K}_{Y_i}\equivd (\K_{Y_i})_{(-Y_{i1})Y_{i2}}.$
 Let $\hat{\mathcal{H}}_{1,\cdots,m}\equivd \{\hat{\K}_{Y_i}, i=1, \cdots , m\}.$
We say that $\{\mathcal{H},\hat{\mathcal{H}}_{1,\cdots,m}\}$ is a \nw{skewed pair} of associative families.
We omit the subscript when the order is clear from the context and 
denote $\hat{\mathcal{H}}_{1,\cdots,m}$ simply by $\hat{\mathcal{H}}.$
\end{definition}
For the Example \ref{eg:associativefamily}, we note that $Y_{11}=\{e_3,e_5,e_7\}, Y_{21}=\emptyset, Y_{31}=\emptyset, Y_{41}= \{e_2,e_4,e_8\}, Y_{51}=\emptyset, Y_{61}=\emptyset.$

We now have the following useful lemma about the operations $\lrar(\cdot)$ and
$\rightleftharpoons(\cdot).$
\begin{lemma}
\label{lem:lrarassociative}
Let $\mathcal{H}= \{\K_{Y_i}, i=1, \cdots , m\}, \hat{\mathcal{H}}\equivd \{\hat{\K}_{Y_i}, i=1, \cdots , m\}$ and let $\{\mathcal{H},\hat{\mathcal{H}}\}$ be a {skewed pair}
 of associative families.
Let $Z$ be the set of all elements of $\bigcup_i Y_i$ which belong to
only one of the $Y_i.$
Then
\begin{enumerate}
\item $\lrar(\mathcal{H})\ =\ (\bigcap_i \K_{Y_i})\circ Z= (\Sigma_i \hat{\K}_{Y_i})\times Z\ =\  (\rightleftharpoons(\hat{\mathcal{H}})),$
\item 
$\rightleftharpoons(\mathcal{H})\ =\ (\Sigma_i \K_{Y_i})\times Z=\ (\bigcap_i \hat{\K}_{Y_i})\circ Z\ =\   (\lrar(\hat{\mathcal{H}})),$
\end{enumerate}
where the intersection and summation operations are over all 
nonvoid collections in $\mathcal{H},\hat{\mathcal{H}}.$
\end{lemma}

\begin{definition}
\label{def:rigidn}
Let ${\mathcal{H}}\equivd \{\A_{Y_1},\cdots , \A_{Y_n}\}$ be an associative family of affine spaces with vector space translates $\{\V_{Y_1},\cdots , \V_{Y_n}\}.$
We say ${\mathcal{H}}$ is \nw{rigid} iff
the following rank conditions are satisfied.
\begin{itemize}
\item (primal) $r(\Sigma_i\V_{Y_i})=
\Sigma_ir(\V_{Y_i});$
\item (dual) $r(\Sigma_i\V^{\perp}_{Y_i})=
\Sigma_ir(\V^{\perp}_{Y_i}).$
\end{itemize}
\end{definition}
We note that the definition of rigid families agrees with that of a pair 
of affine spaces (by part 3 of Theorem \ref{thm:regularrecursive}).
We further have the following simple lemma.
\begin{lemma}
\label{lem:rigidsubfamily}
Let ${\mathcal{H}}\equivd \{\A_{Y_1},\cdots , \A_{Y_n}\}$ be an associative family of affine spaces with vector space translates $\{\V_{Y_1},\cdots , \V_{Y_n}\}.$ Let  $\G_{\mathcal{H}}$ be connected.
%$\V_{Y_1}$ for $\V_{Y_1}, i= 1, \cdots 
If $\mathcal{H}$ is rigid then 
 every subfamily of ${\mathcal{H}}$ is rigid.
\end{lemma}
\begin{proof}
Let $\A_{Y_1},\cdots , \A_{Y_n}$ have vector space translates $\V_{Y_1},\cdots , \V_{Y_n},$ respectively. Without loss of generality, it is adequate 
to show that $\{\A_{Y_1},\cdots , \A_{Y_m}\}, m<n $ is rigid.
Now $r(\V_{Y_1}+\cdots+ \V_{Y_m})\leq r(\V_{Y_1})+\cdots+ r(\V_{Y_m})$
and $r(\Sigma_{i=1}^n\V_{Y_i})\leq r(\Sigma_{i=1}^m\V_{Y_i})+r(\Sigma_{i=m+1}^n\V_{Y_i}).$
Therefore the primal rank condition of rigidity is true for $\{\A_{Y_1},\cdots , \A_{Y_n}\}$  only if it is true for $\{\A_{Y_1},\cdots , \A_{Y_m}\}.$ 
The proof for the dual rank condition is essentially the same.
\end{proof}

By induction on the size of $\mathcal{H},$ we can prove the following 
generalization of part 1 of Theorem \ref{thm:regularrecursive}.

\begin{lemma}
\label{lem:associativetranslate}
Let ${\mathcal{H}}\equivd \{\A_{Y_1},\cdots , \A_{Y_n}\}$ be an associative family of affine spaces with vector space translates $\{\V_{Y_1},\cdots , \V_{Y_n}\}.$ Further let $e\in Y_1- \bigcup_{i=2}^n Y_i.$
%$\V_{Y_1}$ for $\V_{Y_1}, i= 1, \cdots 
If $\mathcal{H}$ is rigid then 
\begin{enumerate}
\item $\lrar({\mathcal{H}})$ is nonvoid and 
has the vector space translate $\lrar(\{\V_{Y_1},\cdots , \V_{Y_n}\}),$ and
\item $\rightleftharpoons({\mathcal{H}})$ is nonvoid and
has the vector space translate $\rightleftharpoons(\{\V_{Y_1},\cdots , \V_{Y_n}\}).$
\end{enumerate}
\end{lemma}
The following lemma addresses the rigidity of a common class of associative families. 

\begin{lemma}
\label{lem:rigidpairrank}
Let $\mathcal{H}_1\equivd \{\A_{X_1},\cdots ,\A_{X_n}\},$ and let $\mathcal{H}_2\equivd \{\A_{Y_1},\cdots ,\A_{Y_m}\},$ with the $X_i,$ being pairwise disjoint and the $Y_j$
also. Let $\V_{X_i},\V_{Y_j},$ be the vector space associates of
$\A_{X_i},\A_{Y_j},$ respectively.
Then the associative family $\mathcal{H}_1\uplus \mathcal{H}_{2}$ is rigid iff the
pair $\{\oplus \V_{X_i},\oplus \V_{Y_j}\},$ is rigid.
\end{lemma}
\begin{proof}
We have $ r(\V_{X_1}\oplus \cdots \oplus \V_{X_n})=r(\V_{X_1})+ \cdots + r(\V_{X_n}),\ 
 r(\V_{Y_1}\oplus \cdots \oplus \V_{Y_m})=r(\V_{Y_1})+ \cdots + r(\V_{Y_n}),$\\$
 r(\V^{\perp}_{X_1}\oplus \cdots \oplus \V^{\perp}_{X_n})=r(\V^{\perp}_{X_1})+ \cdots + r(\V^{\perp}_{X_n}),\ 
 r(\V^{\perp}_{Y_1}\oplus \cdots \oplus \V^{\perp}_{Y_m})=r(\V^{\perp}_{Y_1})+ \cdots + r(\V^{\perp}_{Y_n}).$
\\ The result now follows from part 3 of Theorem \ref{thm:regularrecursive}.
\end{proof}
We now have a result which expresses rigidity of associative families 
recursively.
\begin{theorem}
\label{thm:associativerigidrecursive}
Let ${\mathcal{H}}\equivd \{\A_{Y_1},\cdots , \A_{Y_n}\}$ be an associative family of affine spaces with $\G_{\mathcal{H}}$ connected and with non empty edge 
set.
% with vector space translates $\{\V_{Y_1},\cdots , \V_{Y_n}\}$
Let ${\mathcal{H}}$ be partitioned as ${\mathcal{H}_1}\uplus {\mathcal{H}_2}.$
Then
\begin{enumerate}
\item ${\mathcal{H}}$ is rigid iff ${\mathcal{H}_1},{\mathcal{H}_2}$ and $\{\lrar(\mathcal{H}_1), \lrar(\mathcal{H}_2)\}$
are rigid.
\item ${\mathcal{H}}$ is rigid iff ${\mathcal{H}_1},{\mathcal{H}_2}$ and $\{\rightleftharpoons(\mathcal{H}_1), \rightleftharpoons(\mathcal{H}_2)\}$
are rigid.
\end{enumerate}
\end{theorem}
The proof of Theorem \ref{thm:associativerigidrecursive} is based on the following lemmas.

\begin{lemma}
\label{lem:rankcondition1}
Let $\mathcal{H}\equivd\{\V_{Y_1},\cdots , \V_{Y_m}\}$ be a collection of vector spaces.
We have the following.
\begin{enumerate}
\item  $r(\Sigma_{i=1}^m\V_{Y_i})=
\Sigma_{i=1}^mr(\V_{Y_i})$
 iff every set of nonzero vectors $\{f_{Y_1}, \cdots, f_{Y_m}\}, f_{Y_i}\in \V_{Y_i}$ is linearly independent.
\item  $r(\Sigma_{i=1}^m\V_{Y_i})=
\Sigma_{i=1}^mr(\V_{Y_i})$
 iff for each $t<m, $ we have $r(\Sigma_{i=t}^m\V_{Y_i})=
r(\V_{Y_t})+r(\Sigma_{i=t+1}^m\V_{Y_i}).$
\item If $r(\Sigma_{i=1}^m\V_{Y_i})=
\Sigma_{i=1}^mr(\V_{Y_i}),$ and $T\subseteq \bigcup_i Y_i,$ then 
 $r((\Sigma_{i=1}^m\V_{Y_i})\times T)= \Sigma_{i=1}^mr(\V_{Y_i}\times {Y_i}\cap T)$\\
(where $r(\V_{Y_i}\times {Y_i}\cap T)$ is taken to be zero if ${Y_i}\cap T=\emptyset$).
\item Let $\mathcal{H}$ be an associative family. For each $Y_i$ let $Y_{i1}$ denote the
set of all $e$ such that $e\in Y_i\cap Y_j$ for some $j>i.$
Let ${Y_{i2}}\equivd Y_i- {Y_{i1}}, \ \hat{\V}_{Y_i}\equivd ({\V}_{Y_i})_{(-Y_{i1})Y_{i2}}, i= 1, \cdots, m.$
Then $r(\Sigma_{i=1}^m\V_{Y_i})=
\Sigma_{i=1}^mr(\V_{Y_i})$\\ iff for each $Y_t$ we have
$r[(\V_{Y_t}\times Y_{t1})\cap  ((\Sigma_{i=t+1}^m\V_{Y_i})\times Y_{t1})]=0,$
\\equivalently, iff for each $Y_t$ we have $r[(\hat{\V}_{Y_t}\times Y_{t1})\cap  ((\Sigma_{i=t+1}^m\hat{\V}_{Y_i})\times Y_{t1})]=0,$
\\equivalently, iff $r(\Sigma_{i=1}^m\hat{\V}_{Y_i})=
\Sigma_{i=1}^mr(\hat{\V}_{Y_i}).$
\end{enumerate}
\end{lemma}
\begin{proof}
Part 1 is routine.
Part 2 can be proved  simply by  induction.
Part 3 follows from part 1.
%We only prove part 4.
\\ 4. By Theorem \ref{thm:sumintersection}, $r(\V_A+\V_B)= r(\V_A)+r(\V_B)-r([\V_A\times (A\cap B)]\cap [\V_B\times (A\cap B)]).$ 
\\Therefore $r(\V_A+\V_B)= r(\V_A)+r(\V_B)$ iff  $r([\V_A\times (A\cap B)]\cap [\V_B\times (A\cap B)])=0.$
\\We note that $Y_{t1} = Y_t\cap (Y_{t+1}\cup \cdots \cup Y_m).$
Therefore,\\
$r(\Sigma_{i=t}^m\V_{Y_i})=
r(\V_{Y_t})+r(\Sigma_{i=t+1}^m\V_{Y_i})$
iff $r[(\V_{Y_t}\times Y_{t1})\cap  ((\Sigma_{i=t+1}^m\V_{Y_i})\times Y_{t1})]=0.$
\\ Using part 2, we therefore have $r(\Sigma_{i=1}^m\V_{Y_i})=
\Sigma_{i=1}^mr(\V_{Y_i})$ \\iff $r[(\V_{Y_t}\times Y_{t1})\cap  ((\Sigma_{i=t+1}^m\V_{Y_i})\times Y_{t1})]=0,$
 for each $Y_t.$
\\Next $\hat{\V}_{Y_t}\times Y_{t1}= ({\V}_{Y_i})_{(-Y_{i1})Y_{i2}}\times Y_{t1}={\V}_{Y_t}\times Y_{t1}$ and, since  in $\hat{\V}_{Y_i}, {\V}_{Y_i}, i>t,$
the columns $Y_{t1}$ have the same sign,  we have  
$(\Sigma_{i=t+1}^m\hat{\V}_{Y_i})\times Y_{t1}=(\Sigma_{i=t+1}^m\V_{Y_i})\times Y_{t1}.$
\\Therefore $(\V_{Y_t}\times Y_{t1})\cap  ((\Sigma_{i=t+1}^m\V_{Y_i})\times Y_{t1})= (\hat{\V}_{Y_t}\times Y_{t1})\cap  ((\Sigma_{i=t+1}^m\hat{\V}_{Y_i})\times Y_{t1}).$
Thus  $r(\Sigma_{i=1}^m\V_{Y_i})=
\Sigma_{i=1}^mr(\V_{Y_i})$ \\iff $r[(\hat{\V}_{Y_t}\times Y_{t1})\cap  ((\Sigma_{i=t+1}^m\hat{\V}_{Y_i})\times Y_{t1})]=0,$
for each $Y_t.$
We have shown above that $r(\Sigma_{i=1}^m\hat{\V}_{Y_i})=
\Sigma_{i=1}^mr(\hat{\V}_{Y_i})$ iff $r[(\hat{\V}_{Y_t}\times Y_{t1})\cap  ((\Sigma_{i=t+1}^m\hat{\V}_{Y_i})\times Y_{t1})]=0.$
Therefore, $r(\Sigma_{i=1}^m\V_{Y_i})=
\Sigma_{i=1}^mr(\V_{Y_i})$ iff $r(\Sigma_{i=1}^m\hat{\V}_{Y_i})=
\Sigma_{i=1}^mr(\hat{\V}_{Y_i}).$
\end{proof}
\begin{lemma}
\label{lem:rankcondition}
Let ${\mathcal{H}}\equivd \{\V_{Y_i},i=1, \cdots , m\}$
and let ${\mathcal{H}}_1\equivd \{\V_{Y_i},i=1, \cdots , t\}, t<m,$\\
${\mathcal{H}}_2\equivd   \{\V_{Y_i},i=t+1, \cdots , m\}$
be subfamilies of $\hat{\mathcal{H}}.$
Let $Y^1=\bigcup^t_{i=1} Y_i$ and let $Y^2=\bigcup^m_{i=t+1} Y_i.$
Then ${\mathcal{H}}$ satisfies the rank conditions of Definition \ref{def:rigidn}
 iff
${\mathcal{H}}_1,{\mathcal{H}}_2$ satisfy them and
further\\
$ r([\Sigma^t_{i=1}\V_{Y_i}]\times (Y^1\cap Y^2) \bigcap [\Sigma^m_{i=t+1}\V_{Y_i}]\times (Y^1\cap Y^2))=0$ and\\
$ r([\Sigma^t_{i=1}\V^{\perp}_{Y_i}]\times (Y^1\cap Y^2) \bigcap [\Sigma^m_{i=t+1}\V^{\perp}_{Y_i}]\times (Y^1\cap Y^2))=0.$
\end{lemma}
\begin{proof}
When $Y_1\cap Y_2=\emptyset,$ it is clear that if the rank condition is satisfied 
for ${\mathcal{H}_1},{\mathcal{H}_2}$ the lemma is trivially true. Let $Y_1\cap Y_2$ be 
nonvoid.
By Theorem \ref{thm:sumintersection}, $r(\V_A+\V_B)= r(\V_A)+r(\V_B)-r([\V_A\times (A\cap B)]\cap [\V_B)\times (A\cap B)].$ Therefore,
$r(\Sigma^m_{i=1}\V_{Y_i}) =r(\Sigma^t_{i=1}\V_{X_i})+r(\Sigma^m_{i=t+1}\V_{X_i})-  r([\Sigma^t_{i=1}\V_{Y_i}]\times (Y^1\cap Y^2) \bigcap [\Sigma^m_{i=t+1}\V_{Y_i}]\times (Y^1\cap Y^2)).$
\\
It follows that $r(\Sigma^m_{i=1}\V_{Y_i}) = \Sigma^m_{i=1}r(\V_{Y_i})$
iff $r(\Sigma^t_{i=1}\V_{Y_i}) = \Sigma^t_{i=1}r(\V_{Y_i}), t<m,$ \\
$r(\Sigma^m_{i=t+1}\V_{Y_i}) = \Sigma^m_{i=t+1}r(\V_{Y_i})$
and $ r([\Sigma^t_{i=1}\V_{Y_i}]\times (Y^1\cap Y^2) \bigcap [\Sigma^m_{i=t+1}\V_{Y_i}]\times (Y^1\cap Y^2))=0.$
\\ Working with $\V^{\perp}_{Y_i}$ and repeating the argument, we get the
second half of the statement of the lemma.
\end{proof}
\begin{proof}
(Proof of Theorem \ref{thm:associativerigidrecursive})\\
%Since $\G_{\mathcal{H}}$ is connected, we must have 
%$\lrar(\mathcal{H}_i), \rightleftharpoons(\mathcal{H}_i), i=1,2,$ not 
%equal to $\K_{\emptyset},$
%By Lemma \ref{lem:associativetranslate},\\ $\lrar(\mathcal{H}_i), \rightleftharpoons(\mathcal{H}_i), i=1,2,$ have vector space translates obtained by performing these operations on the families of vector space associates of the affine spaces in the corresponding families.\\
Let $\A_{Y_1}, \cdots , \A_{Y_m}$ 
have vector space translates $\V_{Y_1},\cdots , \V_{Y_m},$
 respectively.
We note that 
rigidity of  $\mathcal{H}\equivd \{\A_{Y_1}, \cdots , \A_{Y_m}\}$    
is defined in terms of rigidity of $\{\V_{Y_1}, \cdots , \V_{Y_m}\}.$ 
Further
by Lemma \ref{lem:associativetranslate}, we know that 
 when $\mathcal{H}_i$ is rigid  and 
$\lrar(\mathcal{H}_i), \rightleftharpoons(\mathcal{H}_i)$ are nonvoid, then these latter have vector space translates\\
$\lrar(\{\V_{Y_1}, \cdots , \V_{Y_m}\}), \rightleftharpoons(\{\V_{Y_1}, \cdots , \V_{Y_m}\})$ respectively.  
Since $\G_{\mathcal{H}}$ is connected, we must have
$\lrar(\mathcal{H}_i), \rightleftharpoons(\mathcal{H}_i), i=1,2,$ not
equal to $\K_{\emptyset}.$
We therefore prove the theorem taking $\mathcal{H}\equivd \{\V_{Y_1}, \cdots , \V_{Y_m}\}.$

By Lemma \ref{lem:rankcondition}, the primal  rank condition of Definition \ref{def:rigidn} ($r(\Sigma_{i}\V_{Y_i}) = \Sigma_{i}r(\V_{Y_i})$)
 holds for 
$\mathcal{H}$ iff it holds for ${\mathcal{H}_1},{\mathcal{H}_2},$ and $ r([\Sigma^t_{i=1}\V_{Y_i}]\times (Y^1\cap Y^2) \bigcap [\Sigma^m_{i=t+1}\V_{Y_i}]\times (Y^1\cap Y^2))=0,$ 
 where $Y^1=\bigcup^t_{i=1} Y_i$ and  $Y^2=\bigcup^m_{i=t+1} Y_i.$
%and \\$ r([\Sigma^t_{i=1}\V^{\perp}_{Y_i}]\times (Y^1\cap Y^2) \bigcap [\Sigma^m_{i=t+1}\V^{\perp}_{Y_i}]\times (Y^1\cap Y^2))=0.$\\
%We will examine the equality involving $\V_{Y_i}.$
We have $r([\Sigma^t_{i=1}\V_{Y_i}]\times (Y^1\cap Y^2) \bigcap [\Sigma^m_{i=t+1}\V_{Y_i}]\times (Y^1\cap Y^2))$\\$=
r([\Sigma^t_{i=1}\V_{Y_i}]\times Y^1\times (Y^1\cap Y^2)\bigcap [\Sigma^m_{i=t+1}\V_{Y_i}]\times Y^2\times (Y^1\cap Y^2)$\\$
=r([\rightleftharpoons(\mathcal{H}_1))\times (Y^1\cap Y^2)]\cap  [(\rightleftharpoons(\mathcal{H}_2))\times(Y^1\cap Y^2)].$
Thus the primal rank condition 
%($r(\Sigma_{i}\V_{Y_i}) = \Sigma_{i}r(\V_{Y_i})$) 
holds for
$\mathcal{H}$ iff it holds for $\mathcal{H}_1,\mathcal{H}_2$ and 
the zero intersection property holds for $\{(\rightleftharpoons(\mathcal{H}_1)), (\rightleftharpoons(\mathcal{H}_2))\}.$
%\\$=  (\rightleftharpoons(\mathcal{H}_1))\lrar (\rightleftharpoons(\mathcal{H}_2)).$

As in Definition \ref{def:skewedpair}, we take $\hat{\mathcal{H}}\equivd \{\hat{\V}_{Y_1}, \cdots , \hat{\V}_{Y_m}\}.$
Let $\hat{\mathcal{H}}\equivd \hat{\mathcal{H}}_{1, \cdots ,m}, \hat{\mathcal{H}}_1\equivd (\hat{\mathcal{H}}_1)_{1, \cdots ,t}, $\\$ \hat{\mathcal{H}}_2\equivd (\hat{\mathcal{H}}_2)_{t+1, \cdots ,m}.$
We note that $\hat{\hat{\mathcal{H}}}=\mathcal{H}$ and 
we have $\lrar(\mathcal{H})= (\rightleftharpoons(\hat{\mathcal{H}})),$ for any associative 
family $\mathcal{H}.$ By part 4 of Lemma \ref{lem:rankcondition1}, we know that 
the primal rank condition holds for $\mathcal{H}$ iff it holds for $\hat{\mathcal{H}}.$
By the above argument, the primal rank condition holds for $\hat{\mathcal{H}}$ 
iff it holds for $\hat{\mathcal{H}}_1,\hat{\mathcal{H}}_2$, i.e., for $\mathcal{H}_1,\mathcal{H}_2, $  and
the zero intersection property holds for $\{(\rightleftharpoons(\hat{\mathcal{H}}_1)), (\rightleftharpoons(\hat{\mathcal{H}}_1))\}= \{(\lrar(\mathcal{H}_1)), (\lrar(\mathcal{H}_2))\}.$
It follows that
the primal rank condition holds for $\mathcal{H}$ 
iff it holds for $\mathcal{H}_1,\mathcal{H}_2$ and
the zero intersection property holds for $\{(\lrar(\mathcal{H}_1)), (\lrar(\mathcal{H}_2))\}
.$

%$r(\Sigma^m_{i=1}\V_{Y_i}) = \Sigma^m_{i=1}r(\V_{Y_i})$ holds for
%$\mathcal{H}$ iff $r(\Sigma^m_{i=1}\hat{\V}_{Y_i}) = \Sigma^m_{i=1}r(\hat{\V}_{Y_i})$ holds for
%$\hat{\mathcal{H}}.$ By the above argument, $r(\Sigma^m_{i=1}\hat{\V}_{Y_i}) = \Sigma^m_{i=1}r(\hat{\V}_{Y_i})$ holds for
%$\hat{\mathcal{H}}$ iff it holds for $\hat{\mathcal{H}}_1,\hat{\mathcal{H}}_2$, i.e., for $\mathcal{H}_1,\mathcal{H}_2, $  and
%the zero intersection property holds for $\{(\rightleftharpoons(\hat{\mathcal{H}}_1)), (\rightleftharpoons(\hat{\mathcal{H}}_1))\}= \{(\lrar(\mathcal{H}_1)), (\lrar(\mathcal{H}_2))\}.$
%It follows that 
%$r(\Sigma^m_{i=1}\V_{Y_i}) = \Sigma^m_{i=1}r(\V_{Y_i})$ holds for
%$\mathcal{H}$ iff it holds for $\mathcal{H}_1,\mathcal{H}_2$ and
%the zero intersection property holds for $\{(\lrar(\mathcal{H}_1)), (\lrar(\mathcal{H}_2))\}
%.$
The dual rank condition for $\mathcal{H}$ is the same as the primal rank condition 
for $\mathcal{H}^{\perp}.$ When $\mathcal{H}$ is associative, so is $\mathcal{H}^{\perp}.$
 Working with $\mathcal{H}^{\perp}$ 
%and using Theorem \ref{thm:idtassociative},
we can show that the primal rank condition holds for
%$r(\Sigma^m_{i=1}\V^{\perp}_{Y_i}) = \Sigma^m_{i=1}r(\V^{\perp}_{Y_i})$ holds for
$\mathcal{H}^{\perp}$ iff it holds for $\mathcal{H}^{\perp}_1,\mathcal{H}^{\perp}_2$ and
the zero intersection property holds for 
$\{(\rightleftharpoons(\mathcal{H}^{\perp}_1)), (\rightleftharpoons(\mathcal{H}^{\perp}_2))\}.$
By Theorem \ref{thm:idtassociative}, \\$\{(\rightleftharpoons(\mathcal{H}^{\perp}_1)), (\rightleftharpoons(\mathcal{H}^{\perp}_2))\}
=
\{(\lrar(\mathcal{H}_1))^{\perp})), (\lrar(\mathcal{H}_2))^{\perp}))\}.$
Now by 
 by part 4 of Theorem \ref{thm:regularrecursive}, the 
zero intersection property holds for $\{(\lrar(\mathcal{H}_1))^{\perp})), (\lrar(\mathcal{H}_2))^{\perp}))\}$
iff 
the 
full sum property holds 
for\\ $\{(\lrar(\mathcal{H}_1)), (\lrar(\mathcal{H}_2))\}
.$
We thus see that the primal and dual rank conditions of rigidity hold for $\mathcal{H}$ iff they hold for $\mathcal{H}_1, \mathcal{H}_2$ and $\{(\lrar(\mathcal{H}_1)), (\lrar(\mathcal{H}_2))\}
$ is rigid.
Thus $\mathcal{H}$ is rigid iff $\mathcal{H}_1, \mathcal{H}_2, \{(\lrar(\mathcal{H}_1)), (\lrar(\mathcal{H}_2))\}$ 
are rigid.

%2. By part 4 of Lemma \ref{lem:rankcondition1}, and using the fact that $(\hat{\mathcal{H}})^{\perp}= \hat{(\mathcal{H}^{\perp})}, $ we have that $\mathcal{H}$ is rigid iff $\hat{\mathcal{H}}$ is rigid.
%By the above proof, $\hat{\mathcal{H}}$ is rigid iff $\hat{\mathcal{H}}_1, \hat{\mathcal{H}}_2$ are rigid 
%and $\{(\lrar(\hat{\mathcal{H}}_1)), (\lrar(\hat{\mathcal{H}}_2))\}= \{(\rightleftharpoons(\mathcal{H}_1)), (\rightleftharpoons(\mathcal{H}_2))\} $ is rigid.
%Thus ${\mathcal{H}}$ is rigid iff ${\mathcal{H}}_1, {\mathcal{H}}_2$ are rigid
%and $\{(\rightleftharpoons(\mathcal{H}_1)), (\rightleftharpoons(\mathcal{H}_2))\} $ is rigid.
%
%=============================================================
2. As before, we only prove the result taking the families to be made up 
of vector spaces. 
%Since $\G_{\mathcal{H}}$ is connected and 

%  the families $ \mathcal{H}_1, \mathcal{H}_2$ are non void and associative, we must have that \\$\rightleftharpoons(\mathcal{H}_1), \rightleftharpoons(\mathcal{H}_2)$

%are nonvoid and these latter have vector space associates which are obtained by replacing affine spaces, wherever they occur, by their vector space associates (Lemma \ref{lem:associativetranslate}).

%So we need only prove the result for the case where $\mathcal{H}$

%is a family of vector spaces.

As in Definition \ref{def:skewedpair}, we
define $\hat{\mathcal{H}}, \hat{\mathcal{H}}_1, \hat{\mathcal{H}}_2.$
By part 4 of Lemma \ref{lem:rankcondition1}, we know that the
primal and dual rank condition holds for $\mathcal{H}$
 iff it holds for $\hat{\mathcal{H}}.$
From the proof above, $\hat{\mathcal{H}}$ is rigid iff $\hat{\mathcal{H}_1},\hat{\mathcal{H}_2}$ and $\{\lrar(\hat{\mathcal{H}_1}) , \lrar(\hat{\mathcal{H}_2})\}$
are rigid.
By Lemma \ref{lem:lrarassociative}, $\lrar(\hat{\mathcal{H}})\ = \ \rightleftharpoons(\mathcal{H}), \lrar(\hat{\mathcal{H}}_i)\ = \ \rightleftharpoons(\mathcal{H}_i).$
The result now follows.

%
%============================================================

\end{proof}

The following is a more convenient version of Theorem \ref{thm:associativerigidrecursive}.
\begin{theorem}
\label{thm:associativerigidrecursiven}
Let ${\mathcal{H}}\equivd \{\A_{Y_1},\cdots , \A_{Y_m}\}$ be an associative family of affine spaces, 
with $\G_{\mathcal{H}}$ connected and with non empty edge set.
  Let ${\mathcal{H}}$ be partitioned as ${\mathcal{H}_1}\uplus \cdots \uplus {\mathcal{H}_n}.$
Then
\begin{enumerate}
\item ${\mathcal{H}}$ is rigid iff ${\mathcal{H}_1},\cdots ,{\mathcal{H}_n}$ and $\{\lrar(\mathcal{H}_1),\cdots , \lrar(\mathcal{H}_n)\}$
are rigid.
\item ${\mathcal{H}}$ is rigid iff ${\mathcal{H}_1},\cdots ,{\mathcal{H}_n}$ and $\{\rightleftharpoons(\mathcal{H}_1), \cdots ,\rightleftharpoons(\mathcal{H}_n)\}$
are rigid.
\end{enumerate}
\end{theorem}
\begin{proof}
1. The proof is by induction over $n.$ We have seen in Theorem \ref{thm:associativerigidrecursive}, that the result is true when $n=2.$
%Suppose it is true for $n-1.$
We then have ${\mathcal{H}}$ is rigid iff ${\mathcal{H}_1}, ({\mathcal{H}_2}\uplus \cdots \uplus {\mathcal{H}_n}), \{\lrar(\mathcal{H}_1), \lrar({\mathcal{H}_2}\uplus \cdots \uplus \mathcal{H}_n)\}$ are rigid. 
\\
By induction, ${\mathcal{H}_2}\uplus \cdots \uplus {\mathcal{H}_n}$ is rigid iff ${\mathcal{H}_2},\cdots ,{\mathcal{H}_n}$ and $\{\lrar(\mathcal{H}_2),\cdots , \lrar(\mathcal{H}_n)\}$
are rigid.
\\ We therefore need only show that, when $\mathcal{H}_1,\cdots , \mathcal{H}_n$ are rigid, rigidity of $\{\lrar(\mathcal{H}_1), \lrar({\mathcal{H}_2}\uplus \cdots \uplus \mathcal{H}_n)\}$ and 
$\{\lrar(\mathcal{H}_2),\cdots , \lrar(\mathcal{H}_n)\}$ implies rigidity of 
$\{\lrar(\mathcal{H}_1),\cdots , \lrar(\mathcal{H}_n)\}.$

Let us denote $\lrar(\mathcal{H}_1),\cdots , \lrar(\mathcal{H}_n),$ respectively by $\A_{Z_1},\cdots , \A_{Z_n}.$ Since $\mathcal{H}$ is an associative 
family, by the definition of $\lrar(\mathcal{H}_i),$ it is clear that no element belongs to more than two of the $Z_i.$ Further, since $\G_{\mathcal{H}}$ is connected, the $Z_i$  are non void. Since $\mathcal{H}_1,\cdots , \mathcal{H}_n,$ are rigid, by Lemma \ref{lem:associativetranslate}, the $\A_{Z_i}$ are nonvoid.
Let $\A_{Z_1},\cdots , \A_{Z_n},$
 have vector space translates $\V_{Z_1},\cdots , \V_{Z_n},$
 respectively. 
We thus need to show that rigidity of $\{\V_{Z_1},(\V_{Z_2}\lrar \cdots \lrar  \V_{Z_n})\}$ and $\{\V_{Z_2},\cdots , \V_{Z_n}\}$
 implies the rigidity of $\{\V_{Z_1},\cdots , \V_{Z_n}\}.$

Suppose $\{\V_{Z_1},(\V_{Z_2}\lrar \cdots \lrar  \V_{Z_n})\}$ and $\{\V_{Z_2},\cdots , \V_{Z_n}\}$
are rigid. Let $P$ be the set of all elements which belong 
precisely to one of $Z_2,\cdots , Z_m$ and let $T\equivd Z_1\cap \bigcup_{i=2}^mZ_i.$ Since no element belongs to more than 
two of $Z_1,\cdots , Z_n,$ it is clear that $T\subseteq P.$ 
From the zero intersection property of $\{\V_{Z_1},(\V_{Z_2}\lrar \cdots \lrar  \V_{Z_n})\},$ we have that 
%$\{\V_{Z_1},(\V_{Z_2}\lrar \cdots \lrar  \V_{Z_m})\},$
$r((\V_{Z_1}\times T)\cap 
(\V_{Z_2}\lrar \cdots \lrar  \V_{Z_n})\times T)=
r((\V_{Z_1}\times T)\cap[(\V_{Z_2}+\cdots + \V_{Z_n})\times P\times T])$\\$=
 r((\V_{Z_1}\times T)\cap[(\V_{Z_2}+\cdots + \V_{Z_n})\times  T])=
0.$
By Theorem \ref{thm:sumintersection}, it follows that 
$r(\V_{Z_1}+\cdots + \V_{Z_m})= r(\V_{Z_1})+r(\V_{Z_2}+\cdots + \V_{Z_m}).$
Since $\{\V_{Z_2},\cdots , \V_{Z_n}\}$
 is rigid, we have $r(\V_{Z_2}+\cdots + \V_{Z_m})= r(\V_{Z_1})+r(\V_{Z_2})+\cdots + r(\V_{Z_n}).$
Therefore $r(\V_{Z_1}+\cdots + \V_{Z_m})= r(\V_{Z_1})+\cdots + r(\V_{Z_n}).$

The proof of the dual rank condition for $\{\V_{Z_1},\cdots , \V_{Z_n}\}$
is by replacing $\V_{Z_i}$ by $\V_{Z_i}^{\perp}$ everywhere in the above proof and by using the fact that rigidity of $\{\V_{Z_2},\cdots , \V_{Z_n}\}$ is equivalent to rigidity of $\{\V^{\perp}_{Z_2},\cdots , \V^{\perp}_{Z_n}\}.$
\\
We thus see that rigidity of $\{\V_{Z_1},(\V_{Z_2}\lrar \cdots \lrar  \V_{Z_n})\}$ and $\{\V_{Z_2},\cdots , \V_{Z_n}\}$
 implies the rigidity of $\{\V_{Z_1},\cdots , \V_{Z_n}\}.$\\
This proves that 
${\mathcal{H}}$ is rigid iff ${\mathcal{H}_1},\cdots ,{\mathcal{H}_n}$ and $\{\lrar(\mathcal{H}_1),\cdots , \lrar(\mathcal{H}_n)\}$
are rigid.

2. 
%Since the families $\mathcal{H}_1, \cdots , \mathcal{H}_n$ are non void and associative, and $\G_{\mathcal{H}}$ is connected, $ \rightleftharpoons(\mathcal{H}_1), \cdots , \rightleftharpoons(\mathcal{H}_n)$ 
%are nonvoid and these latter have vector space associates which are obtained by replacing affine spaces, wherever they occur, by their vector space associates (Lemma \ref{lem:associativetranslate}). 
As before, we need only prove the result for the case where $\mathcal{H}$ 
is a family of vector spaces.

As in Definition \ref{def:skewedpair}, we 
define $\hat{\mathcal{H}}, \hat{\mathcal{H}}_1,\cdots \hat{\mathcal{H}}_n.$
By part 4 of Lemma \ref{lem:rankcondition1}, we know that the 
primal and dual rank condition holds for $\mathcal{H}$ 
 iff it holds for $\hat{\mathcal{H}}.$
From the proof above, $\hat{\mathcal{H}}$ is rigid iff $\hat{\mathcal{H}_1},\cdots ,\hat{\mathcal{H}_n}$ and $\{\lrar(\hat{\mathcal{H}_1}),\cdots , \lrar(\hat{\mathcal{H}_n})\}$
are rigid.
By Lemma \ref{lem:lrarassociative}, $\lrar(\hat{\mathcal{H}})\ = \ \rightleftharpoons(\mathcal{H}), \lrar(\hat{\mathcal{H}}_i)\ = \ \rightleftharpoons(\mathcal{H}_i).$
The result now follows.
\end{proof}

\section{Matroidal rigidity}
\label{sec:matroidrigidity}
\subsection{Matroid preliminaries}
\label{subsec:matroidprelim}
This section follows \cite{STHN2014}, particularly the definition of 
the `$\lrar$' operation for matroids.

{\bf Matroids} are defined using various cryptomorphic sets of axioms. We will use the base axioms to define a matroid $\M_S$ on a finite set $S$ as an ordered pair $(S,{\cal{B}})$, where ${\cal{B}}$ is a collection of subsets of $S$, called `bases', satisfying the following equivalent axioms.

\begin{itemize}
\item If $b_1, b_2 \in {\cal{B}}$ and if $e_2 \in b_2 - b_1$, then there exists $e_1 \in b_1 - b_2$ such that $(b_1 - \{e_1\}) \cup \{e_2\}$ is a member of ${\cal{B}}$.

\item If $b_1, b_2 \in {\cal{B}}$ and if $e_1 \in b_1 - b_2$, then there exists $e_2 \in b_2 - b_1$ such that $(b_1 - \{e_1\}) \cup \{e_2\}$ is a member of ${\cal{B}}$.
\end{itemize}

It follows from the base axioms that all bases of a matroid have the same size. Subsets of bases are called independent sets. {\bf Rank} function $r:2^S \rightarrow \mathbb{Z^+}$, corresponding to a matroid, is defined by size of a largest independent set contained in the given subset. 

The \nw{independence oracle} for matroid $\M_S$ returns for the query `is $X\subseteq S $ 
independent in the matroid $\M_S?$ with a `Yes' or a `No', with the 
former when $X$ is independent and the latter when it is not.
Usually, the complexity of algorithms related to matroids 
is evaluated in terms of the number of queries to the oracle 
needed in the worst case to complete the algorithm.

A matroid on $S$, with only null set as a base is called {\bf zero} matroid, denoted $\0_S$. A matroid on $S$ with only full set $S$ as a base is called a {\bf full} or {\bf complete} matroid, denoted $\F_S$. 
(Note that we use $\0_S$ to denote the zero vector space as well as the 
zero matroid and $\F_S$ to denote the full vector space as well as the
full matroid. This abuse of notation will not cause confusion since the context
would make it clear which entity is involved.)
When sets $S$ and $P$ are disjoint, we denote a matroid on the set $S \cup P$ by $\M_{SP}$ (and whenever we write $\M_{SP}$, it is assumed that $S$ and $P$ are disjoint). Given a matroid $\M_S$ on $S$ and $\M_P$ on $P$, with $S$ and $P$ disjoint, {\bf direct sum} of $\M_S$ and $\M_P$ is a matroid $\M_S \oplus \M_P$ on $S \cup P$, whose bases are unions of a base of $\M_S$ with a base of $\M_P$.

For any subset $b_S$ of $S$, its complement w.r.t. $S$ is denoted by $\bar{b}_S = S - b_S$. Given a matroid $\M = (S, {\cal{B}})$, its {\bf dual} matroid, denoted $\M^*$, is defined to be $(S, {\cal{B}}^*)$, where ${\cal{B}}^*$ is the collection of complements of the subsets present in ${\cal{B}}$. We can see that dual of the dual is same as original matroid ($\M^{**} = \M$). 

We summarize these ideas in the following theorem.
\begin{theorem}
 \label{thm:perperpm}
Let $\M_X$ be a matroid on $X.$ Then
\begin{enumerate}
\item
$r(\M_X)+r(\M^{*}_X)=|X|$ and $(\M_X)^{**}=\M_X;$
\item $T$ is a base of $\M_X$ iff $T$ is a cobase of $\M^{*}_X.$
\end{enumerate}
\end{theorem}
A {\bf circuit} of a matroid $\M$ is a minimal dependent (not independent) set in $\M$. Circuits of $\M^*$ are called {\bf bonds} of $\M$.
For a matroid $\M$ on $S$, given a subset $T \subseteq S$, we define {\bf restriction} of $\M$ to $T$, denoted $\M \circ T$, as a matroid, whose independent sets are all subsets of $T$ which are independent in $\M$ (equivalently, bases in $\M \circ T$ are maximal intersections of bases of $\M$ with $T$). {\bf Contraction} of $\M$ to $T$, denoted $\M \times T$, is defined by the matroid whose independent sets are precisely those $X \subseteq T$ which satisfy the property that $X \cup b_{S-T}$ is independent in $\M$ whenever $b_{S-T}$ is a base of $\M \circ (S-T)$ (equivalently, bases of $\M \times T$ are minimal intersections of bases of $\M$ with $T$). A {\bf minor} of $\M$ is a matroid of the form $(\M \times T_1) \circ T_2$ or $(\M \circ T_1) \times T_2$, where $T_2 \subseteq T_1 \subseteq S.$
We usually omit the bracket when we speak of minors of matroids (eg $\M \times T_1 \circ T_2$ in place of $(\M \times T_1) \circ T_2$ ).
% (these two forms are in fact equivalent, $(\M \times T_1) \circ T_2 = (\M \circ (S - (T_1 - T_2))) \times T_2$).

We have the following set of basic results regarding minors and duals.

\begin{theorem}
\label{thm:dotcrossidentitym}
Let $S\cap P=\emptyset, T_2\subseteq T_1\subseteq S,$
$\M_S$ be a matroid on $S,$ $\Msp$ be a matroid on $S\uplus P.$
We then have the following.
\begin{enumerate}
\item $(\M_S \times T_1) \circ T_2 = (\M_S \circ (S - (T_1 - T_2))) \times T_2.$
\item $r(\Msp)=r(\Msp\circ S)+r(\Msp\times P).$
\item $\M_{SP}^{*}\circ P= (\M_{SP}\times P)^{*}.$
\item $\M_{SP}^{*}\times S= (\M_{SP}\circ S)^{*}.$
\item
$P_1\subseteq P$ is a base of $\Msp\times P$ iff
there exists $S_1\subseteq S$ such that  $S_1$ is a base of $\Msp\circ S$
and $S_1\cup P_1$ is a base of $\Msp.$
\item If $P_1\subseteq P$ is a base of $\Msp\times P,$
then there exists a base $P_2$ of $\Msp\circ P,$
that contains it.
\item
Let $A,B\subseteq S\uplus P,  A\cap B=\emptyset,$ and let there exist a base $D_1$  of $\Msp$ that contains $A$ and a cobase  $E_2$  of $\Msp$ that
contains $B.$ Then there exists a base $D$ of $\Msp$  that contains $A$ but does not intersect $B.$
\end{enumerate}
\end{theorem}
%==================================
%
%Two matroids $(S_1, {\cal{B}}_1)$ and $(S_2, {\cal{B}}_2)$ are said to be {\bf isomorphic} iff there exists a bijection between $S_1$ and $S_2$ such that this bijection maps ${\cal{B}}_1$ exactly to ${\cal{B}}_2$. A matroid is called identically self-dual iff $\M^* = \M$. We can extend this using the idea of isomorphism. A matroid is called self-dual iff $\M^*$ and $\M$ are isomorphic. 
%
%==================================

We associate matroids with vector spaces and graphs as follows.
If $\V_S$ is a vector space $\M(\V_S)$ is defined to be the matroid whose 
bases are the column bases of $\V_S.$ 
It is clear that $(\M(\V))^* = \M(\V^{\perp})$ since, $(I|K)$ is a representative matrix for 
$\V,$ iff $(-K^T|I)$ is a representative matrix for $\V^{\perp}.$ If $\G_S$ is a graph, $\M(\G_S)$ is defined to be the matroid whose
bases are the column bases of $\V^v(\G_S).$
We then have the following results.
\begin{theorem}
\label{thm:tellegenm}
 $\M(\mathcal{G})\equivd \M(\V^v(\mathcal{G}))=(\M(\V^i(\mathcal{G})))^*.$
%\item Let $W\subseteq T\subseteq S$
%and let $\M$ be a matroid on $S.$
%Then $(\mathcal{M}\times T\circ W)^*= \mathcal{M}^*\circ T\times W.$
\end{theorem}
\begin{lemma}
\label{lem:minorvectorspacem}
Let $\V_S$ be a vector space.
Let $W\subseteq T\subseteq S.$
\begin{enumerate}
\item $ \M(\mathcal{V}\circ T)= (\M(\mathcal{V}))\circ T, \ \ \  \M(\mathcal{V}\times T)= (\M(\mathcal{V}))\times T,\ \ \M(\mathcal{V}\circ T\times W)= (\M(\mathcal{V}))\circ T \times W.$
\item $ (\M(\mathcal{V}\circ T))^*= (\M(\mathcal{V})))^*\times T, \ \ \  (\M(\mathcal{V}\times T))^*= (\M(\mathcal{V})))^*\circ T,\ \ (\M(\mathcal{V}\circ T\times W))^*= (\M(\mathcal{V})))^*\times T \circ W.$
\item A subset is a  base of $\M(\V)$ iff it is a column base of
$\V.$
\item A subset  is a  base of $(\M(\V))^{*}$ iff it is a column cobase of
$\V.$
\end{enumerate}
\end{lemma}

\begin{lemma}
\label{lem:minorgraphvectorspacem}
Let $\G$ be a graph on edge set $S.$
Let $W\subseteq T\subseteq S.$
\begin{enumerate}
\item $ \M(\mathcal{G}\circ T)= (\M(\mathcal{G}))\circ T, \ \ \  \M(\mathcal{G}\times T)= (\M(\mathcal{G}))\times T,\ \ \M(\mathcal{G}\circ T\times W)= (\M(\mathcal{G}))\circ T \times W.$
\item $ (\M(\mathcal{G}\circ T))^*= (\M(\mathcal{G})))^*\times T, \ \ \  (\M(\mathcal{G}\times T))^*= (\M(\mathcal{G})))^*\circ T,\ \ (\M(\mathcal{G}\circ T\times W))^*= (\M(\mathcal{G})))^*\times T \circ W.$
\item A subset is a  base of $\M(\G)$ iff it is a tree of
$\G.$
\item A subset  is a  base of $(\M(\G))^{*}$ iff it is a cotree of
$\G.$
\end{enumerate}
\end{lemma}

Let $\M_1$ and $\M_2$ be matroids on $S$. The {\bf union} of these matroids, denoted $\M_1 \vee \M_2$, is defined to be $(S,{\cal{B_\vee}})$, where ${\cal{B}}_{\vee}$ is the collection of maximal sets of the form $b_1 \cup b_2$, where $b_1$ is a base of $\M_1$ and $b_2$ is a base of $\M_2$. It can be shown that $\M_1 \vee \M_2$ is again a matroid.  
The 
 {\bf intersection} of matroids $\M_1$ and $\M_2$, denoted $\M_1 \wedge \M_2$, is defined to be  the matroid  $(S,{\cal{B_\wedge}})$ whose bases are minimal sets of the form $b_1 \cap b_2$, where $b_1$ is a base of $\M_1$ and $b_2$ is a base of $\M_2$. Matroid union is related to this intersection through dualization, $(\M_1 \vee \M_2)^* = \M_1^* \wedge \M_2^*$. Let $S,P,Q$ be pairwise disjoint sets. Let $\M_{SP}$ and $\M_{PQ}$ be matroids on $S \uplus P$ and $P \uplus Q$ respectively. Then, their union and intersection operations are defined by 

\[
\M_{SP} \vee \M_{PQ} \equivd (\M_{SP} \oplus \0_Q) \vee (\M_{PQ} \oplus \0_S)
\]
\[
\M_{SP} \wedge \M_{PQ} \equivd (\M_{SP} \oplus \F_Q) \wedge (\M_{PQ} \oplus \F_S)
.\]

Bases $b_{SP} , b_{PQ}$ of $\M_{SP},\M_{PQ},$ respectively, such that $b_{SP} \cup b_{PQ}$ is a base of $\M_{SP} \vee \M_{PQ}$, are said to be {\bf maximally
distant}.
%Analogous to the vector space case, where we call two vector spaces to be disjoint iff their intersection is the zero vector space, we call two matroids $\M_S^1$ and $\M_S^2$ to be {\bf disjoint} iff their wedge-intersection is the zero matroid (i.e., there is at least one pair, of bases of $\M_S^1$ and $\M_S^2$, that is disjoint)

Let $\M_{SP}, \M_{PQ}$ be matroids defined on sets $S\uplus P,P\uplus Q$ respectively, where $S,P,Q$ are pairwise disjoint sets. 
Then the {\bf linking} of $\M_{SP}, \M_{PQ}$, denoted $\M_{SP} \leftrightarrow \M_{PQ}$ is defined on the set $S\uplus Q$ to be \\$(\M_{SP} \vee \M_{PQ}) \times (S \cup Q).$

We need the following basic identities about ranks of matroids, matroid unions,
 minors and duals for the development
in subsequent sections.

%Theorem \ref{thm:tellegen}$_{\M}.$
%Theorem \ref{thm:sumintersection}$_{\M}.$
\begin{theorem}
\label{thm:sumintersectionm}
Let $\M^1_A, \M^2_B, \M_S,\M'_S $ be matroids. Then
\begin{enumerate}
\item $r(\M_S)+r(\M'_S)=r(\M_S\vee\M'_S)+r(\M_S\wedge \M'_S);$
\item $r(\M^1_A)+r(\M^2_B)=r(\M^1_A\vee\M^2_B)+r[(\M^1_A\times (A\cap B))\wedge (\M^2_B\times (A\cap B))];$

\item $(\M^1_A\vee\M^2_B)^{*}=(\M^1_A)^{*}\wedge (\M^2_B)^{*};$
\item $(\M^1_A\wedge \M^2_B)^{*}=(\M^1_A)^{*}\vee (\M^2_B)^{*};$
\item (a) $(\M^1_A\vee\M^2_B)\circ X= \M^1_A\circ X \vee \M^2_B\circ X ,X\subseteq A\cap B;$\\
(b) $(\M^1_A\wedge\M^2_B)\times X= \M^1_A\times X \wedge \M^2_B\times X ,X\subseteq A\cap B.$
\end{enumerate}
\end{theorem}
\begin{proof}

1. A base of $\M_S\vee\M'_S$ is a union of two maximally distant 
bases of $\M_S,\M'_S.$
By definition,  a base of $\M_S\wedge \M'_S,$
is the intersection of two maximally distant
bases of $\M_S,\M'_S.$
The result follows from the identity $|X|+|Y|=|X\cup Y|+|X\cap Y|,$ 
where $X,Y,$ are sets.

2. We have $\M^1_A\vee\M^2_B\equivd (\M^1_A\oplus \0_{B-A})\vee(\M^2_B\oplus \0_{A-B}).$
From the previous part, we have\\
$r(\M^1_A)+r(\M^2_B)=r(\M^1_A\oplus \0_{B-A})+r(\M^2_B\oplus \0_{A-B})$\\$=
r((\M^1_A\oplus \0_{B-A})\vee (\M^2_B\oplus \0_{A-B}))+
r((\M^1_A\oplus \0_{B-A})\wedge (\M^2_B\oplus \0_{A-B})).$\\
Now it can be seen that a minimal intersection of a pair of bases of 
$\M^1_A\oplus \0_{B-A}, \M^2_B\oplus \0_{A-B}
$ is the same as a minimal intersection of a pair of bases of 
$\M^1_A\times (A\cap B), \M^2_B\times (A\cap B)$
 so that\\ $r((\M^1_A\oplus \0_{B-A})\wedge (\M^2_B\oplus \0_{A-B}))=r[(\M^1_A\times (A\cap B))\wedge (\M^2_B\times (A\cap B))]$
 and the result follows.

Parts 3 and 4 follow from the definition of  
$\M^1_A\vee\M^2_B, \M^1_A\wedge\M^2_B, (\M^1_A)^*\vee(\M^2_B)^*, (\M^1_A)^*\wedge(\M^2_B)^*.$

5(a) follows from the definition of the matroid union operation.

5(b) We have $(\M^1_A\wedge\M^2_B)\times X= ((\M^1_A\wedge\M^2_B)\times X)^{**}
= ((\M^1_A\wedge\M^2_B)^*\circ X)^{*}= ((\M^1_A)^*\vee(\M^2_B)^*)\circ X)^*
= ((\M^1_A)^*\circ X\vee(\M^2_B)^*\circ X)^*= ((\M^1_A)^*\circ X)^*\wedge((\M^2_B)^*\circ X)^*
= \M^1_A\times X\wedge \M^2_B\times X.$

\end{proof}

\begin{theorem}
\label{cor:ranklrarm}
{\it Let $\Msp,\Mpq,$ be matroids on $S\uplus P, P\uplus Q,$ respectively,
with $S,P,Q,$ being pairwise disjoint.
Then\\
$r(\Msp\lrar \Mpq) = r(\Msp\times S)+r(\Mpq\times Q)+r(\Msp\circ P\wedge \Mpq\circ P) - r(\Msp\times P\wedge \Mpq\times P) .$
}
\end{theorem}
\begin{proof}
Here we have used part 2 of Theorem \ref{thm:sumintersectionm}
 and part 2 of Theorem \ref{thm:dotcrossidentitym}.$
\\
$$r(\Msp\lrar \Mpq) \equivd r((\Msp\vee \Mpq) \times SQ)$\\$=
r(\Msp\vee \Mpq)-r((\Msp\vee \Mpq)\circ P)$\\$=r(\Msp)+ r(\Mpq) - r(\Msp\times P\wedge 
 \Mpq\times P)-r(\Msp\circ P\vee\Mpq\circ P)$\\$
= r(\Msp)+ r(\Mpq) - r(\Msp\times P\wedge 
 \Mpq\times P)- r(\Msp\circ P)-r(\Mpq\circ P)+ r((\Msp\circ P)\wedge (\Mpq\circ P))  $
\\$  = [r(\Msp)- r(\Msp\circ P)]+ [r(\Mpq)-r(\Mpq\circ P)] - r(\Msp\times P\wedge 
 \Mpq\times P)+ r((\Msp\circ P)\wedge (\Mpq\circ P))  $\\$
=r(\Msp\times S)+r(\Mpq\times Q)+r(\Msp\circ P\wedge \Mpq\circ P) - r(\Msp\times P\wedge \Mpq\times P) .$

\end{proof}

The following pair of results are from \cite{STHN2014}.
\begin{theorem}
$\M_{SP} \leftrightarrow \M_{PQ}=(\M_{SP} \vee \M_{PQ}) \times (S \cup Q) = (\M_{SP} \wedge \M_{PQ}) \circ (S \cup Q)$
\label{thm:matchedpropm}
\end{theorem}
\begin{theorem}(Implicit Duality Theorem (IDT))
\label{thm:idt0m}
\begin{equation*}
(\M_{SP} \leftrightarrow \M_{PQ})^* = \M_{SP}^* \leftrightarrow \M_{PQ}^*
\end{equation*}
\label{thm:implicitDuality2}
\end{theorem}

\proof

\begin{eqnarray*}
(\M_{SP} \leftrightarrow \M_{PQ})^* & = & [(\M_{SP} \vee \M_{PQ}) \times (S \cup Q)]^*\\
& = & (\M_{SP} \vee \M_{PQ})^* \circ (S \cup Q) \mbox{ (using Theorem \ref{thm:dotcrossidentitym}).}\\
& = & (\M_{SP}^* \wedge \M_{PQ}^*) \circ (S \cup Q)\\
& = & (\M_{SP}^* \vee \M_{PQ}^*) \times (S \cup Q)  \mbox{ (using Theorem \ref{thm:matchedpropm}).} \\
& = & \M_{SP}^* \leftrightarrow \M_{PQ}^*.
\end{eqnarray*}

\qed 

We have deliberately written the above results in a form that mimics 
that of corresponding results for vector spaces.

The substitution rules are 
\begin{align*}
&{Vector\  space\  entity}&\ \ &{Matroid\  entity}&{Vector\  space\  entity}&\ \ &{Matroid\  entity}\\
&\mbox{column base}&\ &\ \ \ \  \mbox{base}&\ \ \ \ \mbox{column cobase}&\ &\mbox{cobase}\\
&\ \ \ \ \ \V_X&\ &\ \ \ \  \M_X
&\ \ \ \ \V_X^{\perp}&\ &\ \ \ \  \M_X^*\\
&\ \ \ \ \F_A&\ &\ \ \ \  \F_A
&\ \ \ \ \0_A&\ &\ \ \ \  \0_A\\
&\ \ \ \ \V_X\oplus \V_Y&\ &\ \ \ \  \M_X\oplus\M_Y&\ \ \ \ \V_X+\V_Y&\ &\ \ \ \  \M_X\vee\M_Y\\
&\ \ \ \ \V_X\cap\V_Y&\ &\ \ \ \  \M_X\wedge\M_Y&\ \ \ \ \V_X\lrar\V_Y&\ &\ \ \ \  \M_X\lrar\M_Y\\
&\ \ \ \ \V_{AB}\circ B&\ &\ \ \ \  \M_{AB}\circ B&\ \ \ \ \V_{AB}\times B&\ &\ \ \ \  \M_{AB}\times B\\
\end{align*}

The corresponding pairs are 
Theorems \ref{thm:perperpm},\ref{thm:perperp}, Theorems \ref{thm:dotcrossidentitym},\ref{thm:dotcrossidentity}, Theorems \ref{thm:tellegenm},\ref{thm:tellegen}, Lemmas \ref{lem:minorgraphvectorspacem},\ref{lem:minorgraphvectorspace}, 
\\Theorems \ref{thm:sumintersectionm},\ref{thm:sumintersection},
Theorems \ref{cor:ranklrarm},\ref{cor:ranklrar},
Theorems \ref{thm:matchedpropm},\ref{thm:matchedprop}  and Theorems \ref{thm:implicitDuality2},\ref{thm:idt0}.

\subsection{Rigidity of matroid pairs}
\label{subsec:rigidpairsm}
We show in this subsection, that there is a notion analogous to rigidity 
of vector space pairs in the matroid context  and the properties of 
rigid pairs of matroids mimic those of rigid pairs of vector spaces.
This fact is used finally in Theorem \ref{thm:rigidmatroidvector}, to examine under what 
conditions rigidity of a vector space pair can be inferred from 
that of the associated matroid pair.

\begin{definition}
\label{def:matroidrigid}
Let $\M_{AB},\M_{BC}$ be  matroids on sets $A\uplus B, B\uplus C,$ respectively, A,B,C being pairwise disjoint.
We say the pair $\{\M_{AB},\M_{BC}\}$ has the \nw{full sum property} iff
$\M_{AB}\circ B\vee \M_{BC}\circ B=\F_B.$\\
We say the pair $\{\M_{AB},\M_{BC}\}$ has the \nw{zero intersection property} iff
$\M_{AB}\times B\wedge \M_{BC}\times B=\0_B.$\\
We say $\M_{AB}$ is \nw{rigid with respect to} $\M_{BC},$ or that the pair $\{\M_{AB},\M_{BC}\}$ is \nw{rigid,
} iff  it has the full sum property and the zero intersection property.
\\When $B=\emptyset,$ we take $\{\M_{A},\M_{C}\}$ to be  {rigid
}.
\end{definition}
When matroids are direct sums of other matroids, it is clear from the definition that 
rigidity involving the former can be inferred from rigidity involving the 
latter.
\begin{lemma}
\label{lem:directsumm}
The pair $\{\M_{AB}\oplus\M_{CD},\M^1_B\oplus\M^2_C\}$ is rigid
iff the pairs $\{\M_{AB},\M^1_B\},$
$\{\M_{CD},\M^2_C\}$ are rigid.
\end{lemma}

The next result states
the basic facts about rigid
 pairs of matroids.

\begin{theorem}
\label{thm:regularrecursivem}
Let $\{\M_{AB},\M_{BC}\}$ be a rigid
 pair of matroids.
We then have the following
\begin{enumerate}
\item The full sum property  of $\{\M_{AB},\M_{BC}\}$ is equivalent to 
$(\M_{AB}\vee \M_{BC})\circ  B=\F_B,$
equivalently, to 
 there being bases $b_{AB},b_{BC}$ of $\M_{AB},\M_{BC},$ 
respectively such that $b_{AB}\cup b_{BC}\supseteq B.$
\item The zero intersection  property of $\{\M_{AB},\M_{BC}\}$ is equivalent to $(\M_{AB}\wedge \M_{BC})\times B=\0_B,$
equivalently, to there being bases $b_{AB},b_{BC}$ of $\M_{AB},\M_{BC},$
respectively such that $b_{AB}\cap b_{BC}=\emptyset .$
\item  The pair  $\{\M_{AB},\M_{BC}\}$ is rigid
 \\iff
$r(\M_{AB}\vee\M_{BC})= r(\M_{AB})+r(\M_{BC})$ and $  r(\M_{AB}^{*}\vee \M_{BC}^{*})= r(\M_{AB}^{*})+r(\M_{BC}^{*}).$
\item The pair $\{\M_{AB},\M_{BC}\}$ has the zero intersection (full sum) property iff $\{\M^{*}_{AB},\M^{*}_{BC}\}$ has the full sum (zero intersection) property.
Therefore
 $\{\M_{AB},\M_{BC}\}$ is {rigid
} iff
$\{\M^{*}_{AB},\M^{*}_{BC}\}$ is rigid.
%\item The pair $\{\M_{AB}\oplus\M_{CD},\M^1_B\oplus\M^2_C\}$ is rigid
%iff the pairs $\{\M_{AB},\M^1_B\},$
%$\{\M_{CD},\M^2_C\}$ are rigid.
%\item Let $\M^1_B,\M^2_B$ be matroids on $B.$
%If $r(\M^1_B)+r(\M^2_B)=|B|,$ then
%the full sum property and zero intersection property are equivalent
%\item The pair  $\{\M_{AB},\M_{BC}\}$ is rigid
% only if 
%$r(\M_{AB}\lrar \M_{BC})= r(\M_{AB})+r(\M_{BC})-|B|.$
\item $\{\M_{AB},\M_{BC}\}$ is rigid
 iff
there exist disjoint bases $b_{AB},b_{BC}$ of $\M_{AB},\M_{BC}$ respectively
such that $b_{AB}\cup b_{BC}\supseteq B.$
\end{enumerate}
\end{theorem}
\begin{proof} 
%{\it Proof of Theorem \ref{thm:regularrecursive}}
1. 
Note that $\M_{AB}\vee \M_{BC}\equivd (\M_{AB}\oplus \0_C)\vee 
(\M_{BC}\oplus \0_A).$\\
It is clear that $(\M_{AB}\circ B) \vee  (\M_{BC}\circ B)=(\M_{AB}\vee \M_{BC})\circ B,$ so that $(\M_{AB}\circ B) \vee  (\M_{BC}\circ B)=\F_B$  is equivalent to
$(\M_{AB}\vee \M_{BC})\circ B=\F_B.$\\
The latter statement is the same as
there being bases $b_{AB},b_{BC}$ of $\M_{AB},\M_{BC},$
respectively such that $b_{AB}\cup b_{BC}\supseteq B.$

2. Note that $\M_{AB}\wedge \M_{BC}\equivd (\M_{AB}\oplus \F_C)\wedge 
(\M_{BC}\oplus \F_A).$\\
We have  $(\M_{AB}\times B) \wedge  (\M_{BC}\times B)=
 (\M_{AB}\wedge \M_{BC})\times B $ (using part 5 of Theorem \ref{thm:sumintersectionm}),
%[(\M_{AB}\times B) ^*\vee (\M_{BC}\times B)^*]^*$\\$=
%[(\M_{AB}^*\circ B)\vee(\M_{BC}^*\circ B)]^*=[(\M_{AB}^*\vee\M_{BC}^*)\circ B]^*
%= (\M_{AB}^*\vee\M_{BC}^*)^*\times B = (\M_{AB}\wedge \M_{BC})\times B
%$ (using parts 3 and 4 of Theorems \ref{thm:sumintersectionm},\ref{thm:dotcrossidentitym}), 
so that \\$(\M_{AB}\times B) \wedge  (\M_{BC}\times B)=\0_B$  is equivalent to
$(\M_{AB}\wedge \M_{BC})\times B=\0_B.$\\
The latter statement is the same as
there being bases $b_{AB},b_{BC}$ of $\M_{AB},\M_{BC},$
respectively such that $b_{AB}\cap b_{BC}=\emptyset .$

3. Let $\{\M_{AB},\M_{BC}\}$ be rigid
.
Then $(\M_{AB}\circ B\vee\M_{BC}\circ B)=\F_B$ and $(\M_{AB}\times B\wedge \M_{BC}\times B)=\0_B.$\\
Now, using Theorem \ref{thm:sumintersectionm} and the zero intersection property, \\ $r(\M_{AB}\vee\M_{BC})
%= r((\V_{AB}\oplus \0_C)+(\0_A\oplus \V_{BC}))=
%r(\V_{AB}\oplus \0_C) + r(\0_A\oplus \V_{BC})- r((\V_{AB}\oplus \0_C)\cap (\0_A\oplus \V_{BC}))$\\ 
= r(\M_{AB})+r(\M_{BC})-r(\M_{AB}\times B\wedge \M_{BC}\times B)=  r(\M_{AB})+r(\M_{BC}).$
%(using the zero intersection property).
\\
%$r(\V_{AB})+r(\0_A+\V_B)-r(\V_{AB}\cap (\0_A+\V_B)) $\\ (using Theorem \ref{thm:sumintersection})\\
%$= r(\V_{AB})+r(\V_B)-r(\V_{AB}\times B\cap \V_B)=  r(\V_{AB})+r(\V_B).$
%\\
Next $(\M_{AB}\circ B\vee \M_{BC}\circ B)=\F_B$ iff $(\M_{AB}\circ B\vee \M_{BC}\circ B)^{*}=\F_B^{*},$ i.e.,
$  (\M_{AB}^{*}\times B\wedge\M_{BC}^{*}\times B)=\0_B.$
Using the above argument, we therefore have
$r(\M_{AB}^{*}\vee\M_{BC}^{*})=r(\M_{AB}^{*})+ r(\M_{BC}^{*}).$
\\
On the other hand, suppose $r(\M_{AB}\vee\M_{BC})= r(\M_{AB})+r(\M_{BC})$ and $  r(\M_{AB}^{*}+\M_{BC}^{*})$\\$= r(\M_{AB}^{*})+r(\M_{BC}^{*}).$
\\
The former condition implies $r(\M_{AB}\times B\wedge \M_{BC}\times B)=0,$ i.e., $(\M_{AB}\times B\wedge \M_{BC}\times B)=\0_B.$
 and the latter implies $r(\M_{AB}^{*}\times B\wedge \M_{BC}^{*}\times B)=0,$ i.e.,
$(\M_{AB}^{*}\times B\wedge \M_{BC}^{*}\times B)=\0_B,$
i.e., $(\M_{AB}^{*}\times B\wedge \M_{BC}^{*}\times B)^{*}=(\M_{AB}\circ B\vee \M_{BC}\circ B)=\F_B$ (using parts 3 and 4 of Theorems \ref{thm:sumintersectionm},\ref{thm:dotcrossidentitym}).

4. We have,
$(\M_{AB}\circ B\vee \M_{BC}\circ B)^{*}=\M_{AB}^{*}\times B\ \wedge \ \M_{BC}^{*}\times B.$ Therefore $(\M_{AB}\circ B\vee \M_{BC}\circ B)=\F_B$
iff $\M_{AB}^{*}\times B\ \wedge \ \M_{BC}^{*}\times B= \0_B.$
Next
$(\M_{AB}\times B\ \wedge\ \M_{BC}\times B)^{*}=\M_{AB}^{*}\circ B\vee \M_{BC}^{*}\circ B.$ Therefore $(\M_{AB}\times B\ \wedge\ \M_{BC}\times B) =\0_B$ iff $\M_{AB}^{*}\circ B\vee \M_{BC}^{*}\circ B= \F_B.$

5. 
%Since $\{\M_{AB},\M_{BC}\}$ is rigid, by part 3 above, we must have  
%$r(\M_{AB}\vee \M_{BC})= r(\M_{AB}) +r(\M_{BC}).$
%Therefore if  $b_{ABC}$ is a base of $\M_{AB}\vee \M_{BC},$
%then there exist disjoint bases $b_{AB},b_{BC}$ of $\M_{AB},\M_{BC},$ respectively
%such that $b_{ABC}=b_{AB}\uplus b_{BC}.$
Let $\{\M_{AB},\M_{BC}\}$ be rigid and let
 $b_{AC}$ be a base of $\M_{AB}\lrar \M_{BC}= (\M_{AB}\vee \M_{BC})\times AC.$
Then there exists a base $b_{ABC}$  of $\M_{AB}\vee \M_{BC}$ such that 
$b_{AC}=b_{ABC}\cap (A\uplus C).$ 
 Since $\{\M_{AB},\M_{BC}\}$ is rigid, by part 3 above, we must have
$r(\M_{AB}\vee \M_{BC})= r(\M_{AB}) +r(\M_{BC}).$
Therefore 
%since  $b_{ABC}$ is a base of $\M_{AB}\vee \M_{BC},$
there exist disjoint bases $b_{AB},b_{BC}$ of $\M_{AB},\M_{BC},$ respectively
such that $b_{ABC}=b_{AB}\uplus b_{BC}.$
We have,\\
$|b_{AC}|=|b_{AB}\uplus b_{BC}|-|(b_{AB}\uplus b_{BC})\cap B|= r(\M_{AB}\lrar \M_{BC})= r((\M_{AB}\vee \M_{BC})\times AC)$\\$=
r(\M_{AB}\vee \M_{BC}) - r((\M_{AB}\vee \M_{BC})\circ  B)=
r(\M_{AB})+r(\M_{BC})-|B|$ (using full sum property of $\{\M_{AB},\M_{BC}\}$).
We conclude that $(b_{AB}\uplus b_{BC})\cap B=B.$
\\Thus rigidity of $\{\M_{AB},\M_{BC}\}$
implies that there exist disjoint bases of $\M_{AB},\M_{BC}$ 
which cover $B.$

On the other hand suppose that there exist disjoint bases of $\M_{AB},\M_{BC}$
which cover $B.$ It is clear that we must have 
$r(\M_{AB}\vee \M_{BC})= r(\M_{AB}) +r(\M_{BC})$
 and $r(\M^*_{AB}\vee \M^*_{BC})= r(\M^*_{AB}) +r(\M^*_{BC}),$
so that the rigidity of $\{\M_{AB},\M_{BC}\}$ follows.

\end{proof}
The next lemma is of use for characterizing rigid networks without ports
and, in the case of rigid multiports, the dimension of the multiport
of the multiport behaviour. Proof is given in \ref{subsec:rigidpairsmp}.
\begin{lemma}
\label{lem:rankfactsmatroids}
\begin{enumerate}
\item Let $\M^1_B,\M^2_B$ be matroids on $B.$
If $r(\M^1_B)+r(\M^2_B)=|B|,$ then
the full sum property and zero intersection property are equivalent
\item If  $\{\M_{AB},\M_{BC}\}$ is rigid
 then 
$r(\M_{AB}\lrar \M_{BC})= r(\M_{AB})+r(\M_{BC})-|B|.$
\end{enumerate}
\end{lemma}

The following result is parallel to Theorem \ref{thm:derivedregularity}.
Its proof is a line by line translation of that of Theorem \ref{thm:derivedregularity} and is shifted to \ref{subsec:rigidpairsmp}.

\begin{theorem}
\label{thm:derivedregularitym}
Let $\M_{WTV},\M_V, \M_T,$ be matroids on $W\uplus T\uplus V, V,T,$ \\
respectively.
% and let $(\V_{WTV},\V_R)$ be regular.
Then $\{\M_{WTV},\M_T\oplus\M_V\}$ is rigid
\begin{enumerate}
\item
iff
$\{\M_{WTV},\M_V\},$  $\{\M_{WTV}\lrar\M_V,\M_T\}$  are rigid;
\item iff
$\{\M_{WTV},\M_V\}$ is rigid, $\{\M_{WTV}\wedge \M_V,\M_T\},$
$\{\M_{WTV}\vee \M_V,\M_T\}$ satisfy  the full sum and zero intersection
properties respectively;
\item iff $\{\M_{WTV},\M_V\},$ $\{\M_{WTV}\wedge \M_V,\M_T\},$
$\{\M_{WTV}\vee \M_V,\M_T\}$ are rigid.
\end{enumerate}
\end{theorem}

\subsection{Testing rigidity of vector space pairs through associated matroid pairs}
Matroids associated with rigid pairs of vector  spaces are always rigid.
But  under certain 
commonly occurring situations, the converse is also true. This fact is useful for testing the 
rigidity of the original pairs of vector spaces avoiding linear algebraic computations. 

%======================================================

\begin{theorem}
\label{thm:rigidmatroidvector}
Let $\V_{AB},\V_{BC}$ be vector spaces on $A\uplus B,B\uplus C,$ respectively
with $A,B,C$ pairwise disjoint.
We then have the following.
\begin{enumerate}
\item  If $\{\V_{AB},\V_{BC}\}$ is rigid then so is $\{\M(\V_{AB}), \M(\V_{BC})\}.$
\item
Suppose $\M(\V_{AB}+\V_{BC})=\M(\V_{AB})\vee \M(\V_{BC})$ and
$\M(\V^{\perp}_{AB}+\V^{\perp}_{BC})=\M(\V^{\perp}_{AB})\vee \M(\V^{\perp}_{BC}).$\\
Then
\begin{enumerate}
\item $\M(\V_{AB}\cap\V_{BC})= \M(\V_{AB})\wedge\M(\V_{BC}),
\M(\V_{AB}^{\perp}\cap\V_{BC}^{\perp})= (\M(\V_{AB}))^*\wedge(\M(\V_{BC}))^*$
\item $\M(\V_{AB}\lrar\V_{BC})= \M(\V_{AB})\lrar  \M(\V_{BC});$
\item
$\{\V_{AB},\V_{BC}\}$ is rigid iff $\{\M(\V_{AB}), \M(\V_{BC})\}$
is rigid.
\end{enumerate}
\end{enumerate}
\end{theorem}

\begin{proof}
We will use the following facts.
%\begin{enumerate}
\\(a)  $\M(\V_X\circ Z)=(\M(\V_X))\circ Z, \M(\V_X\times Z)=(\M(\V_X))\times Z,$\\
$(\M(\V_X)^{\perp})= (\M(\V_X))^*, Z\subseteq X$ (Lemma \ref{lem:minorvectorspacem}).
\\(b) $\M^1_X\wedge \M^2_X\equivd ((\M^1_X)^*\vee (\M^2_X)^*)^*.$
\\When $A,B,C$ are pairwise disjoint, $\M_{AB}\wedge \M_{BC}\equivd (\M_{AB}\oplus \F_C)\wedge(\M_{BC}\oplus \F_A)$\\$= ((\M_{AB}\oplus \F_C)^*\vee(\M_{BC}\oplus \F_A)^*)^*= (\M_{AB}^*\oplus \0_C)\vee(\M_{BC}^*\oplus \0_A))^* = (\M_{AB}^*\vee\M_{BC}^*)^*$ (Theorem \ref{thm:idt0m}).
\\(c) When $A,B,C$ are pairwise disjoint, $(\M_{AB}\vee \M_{BC})\circ B=
(\M_{AB}\circ B)\vee (\M_{BC}\circ B), (\M_{AB}\wedge \M_{BC})\times B=
(\M_{AB}\times B)\wedge (\M_{BC}\times B)$ (Theorem \ref{thm:sumintersectionm}).
\\(d) $r(\M(\V_{AB}+\V_{BC}))\leq r(\M(\V_{AB})\vee \M(\V_{BC})).$

1.  $\{\V_{AB},\V_{BC}\}$ is rigid iff  
$r(\V_{AB})+r(\V_{BC})=r(\V_{AB}+\V_{BC})$ and 
$r(\V_{AB}^{\perp})+ r(\V_{BC}^{\perp})$\\$= r(\V_{AB}^{\perp}+ \V_{BC}^{\perp}),$
 by part 4 of Theorem \ref{thm:regularrecursive}.
Therefore, if $\{\V_{AB},\V_{BC}\}$ is rigid 
then we have, \\
$r(\M(\V_{AB}))+r(\M(\V_{BC}))=r(\M(\V_{AB}+\V_{BC}))\leq r(\M(\V_{AB})\vee \M(\V_{BC}))$ and
$r(\M(\V_{AB}^{\perp}))+ r(\M(\V_{BC}^{\perp}))$\\$= r(\M(\V_{AB}^{\perp}+ \V_{BC}^{\perp}))
\leq r(\M(\V_{AB}^{\perp})\vee \M(\V_{BC}^{\perp})).$
But, by definition of the matroid union operation, $r(\M(\V_{AB}))+r(\M(\V_{BC}))\geq r(\M(\V_{AB})\vee \M(\V_{BC}))$
and $r(\M(\V_{AB}^{\perp}))+ r(\M(\V_{BC}^{\perp}))$\\$\geq r(\M(\V_{AB}^{\perp})\vee \M(\V_{BC}^{\perp})).$
We therefore conclude that 
$r(\M(\V_{AB}))+r(\M(\V_{BC}))= r(\M(\V_{AB})\vee \M(\V_{BC}))$
and $r((\M(\V_{AB}))^*)+r((\M(\V_{BC}))^*)=r(\M(\V_{AB}^{\perp}))+ r(\M(\V_{BC}^{\perp}))= 
r((\M(\V_{AB}))^*\vee (\M(\V_{BC}))^{*})$
 and by part 3 of Theorem \ref{thm:regularrecursivem}, it follows that $\{\M(\V_{AB}), \M(\V_{BC})\}$ is rigid.

2(a) We are given that $\M(\V_{AB}+\V_{BC})=\M(\V_{AB})\vee \M(\V_{BC}).$ 
\\We then have
$\M(\V_{AB}\cap\V_{BC})= \M((\V_{AB}^{\perp}+\V_{BC}^{\perp})^{\perp})=
 (\M(\V_{AB}^{\perp})\vee\M(\V_{BC}^{\perp}))^*$\\$=((\M(\V_{AB}))^*\vee(\M(\V_{BC}))^*)^* = \M(\V_{AB})\wedge \M(\V_{BC}),$
using facts (a) and (b).
\\ Similarly, from $\M(\V^{\perp}_{AB}+\V^{\perp}_{BC})=\M(\V^{\perp}_{AB})\vee \M(\V^{\perp}_{BC}),$ it follows that\\ 
$\M(\V_{AB}^{\perp}\cap\V_{BC}^{\perp})= (\M(\V_{AB}))^*\wedge(\M(\V_{BC}))^*.$

2(b) We have $\M(\V_{AB}\lrar \V_{BC})= \M((\V_{AB}+\V_{BC})\times B)=
(\M(\V_{AB}+\V_{BC}))\times B$\\$ =(\M(\V_{AB})\vee\M(\V_{BC}))\times B
= \M(\V_{AB})\lrar \M(\V_{BC}),$
using the hypothesis of the theorem,  fact (a), and the definition of
the `$\lrar$' operation for matroids.

2(c) The pair  $\{\V_{AB},\V_{BC}\}$ is  rigid iff
$\V_{AB}\circ B+\V_{BC}\circ B=\F_B$ and $\V_{AB}\times B\cap \V_{BC}\times B=\0_B,$\\ i.e., iff
$(\V_{AB}+\V_{BC})\circ B=\F_B$ and $(\V_{AB}\cap \V_{BC})\times B=\0_B,$\\ i.e., iff
$\M((\V_{AB}+\V_{BC})\circ B)=\F_B$ and $\M((\V_{AB}\cap \V_{BC})\times B)=\0_B,$
\\
\i.e., iff  $(\M(\V_{AB})\vee\M(\V_{BC}))\circ B=\F_B$ and  $(\M(\V_{AB})\wedge\M(\V_{BC}))\times B=\0_B$
\\i.e., iff $\{\M(\V_{AB}), \M(\V_{BC})\}$
is rigid, using  fact (c), hypothesis of part 2 of the theorem  and part 2(a) above.
\end{proof}

\subsection{Rigidity and matroid duality}
\label{sec:rigiddualm}
We show in this subsection that
when we have matroid duality, we automatically have rigidity.
This suggests
 rigidity can be regarded as  a weak form of 
matroid duality.
\begin{lemma}
\label{lem:gyrator_r2m}
Let $\M_{AB},\M_{BC}$ be matroids, $A,B,C,$ being pairwise disjoint,
such that\\  $(\M_{AB}\circ B)^{*}= \M_{BC}\times B.$
We then have the following.
\begin{enumerate}
\item
  $\{\M_{AB}, \M_{BC}\}$ is rigid;
\item $\{\M_{AB}, (\M^{*}_{AB})_{\tilde{A}B}\},$ is rigid, where $\tilde{A}$ is
a copy of $A$ disjoint from $A,B;$
\end{enumerate}
\end{lemma}
\begin{proof}
1.
Since $(\M_{AB}\circ B)^{*}= \M_{BC}\times B,$ we must have 
$\M_{AB}\circ B\vee \M_{BC}\times B=\F_B.$
Now every independent set of $\M_{BC}\times B$ is contained 
in an independent set of $\M_{BC}\circ B.$ 
%and also every independent set 
%of the latter contains an independent set of the former.
Therefore $\M_{AB}\circ B\vee \M_{BC}\circ B=\F_B,$
i.e., the full sum property holds for $\{\M_{AB}, \M_{BC}\}.$
\\ Next $(\M_{AB}\circ B)^{**}= (\M_{BC}\times B)^*.$ 
Therefore, $ (\M_{AB}^*\times B)^{*}= \M_{BC}^*\circ B,$
i.e., $ (\M_{BC}^*\circ B)^{*}= \M_{AB}^*\times B.$
By the argument above, this means $ \M_{BC}^*\circ B\vee \M_{AB}^*\times B=
\F_B,$ i.e., the full sum property holds for $\{\M_{AB}^*, \M_{BC}^*\}.$ By part 4 of Theorem \ref{thm:regularrecursivem}, this means that the zero intersection 
property holds for $\{\M_{AB}, \M_{BC}\}.$

We conclude that $\{\M_{AB}, \M_{BC}\}$ is rigid.

2. We note that $(\M_{AB}\circ B)^{*}=\M_{AB}^{*}\times B=
(\M^{*}_{AB})_{\tilde{A}B}\times B.$
The result now follows from part 1 above.
\end{proof}
\begin{remark}
We have that $\{\M_X,\M_X^{*}\}$ is rigid.  By Theorem \ref{thm:idt0m}, $\{\M_A\lrar \M_B, \M^{*}_A\lrar \M^{*}_B\},B\subseteq A,$ is also rigid.
We have that $\{\M^1_A\vee \M^2_A, (\M^1_A)^{*}\wedge  (\M^2_A)^{*}\}$ is rigid since $(\M^1_A\vee \M^2_A)^{*}=(\M^1_A)^{*}\wedge  (\M^2_A)^{*}.$
\\There appear to be no counterparts to these facts in terms of rigidity
 without using duality.
\end{remark}

\subsection{Computing a column base of $\V_{AB}\lrar\V_{BC}$}
Finding a column base of $\V_{AB}\lrar \V_{BC}$
%a vector space which is specified 
%by a general representative matrix, 
requires linear algebraic
computations. However, when the conditions of part 2 of Theorem \ref{thm:rigidmatroidvector} are satisfied,
we have $\M(\V_{AB})\lrar \M(\V_{BC})= \M(\V_{AB}\lrar \V_{BC}).$
Bases of $\M(\V_{AB}\lrar \V_{BC})$ and  column bases of
$\V_{AB}\lrar \V_{BC}$ are identical.
The problem of computing a base of $\M_{AB}\lrar\M_{BC}$ given 
an `independence oracle' of the matroids $\M_{AB},\M_{BC}$ 
is one whose solution is nontrivial but known.
First, we find a pair of maximally distant bases $b^1_B,b^2_B$ for the matroids $\M_{AB}\circ B,\M_{BC}\circ B$ using the  well known matroid union algorithm essentially due to
J. Edmonds \cite{edm65a}.
 It follows that $b^1_B\cup b^2_B$ is a basis for $(\M_{AB}\vee\M_{BC})\circ B= \M_{AB}\circ B\vee \M_{BC}\circ B.$
We then extend these  bases to $b^1_{AB},b^2_{BC}$ of $\M_{AB},\M_{BC}$
respectively. Then $b^1_{AB}\cup b^2_{BC}$ is a basis of 
$\M_{AB}\vee \M_{BC}$ that contains the basis $b^1_B\cup b^2_B$
of $(\M_{AB}\vee\M_{BC})\circ B.$ Therefore, $(b^1_{AB}-b^1_B)\uplus 
(b^2_{BC}-b^2_B)$ is a base of $(\M_{AB}\vee\M_{BC})\times (A\uplus C)=
\M_{AB}\lrar\M_{BC}.$
A detailed discussion of the matroid union algorithm may be found in
\cite{recski89},\cite{HNarayanan1997}.
Fortunately for us, in the most important cases where we have to find maximally distant bases, we can do so by computing one of them to satisfy some  
simple properties.

\subsection{Rigid families of matroids}
\label{subsec:rigidfamiliesmatroid}
The notion of rigidity of families carries over from vector spaces 
to matroids routinely. 
 For this purpose  we first give a definition of rigidity
for families relevant to our problem (`associative families') and then justify the definition by describing
the properties of rigid associative families through subsequent lemmas.

\begin{definition}
\label{def:gen_lrar1m}
An  \nw{associative family}  $\mathcal{H}\equivd \{\M_{Y_i},i=1, \cdots , m\}$ is a family of matroids where no element in $Y\equivd \bigcup_{i=1}^mY_i$ belongs to more than two of the $Y_i.$
When $Y_i=\emptyset,$ we will take $\M_{Y_i}=\M_{\emptyset}\equivd \emptyset .$
Let $\mathcal{H}_{nonvoid}\equivd \{\M_{Y_i},i=1, \cdots , t\}$ be the subfamily 
of $\mathcal{H}$ whose elements are non void matroids that belong to
$\mathcal{H}.$
%We will permit $\K_{\emptyset}\equivd \emptyset $ to be one of the 
%elements of $\mathcal{H}.$ 
For $i=1, \cdots , t,$ let $Z_i$ be the subset of elements of $Y_i$
which belong to none of the $Y_j, j\ne i$ and let $Z\equivd \bigcup_{i=1}^tZ_i.$\\
We define $\mnw{\lrar(\mathcal{H})} \equivd
(\bigvee_{i=1}^t \M_{Y_i}))\times Z= (\M_{Y_1}\vee \cdots \vee \M_{Y_t})\times Z.$ 
\end{definition}
\begin{definition}
For an associative family $\mathcal{H}\equivd \{\M_{Y_1}, \cdots , \M_{Y_m}\},$
the graph $\mnw{\G_{\mathcal{H}}}$ is defined to have  vertices $\M_{Y_1}, \cdots , \M_{Y_m},$
edges $Y_1\cup \cdots \cup Y_m$ with an edge $e$ lying between
vertices $\M_{Y_i},\M_{Y_j}$ iff $e\in Y_i\cap Y_j.$
We will refer to $\G_{\mathcal{H}}$ as the graph of $\mathcal{H}.$
\end{definition}

\begin{remark}
We note the following.
\begin{enumerate}
\item The operation $\lrar(\cdot)$ is well defined over associative families and agrees with the corresponding 
operation over pairs of matroids.
If $\mathcal{H}\equivd \{\M_{Y_1},\M_{Y_2},\M_{Y_3}\},$ is an associative family, it can be shown that
\\$\lrar(\mathcal{H})= (\M_{Y_1}\lrar\M_{Y_2})\lrar \M_{Y_3}= \M_{Y_1}\lrar(\M_{Y_2}\lrar \M_{Y_3})$ \cite{STHN2014}. 
By induction it follows that, treated as a binary operation,  $\lrar(\cdot)$  is associative over an associative family 
$\mathcal{H}$ and, independent of the way the brackets are constructed this operation
reduces to $\lrar(\mathcal{H}).$
\item Suppose $\mathcal{H}\equivd \{\M_{Y_i},i=1, \cdots , m\}$ is an associative
family with every element of $\bigcup _{i=1}^{m}\M_{Y_i}$ belonging to exactly
two of the $Y_i.$ Then $(\lrar(\mathcal{H}))=\emptyset.$ For consistency of notation,
 in this case we denote $\lrar(\mathcal{H})$ by $\M_{\emptyset}.$
We note that, by the definition of the operations $\vee, \wedge, \lrar(\cdot),$  we must have  $(\M_{\emptyset}\vee \M_X)= (\M_{\emptyset}\wedge \M_X)= (\M_{\emptyset}\lrar \M_X)= \M_X.$ 
The following result is known \cite{STHN2014}.
\begin{lemma}
\label{lem:veewedgeassociative}
Let $\mathcal{H}\equivd \{\M_{Y_1}, \cdots , \M_{Y_m}\}$ be an associative family of matroids. Let $Z$ be the subset of elements of $\bigcup_iY_i$ that belong
 to only one of the $Y_i.$
Then \\
 $\mnw{\lrar(\mathcal{H})}$ $ \equivd
(\bigwedge_{i=1}^m \M_{Y_i}))\circ Z= (\M_{Y_1}\wedge \cdots \wedge \M_{Y_m})\circ Z.$
\end{lemma}
\item 
If the graph $\G_{\mathcal{H}}$ has edges and is connected, then
it does not have isolated vertices. This means $\mathcal{H}$ has no element
which is equal to $\M_{\emptyset}.$ Further, if it is connected and  has at least one selfloop, no subfamily $\mathcal{H}_1$ of $\mathcal{H}$ would be such that $\lrar(\mathcal{H}_1)=\M_{\emptyset}.$
\item Let $\mathcal{H}=\mathcal{H}_1\uplus \mathcal{H}_2.$ Then $\lrar(\mathcal{H})\ =\ (\lrar(\mathcal{H}_1)) \lrar  (\lrar(\mathcal{H}_2)).$
%\item  It is clear that 
%$\lrar(\M)\ =\ (\bigcap_i \K_{Y_i})\circ Z,$
%$\rightleftharpoons(\M)\ =\ (\Sigma_i \K_{Y_i})\times Z,$
%where $Z$ is the set of all elements of $\bigcup_i Y_i$ which belong to 
%only one of the $Y_i.$\\
%For each $e \in (\bigcup_i Y_i-Z),$ pick one of the $Y_i$ such that $e\in \K_{Y_i}$
%%and denote that $\K_{Y_i}$ by $\K^e_{Y_i}.$ \\Let $\hat{\K}_{Y_i}\equivd \{g_{Y_i},
%g_{Y_i}(e)=-f_{Y_i}(e),\ \mbox{if} \ \K_{Y_i}=\K^e_{Y_i}\ \mbox{and}\ 
%g_{Y_i}(e)=f_{Y_i}(e),\ \mbox{if} \ \K_{Y_i}\ne \K^e_{Y_i}, f_{Y_i}\in \K_{Y_i}\}$ and let $\hat{\M}\equivd \{\hat{\K}_{Y_i}, i=1, \cdots , m\}.$
%We then have 
%$\lrar(\M)\ =\ (\Sigma_i \hat{\K}_{Y_i})\times Z$ and
%$\rightleftharpoons(\M)\ =\ (\bigcap_i \hat{\K}_{Y_i})\circ Z.$

\item By induction on the size of $\mathcal{H},$ one can prove the following generalization of 
Theorem \ref{thm:idt0}:
\begin{theorem}
\label{thm:idtassociativem}
Let $\mathcal{H}\equivd \{\M_{Y_i},i=1, \cdots , m\}$ be an associative family of matroids
and let\\ $\mathcal{H}^{*}\equivd \{\M^{*}_{Y_i},i=1, \cdots , m\}.$
Then
$(\lrar(\mathcal{H}))^{*}=(\lrar(\mathcal{H}^{*})).$
\end{theorem}
\end{enumerate}
\end{remark}
\begin{definition}
\label{def:rigidnm}
%A family $\hat{\M}\equivd \{\A_{Y_1},\cdots , \A_{Y_n}\}$ of affine spaces
%is said to be rigid iff 
Let ${\mathcal{H}}\equivd \{\M_{Y_1},\cdots , \M_{Y_n}\}$ be an associative family of matroids.
We say ${\mathcal{H}}$ is \nw{rigid} iff
the following rank conditions are satisfied.
\begin{itemize}
\item (primal) $r(\Sigma_i\M_{Y_i})=
\Sigma_ir(\M_{Y_i});$
\item (dual) $r(\Sigma_i\M^{*}_{Y_i})=
\Sigma_ir(\M^{*}_{Y_i}).$
\end{itemize}
\end{definition}
We note that the definition of rigid families of matroids agrees with that of a pair 
of matroids (by part 3 of Theorem \ref{thm:regularrecursivem}).

We  have the following simple lemma.
\begin{lemma}
\label{lem:rigidsubfamilym}
Let ${\mathcal{H}}\equivd \{\M_{Y_1},\cdots , \M_{Y_n}\}$ be an associative family of matroids with $\G_{\mathcal{H}}$ connected and with non empty edge set.
%$\V_{Y_1}$ for $\V_{Y_1}, i= 1, \cdots 
If $\mathcal{H}$ is rigid then
 every subfamily of ${\mathcal{H}}$ is rigid.
\end{lemma}
\begin{proof}
%Let $\A_{Y_1},\cdots , \A_{Y_n}$ have vector space translates $\V_{Y_1},\cdots , \V_{Y_n},$ respectively. 
Without loss of generality, it is adequate
to show that $\{\M_{Y_1},\cdots , \M_{Y_m}\}, m<n $ is rigid.
Now $r(\M_{Y_1}\vee\cdots\vee \M_{Y_m})\leq r(\M_{Y_1})+\cdots+ r(\M_{Y_m})$
and $r(\bigvee_{i=1}^n\M_{Y_i})\leq r(\bigvee_{i=1}^m\M_{Y_i})+r(\bigvee_{i=m+1}^n\M_{Y_i}).$
Therefore the primal rank condition of rigidity is true for $\{\M_{Y_1},\cdots , \M_{Y_n}\}$  only if it is true for $\{\M_{Y_1},\cdots , \M_{Y_m}\}.$
The proof for the dual rank condition is essentially the same.
\end{proof}

The following lemma addresses the rigidity of a common class of associative families. 

\begin{lemma}
\label{lem:rigidpairrankm}
Let $\mathcal{H}_1\equivd \{\M_{X_1},\cdots ,\M_{X_n}\},$ and let $\mathcal{H}_2\equivd \{\M_{Y_1},\cdots ,\M_{Y_m}\},$ with the $X_i,$ being pairwise disjoint and the $Y_j$
also. 
%Let $\V_{X_i},\V_{Y_j},$ be the vector space associates of
%$\A_{X_i},\A_{Y_j},$ respectively.
Then the associative family $\mathcal{H}_1\uplus \mathcal{H}_{2}$ is rigid iff the
pair $\{\oplus \M_{X_i},\oplus \M_{Y_j}\},$ is rigid.
%where $\V_{X_i}\in \B_l, \V_{Y_j}\in \B_r,$ is rigid.
%Let $\M\equivd \{\V_{X_1},\cdots ,\V_{X_n}\}, \ \M_{link}\equivd \{\V_{Z_1\tilde{Z}_1},\cdots ,\V_{Z_k\tilde{Z}_k}\},$ 
%as in Definition \ref{def:connection}.
%Then the connection $\M\uplus \M_{link}$
%is rigid iff the pair $\{\V_{X_1}\oplus \cdots \oplus \V_{X_n}, \V_{Z_1\tilde{Z}_1}\oplus \cdots \oplus \V_{Z_k\tilde{Z}_k}\}$ is rigid.
\end{lemma}
\begin{proof}
We have\\ $ r(\M_{X_1}\oplus \cdots \oplus \M_{X_n})=r(\M_{X_1})+ \cdots + r(\M_{X_n}),\ 
 r(\M_{Y_1}\oplus \cdots \oplus \M_{Y_m})=r(\M_{Y_1})+ \cdots + r(\M_{Y_n}),$\\$
 r(\M^{*}_{X_1}\oplus \cdots \oplus \M^{*}_{X_n})=r(\M^{*}_{X_1})+ \cdots + r(\M^{*}_{X_n}),\ 
 r(\M^{*}_{Y_1}\oplus \cdots \oplus \M^{*}_{Y_m})=r(\M^{*}_{Y_1})+ \cdots + r(\M^{*}_{Y_n}).$
\\ The result now follows from part 3 of Theorem \ref{thm:regularrecursivem}.
\end{proof}
We now have a result which expresses rigidity of associative families 
recursively.
\begin{theorem}
\label{thm:associativerigidrecursivem}
Let ${\mathcal{H}}\equivd \{\M_{Y_1},\cdots , \M_{Y_n}\}$ be an associative family of matroids with $\G_{\mathcal{H}}$ connected and with non empty edge set.
%affine spaces with vector space translates $\{\V_{Y_1},\cdots , \V_{Y_n}\}$
Let ${\mathcal{H}}$ be partitioned as ${\mathcal{H}_1}\uplus {\mathcal{H}_2}.$
Then
${\mathcal{H}}$ is rigid iff ${\mathcal{H}_1},{\mathcal{H}_2}$ and $\{\lrar(\mathcal{H}_1), \lrar(\mathcal{H}_2)\}$
are rigid.
\end{theorem}
The proof of Theorem \ref{thm:associativerigidrecursivem} is based on the following lemma.
\begin{lemma}
\label{lem:rankconditionm}
Let ${\mathcal{H}}\equivd \{\M_{Y_i},i=1, \cdots , m\}$
and let ${\mathcal{H}}_1\equivd \{\M_{Y_i},i=1, \cdots , t\}, t<m,$\\
${\mathcal{H}}_2\equivd   \{\M_{Y_i},i=t+1, \cdots , m\}$
be subfamilies of ${\mathcal{H}}.$
Let $Y^1=\bigcup^t_{i=1} Y_i$ and let $Y^2=\bigcup^m_{i=t+1} Y_i.$
Then ${\mathcal{H}}$ satisfies the rank conditions of Definition \ref{def:rigidn}
 iff
${\mathcal{H}}_1,{\mathcal{H}}_2$ satisfy them and
further\\
$ r([\bigvee^t_{i=1}\M_{Y_i}]\times (Y^1\cap Y^2) \bigwedge [\bigvee^m_{i=t+1}\M_{Y_i}]\times (Y^1\cap Y^2))=0$ and\\
$ r([\bigvee^t_{i=1}\M^{*}_{Y_i}]\times (Y^1\cap Y^2) \bigwedge [\bigvee^m_{i=t+1}\M^{*}_{Y_i}]\times (Y^1\cap Y^2))=0.$
\end{lemma}
\begin{proof}
When $Y_1\cap Y_2=\emptyset,$ it is clear that if the rank condition is satisfied
for ${\mathcal{H}_1},{\mathcal{H}_2}$ the lemma is trivially true.
Let $Y_1\cap Y_2$ be 
nonvoid.
By Theorem \ref{thm:sumintersectionm}, \\$r(\M_A\vee\M_B)= r(\M_A)+r(\M_B)-r([\M_A\times (A\cap B)]\wedge [\M_B\times (A\cap B)].$ Therefore,\\
$r(\bigvee^m_{i=1}\M_{Y_i}) =r(\bigvee^t_{i=1}\M_{Y_i})+r(\bigvee^m_{i=t+1}\M_{Y_i})-  r([\bigvee^t_{i=1}\M_{Y_i}]\times (Y^1\cap Y^2) \wedge [\bigvee^m_{i=t+1}\M_{Y_i}]\times (Y^1\cap Y^2)).$
\\
It follows that $r(\bigvee^m_{i=1}\M_{Y_i}) = \Sigma^m_{i=1}r(\M_{Y_i})$
iff $r(\bigvee^t_{i=1}\M_{Y_i}) = \Sigma^t_{i=1}r(\M_{Y_i}), t<m,$ \\
$r(\bigvee^m_{i=t+1}\M_{Y_i}) = \Sigma^m_{i=t+1}r(\M_{Y_i})$
and $ r([\bigvee^t_{i=1}\M_{Y_i}]\times (Y^1\cap Y^2) \wedge [\bigvee^m_{i=t+1}\M_{Y_i}]\times (Y^1\cap Y^2))=0.$
\\ Working with $\M^{*}_{Y_i}$ and repeating the argument, we get the
second half of the statement of the lemma.
\end{proof}
\begin{proof}
(Proof of Theorem \ref{thm:associativerigidrecursivem})\\
%We note that 
%when an associative family $\mathcal{H}\equivd \{\A_{Y_1}, \cdots , \A_{Y_m}\}$ has   
%vector space translates $\{\V_{Y_1}, \cdots , \V_{Y_m}\},$ its rigidity is defined in terms of that of the latter.
%Further
%by Lemma \ref{lem:associativetranslate}, we know that 
% when $\mathcal{H}$ is rigid, 
%$\lrar(\M), \rightleftharpoons(\M)$ are nonvoid and have vector space translates
%$\lrar(\{\V_{Y_1}, \cdots , \V_{Y_m}\}), \rightleftharpoons(\{\V_{Y_1}, \cdots , \V_{Y_m}\})$ respectively.  
%We therefore prove the theorem taking $\M\equivd \{\V_{Y_1}, \cdots , \V_{Y_m}\}.$
By Lemma \ref{lem:rankconditionm}, the primal  rank condition of Definition \ref{def:rigidn} ($r(\bigvee_{i}\M_{Y_i}) = \Sigma_{i}r(\M_{Y_i})$)
 holds for 
$\mathcal{H}$ iff it holds for ${\mathcal{H}_1},{\mathcal{H}_2},$ and $ r([\bigvee^t_{i=1}\M_{Y_i}]\times (Y^1\cap Y^2) \wedge [\bigvee^m_{i=t+1}\M_{Y_i}]\times (Y^1\cap Y^2))=0.$ 
%and \\$ r([\Sigma^t_{i=1}\M^{\perp}_{Y_i}]\times (Y^1\cap Y^2) \bigcap [\Sigma^m_{i=t+1}\M^{\perp}_{Y_i}]\times (Y^1\cap Y^2))=0.$\\
%We will examine the equality involving $\M_{Y_i}.$
\\We have $r([\bigvee^t_{i=1}\M_{Y_i}]\times (Y^1\cap Y^2) \wedge [\bigvee^m_{i=t+1}\M_{Y_i}]\times (Y^1\cap Y^2))$\\$=
r([\bigvee^t_{i=1}\M_{Y_i}]\times Y^1\times (Y^1\cap Y^2)\wedge [\bigvee^m_{i=t+1}\M_{Y_i}]\times Y^2\times (Y^1\cap Y^2)
$\\$= (\lrar(\mathcal{H}_1))\times (Y^1\cap Y^2)\wedge (\lrar(\mathcal{H}_2))\times (Y^1\cap Y^2).$
Thus the primal rank condition 
%($r(\Sigma_{i}\M_{Y_i}) = \Sigma_{i}r(\M_{Y_i})$) 
holds for
$\mathcal{H}$ iff it holds for $\mathcal{H}_1,\mathcal{H}_2$ and 
the zero intersection property holds for $\{(\lrar(\mathcal{H}_1)), (\lrar(\mathcal{H}_2))\}.$
%\\$=  (\rightleftharpoons(\M_1))\lrar (\rightleftharpoons(\M_2)).$

%As in Definition \ref{def:skewedpair}, we take $\hat{\M}_{1, \cdots ,m}\equivd \{\hat{\M}_{Y_1}, \cdots , \hat{\M}_{Y_m}\}.$
%Let $\hat{\M}\equivd \hat{\M}_{1, \cdots ,m}, \hat{\M}_1\equivd (\hat{\M}_1)_{1, \cdots ,t}, $\\$ \hat{\M}_2\equivd (\hat{\M}_2)_{t+1, \cdots ,m}.$
%We note that $\hat{\hat{\M}}=\M$ and 
%we have $\lrar(\M)= (\rightleftharpoons(\hat{\M})),$ for any associative 
%family $\M.$ By part 4 of Lemma \ref{lem:rankcondition1}, we know that 
%the primal rank condition holds for $\M$ iff it holds for $\hat{\M}.$
%By the above argument, the primal rank condition holds for $\hat{\M}$ 
%iff it holds for $\hat{\M}_1,\hat{\M}_2$, i.e., for $\M_1,\M_2, $  and
%the zero intersection property holds for $\{(\rightleftharpoons(\hat{\M}_1)), (\rightleftharpoons(\hat{\M}_1))\}= \{(\lrar(\M_1)), (\lrar(\M_2))\}.$
%It follows that
%the primal rank condition holds for $\M$ 
%iff it holds for $\M_1,\M_2$ and
%the zero intersection property holds for $\{(\lrar(\M_1)), (\lrar(\M_2))\}
%.$

The dual rank condition for $\mathcal{H}$ is the same as the primal rank condition 
for $\mathcal{H}^{*}.$ When $\mathcal{H}$ is associative, so is $\mathcal{H}^{*}.$
 Working with $\mathcal{H}^{*}$ 
%and using Theorem \ref{thm:idtassociative},
we can show that the primal rank condition holds for
%$r(\Sigma^m_{i=1}\V^{\*}_{Y_i}) = \Sigma^m_{i=1}r(\V^{\*}_{Y_i})$ holds for
$\mathcal{H}^{*}$ iff it holds for $\mathcal{H}^{*}_1,\mathcal{H}^*_2$ and
the zero intersection property holds for 
$\{(\lrar(\mathcal{H}^{*}_1)), (\lrar(\mathcal{H}^{*}_2))\}.$
By Theorem \ref{thm:idtassociativem}, \\$\{(\lrar(\mathcal{H}^{*}_1)), (\lrar(\mathcal{H}^{*}_2))\}
=
\{(\lrar(\mathcal{H}_1))^{*})), (\lrar(\mathcal{H}_2))^{*}))\}.$
Now by 
 by part 4 of Theorem \ref{thm:regularrecursive}, the 
zero intersection property holds for $\{(\lrar(\mathcal{H}_1))^{*})), (\lrar(\mathcal{H}_2))^{*}))\}$
iff 
the 
full sum property holds 
for\\ $\{(\lrar(\mathcal{H}_1)), (\lrar(\mathcal{H}_2))\}
.$
We thus see that the primal and dual rank conditions of rigidity hold for $\mathcal{H}$ iff they hold for $\mathcal{H}_1, \mathcal{H}_2$ and $\{(\lrar(\mathcal{H}_1)), (\lrar(\mathcal{H}_2))\}
$ is rigid.
Thus $\mathcal{H}$ is rigid iff $\mathcal{H}_1, \mathcal{H}_2, \{(\lrar(\mathcal{H}_1)), (\lrar(\mathcal{H}_2))\}$ 
are rigid.
\end{proof}

The following is a more convenient version of Theorem \ref{thm:associativerigidrecursivem}.
\begin{theorem}
\label{thm:associativerigidrecursivenm}
Let ${\mathcal{H}}\equivd \{\M_{Y_1},\cdots , \M_{Y_m}\}$ be an associative family of matroids with $\G_{\mathcal{H}}$ connected and with non empty edge set.
Let ${\mathcal{H}}$ be partitioned as ${\mathcal{H}_1}\uplus \cdots \uplus {\mathcal{H}_n}.$
Then
 ${\mathcal{H}}$ is rigid iff ${\mathcal{H}_1},\cdots ,{\mathcal{H}_n}$ and $\{\lrar(\mathcal{H}_1),\cdots , \lrar(\mathcal{H}_n)\}$
are rigid.
\end{theorem}
\begin{proof}
1. The proof is by induction over $n.$ We have seen in Theorem \ref{thm:associativerigidrecursivem}, that the result is true when $n=2.$
%Suppose it is true for $n-1.$
We then have ${\mathcal{H}}$ is rigid iff ${\mathcal{H}_1}, ({\mathcal{H}_2}\uplus \cdots \uplus {\mathcal{H}_n}), \{\lrar(\mathcal{H}_1), \lrar({\mathcal{H}_2}\uplus \cdots \uplus \mathcal{H}_n)\}$ are rigid. 
\\
By induction, ${\mathcal{H}_2}\uplus \cdots \uplus {\mathcal{H}_n}$ is rigid iff ${\mathcal{H}_2},\cdots ,{\mathcal{H}_n}$ and $\{\lrar(\mathcal{H}_2),\cdots , \lrar(\mathcal{H}_n)\}$
are rigid.
\\ We therefore need only show that,  rigidity of $\{\lrar(\mathcal{H}_1), \lrar({\mathcal{H}_2}\uplus \cdots \uplus \mathcal{H}_n)\}$ and 
$\{\lrar(\mathcal{H}_2),\cdots , \lrar(\mathcal{H}_n)\}$ implies rigidity of 
$\{\lrar(\mathcal{H}_1),\cdots , \lrar(\mathcal{H}_n)\}.$

Let us denote $\lrar(\mathcal{H}_1),\cdots , \lrar(\mathcal{H}_n),$ respectively by $\M_{Z_1},\cdots , \M_{Z_m}.$ Since $\mathcal{H}$ is an associative 
family, by the definition of $\lrar(\mathcal{H}_i),$ it is clear that no element belongs to more than two of the $Z_i.$ 
%Since $\mathcal{H}_1,\cdots , \mathcal{H}_n,$ are rigid, by Lemma \ref{lem:associativetranslate}, the $\A_{Z_i}$ are nonvoid.
%Let $\A_{Z_1},\cdots , \A_{Z_m},$
% have vector space translates $\V_{Z_1},\cdots , \V_{Z_m},$
% respectively. 
We thus need to show that rigidity of $\{\M_{Z_1},(\M_{Z_2}\lrar \cdots \lrar  \M_{Z_m})\}$ and $\{\M_{Z_2},\cdots , \M_{Z_m}\}$
 implies the rigidity of $\{\M_{Z_1},\cdots , \M_{Z_m}\}.$

Suppose $\{\M_{Z_1},(\M_{Z_2}\lrar \cdots \lrar  \M_{Z_m})\}$ and $\{\M_{Z_2},\cdots , \M_{Z_m}\}$
are rigid. Let $P$ be the set of all elements which belong 
precisely to one of $Z_2,\cdots , Z_m$ and let $T\equivd Z_1\cap \bigcup_{i=2}^mZ_i.$ Since no element belongs to more than 
two of $Z_1,\cdots , Z_m,$ it is clear that $T\subseteq P.$ 
From the zero intersection property of $\{\M_{Z_1},(\M_{Z_2}\lrar \cdots \lrar  \M_{Z_m})\},$ we have that 
$r((\M_{Z_1}\times T)\wedge 
(\M_{Z_2}\lrar \cdots \lrar  \M_{Z_m})\times T)=
r((\M_{Z_1}\times T)\wedge[(\M_{Z_2}\vee \cdots \vee \M_{Z_m})\times P\times T])$\\$=
 r((\M_{Z_1}\times T)\wedge[(\M_{Z_2} \vee\cdots \vee \M_{Z_m})\times  T])=
0.$
By Theorem \ref{thm:sumintersection}, it follows that 
$r(\M_{Z_1}\vee\cdots \vee \M_{Z_m})= r(\M_{Z_1})+r(\M_{Z_2}\vee\cdots \vee \M_{Z_m}).$
Since $\{\M_{Z_2},\cdots , \M_{Z_m}\}$
 is rigid, we have $r(\M_{Z_2}\vee\cdots \vee \M_{Z_m})= r(\M_{Z_1})+r(\M_{Z_2})+\cdots + r(\M_{Z_m}).$
Therefore $r(\M_{Z_1}\vee\cdots \vee \M_{Z_m})= r(\M_{Z_1})+\cdots + r(\M_{Z_m}).$

The proof of the dual rank condition for $\{\M_{Z_1},\cdots , \M_{Z_m}\}$
is by replacing $\M_{Z_i}$ by $\M_{Z_i}^{*}$ everywhere in the above proof and by using the fact that rigidity of $\{\M_{Z_2},\cdots , \M_{Z_m}\}$ is equivalent to rigidity of $\{\M^{*}_{Z_2},\cdots , \M^{*}_{Z_m}\}.$
\\
We thus see that rigidity of $\{\M_{Z_1},(\M_{Z_2}\lrar \cdots \lrar  \M_{Z_m})\}$ and $\{\M_{Z_2},\cdots , \M_{Z_m}\}$
 implies the rigidity of $\{\M_{Z_1},\cdots , \M_{Z_m}\}.$\\
This proves that 
${\mathcal{H}}$ is rigid iff ${\mathcal{H}_1},\cdots ,{\mathcal{H}_n}$ and $\{\lrar(\mathcal{H}_1),\cdots , \lrar(\mathcal{H}_n)\}$
are rigid.

\end{proof}
\section{Connecting rigid multiports to obtain a rigid multiport}
\label{subsec:connect}
In this section we combine the results of Subsections \ref{subsec:rigidpairs} and \ref{subsec:rigidpairsm} to give simple rules for connecting rigid multiports to obtain a rigid multiport, when the device characteristic has a 
special form. This form  is sufficiently general to capture  
commonly obtained device characteristics.

Suppose we have rigid multiports $\N_{WQ},\N_{TW}$ and we wish to connect 
them at ports $W.$ How do we test for rigidity of the resulting multiport 
$\N_{TQ}?$ How do we connect them so that the resulting multiport is rigid?

The general answer is given in Theorems \ref{lem:derivedregularity}, 
 \ref{thm:derivedregularity} and \ref{thm:regularmultiportrecursive}.
%\\If we directly connect them (i.e., by forcing the voltages at $W$ 
%the same for the two networks and making the currents at $W$ negatives 
%of each other) then the resulting multiport would be rigid iff
%the pair $\{\V_{R'W'R"W"}, \V_{{W}'Q'{W}"Q"}\}$ is rigid, where $\V_{R'W'R"W"}, \V_{{W}'Q'{W}"Q"}$ are the vector space translates of the
%port behaviours of $\N_{WQ},\N_{TW}$ respectively.
\\ As in Theorem \ref{thm:regularmultiportrecursive},
we could have rigid multiports $\N_{TW},{\N}_{\tilde{W}Q}$  on graphs
$\G_{RSW},\G_{\tilde{W}MQ}$ respectively and device characteristics
$\A_{S'S"},{\A}_{M'M"}$ respectively.
The multiport $\N_{TQ}\equivd {[\N_{TW}\oplus {\N}_{\tilde{W}Q}]\cap \A^{W\tilde{W}}},$ obtained by connecting 
$\N_{TW},{\N}_{\tilde{W}Q}$
through $\A^{W\tilde{W}}$ is rigid, provided the pair
$\{\V_{R'W'R"W"}\oplus \V_{\tilde{W}'Q'\tilde{W}"Q"},\A^{W\tilde{W}}_{W'W"\tilde{W}'\tilde{W}"}\},$
is rigid, where $\V_{R'W'R"W"}\oplus \V_{\tilde{W}'Q'\tilde{W}"Q"}$ are the vector  space translates of port behaviours of $\N_{TW},{\N}_{\tilde{W}Q}$
respectively.
This test requires solution of linear equations whose nature is very general.
%is a  comprehensive answer to the question of rigidity 
%of the resulting multiport, but as it stands, it is a test involving 
%solution of linear equations whose nature is very general.
Fortunately, in some  important special cases, the test only involves working with 
simple matroids and is therefore very efficient.
%\label{thm:regularmultiportrecursive}
\subsection{Sufficiency of matroidal tests for rigidity of multiports}
\label{subsec:suffmatroidtest}
We simplify the problem of testing for rigidity of multiports by making an assumption that is justified 
in practice.
We assume that the parameters associated with the various devices,
such as {\it resistance} in the case of resistors, {\it gain}, {\it transresistance}, {\it transconductance} in the case 
of controlled sources,  are not known exactly but may be treated 
as being picked at random from an interval of positive width $\epsilon,$
 about some real value. This width may be as small as we please, 
 as long as it is nonzero. This is a way of saying that parameters 
are not known exactly. Under this condition, we may treat 
the parameters to be algebraically independent over $\Q.$
We remind the reader that elements $x_i\in \Re, i=1, \cdots , k$ are
\nw{algebraically independent} over $\mathbb{Q}$ iff
every nontrivial multivariable polynomial in $x_i$ with coefficients
from $\mathbb{Q}$ takes nonzero value.
Essentially, once this assumption is made, we have to 
test for rigidity of a pair $\{\V_{AB},\V_B\},$
which has the property that $\M(\V_{AB}+\V_B)=\M(\V_{AB})\vee \M(\V_B)$
 and $\M(\V^{\perp}_{AB}+\V^{\perp}_B)=\M(\V^{\perp}_{AB})\vee \M(\V^{\perp}_B).$
So, by Theorem \ref{thm:rigidmatroidvector}, the problem reduces to testing the rigidity of the pair
$\{\M(\V_{AB}), \M(\V_B)\}$  which turns out to be simple 
because both the matroids
are simple,  the matroid $\M(\V_{B}),$ being particularly so.

%In our case, the building blocks are rigid multiports and we need that rigidity is preserved even after connection 
%of these building blocks. This requires some additional techniques,
% such as the use of the implicit duality theorem for vector spaces 
%and matroids (Theorems \ref{thm:idt0}, \ref{thm:idt0m}).

In order to deal with connection of multiports in a simple way, we use ideas from topological multiport decomposition \cite{HNarayanan1986a,HNarayanan1997}, which we describe below.

Let $\G_{S_1S_2}$ be a graph. Then it is always possible to construct 
graphs $\G_{S_1W}, \G_{S_2\tilde{W}}, \G_{W\tilde{W}},$ where $W,\tilde{W}$ are additional sets of edges which 
are disjoint from each other as well as disjoint from $S_1,S_2,$ such that \\
$\V^v(\G_{S_1S_2})=(\V^v(\G_{S_1W})\oplus \V^v(\G_{S_2\tilde{W}}))\lrar \V^v(\G_{W\tilde{W}})$ and 
$\V^i(\G_{S_1S_2})=(\V^i(\G_{S_1W})\oplus \V^i(\G_{S_2\tilde{W}}))\lrar \V^i(\G_{W\tilde{W}}).$
\\We can always choose $W,\tilde{W}$ to have size 
$|W|=|\tilde{W}|= \xi(S_1,S_2)\equivd r(\G_{S_1S_2}\circ S_1)-r(\G_{S_1S_2}\times S_1)$\\$
=r(\G_{S_1S_2}\circ S_2)-r(\G_{S_1S_2}\times S_2).$
If $\G_{S_1S_2}\circ S_1, \G_{S_1S_2}\circ S_2$
are connected, these graphs would have common vertices whose number would 
be equal to $\xi(S_1,S_2)+1.$ In this case $\G_{S_1W}, \G_{S_2W}$ can be built as given below.

Let $v_1, \cdots, v_k,$ be the common vertices between the above mentioned 
subgraphs of $\G_{S_1S_2}.$
Let $e_1, \cdots , e_{k-1},$ be edges distinct from those in 
$S_1\uplus S_2$ with $e_i$ having end points $v_i,v_k$ for $i=1, \cdots, k-1.$
We will call $v_k$ as the datum node and
 we will say $\G_{S_1W}, \G_{S_2W}$ have \nw{node to datum type 
of ports}.
In this case, $\G_{S_1S_2}$ is obtained from $\G_{S_1S_2}\circ S_1$
and $\G_{S_1S_2}\circ S_2$ by identifying the nodes $v_1, \cdots, v_k,$  in the two graphs and, further, it turns out that 
$\V^v(\G_{S_1S_2})=(\V^v(\G_{S_1W})\lrar \V^v(\G_{S_2{W}}))$ and    
$\V^i(\G_{S_1S_2})=(\V^i(\G_{S_1W})\rightleftharpoons \V^i(\G_{S_2{W}})).$
\\ It may be noted that, in practice, this is the way larger networks 
are built from smaller modules.

Let us suppose we have built rigid multiports $\N_{TW},\N_{WQ}$ on graphs 
$\G_{S_1W},\G_{S_2W}$ respectively, where $S_1\supseteq T, S_2\supseteq Q.$
Further  
let $\G_{S_1W}\circ S_1,\G_{S_2W}\circ S_2$ be connected and let
 $\G_{S_1W},\G_{S_2W}$ have node to datum type of ports.
Let the device characteristics for both multiports be composed in terms 
of  resistors, voltage sources, current sources and controlled sources 
(CCVS,VCCS,CCCS,VCVS). Note that the first two letters in the acronym 
refer to `current controlled' or `voltage controlled' and the last two, to  
`current source' or `voltage source'.
The combined device characteristic is given below.

\begin{align}
\label{eqn:ccvsvvcs1new}
\ppmatrix{I_{\tilde{Y}_1'}&0_{\tilde{Z}_1'}&0_{\tilde{Y}_1"}&0_{\tilde{Z}_1"}\\
0_{\tilde{Y}_1'}&I_{\tilde{Z}_1'}&-r_{\tilde{Z}_1'\tilde{Y}_1"}&0_{\tilde{Z}_1"}}
\ppmatrix{v_{\tilde{Y}_1'}\\v_{\tilde{Z}_1'}\\i_{\tilde{Y}_1"}\\i_{\tilde{Z}_1"}}
%\ \ \ \ \ \ \ \ \ \ \ \ \ \ \ \ 
=\ppmatrix{0\\0};
\ppmatrix{0_{\tilde{Y}_2'}&0_{\tilde{Z}_2'}&I_{\tilde{Y}_2"}&0_{\tilde{Z}_2"}\\
-g_{\tilde{Z}_2"\tilde{Y}_2'}&0_{\tilde{Z}_2'}&0_{\tilde{Y}_2"}&I_{\tilde{Z}_2"}}\ppmatrix{v_{\tilde{Y}_2'}\\v_{\tilde{Z}_2'}\\i_{\tilde{Y}_2"}\\i_{\tilde{Z}_2"}}
%\ \ \ \ \ \ \ \ \ \ \ \ \ \ \ \ 
=\ppmatrix{0\\0}
\end{align}

\begin{align}
\label{eqn:ccvsvvcsanew}
\ppmatrix{I_{\hat{Y}_1'}&0_{\hat{Z}_2'}&0_{\hat{Y}_1"}&0_{\hat{Z}_2"}\\
0_{\hat{Y}_1'}&0_{\hat{Z}_2'}&-\alpha_{\hat{Z}_2"\hat{Y}_1"}&I_{\hat{Z}_2"}}
\ppmatrix{v_{\hat{Y}_1'}\\v_{\hat{Z}_2'}\\i_{\hat{Y}_1"}\\i_{\hat{Z}_2"}}
=\ppmatrix{0\\0};
\ppmatrix{0_{\hat{Y}_2'}&0_{\hat{Z}_1'}&I_{\hat{Y}_2"}&0_{\hat{Z}_1"}\\
-\beta_{\hat{Z}_1'\hat{Y}_2'}&I_{\hat{Z}_1'}&0_{\hat{Y}_2"}&0_{\hat{Z}_1"}}\ppmatrix{v_{\hat{Y}_2'}\\v_{\hat{Z}_1'}\\i_{\hat{Y}_2"}\\i_{\hat{Z}_1"}}
=\ppmatrix{0\\0}
\end{align}
\begin{align}
\label{eqn:ccvsvvcs3new}
\ppmatrix{I& -\hat{R}}\ppmatrix{v_{R'}\\i_{R"}}
=0;
\ppmatrix{I_{E'}& 0_{E"}}\ppmatrix{v_{E'}\\i_{E"}}
=s_E;
\ppmatrix{0_{J'}& I_{J"}}\ppmatrix{v_{J'}\\i_{J"}}
=s_J.
\end{align}
The matrices $r_{\tilde{Z}_1'\tilde{Y}_1"}, g_{\tilde{Z}_2"\tilde{Y}_2'}, \alpha_{\hat{Z}_2"\hat{Y}_1"},
\beta_{\hat{Z}_1'\hat{Y}_2'}, \hat{R}$ are diagonal matrices and their
diagonal entries are together assumed to be algebraically independent over $\Q.$
It is clear that a device characteristic composed entirely of the above mentioned devices is a proper affine space. A multiport composed only of these devices
is proper, by definition, iff it is rigid.

If we connect $\N_{TW},\N_{WQ}$ at ports $W$
 will we  obtain a rigid multiport $\N_{TQ}?$
\\To answer this question, we need to examine whether the pair 
$\{\V_{AX},\V_X\}$ is rigid, where $\V_X$ is the vector space translate of the device characteristic composed of the devices described in 
Equations \ref{eqn:ccvsvvcs1new}, \ref{eqn:ccvsvvcsanew}, \ref{eqn:ccvsvvcs3new} and $\V_{AX}\equivd (\V^v(\G_{S_1S_2}))_{S_1'S_2'}\oplus (\V^i(\G_{S_1S_2}))_{S_1"S_2"}.$
Because of the algebraic independence of the above mentioned parameters, and 
the structure of the device characteristic, under certain additional 
topological conditions, this problem reduces to checking whether the pair $\{\M_{AX},\M_X\},$ is rigid, where $\M_{AX}\equivd \M(\V_{AX}),$\\$ \M_X\equivd \M(\V_{X})$ 
(Theorem \ref{thm:maxdistancenew2} in the Appendix).
This checking is done by applying the matroid union algorithm \cite{edm65a, recski89, HNarayanan1997}. 
%In this case this turns out to be simple, because both the matroids 
%are simple,  the matroid $\M(\V_{C}),$ particularly so.
%There is a  practical solution to this problem.
%The result however needs tedious description of simplifying notation
%and is given in the appendix (Theorem \ref{thm:purslowmatroidnew}).

In this section we give a simpler result,  which gives sufficient conditions
for the multiport $\N_P\equivd \N_{TQ}$
 to be rigid. This result is easy to state, and additionally 
the sufficient conditions given are easy to test and practically 
speaking, invariably satisfied.
Essentially, we only need to  check if the voltage source like branches (controlled voltage source branches $Z_{1}$ 
together with the (voltage zero) controlling current branches $Y_{1}$ and the voltage source
 branches $E$) do not form a loop and the current source like branches (controlled current source branches $Z_{2}$ 
together with the (current zero) controlling voltage branches $Y_{2}$ and the current source
 branches $J$) do not form a cutset in the graph $\G_{S_1S_2}.$
Additionally, when these conditions are satisfied, the result states that the multiport behaviour has a \nw{hybrid representation}, i.e., a representation of the form
\begin{align*} 
\ppmatrix{i_{P_1"}\\
v_{(P-P_1)'}}=\ppmatrix{g_{11}&h_{12}\\h_{21}&r_{22}}\ppmatrix{v_{P_1'}\\
i_{(P-P_1)"}} +\ppmatrix{s_{P_1"}\\
s_{(P-P_1)'}},
\end{align*} 
and the corresponding column base is easy to compute directly from the graph $\G_{S_1S_2}.$ This makes the computation of the multiport behaviour easier 
than when the column base is not available.

The proof is given in  \ref{sec:matroidalrigidity}.

\begin{theorem}
\label{thm:purslowgraphnew}
Let $\N_P$ be a multiport on graph $\G_{SP}$ whose device characteristic
 $\A_{S'S"}$ is composed
entirely of voltage sources, current sources,   CCVS,VCCS,CCCS,VCVS and resistors as in Equations \ref{eqn:ccvsvvcs1new}, \ref{eqn:ccvsvvcsanew}, \ref{eqn:ccvsvvcs3new}.
Let $Y_1\equivd \tilde{Y}_1\uplus \hat{Y}_1,Y_2\equivd \tilde{Y}_2\uplus \hat{Y}_2,Z_1\equivd \tilde{Z}_1\uplus \hat{Z}_1,Z_2\equivd \tilde{Z}_2\uplus \hat{Z}_2.$

Let the values of resistors
and the controlled source parameters be algebraically independent over $\Q.$
We then have the following.
\begin{enumerate}
\item $\N_P$ is a  proper multiport only if
\begin{enumerate}
\item the controlling current branches $Y_{1}$ together with the voltage source
 branches $E$
 do not contain  a loop of $\G_{SP}$
and
\item  the controlling voltage branches $Y_{2}$ together with the current source
 branches $J$
 do not contain  a cutset of  $\G_{SP}.$
\end{enumerate}
\item  $\N_P$ is a  proper multiport if
\begin{enumerate}
\item the set of controlled voltage source branches $Z_{1}$
together with the controlling current branches $Y_{1}$ and the voltage source
 branches $E$
 do not contain a loop
of $\G_{SP}$
\item the set of controlled current source branches $Z_{2}$ 
together with the controlling voltage  branches $Y_{2}$ and the current  source
 branches $J$
 do not contain  a cutset of $\G_{SP}.$
\end{enumerate}
\item Let the above conditions of 2 be satisfied.
Then   there exists a
 tree $t$ of $\G_{SP}$ that contains $Z_{1}\uplus Y_{1}\uplus E$
and does not intersect $Z_{2}\uplus Y_{2}\uplus J.$
For such a tree, the port behaviour of multiport $\N_P$
has a hybrid representation
with $(t\cap P)'\uplus (P-t)"$ as a column base for the vector space
translate of the behaviour.
\end{enumerate}
\end{theorem}

\begin{remark}
The sufficient conditions of Theorem \ref{thm:purslowgraphnew} are not necessary.
Let $\G$ be a graph on two edges $e_1,e_2,$ where the edges are in parallel.
Let these be the device ports of a current controlled voltage source.
Then $Y_1\equivd \{e_1\}, Z_1\equivd \{e_2\}$ and $Y_1\uplus Z_1$ contains
a loop. Let $r_{e_1"e_2'}\ne 0.$ This network has the unique solution
 $0_{e_1'e_2'e_1"e_2"}.$ 

From a practical point of view the sufficient 
conditions are reasonable since, while designing a circuit, one will never force a dependency on controlled voltages
 by making them form a loop with predetermined voltages
 or on controlled current sources by making them form a cutset with predetermined currents.

%The most that can be said about this problem is given in Theorem \ref{thm:maxdistancenew2}.
%This theorem essentially reduces the problem, once the necessary conditions of Theorem \ref{thm:purslowgraphnew} are satisfied, to checking the rigidity of  the matroids associated with 
%two  vector spaces derived from $(\V^v(\Gsw))_{S'W'}\oplus (\V^i(\Gsw))_{S"W"}$
%and $\V_{S'S"}.$
%This checking is done by applying the matroid union algorithm \cite{edm65a, recski89, HNarayanan1997}. 
\end{remark}

\subsection{Algorithm for constructing hybrid representation of 
the port behaviour}
We present below an algorithm for construction of the hybrid representation of the port behaviour of a multiport satisfying the sufficiency conditions in Theorem \ref{thm:purslowgraphnew}.

We remind the reader that hybrid representation for a port behaviour ${\V}_{P'P"}$ is a representative matrix of the form
\begin{align}
(F_{P_1'}|F_{P_2'}|F_{P_1"}|F_{P_2"})=\ppmatrix{I&h_{21}^T&g_{11}^T&0\\
0&r_{22}^T&h_{12}^T&I} 
\end{align}
Suppose ${\A}_{P'P"}= (\alpha_{P'},\beta_{P"})+{\V}_{P'P"}.$ 
Then we can write ${\A}_{P'P"}$ as the solution space of the equations
\begin{align}
\label{eqn:hybridconstraints1}
\ppmatrix{i_{P_1"}\\v_{P_2'}}=\ppmatrix{g_{11}&h_{12}\\h_{21}&r_{22}}\ppmatrix{v_{P_1'}\\i_{P_2"}}+\ppmatrix{s_{P_1"}\\s_{P_2'}},
\end{align}
where 
\begin{align}
\ppmatrix{s_{P_1"}\\s_{P_2'}}=-\ppmatrix{g_{11}&h_{12}\\h_{21}&r_{22}}\ppmatrix{\alpha_{P_1'}\\\beta_{P_2"}}+\ppmatrix{\beta_{P_1"}\\\alpha_{P_2'}},
\end{align}

\begin{algorithm}
\label{alg:hybridgraphnew}
Input:  A multiport $\N_P\equivd (\G_{SP}, \A_{S'S"}),$ 
where 
$\A_{S'S"}$ is
composed
 entirely of voltage sources, current sources, CCVS,VCCS,CCCS,VCVS
  and resistors
and satisfies sufficiency conditions of\\ Theorem  \ref{thm:purslowgraphnew}.\\
Output: A hybrid representation for the port behaviour ${\A}_{P'P"}$ of $\N_P.$
\\Step 1. Construct a 
forest $t$ of $\G_{SP}$ that contains $Z_1\uplus Y_1\uplus E$ and does not 
intersect\\  $Z_2\uplus Y_2\uplus J.$ 
\\ Step 2. Let 
$P_1\equivd t\cap P, P_2\equivd P-t.$
The set of columns  $P_1'\uplus P_2"$ is a column base of ${\V}_{P'P"}.$
\\Step 3. Set all internal sources to zero (i.e., the device characteristic 
is the\\ vector space associate of the original device characteristic). \\For each element $e_m\in P_1, $ in turn, set $v_{e_m'}=1,
v_{e_j'}=0, e_j\in P_1, j\ne m, i_{P_2"}=0_{P_2"}$ and \\for each $e_r\in P_2 $ in turn, set $i_{e_r"}=1,
i_{e_j"}=0, e_j\in P_2, j\ne r,  v_{P_1'}=0_{P_1'}$ and compute \\the corresponding vector
$\ppmatrix{i_{P_1"}\\v_{P_2'}}.$
These yield the $m^{th}$ and $(|P_1|+r)^{th}$ columns respectively \\of the matrix 
$\ppmatrix{g_{11}&h_{12}\\h_{21}&r_{22}}.$\\
Step 4. Keep the internal sources active (i.e., the device characteristic
\\remains the  original device characteristic of the multiport),
 set $v_{P_1'}=0_{P_1'}, i_{P_2"}=0_{P_2"}$ and \\solve the multiport 
to compute the corresponding vector
$\ppmatrix{i_{P_1"}\\v_{P_2'}}.$
This yields $\ppmatrix{s_{P_1"}\\s_{P_2'}}.$
\\STOP
\end{algorithm}
Algorithm \ref{alg:hybridgraphnew}
is justified by part 3 of Theorem \ref{thm:purslowgraphnew}:
if $t$ is a forest  of $\G_{SP}$ that contains $Z_1\uplus Y_1\uplus E$ and does not
intersect  $Z_2\uplus Y_2\uplus J,$ and $\bar{t} \equivd S\uplus P -t,$
 then $t' \uplus \bar{t}"$ is a column base of $\hat{\V}^1_{S'P'S"P"}$\\$\equivd (\V^v(\G_{SP}))_{S'P'} \oplus (\V^i(\G_{SP}))_{S"P"},$ such that $(t' \uplus \bar{t}") \cap (P'\uplus P")=P_1'\uplus P_2"$
is a column base of \\$\hat{\V}^1_{S'P'S"P"}\lrar \V_{S'S"}.$ 
\subsection{Generalized multiports based on Dirac structures}
\label{subsec:dirac}
In this subsection we examine the case where the device characteristic 
contains gyrators and ideal transformers. There are three ways in which we can 
approach this problem. First, we could model these devices in terms of controlled sources. If we take the parameters of these controlled sources 
to be algebraically indepedent over $\Q,$ Theorem \ref{thm:purslowgraphnew} would be applicable. But this would only be an approximate model 
of gyrators and ideal transformers.  Since the approximation can be made arbitrarily 
accurate, this may be adequate for practical purposes. A second approach would be to  take the parameters of the gyrators 
and ideal transformers to be algebraically independent over $\Q.$
This problem has been tackled with some success in the case of electrical 
networks with no ports (see \cite{recski89} for a detailed 
discussion). However, the algorithm for testing whether the network 
has a unique solution in this case is very cumbersome in comparison 
with the case where we have only resistors and controlled and independent 
sources.
 Further, there are important situations
where the gyrators and transformers are used to impose constraints analogous to
topological ones. Here the assumption of algebraic independence is no 
longer valid. We use below an approach with uses conventional circuit simulators
for computing the port behaviour.
In brief, we proceed as follows.
Let 
$\hat{\N}_P\equivd(\G_{\hat{S}P}, \D_{S_1'S_1"}\oplus \A_{S'S"}),$
 where $\hat{S}= S_1\uplus S, $ $\D_{S_1'S_1"}$ is a Dirac structure (i.e., satisfies 
$(\D_{S_1'S_1"})_{S_1"S_1'}= \D_{S_1'S_1"}^{\perp}$) and $\A_{S'S"}$ is composed 
of resistors, independent and controlled sources as in  Theorem \ref{thm:purslowgraphnew}.
By Theorem \ref{thm:regularmultiportrecursive}, we know that $\hat{\N}_P$ is rigid iff the multiport 
${\N}_P\equivd(\G_{\hat{S}P}, \D_{S_1'S_1"})$  
and the generalized multiport $\N^g_P\equivd ({\V}^1_{S'P'S"P"},\A_{S'S"})$ 
are rigid, where ${\V}^1_{S'P'S"P"}\equivd [(\V^v(\G_{\hat{S}P}))_{\hat{S}'P'}, (\V^i(\G_{\hat{S}P}))_{\hat{S}"P"}]\lrar \D_{S_1'S_1"}$
 is a Dirac structure (Corollary \ref{cor:reciprocalDirac}).
We test for rigidity of ${\N}_P$ using a 
conventional simulator to compute its port behaviour ${\V}^1_{S_1'P'S_1"P"}.$
The following theorem states, analogous to Theorem \ref{thm:purslowgraphnew}, 
that if certain columns of ${\V}^1_{S'P'S"P"}$ can be included in column 
bases and cobases of appropriate minors of ${\V}^1_{S'P'S"P"},$ then the generalized  multiport $\N^g_P\equivd ({\V}^1_{S'P'S"P"},\A_{S'S"})$ is rigid.
\begin{theorem}
\label{thm:purslow2}
Let $\N^g_P\equivd ({\V}^1_{S'P'S"P"}, \A_{S'S"}),$ be a generalized multiport  whose device characteristic
 $\A_{S'S"}$ is composed
entirely of voltage sources, current sources,   CCVS,VCCS,CCCS,VCVS and resistors as in Equations \ref{eqn:ccvsvvcs1new}, \ref{eqn:ccvsvvcsanew}, \ref{eqn:ccvsvvcs3new}.
Let $Y_1\equivd \tilde{Y}_1\uplus \hat{Y}_1,Y_2\equivd \tilde{Y}_2\uplus \hat{Y}_2,Z_1\equivd \tilde{Z}_1\uplus \hat{Z}_1,Z_2\equivd \tilde{Z}_2\uplus \hat{Z}_2.$

%\ref{eqn:ccvsvvcs1},\ref{eqn:ccvsvvcs2},\ref{eqn:ccvsvvcs3}. i
Let the values of resistors
and controlled source parameters be algebraically independent over $\Q.$ 
Further, let ${\V}^1_{S'P'S"P"}$ be  Dirac.

1.  {\bf Necessity for multiport being rigid} $\ \ $ Let $\V_{S'S"}$ be the vector space translate of $\A_{S'S"}.$
Then $\V_{S'S"}$
has full sum property with
${\V}^1_{S'P'S"P"}\circ S'S"$ and zero intersection property with
${\V}^1_{S'P'S"P"}\times S'S",$
only if
the sets of columns $ Y'_1\uplus E', Y"_2\uplus J",$ of ${\V}^1_{S'P'S"P"},$ are independent.

2. {\bf Sufficiency for multiport being rigid} $\ \ $ The generalized multiport  $\N^g_P$ is rigid \\a) if
the set of columns $Z'_1\uplus Y'_1\uplus E',$ can be included in a column base of ${\V}^1_{S'P'S"P"}\circ S'P'$ and the set of columns
$Z'_2\uplus Y'_2\uplus J"$ can be included in a  column cobase of ${\V}^1_{S'P'S"P"}\circ S'P',$ i.e., $Z"_2\uplus Y"_2\uplus J"$ is an independent set of columns
of ${\V}^1_{S'P'S     "P"}\times S"P".$
$$or $$
b) if the set of columns $Z"_2\uplus Y"_2\uplus J",$ can be included in a column base of ${\V}^1_{S'P'S"P"}\circ S"P"$ and the set of columns
$Z"_1\uplus Y"_1\uplus E"$ can be included in a  column cobase of ${\V}^1_{S'P'S"P"}\circ S"P", $ i.e., $Z'_1\uplus Y'_1\uplus E'$ is an independent set of columns
 of ${\V}^1_{S'P'S"P"}\times S'P'.$
%\\
%($Z_1,Y_1,Z_2,Y_2,$ are respectively the sets of controlled voltage source branches, controlling current branches, controlled current source branches and controlling voltage branches).

3. {\bf Existence of hybrid representation}  $\ \ $ 
Let the above conditions of 2(a) or 2(b) be satisfied.
Then   there exists a
 column base $b_1$ of ${\V}^1_{S'P'S"P"},$ that contains $Z'_{1}\uplus Y'_{1}\uplus E'\uplus Z"_{2}\uplus Y"_{2}\uplus J".$
For such a column  base, the port behaviour of  $\N^g_P$
has a hybrid representation
with $(b_1\cap P')\uplus (b_1\cap P")$ as a column base for the vector space
translate of the behaviour.

\end{theorem}

\begin{remark}
We note that Theorem \ref{thm:purslow2} requires testing of independence of certain columns of the port behaviour of $\N_{SP}\equivd (\G_{S_1SP}, D_{S_1'S_1"})$ 
Usually the size of $S\uplus P$ would be much greater than that of $S_1.$
Computing the port behaviour of $\N_{SP}$ directly would be expensive.
However, the ports $S\uplus P$ would have redundancies in $\N_{SP}.$
We show in \ref{subsec:porttransformation}, that in such cases we can compute the port behaviour of a multiport in which the ports are minimized and the independence can be reduced to testing that of a reduced set of columns in the minimized port behaviour. Further, it is shown there that the minimized number of ports
 is not greater than $|S_1|.$
\end{remark}

\section{Conclusions and future work}
\label{sec:conclusions}
We have introduced the notion of rigidity for linear multiports as 
 the counterpart of uniqueness of solution for linear networks 
and shown its utility in understanding the behaviour of linear networks 
composed of multiports.

We have shown that there is an analogous notion of rigidity for pairs of matroids and that its properties mimic those for multiports.
We have exploited this analogy for giving simple and very efficient topological tests 
for rigidity of electrical multiports composed of nonzero resistors, independent and 
controlled  voltage and independent and
controlled  current sources assuming that the parameters of the devices
are algebraically independent over $\Q.$

We mention below a few lines of future research.

Gyrators and ideal transformers are important electrical components
 and also means of imposing conditions similar to topological 
ones. 
In  this paper, we have handled 
 the case where the algebraic independence condition
 is not valid. 
If one assumes algebraic independence over $\Q$ 
for their parameters, one can develop purely matroidal tests for 
rigidity of multiports containing them, analogous to the theory for 
uniqueness of solution of networks containing them \cite{recski89, lp86}.

Orthogonality is an important property in computational linear algebra.
 Rigidity is a weaker form of complementary orthogonality (Subsection \ref{sec:rigiddual}). It  is the best that can be asked 
for, when the two spaces model different physical entities  
(such as topological and device characteristic).
Similarly for matroids, rigidity can be regarded as a weaker form of duality (Subsection \ref{sec:rigiddualm}).
It may be interesting to explore this aspect of rigidity.

Both matched composition and the implicit duality theorem (Theorems \ref{thm:idt0}, \ref{thm:idt0m}) hold for polymatroid rank functions \cite{STHN2014}. It appears worthwhile
to explore the relevance of the notion of rigidity for such functions.
\appendix

\section{Rigidity of matroid pairs: proofs}
\label{subsec:rigidpairsmp}
\begin{proof}
({\it Proof of Lemma \ref{lem:rankfactsmatroids}})

1. We have $r(\M^1_B)+r(\M^2_B)= r(\M^1_B\vee\M^2_B) + r(\M^1_B\wedge \M^2_B),$ so that $|B|=r(\M^1_B)+r(\M^2_B)= r(\M^1_B\vee\M^2_B) $ iff
$r(\M^1_B\wedge \M^2_B)=0.$
%Therefore, if $\{\A^1_B,\A^2_B\}$ is not rigid, Equation \ref{eqn:6a}
%does not have a unique solution since it has a  singular coefficient matrix.

2. By Theorem \ref{cor:ranklrarm}, we have\\
$r(\M_{AB}\lrar \M_{BC})=r(\M_{AB}\times A)+ r(\M_{BC}\times C) +r(\M_{AB}\circ B\wedge \M_{BC}\circ B)-r(\M_{AB}\times B\wedge \M_{BC}\times B).$
\\
Now $r(\M_{AB}\circ B\wedge \M_{BC}\circ B)=r(\M_{AB}\circ B)+r(\M_{BC}\circ B)-
r(\M_{AB}\circ B\vee \M_{BC}\circ B).$\\
Therefore, using Theorems \ref{thm:sumintersectionm},\ref{thm:dotcrossidentitym}, we have
 $r(\M_{AB}\lrar \M_{BC})$\\$=r(\M_{AB}\times A)+ r(\M_{BC}\times C) +
r(\M_{AB}\circ B)+r(\M_{BC}\circ B)-
r(\M_{AB}\circ B\vee \M_{BC}\circ B)-r(\M_{AB}\times B\wedge \M_{BC}\times B)$
\\
$=r(\M_{AB})+r(\M_{BC})-r(\M_{AB}\circ B\vee\M_{BC}\circ B)-r(\M_{AB}\times B\wedge \M_{BC}\times B).$ Therefore,
\\
%Since $r(\V_{AB}\circ B+\V_{BC}\circ B)\leq |B|$ and $r(\V_{AB}\times B\cap \V_{BC}\times B)\geq 0,$
we have $r(\M_{AB}\lrar \M_{BC})=r(\M_{AB})+r(\M_{BC})-|B|,$ if
$\M_{AB}\circ B\vee\M_{BC}\circ B=\F_B$ and\\ $\M_{AB}\times B\wedge \M_{BC}\times B=\0_B,$
i.e., if $(\M_{AB}, \M_{BC})$ is rigid.

\end{proof}

We need the following lemma for the proof of the  result on rigidity of 
pairs of matroids which are derived through the `$\lrar$' operation 
from a rigid pair of matroids.
\begin{lemma}
\label{lem:factsmatroids}
\begin{enumerate} 
\item  $[\M_{WTV}\circ TV\vee(\M_T\oplus\M_V)]\circ V=\M_{WTV}\circ V\vee\M_V;$
\item  $[\M_{WTV}\circ TV\vee(\M_T\oplus\M_V)]\times T=[(\M_{WTV}\circ TV\vee\M_V)\times T]\vee\M_T;$
\item  $\M_{WTV}\circ TV\lrar \M_V=(\M_{WTV}\lrar \M_V)\circ T.$
\end{enumerate} 
\end{lemma}
\begin{proof}
Let us denote $(\M_{WTV}\circ TV\vee\M_V)$ by $\hat{\M}_{TV}.$

1. 
We need to show $(\hat{\M}_{TV}\vee \M_T)\circ V= (\M_{WTV}\circ V\vee\M_V).$
\\We have $(\hat{\M}_{TV}\vee \M_T)\circ V=\hat{\M}_{TV}\circ V\vee
((\0_V\oplus \M_T)\circ V)= \hat{\M}_{TV}\circ V = (\M_{WTV}\circ TV \circ V\vee\M_V)= (\M_{WTV}\circ V\vee\M_V).$

2. 
%We have, $[(\M_{WTV}\circ TV)\vee(\M_T\oplus\M_V)]\times T=
% [(\M_{WTV}\circ TV\vee\M_V)\vee\M_T)]\times T.$\\
%Let us denote $(\M_{WTV}\circ TV\vee\M_V)$ by $\hat{\M}_{TV}.$
We need to show that $(\hat{\M}_{TV}\vee \M_T)\times T=
(\hat{\M}_{TV}\times T) \vee \M_T.$

%====================================================
%
%Let $\hat{b}_T$ be a base of $(\hat{\M}_{TV}\times T) \vee \M_T.$
%Then there exist maximally distant bases $b^1_{TV}, b^2_T$ of \\ $\hat{\M}_{TV}, (\M_T\oplus \0_V),$
%respectively such that $\hat{b}_T=(b^1_{TV}\cap T)\cup b^2_T$
%and
%$b^1_{TV}\cap V$ is a base of $\hat{\M}_{TV}\circ V,$
% i.e., $(b^1_{TV}\cup b^2_T)\cap V$ is a base of $(\hat{\M}_{TV}\vee \M_T)\circ V.$
%
%
%
%=====================================================

Let $\hat{b}_T$ be a base of $(\hat{\M}_{TV}\times T) \vee \M_T.$
Then there exist bases $b^1_{TV}, b^2_T$ of $\hat{\M}_{TV}, \M_T,$
respectively such that $\hat{b}_T=(b^1_{TV}\cap T)\cup b^2_T$ 
and 
$b^1_{TV}\cap V$ is a base of $\hat{\M}_{TV}\circ V,$
 i.e., $(b^1_{TV}\cup b^2_T)\cap V$ is a base of $(\hat{\M}_{TV}\vee \M_T)\circ V.$
It follows that $(b^1_{TV}\cup b^2_T)\cap T= \hat{b}_T$ is independent in $(\hat{\M}_{TV}\vee \M_T)\times T.$

Let $\hat{b}_T$ be a base of $(\hat{\M}_{TV}\vee \M_T)\times T.$
Then there exist  bases 
%$b^1_{T}, b^2_T$ of 
$b^1_{TV}, b^2_T$ of $ \hat{\M}_{TV}, \M_T, $ respectively such 
that $(b^1_{TV}\cap T) \cup b^2_T = \hat{b}_T,$ 
 and $(b^1_{TV}\cup b^2_T)\cap V = b^1_{TV}\cap V$ is a base of 
$(\hat{\M}_{TV}\vee \M_T)\circ V= \hat{\M}_{TV}\circ V.$
It follows that $(b^1_{TV}\cap T) $ is independent in 
$\hat{\M}_{TV}\times T$ and $(b^1_{TV}\cap T) \cup b^2_T = \hat{b}_T,$ is independent in $(\hat{\M}_{TV}\times T) \vee \M_T.$

3. $\M_{WTV}\circ TV\lrar \M_V=
(\M_{WTV}\circ TV \vee \M_V)\times T= 
(\M_{WTV}\vee \M_V)\circ TV\times T
$\\$= (\M_{WTV}\vee \M_V)\times WT\circ T
$(using part 1 of Theorem \ref{thm:dotcrossidentitym})
$
=(\M_{WTV}\lrar \M_V)\circ T.$
\end{proof}

\begin{proof}
({\it Proof of Theorem \ref{thm:derivedregularitym}})\\
%1. We will use the following simple facts in the proof below.
%\\(a) $(\M_{WTV}\circ TV\vee\M_T\oplus\M_V)\circ V=(\M_{WTV}\circ V)\vee\M_V;$
%\\(b) $[\M_{WTV}\circ TV\vee(\M_T\oplus\M_V)]\times T=[((\M_{WTV}\circ TV)\vee\M_V)\times T]\vee\M_T;$
%\\(c) $\M_{WTV}\circ TV\lrar \M_V=(\M_{WTV}\lrar \M_V)\circ T.$
%
We have $(\M_{WTV}\circ TV)\vee(\M_T\oplus\M_V)=\F_{TV}$ iff 
$r[(\M_{WTV}\circ TV)\vee(\M_T\oplus\M_V)]$\\$=r[((\M_{WTV}\circ TV)\vee(\M_T\oplus\M_V))\circ V]+r[((\M_{WTV}\circ TV)\vee(\M_T\oplus\M_V))\times T]=|{T\uplus V}|,$\\
using Theorem \ref{thm:dotcrossidentitym}.
Since $r(\M_X)\leq |X|,$
we have $(\M_{WTV}\circ TV)\vee(\M_T\oplus\M_V)=\F_{TV}$ iff \\
$((\M_{WTV}\circ TV)\vee(\M_T\oplus\M_V))\circ V=\F_{V},$
and
$((\M_{WTV}\circ TV)\vee(\M_T\oplus\M_V))\times T=\F_{T},$
\\i.e., iff
$\M_{WTV}\circ V\vee\M_V=\F_{V},$
and
$((\M_{WTV}\circ TV\vee\M_V)\times T)\vee\M_T=\F_{T},$
\\using part 1 and 2 of Lemma \ref{lem:factsmatroids},
i.e., iff
$\M_{WTV}\circ V\vee\M_V=\F_{V},$\\
and
$(\M_{WTV}\circ TV\lrar \M_V)\vee\M_T= ((\M_{WTV}\lrar \M_V)\circ T)\vee\M_T=\F_{T},$\\
using Theorem \ref{thm:matchedpropm}, using part 3 of Lemma \ref{lem:factsmatroids}.
\\
Thus the full sum property holds for $\{\M_{WTV},(\M_T\oplus\M_V)\}$ iff it holds for $\{\M_{WTV},\M_V\}$ and for 
$\{(\M_{WTV}\lrar \M_V),\M_T\}.$

We will now use Theorems \ref{thm:perperpm}, \ref{thm:sumintersectionm}, \ref{thm:dotcrossidentitym} and the Implicit
Duality Theorem (Theorem \ref{thm:idt0m}) and get the second half
of the property of rigidity that is required to complete the proof.

From the above argument, we have
$(\M_{WTV}^{*}\circ TV)\vee(\M_T^{*}\oplus\M_V^{*})=\F_{TV},$\\ iff
$(\M_{WTV}^{*}\circ V)\vee\M_V^{*}=\F_{V}$ and
$ ((\M_{WTV}^{*}\lrar \M_V^{*})\circ T)\vee\M_T^{*}=\F_{T},$
\\i.e., $((\M_{WTV}^{*}\circ TV)\vee(\M_T^{*}\oplus\M_V^{*}))^{*}=\F^{*}_{TV}$\\ iff
$(\M_{WTV}^{*}\circ V\vee\M_V^{*})^{*}=\F^{*}_{V}$ and
$ (((\M_{WTV}^{*}\lrar \M_V^{*})\circ T)\vee\M_T^{*})^{*}=\F^{*}_{T},$
\\i.e., (using Theorems \ref{thm:perperpm}, \ref{thm:sumintersectionm})
$(\M_{WTV}\times TV)\ \wedge\ (\M_T\oplus\M_V)=\0_{TV},$
\\iff
$(\M_{WTV}^{*}\circ V)^{*}\ \wedge \ \M_V=\0_{V}$ and
$ (\M_{WTV}^{*}\lrar \M_V^{*})^{*}\times T\ \wedge \ \M_T=\0_{T},$
\\i.e., (using Theorems \ref{thm:perperpm}, \ref{thm:dotcrossidentitym}, \ref{thm:idt0m}) $(\M_{WTV}\times TV)\ \wedge\ (\M_T\oplus\M_V)=\0_{TV},$
\\iff
$(\M_{WTV}\times V)\ \wedge \ \M_V=\0_{V}$ and
$ (\M_{WTV}\lrar \M_V)\times T\ \wedge \ \M_T=\0_{T}.$
\\
Thus the zero intersection property holds for $\{\M_{WTV},(\M_T\oplus\M_V)\}$ iff it holds for $\{\M_{WTV},\M_V\}$ and for
$\{(\M_{WTV}\lrar \M_V),\M_T\}.$

2. We have $\M_{WTV}\lrar \M_V\equivd (\M_{WTV}\wedge \M_V)\circ WT=
(\M_{WTV}\vee\ \M_V)\times WT$ (using Theorem \ref{thm:sumintersectionm}).\\
Rigidity of $\{\M_{WTV}\lrar \M_V, \M_T\}$ 
is equivalent to the full sum and zero intersection conditions\\
$((\M_{WTV}\lrar \M_V)\circ T)\vee \M_T=\F_T,$
and $((\M_{WTV}\lrar \M_V)\times T)\wedge \M_T=\0_T.$\\
The full sum condition is equivalent to $((\M_{WTV}\wedge \M_V)\circ WT\circ T)\vee\M_T=((\M_{WTV}\wedge \M_V)\circ T)\vee\M_T=\F_T$
and the zero intersection condition is equivalent to $((\M_{WTV}\vee \M_V)\times WT\times T)\wedge \M_T$\\$=((\M_{WTV}\vee \M_V)\times T)\wedge \M_T=\0_T.$

3. Let us say matroid $\M^1_X\geq \M^2_X,$ 
iff every base of $\M^1_X$ contains a base of $\M^2_X,$ 
and every base of $\M^2_X,$  is contained in a base of $\M^1_X.$
We observe that 
$(\M_{WTV}\vee \M_V)\geq (\M_{WTV}\wedge \M_V),$ so that 
$(\M_{WTV}\vee \M_V)\circ T\geq (\M_{WTV}\wedge \M_V)\circ T$ and $(\M_{WTV}\vee \M_V)\times T\geq (\M_{WTV}\wedge \M_V)\times T.$
\\ Therefore, if $\{\M_{WTV}\wedge \M_V,\M_T\}$ satisfies the full
sum condition, so does  $\{\M_{WTV}\vee \M_V,\M_T\}$
and if\\ $\{\M_{WTV}\vee \M_V,\M_T\}$ satisfies the zero
intersection condition, so does  $\{\M_{WTV}\wedge \M_V,\M_T\}.$
\\ 
The result now follows from the previous part.
\end{proof}

%\bibliographystyle{elsarticle1-num}
% \bibliography{references}
% 
%
%\end{document}

\section{Matroidal conditions for rigidity of multiports}
\label{sec:matroidalrigidity}
The matroidal approach to testing rigidity of multiports
%The proof of Theorem \ref{thm:purslowgraphnew} 
is based on the following lemma 
which speaks of the rigidity of $\{\V_{AB},\A_B\}$
where $\V_{AB}$ is over $\Q$ and $\V_B$ is in terms of parameters which are algebraically 
independent over $\Q.$
\begin{lemma}
\label{lem:generalitytomaxdistancenew}
Let $\A^2_{B}\equivd \A^2_{B_1B_2}$ be the solution space of the equation
\begin{align}
\label{eqn:devcharnew}
{y_{B_2}}={D}{x_{B_1}}+ {s}.
%\ppmatrix{v_{T_1'}\\i_{W_2"}}=\ppmatrix{R_1&0\\0&G_2}\ppmatrix{i_{W_1"}\\v_{T_2'}}+\ppmatrix{E_1\\J_2},
\end{align}
where $D$ is a diagonal matrix.
Let $\A^2_{B_1B_2}$ have $\V^2_{B_1B_2}$ as its vector space translate.
Let the diagonal entries of $D$ be algebraically independent over
$\Q.$ 
%Let $\V^1_{B_1B_2}$ be a vector space over $\Q.$

Let $(C_{B_1}|C_{B_2})$ be the representative matrix of a vector space $\V^1_{B_1B_2}$
with entries from $\Q.$
Let
there be a  set of columns $j_1,j_2,\cdots ,j_{n}$
of $(C_{B_1}|C_{B_2}),$ 
which are linearly independent,
and whose complement with respect to $B_1\uplus B_2,$
say $i_1,i_2, \cdots ,i_{m}$ is a
column base of the matrix
\begin{align}
\label{eqn:devchar2}
F= \ppmatrix{F_{B_1}|F_{B_2}}\equivd
\ppmatrix{
I & D}.
\end{align}
We then have the following.
\begin{enumerate}
\item Let $r(\V^{1}_{B_1B_2})+r(\V^2_{B_1B_2})\geq |B_1\uplus B_2|=|B|.$
Then $\V^1_{B_1B_2}, \V^2_{B_1B_2}$ have the full sum property.
\item Let $r(\V^{1}_{B_1B_2})+r(\V^2_{B_1B_2})\leq |B_1\uplus B_2|=|B|.$
Then $\V^1_{B_1B_2}, \V^2_{B_1B_2}$ have the zero intersection property.
%\item Let $\V^1_{AB}$ be a vector space over $\Q$
%such that $r(\V^{1}_{AB}\circ B )+r(\V^2_{B})\geq |B|$
%and $r((\V^{1}_{AB})^{\perp}\circ B )+r((\V^2_{B})^{\perp})\geq |B|.$
%We then have the following.
%\begin{enumerate}
%\item $\{\V^1_{AB},\V^2_B\}$ is rigid.
%\item $\M(\V^1_{AB}+\V^2_B)=\M(\V^1_{AB})\vee \M(\V^2_B)$ and
%$\M((\V^1_{AB})^{\perp}+(\V^2_B)^{\perp})= (\M(\V^1_{AB}))^*\vee (\M(\V^2_B))^*.$
%\end{enumerate}
\item Let $\V^1_{AB}$ be a vector space over $\Q$ and let $\V^2_B$ 
be the vector space translate of $\A^2_B$ defined by Equation \ref{eqn:devcharnew}.
\begin{enumerate}
\item $\M(\V^1_{AB}+\V^2_{B})=\M(\V^1_{AB})\vee \M(\V^2_{B})$ and
$\M((\V^1_{AB})^{\perp}+(\V^{2}_{B})^{\perp})=
\M((\V^1_{AB})^{\perp})\vee \M((\V^{2}_{B})^{\perp}).$
\item
The following are equivalent.
\begin{enumerate}
\item There exists a pair of disjoint bases $b^1,b^2$ of $\V^1_{AB},\V^2_B,$
respectively such that $b^1\cup b^2\supseteq B.$
\item $r(\V^{1}_{AB}\circ B )+r(\V^2_{B})\geq |B|$
and $r((\V^{1}_{AB})^{\perp}\circ B )+r((\V^2_{B})^{\perp})\geq |B|.$
\item $\{\V^1_{AB},\V^2_B\}$ is rigid.
\item $\{\M(\V^1_{AB}),\M(\V^2_B)\}$ is rigid.
%\item $\M(\V^1_{AB}+\V^2_B)=\M(\V^1_{AB})\vee \M(\V^2_B)$ and
%$\M((\V^1_{AB})^{\perp}+(\V^2_B)^{\perp})= (\M(\V^1_{AB}))^*\vee (\M(\V^2_B))^*.$
\end{enumerate}
\end{enumerate}

\end{enumerate}
\end{lemma}

\begin{proof}
1. It is clear that $F$ is a representative matrix
for $\V^2_{B}.$
Let $(C_{2B_1}|C_{2B_2})$
be a suitable submatrix of 
$(C_{B_1}|C_{B_2})$
such that
\begin{align}
\label{eqn:repmatrixsumnew}
\ppmatrix{C_{2B_1}&C_{2B_2}
\\
I&D}
\end{align}
is square. This is possible since
 $r(\V^{1}_{B_1B_2})+r(\V^2_{B_1B_2})\geq |B_1\uplus B_2|=|B|.$
Perform a Laplace expansion of the determinant
of this matrix using the first set of rows 
as one block in the partition of rows.
One term in this expansion is $\pm det(C_{2}^j)\times det(F^i),$
where $C_{2}^j$ denotes the submatrix of 
$(C_{2B_1}|C_{2B_2})$
corresponding to set of columns $j_1,j_2,\cdots ,j_{n}$
and all rows, $F^i$ denotes the submatrix of
$F$
corresponding to set of columns $i_1,i_2,\cdots ,i_{m}$
and all rows.

We know that the determinants of $C_{2}^j$ and $F^i$ are nonzero
since their columns are linearly independent.
Therefore this term cannot be zero.
The determinant of $F^i$ has only one nonzero term
which will have the form $\pm D _{k1}\times \cdots \times D_{kt}.$
Suppose another pair $C_2^{j'}$ and $F^{i'}$ in the Laplace
expansion also has nonzero determinants. The determinant of $F^{i'}$ has only one nonzero term
which will have the form \\$\pm D_{k'1}\times \cdots \times D_{k't'}.$
 Since the columns of $F^{i}$ as well as $F^{i'}$ form column bases of 
the matrix $F,$ it is clear that the sets $\{D_{k1}, \cdots, D_{kt} \},
\{D_{k'1}, \cdots, D_{k't'} \},$ are distinct. Therefore the Laplace expansion
will be a polynomial involving distinct subsets of $\{D_{1}, \cdots D_{|B_2|},
\}$ in different terms and by algebraic independence
of these  variables over $\Q,$ the polynomial cannot
take the value zero.
Thus the matrix in Equation \ref{eqn:repmatrixsumnew}
is nonsingular. We conclude that
$r(\V^{1}_{B_1B_2}+\V^2_{B_1B_2})= |B_1\uplus B_2|$
and therefore that $\V^{1}_{B_1B_2},\V^2_{B_1B_2}$ have full sum property.

2. For handling the zero intersection case, when $r(\V^{1}_{B_1B_2})+r(\V^2_{B_1B_2})\leq |B_1\uplus B_2|,$
 we can make use of the fact that if 
$r(\V^{1}_{B_1B_2})+r(\V^2_{B_1B_2})\leq |B_1\uplus B_2|,$
then $r((\V^{1}_{B_1B_2})^{\perp})+r((\V^2_{B_1B_2})^{\perp})\geq |B_1\uplus B_2|.$
Noting that
\begin{align}
\label{eqn:repmatrixsum2new}
\ppmatrix{
-D& I}
\end{align}
is a representative matrix for $(\V^2_{B_1B_2})^{\perp},$
we can use the above proof to show that $(\V^{1}_{B_1B_2})^{\perp},(\V^2_{B_1B_2})^{\perp}$
have full sum property.
Using  part 4 of Theorem \ref{thm:regularrecursive},
we can then infer that
$\V^{1}_{B_1B_2},\V^2_{B_1B_2}$
 have zero intersection property.

3(a). The space $\V^1_{AB}+\V^2_{B}$ has a representative matrix of the form 
\begin{align}
\label{eqn:repmatrixsumnew2}
\ppmatrix{C_A&C_{B_1}&C_{B_2}
\\
0&I&D},
\end{align}
where $(I|D)$ is a representative matrix for $\V^2_B$
 and $(C_A|C_{B_1}|C_{B_2})$ is a matrix over $\Q.$
By using Laplace expansion as in part 1 above, we can see firstly that 
every column base of $\V^1_{AB}+\V^2_{B}$ is independent in 
$\M(\V^1_{AB})\vee \M(\V^2_{B}).$
Secondly, by the argument of part 1 
that each term in the expansion
which involves diagonal entries of $D$ has a distinct set of them,
using algebraic independence over $\Q$ of diagonal entries of $D,$ 
we can see that no cancellation of such terms can occur. Therefore every base of 
$\M(\V^1_{AB})\vee \M(\V^2_{B})$
 is column independent in 
$\V^1_{AB}+\V^2_{B}.$ We therefore have 
$\M(\V^1_{AB}+\V^2_{B})=\M(\V^1_{AB})\vee \M(\V^2_{B}).$
Since $(-D|I)$ is a representative matrix of $(\V^2_B)^{\perp}$
we can similarly prove that 
$\M((\V^1_{AB})^{\perp}+(\V^{2}_{B})^{\perp})=
\M((\V^1_{AB})^{\perp})\vee \M((\V^{2}_{B})^{\perp}).$

3(bi) $\rightarrow $ 3(bii). 
If $b^1,b^2$ are disjoint bases of $\V^1_{AB},\V^2_B,$
such that $b^1\cup b^2\supseteq B,$
it is clear that\\ $r(\V^1_{AB}\circ B)+r(\V^2_B)\geq |B|.$
Further,
if $b^1,b^2$ are disjoint bases of $\V^1_{AB},\V^2_B,$ such that 
$b^1\cup b^2\supseteq B,$\\
   $A\cup B- b^1,B-b^2$ are disjoint bases of $(\V^1_{AB})^{\perp},(\V^2_B)^{\perp},$ respectively and  $(A\cup B- b^1)\cup (B-b^2)\supseteq B.$
It is therefore clear that 3(bi) implies 3(bii).

3(bii)  $\rightarrow $ 3(biii).   Since $r(\V^1_{AB}\circ B)+r(\V^2_B)\geq |B|,$ by part 1 above, full sum property of $\{\V^1_{AB},\V^2_B\}$ follows.
Since $r((\V^{1}_{AB})^{\perp}\circ B )+r((\V^2_{B})^{\perp})\geq |B|,$
 we must have $|B|- r(\V^1_{AB}\times B)+|B|- r(\V^2_B)\geq |B|,$
\\ i.e., $r(\V^1_{AB}\times B)+r(\V^2_B)\leq |B|$
 (using Theorems \ref{thm:perperp}, \ref{thm:dotcrossidentity}).
\\Therefore, by part 2 above zero intersection property of $\{\V^1_{AB},\V^2_B\}$ follows.

3(biii)  $\rightarrow $ 3(biv).  
This is part 1 of Theorem \ref{thm:rigidmatroidvector}.
%
%
%3(c) is equivalent to 3(a) by part 4 of Theorem \ref{thm:regularrecursive}. The argument in part 1 shows that whenever we have a column base $b^1 $ of $\V^1_{AB}$ that is disjoint from a column base $b^2$ of $\V^2_B,$ 
%then $b^1\cup b^2$ is a base of $\V^1_{AB}+\V^2_B.$
%Thus every base of $\M(\V^1_{AB})\vee \M(\V^2_{B})$ is a base of 
%$\M(\V^1_{AB}+\V^2_B).$
%Since it is clear that a base of $\M(\V^1_{AB}+\V^2_B)$ is always a base 
%of $\M(\V^1_{AB})\vee \M(\V^2_{B})$ we conclude that 
%$\M(\V^1_{AB}+\V^2_B)=\M(\V^1_{AB})\vee \M(\V^2_{B}).$ 
%This argument is valid because $\V^2_B$ has the representative matrix 
%$(I|D),$ where $D$ is a diagonal matrix whose diagonal entries are 
%algebraically independent over $\Q.$\\
%Now $(\V^2_B)^{\perp}$ has the representative matrix
%$(-D|I).$ It follows that whenever we have a column base $\hat{b}^1 $ of $(\V^1_{AB})^{\perp}$ that is disjoint from a column base $\hat{b}^2$ of $(\V^2_B)^{\perp},$
%then $\hat{b}^1\cup \hat{b}^2$ is a base of $(\V^1_{AB})^{\perp}+(\V^2_B)^{\perp}.$
%Therefore, we must have $\M((\V^1_{AB})^{\perp}+(\V^2_B)^{\perp})=\M((\V^1_{AB})^{\perp})\vee \M((\V^2_{B})^{\perp})= (\M(\V^1_{AB}))^*\vee (\M(\V^2_{B}))^*.$ 
%By part 3 of Theorem \ref{thm:rigidmatroidvector}, it follows that $\{\M(\V_{AB}), \M(\V_{BC})\}$
%is rigid.

3(biv) $\rightarrow $ 3(bi).
This is part 5 of Theorem \ref{thm:regularrecursivem}.
%implies that $\{\V_{AB},\V_{BC}\}$ is rigid iff $\{\M(\V_{AB}), \M(\V_{BC})\}$
%is rigid (part 3 of Theorem \ref{thm:rigidmatroidvector}).
%We have  $\{\M(\V_{AB}), \M(\V_{BC})\}$
%is rigid iff $r(\M(\V_{AB})\vee \M(\V_{BC}))$\\$=r(\M(\V_{AB}))+r(\M(\V_{BC}))$
%and $r((\M(\V_{AB}))^*\vee (\M(\V_{BC}))^*)= r((\M(\V_{AB}))^*)+r((\M(\V_{BC}))^*).$ 
%%and $r(\M((\V^1_{AB})^{\perp})\vee \M((\V^2_B)^{\perp}))=r(\M((\V^1_{AB})^{\perp}))+r(\M((\V^2_B)^{\perp}))$ 
%(part 3 of Theorem \ref{thm:regularrecursivem}). 
%%and using the fact $\M(\V^{\perp})=(\M(\V))^*$).
%Therefore rigidity of $\{\M(\V_{AB}), \M(\V_{BC})\}$
%implies the existence of a pair of disjoint bases $b^1,b^2$ of
%$\M(\V_{AB}), \M(\V_{BC})$ which cover $B,$
%i.e., the existence of a pair of disjoint column bases $b^1,b^2$ of
%$\V_{AB}, \V_{BC}$ which cover $B.$
%
%Thus 3(d) implies 3(a).
\end{proof}
The usual situation  encountered in electrical circuit theory is one where 
the device characteristic 
 is composed of independent voltage and current sources, resistors
and controlled sources, as in Equations \ref{eqn:ccvsvvcs1new}, \ref{eqn:ccvsvvcsanew}, \ref{eqn:ccvsvvcs3new}. In this case the vector space translate of the device 
characteristic has  additional variables which have no constraints on them
 (eg. currents of independent and controlled voltage sources) and others which are zero
(eg. voltages of independent voltage sources and of controlling current branches).
%In this case  $\V^2_X$ is the vector space translate of a device characteristic which 
% is more complicated than that described by
%Equation \ref{eqn:devcharnew}.
%
% is more complicated than that described by 
%Equation \ref{eqn:devcharnew}. When it is composed of independent voltage and current sources, resistors
%and controlled sources, there are additional variables which have no constraints on them 
% (eg. currents of independent and controlled voltage sources) and others which are zero
%(eg. voltages of independent voltage sources and of controlling current branches).
Lemma \ref{lem:generalitytomaxdistancenew} needs some further processing before it can be applied to determine the rigidity 
of $\{\V^1_{AX},\V^2_X\},$
where $\V^2_X\equivd \F_Z\oplus \0_Y\oplus \V^2_B,$ and $
\V^1_{AB}=\V^1_{AX}\circ A (X-Z)\times AB.$ 
The most that can be said about this problem is given in Theorem \ref{thm:maxdistancenew2}.
This theorem essentially reduces the problem, once certain  necessary topological conditions 
 are satisfied, to checking the rigidity of  the matroids associated with
$\{\V^1_{AX},\V^2_X\}.$
This checking can be done by applying the matroid union algorithm \cite{edm65a}.
\begin{theorem}
\label{thm:maxdistancenew2}
Let 
%$A,X$ be disjoint sets and let 
$\V^1_{AX},\V^2_X$ be vector spaces 
%and let $\M^1_{AX},\M^2_X$ be matroids on
on $A\uplus X,X,$ respectively. Let $X=Z\uplus Y \uplus B.$ 
\\Let $\V^1_{AB}\equivd \V^1_{AX}\circ A (X-Z)\times AB.$ 
Let $\V^2_X= \F_Z\oplus \0_Y\oplus \V^2_{B}$
 and let  $\V^1_{AX}\times Z=\0_Z, \ \V^1_{AX}\circ Y=\F_Y.$
We then have 
the following.
\begin{enumerate}
\item \begin{enumerate}
\item $\{\V^1_{AX}, \V^2_X\}$ is rigid iff $\{\V^1_{AB}, \V^2_B\}$
 is rigid.
\item $\{\M(\V^1_{AX}), \M(\V^2_X)\}$ is rigid iff 
$\{\M(\V^1_{AB}), \M(\V^2_B)\}$ is rigid.
\item $\V^1_{AX}\lrar \V^2_X=\V^1_{AB}\lrar \V^2_B.$
\item $\M(\V^1_{AX})\lrar \M(\V^2_X)= \M(\V^1_{AB})\lrar \M(\V^2_B).$
\end{enumerate}
%\item Let $\M(\V^1_{AB}\lrar \V^2_B)= \M(\V^1_{AB})\lrar \M(\V^2_{B}).$
%Then, $\M(\V^1_{AX}\lrar \V^2_X)= \M(\V^1_{AX})\lrar \M(\V^2_{X}).$
\item Let $\M(\V^1_{AB}+\V^2_B)= \M(\V^1_{AB})\vee \M(\V^2_B)$ and let
$\M((\V^1_{AB})^{\perp}+(\V^2_B)^{\perp})= \M((\V^1_{AB})^{\perp})\vee \M((\V^2_B)^{\perp}).$ 
We then have the following.
\begin{enumerate}
\item  $\M(\V^1_{AX}\lrar \V^2_X)= \M(\V^1_{AX})\lrar \M(\V^2_{X})= \M(\V^1_{AB})\lrar \M(\V^2_{B})=\M(\V^1_{AB}\lrar \V^2_B).$
\item $\{\V^1_{AX}, \V^2_X\}$ is rigid iff $\{\M(\V^1_{AX}), \M(\V^2_X)\}$ is rigid.
\item $\{\V^1_{AB}, \V^2_B\} \ (\{\V^1_{AX}, \V^2_X\})$ is rigid, iff  there exists 
a pair of disjoint bases $b^1,b^2$ of 
$\M(\V^1_{AB}), \M(\V^2_B)\ $\\$(\M(\V^1_{AX}), \M(\V^2_X)) ,$ respectively, such that $(b^1\cup b^2)\supseteq B \ ((b^1\cup b^2)\supseteq X).$ 
For such a pair $b^1,b^2,$
  $b^1\cap A$ is a column base of $\V^1_{AX}\lrar \V^2_X.$
\end{enumerate}
\end{enumerate}
\end{theorem}
\begin{proof}
%({\it Proof of Theorem \ref{thm:maxdistancenew2}})
%
%
%====================================================================
%
%\begin{theorem}
%\label{lem:derivedregularity}
%Let $\V_{WTV},\V_V, \V_T,$ be vector spaces on $W\uplus T\uplus V, V,T,$
%respectively.
%% and let $(\V_{WTV},\V_R)$ be regular.
%Then $\{\V_{WTV},\V_T\oplus\V_V\}$ is rigid
%\begin{enumerate}
%\item
%iff
%$\{\V_{WTV},\V_V\},$  $\{\V_{WTV}\lrar\V_V,\V_T\}$  are rigid;
%\item iff
%$\{\V_{WTV},\V_V\}$ is rigid and $\{\V_{WTV}\cap \V_V,\V_T\},$
%$\{\V_{WTV}+ \V_V,\V_T\}$ satisfy  the full sum and zero intersection
%properties respectively;
%%\item iff $\{\V_{WTV},\V_V\},$ $\{\V_{WTV}\cap \V_V,\V_T\},$
%$\{\V_{WTV}+ \V_V,\V_T\}$ are rigid.
%\end{enumerate}
%\end{theorem}
%
%
%======================================================

1(a). We apply Theorem \ref{lem:derivedregularity},
taking $T=B, V=X-B=Y\uplus Z, W= A,\V_T=\V_B, \V_V=\F_Z\oplus \0_Y.$\\
We then have $\{\V^1_{AX}, \V^2_X\}$ is rigid iff
 $\{\V^1_{AX},\F_Z\oplus \0_Y\}$ and $\{\V^1_{AX}\lrar (\F_Z\oplus \0_Y),\V_B\}$ are rigid.
Let us examine the implication of $\{\V^1_{AX},\F_Z\oplus \0_Y\}$ being rigid.
We need $\V^1_{AX}\circ ZY+(\F_Z\oplus \0_Y)= \F_{ZY} $ (full sum property).
This is equivalent to  $\V^1_{AX}\circ Y=\F_Y.$
Next we need  $(\V^1_{AX}\times ZY)\cap (\F_Z\oplus \0_Y)= \0_{ZY} $ (zero intersection property).
This is equivalent to  $\V^1_{AX}\times  Z=\0_Z.$
\\
Next  $\V^1_{AX}\lrar (\F_Z\oplus \0_Y)= \V^1_{AX}\circ A (X-Z)\times AB= \V^1_{AB}.$
The result follows.

1(b). 
Since $\V^2_X= \F_Z\oplus \0_Y\oplus  \V^2_B,$
we have $\M(\V^2_X)= \F_Z\oplus \0_Y\oplus  \M(\V^2_B).$
Since $\M(\V^1_{AX}\times Z)$\\$=\M(\V^1_{AX})\times Z, $
and $\V^1_{AX}\times Z=\0_Z,$ we have
$\M(\V^1_{AX})\times Z=\0_Z.$
 Since $\M(\V^1_{AX}\circ Y)=\M(\V^1_{AX})\circ Y, $
and $\V^1_{AX}\circ Y=\F_Y,$ we have
$\M(\V^1_{AX})\circ Y=\F_Y.$
Finally $\M(\V^1_{AB})= \M(\V^1_{AX}\circ A(X-Z)\times AB)$\\$= \M(\V^1_{AX})\circ A(X-Z)\times AB.$

The next part of the proof  is a line by line translation of the proof of 1(a) above, taking\\ $\M^1_{AX}\equivd \M(\V^1_{AX}), \M^2_X\equivd \M(\V^2_X),
\M^1_{AB}\equivd \M(\V^1_{AB}), \M^2_B\equivd \M(\V^2_B).$
\\
We apply Theorem \ref{thm:derivedregularitym},
taking $T=B, V=X-B=Y\uplus Z, W= A,\M_T=\M_B, \M_V=\F_Z\oplus \0_Y.$\\
We then have $\{\M^1_{AX}, \M^2_X\}$ is rigid iff
 $\{\M^1_{AX},\F_Z\oplus \0_Y\}$ and $\{\M^1_{AX}\lrar (\F_Z\oplus \0_Y),\M_B\}$ are rigid.
Let us examine the implication of $\{\M^1_{AX},\F_Z\oplus \0_Y\}$ being rigid.
We need $\M^1_{AX}\circ ZY\vee (\F_Z+\0_Y)= \F_{ZY} $ (full sum property).
This is equivalent to  $\M^1_{AX}\circ Y=\F_Y.$
Next we need  $(\M^1_{AX}\times ZY)\wedge (\F_Z+\0_Y)= \0_{ZY} $ (zero intersection property).
This is equivalent to  $\M^1_{AX}\times  Z=\0_Z.$
\\
Next  $\M^1_{AX}\lrar (\F_Z\oplus \0_Y)= \M^1_{AX}\circ A (X-Z)\times AB= \M^1_{AB}.$
The result follows.

%
%
%
%=====================================================
%
%1(a) This is just part 3(a) of Lemma \ref{lem:maxdistancenew21}
%restated here for convenience.
%
%
%1(b) Since $\V^2_X= \F_Z\oplus \0_Y\oplus  \V^2_B,$
%we have $\M(\V^2_X)= \F_Z\oplus \0_Y\oplus  \M(\V^2_B).$
%Since $\M(\V^1_X\times Z)=\M(\V^1_X)\times Z, $
%and $\V^1_X\times Z=\0_Z,$ we have
%$\M(\V^1_X)\times Z=\0_Z.$
% Since $\M(\V^1_X\circ Y)=\M(\V^1_X)\circ Y, $
%and $\V^1_X\circ Y=\F_Y,$ we have
%$\M(\V^1_X)\circ Y=\F_Y.$
%The result follows now from part 3(a) of Lemma \ref{lem:maxdistancenew22}.
%
%==============================================================

1(c). We observe that, by the definition of the matched composition operation,  $\V_{WTV}\lrar (\V_V\oplus \V_T)$\\$=
(\V_{WTV}\lrar \V_V) \lrar \V_T,$ where $W,T,V $ are pairwise disjoint.
It follows, using the substitutions in the proof of 1(a) above that 
$\V^1_{AX}\lrar \V^2_X=\V^1_{AB}\lrar \V^2_B.$

1(d). As in proof of 1(b), we take $\M^1_{AX}\equivd \M(\V^1_{AX}), \M^2_X\equivd \M(\V^2_X),
\M^1_{AB}\equivd \M(\V^1_{AB}),$\\$ \M^2_B\equivd \M(\V^2_B).$
We observe that, by the definition of the matched composition operation,\\  $\M_{WTV}\lrar (\M_V\oplus \M_T)=
(\M_{WTV}\lrar \M_V) \lrar \M_T,$ where $W,T,V $ are pairwise disjoint.
It follows, using the substitutions in the proof of 1(b) above that
$\M^1_{AX}\lrar \M^2_X=\M^1_{AB}\lrar \M^2_B.$

%===============================================================
%
%1(c). This is part 3(b) of Lemma \ref{lem:maxdistancenew21}, restated here for 
%convenience.
%
%===============================================================

2(a). By part 2(b) of Theorem \ref{thm:rigidmatroidvector}, we have that
$\M(\V^1_{AB}\lrar \V^2_B)= \M(\V^1_{AB})\lrar \M(\V^2_{B}).$
By 1(c) above, 
 we have $\V^1_{AB}\lrar \V^2_{B}= \V^1_{AX}\lrar \V^2_{X}$
and therefore $\M(\V^1_{AB}\lrar \V^2_{B})= \M(\V^1_{AX}\lrar \V^2_{X}).$
By 1(d) above $\M(\V^1_{AX})\lrar \M(\V^2_X)=\M(\V^1_{AB})\lrar \M(\V^2_B).$
The result follows.

%By part 3(b) of Lemma \ref{lem:maxdistancenew21}
% we have $\V^1_{AB}\lrar \V^2_{B}= \V^1_{AX}\lrar \V^2_{X}$
%and therefore $\M(\V^1_{AB}\lrar \V^2_{B})= \M(\V^1_{AX}\lrar \V^2_{X}).$
%By part 3(b) of Lemma \ref{lem:maxdistancenew22}
% we have $\M(\V^1_{AB})\lrar \M(\V^2_{B})= \M(\V^1_{AX})\lrar \M(\V^2_{X}).$
%The result follows.

2(b). 
By part 2(c) of Theorem \ref{thm:rigidmatroidvector}, under the conditions of 
this part, we must have that $\{\V^1_{AB}, \V^2_B\}$ is rigid iff $\{\M(\V^1_{AB}), \M(\V^2_B)\}$  is rigid.
%We have, since $\{\M(\V^1_{AB}, \M(\V^2_B)\}$ is rigid, by 
By part 1(a), 1(b) above, we have  that $\{\V^1_{AB}, \V^2_B\}$ is rigid iff $\{\V^1_{AX}, \V^2_X\}$ is rigid 
and  $\{\M(\V^1_{AB}), \M(\V^2_B)\}$  is rigid iff $\{\M(\V^1_{AX}), \M(\V^2_X)\}$  is rigid.
The result follows.
%
%===========================================================================
%
%
%part 2 of Theorem \ref{thm:rigidmatroidvector}

2(c).   
 We have $\{\V^1_{AB}, \V^2_B\}$ is rigid iff $\{\M(\V^1_{AB}), \M(\V^2_B)\}$  is rigid.
Therefore, by 
part 5 of Theorem \ref{thm:regularrecursivem}, 
$\{\V^1_{AB}, \V^2_B\}$ is rigid iff  
there exists a pair of disjoint bases $b^1,b^2$ of 
$\M(\V^1_{AB}), \M(\V^2_B),$ respectively, such that $(b^1\cup b^2)\supseteq B.$ 
Now $(b^1\cup b^2)$ is a base of $\M(\V^1_{AB})\vee \M(\V^2_B),$ that contains 
a maximal intersection with $B$ and therefore $b\equivd  (b^1\cup b^2)\cap A= b^1\cap A$
is a base of $\M(\V^1_{AB})\lrar \M(\V^2_{B})\equivd (\M(\V^1_{AB})\vee \M(\V^2_B))\times A.$
By part 2(b) of Theorem \ref{thm:rigidmatroidvector}, we have 
$\M(\V^1_{AB}\lrar \V^2_{B})=\M(\V^1_{AB})\lrar \M(\V^2_{B}).$
Therefore, $b$ is a column base of $\V^1_{AB}\lrar \V^2_{B}=\V^1_{AX}\lrar \V^2_{X}$
(by part 1(c) above).

We have $\{\V^1_{AX}, \V^2_X\}$ is rigid iff $\{\M(\V^1_{AX}), \M(\V^2_X)\}$  is rigid
 and $\M(\V^1_{AX}\lrar \V^2_{X})=\M(\V^1_{AX})\lrar \M(\V^2_{X}).$ 
Arguing as above, $\{\V^1_{AX}, \V^2_X\}$ is rigid iff  
there exists a pair of disjoint bases $b^1,b^2$ of 
$\M(\V^1_{AX}), \M(\V^2_X),$ respectively, such that $(b^1\cup b^2)\supseteq X$ 
 and for such a pair $b^1,b^2,$ we have $b^1\cap A$ as a column base of $\V^1_{AX}\lrar \V^2_{X}.$
\end{proof}
%\begin{remark}
%In Theorem \ref{thm:maxdistancenew2}, to prove that $\V^1_{AX}\lrar \V^2_X= \V^1_{AB}\lrar \V^2_B$ 
% and that $\M^1_{AX}\lrar \M^2_X= \M^1_{AB}\lrar \M^2_B$ 
%does not require any condition on $\V^1_{AX}, \M^1_{AX}$ but only that 
%$\V^2_X=\0_Y\oplus \F_Z\oplus \V^2_B$ and that\\ 
%$\M^2_X=\0_Y\oplus \F_Z\oplus \M^2_B$ respectively.\\
%On the other hand, to show that the rigidity of $\{\V^1_{AX},\V^2_X\}
%$ is equivalent to the rigidity of $\{\V^1_{AB},\V^2_B\}$ 
%and that the rigidity of $\{\M^1_{AX},\M^2_X\}
%$ is equivalent to the rigidity of $\{\M^1_{AB},\M^2_B\}$ we need conditions 
%on both $\V^1_{AX},\V^2_X$ and on both $\M^1_{AX}, \M^2_X$ respectively.
%\end{remark}

%===========================================================================
We use the following simple lemma in the proof of Theorem \ref{thm:purslowgraphnew}.
\begin{lemma}
\label{lem:restictcontractbase}
Let $\V_X$ be a vector space on $X$ and let $Y\subseteq X.$
Let $T$ be a column base of $\V_X.$
\begin{enumerate}
\item Let  $T\subseteq Y.$
Then $T$ is a column base of  $\V_X\circ Y.$
\item  Let $T\supseteq (X-Y).$
Then $T-(X-Y)$ is a column base of  $\V_X\times Y.$
\end{enumerate}
\end{lemma}
\begin{proof}
%({\it Proof of Lemma \ref{lem:restictcontractbase}})\\
1. The columns $T$ continue to remain independent in $\V_X\circ Y.$
Since any column base $\V_X\circ Y$ is contained in a column base of $\V_X,$
the result follows.

2. We have $X-T$ as a column cobase of $\V_X$ and therefore as a column base
of $\V_X^{\perp}.$ Now since $T\supseteq (X-Y),$ we must have $X-T\subseteq Y.$
By part 1 above $X-T$ is a column cobase of $\V_X^{\perp}\circ Y= (\V_X\times Y)^{\perp}.$ Therefore  $Y-(X-T)= (T-(X-Y))$ is a column base of $\V_X\times Y.$
\end{proof}

\begin{proof}
 ({\it Proof of Theorem \ref{thm:purslowgraphnew}})\\
For better readability, we introduce simplifying notation.
Let $A\equivd P'\uplus P",X\equivd S'\uplus S",$ \\where
$S= Y_{1}\uplus Z_{1}\uplus Y_{2}\uplus Z_{2}\uplus R\uplus E\uplus J,$
$B=B_1\uplus B_2,$ where $B_1\equivd Y_{1}"\uplus Y_{2}'\uplus R"$
and
 $B_2\equivd Z_{1}'\uplus Z_{2}"\uplus R'.$
\\
Let $X= X_1\uplus X_2,$ where $X_1\equivd E"\uplus J'\uplus Z_{1}"\uplus Z_{2}'\uplus B_1,$
 and let $X_2\equivd E'\uplus J"\uplus Y_{1}'\uplus Y_{2}"\uplus B_2.$
Note that, in the device characteristic $\A_{S'S"},
$ $X_1-B_1$ are the variables without constraint,
$X_2-B_2$ are the `predetermined' variables,
$B_1$ are the controlling variables and $B_2$ are the controlled variables.
Observe that $|X_1|=|X_2|=|S|$ and $|B_1|=|B_2|=|S|-|X_1-B_1|=|S|-|X_2-B_2|.$
\\
We define $\V^2_X\equivd \V_{S'S"}$ and $\V^1_{AX}\equivd  
\V^1_{P'P"S'S"}=(\V^v(\Gsp))_{S'P'}\oplus (\V^i(\Gsp))_{S"P"}.$

%===========================================================================
The solution space of combined Equations \ref{eqn:ccvsvvcs1new}, \ref{eqn:ccvsvvcsanew}, \ref{eqn:ccvsvvcs3new} is the affine space $\A^2_X\equivd \A_{S'S"}.
$ It can be seen that the vector space translate $\V^2_{X}$
 of $\A^2_{X}$ 
has the form $\V^2_{X}=\V_{S'S"}=\F_{E"Z_{1}" Z_{2}'J'}\oplus \0_{E'Y_{1}' Y_{2}"J"}\oplus \V^2_B,$ 
where $\V^2_B$ has the representative matrix $(I|D)$ with the partition corresponding to column division
$B_1,B_2$ and with the diagonal entries of the diagonal matrix $D$ being algebraically independent 
over $\Q.$
%where $B\equivd Y_{1}"\uplus Y_{2}'\uplus R"\uplus Z_{1}'\uplus Z_{2}"\uplus R'.$ 
Note that $\V^2_{X}$ is the source free device 
characteristic, i.e., the voltage sources and current sources have zero values.
From the form of the device characteristic in Equations \ref{eqn:ccvsvvcs1new}, \ref{eqn:ccvsvvcsanew}, \ref{eqn:ccvsvvcs3new}, it is clear 
that $r(\V^2_{X})=r(\V_{S'S"})= |S|.$
Therefore, by definition of proper multiports, $\N_P$ is proper iff it is rigid.
By the definition of rigidity of multiports, $\N_P\equivd (\V^1_{AX}, \A^2_{X})$ is rigid iff 
 $\{\V^1_{AX}, \V^2_{X}\}=\{\V^1_{P'P"S'S"},\V_{S'S"}\}$ is rigid.
Let $\V^1_{AB}\equivd \V^1_{AX}\circ A B_1 X_2 \times AB.$

The solution space of the topological constraints of the multiport $\N_{P}$ 
has the form \\$\V^1_{AX}=\V^1_{P'P"S'S"}=(\V^v(\Gsp))_{S'P'}\oplus (\V^i(\Gsp))_{S"P"}.$
We now examine the nature of $\V^1_{AX}$ and construct a column base for it.

%We have $\V^1_{AX}= (\V^v(\Gsw))_{S'W'}\oplus (\V^i(\Gsw))_{S"W"}.$
A set of columns of $\V^1_{AX}$
is a column base iff it has the form $t_1'\uplus ((S\uplus P)-t_2)",$ where
$t_1,t_2$ are trees of $\Gsp$ (Lemma \ref{lem:minorgraphvectorspace}). Pick a tree $t$ that contains $Y_{1}\uplus 
Z_{1}\uplus E$ and does not intersect $Y_{2}\uplus 
Z_{2}\uplus J.$ This is possible by Theorem \ref{thm:dotcrossidentity}.
Choose as a column base of $\V^1_{AX}\equivd (\V^v(\Gsp))_{S'P'}\oplus (\V^i(\Gsp))_{S"P"},$ the set
$t'\uplus ((S\uplus P)-t)".$
Let $t\equivd Y_{1}\uplus 
Z_{1}\uplus E\uplus R_1\uplus P_1,$ where $P_1\subseteq P.$
We have $S= Y_{1}\uplus 
Z_{1}\uplus E \uplus R\uplus  
Y_{2}\uplus Z_{2}\uplus J.$
Therefore $\bar{t}\equivd S\uplus P-t\equivd Y_{2}\uplus 
Z_{2}\uplus J\uplus R_2\uplus P_2,$ where $P_2\equivd P-P_1,$ $R_2\equivd R-R_1.$

Before we begin the main part of the proof, we prove the following claim.
\\{\bf Claim.}
The pair $\{\V^1_{AB},\V^2_{B}\}$ is rigid.\\
 {\it Proof of Claim.}

We will show that $(t-(E\uplus Y_{1}))'\uplus (\bar{t}-(J\uplus Y_{2}))"$ is a column base of
$\V^1_{AB}\equivd
\V^1_{AX}\circ A B_1 X_2 \times AB_1 B_2.$
\\
%We have $\V^1_{AB}\equivd
%\V^1_{AX}\circ A B_1 X_2 \times AB_1 B_2.$
%$= \V^1_{P'P"S'S"}\circ P'P"Y_{1}"Y_{2}'R"X_2\times P'P"Y_{1}"Y_{2}'R'Z_{1}'Z_{2}" R".$\\
Since $X_1\equivd E"\uplus J'\uplus Z_{1}"\uplus Z_{2}'\uplus B_1,$
$X_2\equivd E'\uplus J"\uplus Y_{1}'\uplus Y_{2}"\uplus B_2,$
$t'= Y'_{1}\uplus  
Z'_{1}\uplus E'\uplus R'_1\uplus P'_1,$
$\bar{t}"= Y"_{2}\uplus 
Z"_{2}\uplus J"\uplus R"_2\uplus P"_2,$
we see that
$t'\uplus \bar{t}"$ does not intersect $X_1-B_1= E"\uplus J'\uplus Z_{1}"\uplus Z_{2}'.$
Therefore, $t'\uplus \bar{t}"$ continues to remain a column base of
$\V^1_{AX}\circ A B_1 X_2 $ (by part 1 of Lemma \ref{lem:restictcontractbase}).

Next $X_2-B_2= E'\uplus J"\uplus Y_{1}'\uplus Y_{2}"\subseteq t'\uplus \bar{t}",$ Therefore $(t-(E\uplus Y_{1}))'\uplus (\bar{t}-(J\uplus Y_{2}))"$ is a column base of
$\V^1_{AB}= \V^1_{AX}\circ A B_1 X_2\times (A B_1 X_2-( X_2-B_2))= \V^1_{AX}\circ A B_1 X_2\times A B_1 B_2$
(by part 2 of Lemma \ref{lem:restictcontractbase}).

We now examine the column bases of $\V^2_B.$
Let $R=R_1\uplus R_2,$ where $R_1=t\cap R.$
It can be verified that $\V^2_B$ after column permutation within $B_1$
and $B_2$ has a representative matrix of the form $(I|D)\equiv \ppmatrix{I&0&\vline &D_1&0\\
0&I&\vline&0&\hat{R}},$
where the identity matrix corresponds to columns $B_1\equivd Y_{1}"\uplus Y_{2}'\uplus R",$
the diagonal matrix $D$ corresponds to columns $B_2\equivd Z_{1}'\uplus Z_{2}"\uplus R',$
the top left identity submatrix corresponds to $Y_{1}"\uplus Y_{2}',$
the bottom right  identity submatrix corresponds to $R".$
We are given that the diagonal entries of the diagonal matrix $D$ are algebraically
independent over $\Q.$  In particular, they are nonzero.
Now since $B_1$ is a column base for $\V^2_B$
 so would $(B_1-R_2")\uplus R_2'= (Y_{1}"\uplus Y_{2}'\uplus R_2')\uplus R_1"$ be.
%It is therefore clear that
%$Y_{1}"\uplus Y_{2}'\uplus R_2'\uplus R_1"$
%is a column base of $\V^2_B.$ 
Since $t\equivd Y_{1}\uplus 
Z_{1}\uplus E\uplus R_1\uplus P_1,  \bar{t}\equivd S\uplus P-t\equivd Y_{2}\uplus 
Z_{2}\uplus J\uplus R_2\uplus P_2,$ it is clear that $(t-(Y_{1}\uplus E))'\uplus (\bar{t}-(Y_{2}\uplus J))"$ does not intersect this column base of $\V^2_B.$
%To see the latter fact, note that since $B_1$ is a base for $\V^2_B$
% so would $(B_1-R_2')\uplus R_2"= (Y_{1}"\uplus Y_{2}'\uplus R_1')\uplus R_2"$ be.
%We will now show that Lemma \ref{lem:generalitytomaxdistancenew} can
%applied to infer the rigidity of $\{\V^1_{AB}, \V^2_{B}\}.$
\\Observe that $(t-(Y_{1}\uplus E))'\uplus (\bar{t}-(Y_{2}\uplus J))"\uplus Y_{1}"\uplus Y_{2}'\uplus R_2'\uplus R_1"
\supseteq Z_{1}'\uplus R_1'\uplus R_2'\uplus Z_{2}"\uplus R_1"\uplus R_2"\uplus Y_{1}"\uplus Y_{2}'=B.$
\\
Therefore  we have disjoint column bases
$(t-(Y_{1}\uplus E))'\uplus (\bar{t}-(Y_{2}\uplus J))"$
 and
$Y_{1}"\uplus Y_{2}'\uplus R_2'\uplus R_1"$ of $\V^1_{AB}, \V^2_{B}$
 respectively, which together cover $B.$
 By part 3(b) of Lemma \ref{lem:generalitytomaxdistancenew}, it follows that $\{\V^1_{AB}, \V^2_{B}\}$ is rigid. 
\\{\it Proof of Claim ends}
%
%===========================================================================
\\({\it Proof of Theorem \ref{thm:purslowgraphnew} continued.})\\
1. 
%$\N_W$ 
%is rigid  iff $\{\V^1_{AX}, \V^2_X\}= \{\V^1_{W'W"S'S"},\V_{S'S"}\}$ is rigid.
Since $\N_P$ is given to be rigid, $\{\V^1_{AX}, \V^2_X\}$ must be rigid.
Therefore, we must have the full sum property satisfied,
 i.e., that $r(\V^1_{AX}+ \V^2_X)\circ X= r((\V^1_{P'P"S'S"}+\V_{S'S"})\circ S'S")= 
2|S|.$ \\Now $\V_{S'S"}\circ E'Y_{1}' Y_{2}"J"=\0_{E'Y_{1}' Y_{2}"J"}.
$ It follows that  $\V^1_{AX}\circ E'Y_{1}' Y_{2}"J"$\\$=
(\V^v(\Gsp))_{S'P'}\circ E'Y_{1}'\oplus (\V^i(\Gsp))_{S"P"}\circ
 Y_{2}"J" = \F_{E'Y_{1}'}\oplus \F_{Y_{2}"J"},$ \\i.e. (using Lemma \ref{lem:minorgraphvectorspace}), that 
$E\uplus Y_{1}$ can be included in a tree of $\Gsp$ and therefore is loop free  and $J\uplus Y_{2}$ 
can be included in a cotree of $\Gsp$ and therefore is cutset free in $\Gsp .$

2. We will first show that $\{\V^1_{AB},\V^2_{B}\}$ is rigid iff $\{\V^1_{AX},\V^2_{X}\}$ is rigid.
\\
We are given that $Y_{1}\uplus Z_{1}\uplus E$ is loop free and $Y_{2}\uplus Z_{2}\uplus J$
is cutset free. Therefore there is a tree containing $Y_{1}\uplus Z_{1}\uplus E$ and a cotree containing $Y_{2}\uplus Z_{2}\uplus J.$ We therefore have,\\ 
$\V^1_{AX}\times (X_1-B_1)=\V^1_{P'P"S'S"}\times E"Z_{1}"Z_{2}'J'= (\V^v(\Gsp))_{S'P'}\times Z_{2}'J'\oplus 
(\V^i(\Gsp))_{S"P"}\times Z_{1}"E"=$\\$ (\V^v(\Gsp\times Z_{2}J))_{Z_{2}'J'}\oplus (\V^i(\Gsp\circ Z_{1}E))_{Z_{1}"E"}= \0_{Z_{2}'J'}\oplus \0_{E"Z_{1}"}= \0_{E"Z_{1}"Z_{2}'J'}=\0_{X_1-B_1}$ 
and $\V^1_{AX}\circ (X_2-B_2)$\\$=\V^1_{P'P"S'S"}\circ E'Y_{1}'Y_{2}"J"= (\V^v(\Gsp\circ EY_{1}))_{E'Y_{1}'}\oplus (\V^i(\Gsp\times Y_{2}J))_{Y_{2}"J"}=  \F_{E'Y_{1}'Y_{2}"J"}= \F_{X_2-B_2}$ (using Lemma \ref{lem:minorgraphvectorspace}).

Now we have, 
$\V^2_X\equivd \F_{X_1-B_1}\oplus \0_{X_2-B_2}\oplus \V^2_B,$
$\V^1_{AX}\times (X_1-B_1)=\0_{X_1-B_1},$ 
$\V^1_{AX}\circ (X_2-B_2)=\F_{X_2-B_2}$
and  $\V^1_{AB}\equivd \V^1_{AX}\circ A B_1 X_2 \times AB.$
We can therefore apply Theorem \ref{thm:maxdistancenew2} (taking $Z\equivd X_1-B_1, Y\equivd X_2-B_2$)
and conclude that $\{\V^1_{AB},\V^2_{B}\}$ is rigid iff $\{\V^1_{AX},\V^2_{X}\}$ is rigid.
\\We have proved above the  claim that $\{\V^1_{AB},\V^2_{B}\}$ is rigid.
Therefore, $\{\V^1_{AX},\V^2_{X}\}$ is rigid.
 
3. We have seen that when the conditions of part 2 are satisfied, there exists a tree $t$ that
contains $E\uplus Y_{1}\uplus Z_{1}$ and that does not intersect $J\uplus Y_{2}\uplus Z_{2}.$
We have $t\equivd Y_{1}\uplus 
Z_{1}\uplus E\uplus R_1\uplus P_1,$ where $P_1\subseteq P$
and we have $S= Y_{1}\uplus 
Z_{1}\uplus E \uplus R\uplus  
Y_{2}\uplus Z_{2}\uplus J,$
$\bar{t}\equivd S\uplus P-t= Y_{2}\uplus 
Z_{2}\uplus J\uplus R_2\uplus P_2,$ where $P_2\equivd P-P_1,$ $R_2\equivd R-R_1.$
We have seen that $b_1\equivd t'\uplus \bar{t}"$ is a column base for 
$\V^1_{AX}\equivd  \V^1_{P'P"S'S"}=(\V^v(\Gsp))_{S'P'}\oplus (\V^i(\Gsp))_{S"P"}.$
\\
We note that 
$\V^2_{X}=\V_{S'S"}=\F_{E"Z_{1}" Z_{2}'J'}\oplus \0_{E'Y_{1}' Y_{2}"J"}\oplus \V^2_B,$
where $\V^2_B$ has 
column base $ Y_{1}"\uplus Y_{2}'\uplus R_2'\uplus R_1".$
\\
Therefore $b_2\equivd E"\uplus Z_{1}" \uplus Z_{2}'\uplus J' \uplus Y_{1}"\uplus Y_{2}'\uplus R_2'\uplus R_1"$ is a column base for $\V^2_{X}.$
\\
It can be seen that $b_1,b_2$ are disjoint and together cover $X\equivd S'\uplus S".$
\\
By part 2 of Theorem \ref{thm:maxdistancenew2},
$b^1\cap A\equivd 
(t'\uplus \bar{t}")\cap A=(t\cap P)'\uplus (P-t)",$
is a column base of $\V^1_{AX}\lrar \V^2_{X}.$
Taking $v_{P_1'}\equivd (t\cap P)',i_{P_2"}\equivd (P-t)",$
 we see that $\V^1_{AX}\lrar \A^2_{X}$ has a hybrid representation of the form 
\begin{align}
\label{eqn:hybridconstraints2}
\ppmatrix{i_{P_1"}\\v_{P_2'}}=\ppmatrix{g_{11}&h_{12}\\h_{21}&r_{22}}\ppmatrix{v_{P_1'}\\i_{P_2"}}+\ppmatrix{s_{P_1"}\\s_{P_2'}}.
\end{align}

\end{proof}

\section{Proof of Theorem \ref{thm:purslow2}}
\label{sec:purslow2}
\begin{lemma}
\label{lem:diracrestriction}
Let $\V^1_{S'P'S"P"}$ be Dirac. 
We then have the following.
\begin{enumerate}
\item $r(\V^1_{S'P'S"P"})=|S|+|P|= |S'|+|P'|=|S"|+|P"|.$
\item Let
$T\subseteq S\uplus P$ be such that
$T'$ is a column base of $\V^1_{S'P'S"P"}\circ S'P'.$
Then
$(S\uplus P -T)"$ is a column base of $\V^1_{S'P'S"P"}\times S"P"$ and
$T'\uplus (S\uplus P -T)"$ is a column base of $\V^1_{S'P'S"P"}.$
%\item Let $X,Y$ be disjoint subsets of $S\uplus P$ and let $X'\uplus Y"$ 
%be independent columns of $\V^1_{S'P'S"P"}.$ Then there is a column base $T$ of 
%$\V^1_{S'P'S"P"}\circ S'P'$ such that $T'\uplus ((S\uplus P)-T)"$ is 
%a column base of $\V^1_{S'P'S"P"}.$
\end{enumerate}
\end{lemma}
\begin{proof}
1. We have $r(\V^1_{S'P'S"P"})+r((\V^1_{S'P'S"P"})^{\perp})= |S'|+|P"|+|S"|+|P"|$ 
(Theorem \ref{thm:perperp}). Since $\V^1_{S'P'S"P"}$ is Dirac,
 we must have $(\V^1_{S'P'S"P"})^{\perp}=(\V^1_{S'P'S"P"})_{S"P"S'P'}.$ 
Since $S',S"$ are copies of $S$ and $P',P"$ are copies of $P,$ 
the result follows.

2. We denote $(\V^1_{S'P'S"P"})_{S"P"S'P'}$ by $\V^1_{S"P"S'P'}.$
Since $\V^1_{S'P'S"P"}$ is Dirac, we have\\ $(\V^1_{S'P'S"P"})^{\perp}
= \V^1_{S"P"S'P'}.$
We have that $(S\uplus P -T)'$ is a column cobase of $\V^1_{S'P'S"P"}\circ S'P',$
 i.e., a column base of $(\V^1_{S'P'S"P"}\circ S'P')^{\perp}=(\V^1_{S'P'S"P"})^{\perp}\times S'P'= \V^1_{S"P"S'P'}\times S'P'.$
Therefore, $(S\uplus P -T)"$ is a column base of
$(\V^1_{S"P"S'P'}\times S'P')_{S"P"}= (\V^1_{S"P"S'P'})_{S'P'S"P"}\times S"P"=\V^1_{S'P'S"P"}\times  S"P".$ By Theorem \ref{thm:dotcrossidentity}, part 4,
 $T'\uplus (S\uplus P -T)"$ is a column base of  $\V^1_{S'P'S"P"}.$
\end{proof}

\begin{proof}
({\it Proof of Theorem \ref{thm:purslow2}})\\
We will prove sufficiency assuming conditions in part (a). 
The proof assuming part (b) is by interchanging prime and double prime for all
the variables in the proof assuming conditions in part (a). 

%===========================================================================
For better readability, we introduce simplifying notation.
Let $A\equivd P'\uplus P",X\equivd S'\uplus S",$ \\where
$S= Y_{1}\uplus Z_{1}\uplus Y_{2}\uplus Z_{2}\uplus R\uplus E\uplus J,$
$B=B_1\uplus B_2,$ where $B_1\equivd Y_{1}"\uplus Y_{2}'\uplus R"$
and
 $B_2\equivd Z_{1}'\uplus Z_{2}"\uplus R'.$
\\
Let $X= X_1\uplus X_2,$ where $X_1\equivd E"\uplus J'\uplus Z_{1}"\uplus Z_{2}'\uplus B_1,$
 and let $X_2\equivd E'\uplus J"\uplus Y_{1}'\uplus Y_{2}"\uplus B_2.$
Note that, in the device characteristic $\A_{S'S"},
$ $X_1-B_1$ are the variables without constraint,
$X_2-B_2$ are the `predetermined' variables,
$B_1$ are the controlling variables and $B_2$ are the controlled variables.
Observe that $|X_1|=|X_2|=|S|$ and $|B_1|=|B_2|=|S|-|X_1-B_1|=|S|-|X_2-B_2|.$
\\
We define $\V^2_X\equivd \V_{S'S"}$ and $\V^1_{AX}\equivd  
\V^1_{P'P"S'S"}.$

%===========================================================================
The solution space of combined Equations \ref{eqn:ccvsvvcs1new}, \ref{eqn:ccvsvvcsanew}, \ref{eqn:ccvsvvcs3new} is the affine space $\A^2_X\equivd \A_{S'S"}.
$ It can be seen that the vector space translate $\V^2_{X}$
 of $\A^2_{X}$ 
has the form $\V^2_{X}=\V_{S'S"}=\F_{E"Z_{1}" Z_{2}'J'}\oplus \0_{E'Y_{1}' Y_{2}"J"}\oplus \V^2_B,$ 
where $\V^2_B$ has the representative matrix $(I|D)$ with the partition corresponding to column division
$B_1,B_2$ and with the diagonal entries of the diagonal matrix $D$ being algebraically independent 
over $\Q.$
%where $B\equivd Y_{1}"\uplus Y_{2}'\uplus R"\uplus Z_{1}'\uplus Z_{2}"\uplus R'.$ 
Note that $\V^2_{X}$ is the source free device 
characteristic, i.e., the voltage sources and current sources have zero values.
%From the form of the device characteristic in Equations \ref{eqn:ccvsvvcs1new}, \ref{eqn:ccvsvvcsanew}, \ref{eqn:ccvsvvcs3new}, it is clear 
%that $r(\V^2_{X})=r(\V_{S'S"})= |S|.$
%Therefore, by definition, $\N^g_P$ is proper iff it is rigid.
By the definition of rigidity of pairs of affine spaces, $\N^g_P\equivd (\V^1_{AX}, \A^2_{X})$ is rigid iff 
 $\{\V^1_{AX}, \V^2_{X}\}=\{\V^1_{P'P"S'S"},\V_{S'S"}\}$ is rigid.
Let $\V^1_{AB}\equivd \V^1_{AX}\circ A B_1 X_2 \times AB.$

We first construct a column base for the Dirac space $\V^1_{AX}= \V^1_{P'P"S'S"}.$
\\
%We have $\V^1_{AX}= (\V^v(\Gsw))_{S'W'}\oplus (\V^i(\Gsw))_{S"W"}.$
By Lemma \ref{lem:diracrestriction}, we have that a set of columns of $\V^1_{AX}$
is a column base if it has the form $T'\uplus (S\uplus P -T)",$
where $T'$ is a column base of $\V^1_{S'P'S"P"}\circ S'P'.$
Pick a column base $T'$ of $\V^1_{S'P'S"P"}\circ S'P'$ that contains $Y'_{1}\uplus 
 Z'_{1}\uplus E'$ and does not intersect $Y'_{2}\uplus 
 Z'_{2}\uplus J'.$
This is possible by part 7 of Theorem \ref{thm:dotcrossidentity}.
Choose the set $T'\uplus (S\uplus P -T)"$ as a column base of $\V^1_{AX}\equivd \V^1_{S'P'S"P"}.$ 

Before we begin the main part of the proof, we prove the following claim.
\\{\bf Claim.} 
The pair $\{\V^1_{AB},\V^2_{B}\}$ is rigid.\\
%Let $\V^1_{AB}\equivd \V^1_{AX}\circ A B_1 X_2 \times AB.$
%= (\V^v(\G_{WT_1}))_{T_1'}\oplus (\V^i(\G_{WT_2}))_{T_2"}.$
 {\it Proof of Claim.}
%The solution space of the topological constraints of the multiport $\N_{P}$ 
%has the form 

Let $T\equivd Y_{1}\uplus  
Z_{1}\uplus E\uplus R_1\uplus P_1,$ where $P_1\subseteq P.$
We have $S= Y_{1}\uplus 
Z_{1}\uplus E \uplus R\uplus  
Y_{2}\uplus Z_{2}\uplus J,$
Therefore $\bar{T}\equivd S\uplus P-T=Y_{2}\uplus 
Z_{2}\uplus J\uplus R_2\uplus P_2,$ where $P_2\equivd P-P_1,$ $R_2\equivd R-R_1.$\\
We will show that $(T-(E\uplus Y_{1}))'\uplus (\bar{T}-(J\uplus Y_{2}))"$ is a column base of
$\V^1_{AB}.$

We have $\V^1_{AB}\equivd
\V^1_{AX}\circ A B_1 X_2 \times AB_1 B_2.$
%$= \V^1_{P'P"S'S"}\circ P'P"Y_{1}"Y_{2}'R"X_2\times P'P"Y_{1}"Y_{2}'R'Z_{1}'Z_{2}" R".$\\
Since $X_1\equivd E"\uplus J'\uplus Z_{1}"\uplus Z_{2}'\uplus B_1,$
$X_2\equivd E'\uplus J"\uplus Y_{1}'\uplus Y_{2}"\uplus B_2,$
$T'= Y'_{1}\uplus  
Z'_{1}\uplus E'\uplus R'_1\uplus P'_1,$
$\bar{T}"= Y"_{2}\uplus 
Z"_{2}\uplus J"\uplus R"_2\uplus P"_2,$
we see that 
$T'\uplus \bar{T}"$ does not intersect $X_1-B_1= E"\uplus J'\uplus Z_{1}"\uplus Z_{2}'.$
Therefore, $T'\uplus \bar{T}"$ continues to remain a column base of
$\V^1_{AX}\circ A B_1 X_2 $ (by part 1 of Lemma \ref{lem:restictcontractbase}).

%======================================================
%and $B_1\equivd Y_{1}"\uplus Y_{2}'\uplus R"$
%and
% $B_2\equivd Z_{1}'\uplus Z_{2}"\uplus R',$
% we have $T'\uplus \bar{T}"= (Y_{1}\uplus  
%Z_{1}\uplus E\uplus R_1\uplus P_1)'\uplus (Y_{2}\uplus 
%Z_{2}\uplus J\uplus R_2\uplus P_2)"$\\$\subseteq P'\uplus P"\uplus [Y_{1}"\uplus Y_{2}'\uplus R"]\uplus  [E'\uplus J"\uplus Y_{1}'\uplus Y_{2}"\uplus  Z_{1}'\uplus Z_{2}"\uplus R']= P'\uplus P"\uplus B_1\uplus X_2= A\uplus  B_1 \uplus X_2 .$
%===================================================================

%Therefore, $T'\uplus \bar{T}"$ continues to remain a column base of 
%$\V^1_{AX}\circ A B_1 X_2 $ (by part 1 of Lemma \ref{lem:restictcontractbase}).\\
Next $X_2-B_2= E'\uplus J"\uplus Y_{1}'\uplus Y_{2}"\subseteq T'\uplus \bar{T}",$ Therefore $(T-(E\uplus Y_{1}))'\uplus (\bar{T}-(J\uplus Y_{2}))"$ is a column base of
$\V^1_{AB}= \V^1_{AX}\circ A B_1 X_2\times (A B_1 X_2-( X_2-B_2))= \V^1_{AX}\circ A B_1 X_2\times A B_1 B_2$
(by part 2 of Lemma \ref{lem:restictcontractbase}).

We now examine the column bases of $\V^2_B.$
We have $R=R_1\uplus R_2,$ where $R_1=T\cap R.$
It can be verified that $\V^2_B$ after column permutation within $B_1$
and $B_2$ has a representative matrix of the form $(I|D)\equiv \ppmatrix{I&0&\vline &D_1&0\\
0&I&\vline&0&\hat{R}},$
where the identity matrix corresponds to columns $B_1\equivd Y_{1}"\uplus Y_{2}'\uplus R",$
the diagonal matrix $D$ corresponds to columns $B_2\equivd Z_{1}'\uplus Z_{2}"\uplus R',$
the top left identity submatrix corresponds to $Y_{1}"\uplus Y_{2}',$
the bottom right  identity submatrix corresponds to $R".$
We are given that the diagonal entries of the diagonal matrix $D$ are algebraically
independent over $\Q.$  In particular, they are nonzero.
Now since $B_1$ is a column base for $\V^2_B$
 so would $(B_1-R_2")\uplus R_2'= (Y_{1}"\uplus Y_{2}'\uplus R_2')\uplus R_1"$ be.
%It is therefore clear that
%$Y_{1}"\uplus Y_{2}'\uplus R_2'\uplus R_1"$
%is a column base of $\V^2_B.$ 
Since $T\equivd Y_{1}\uplus 
Z_{1}\uplus E\uplus R_1\uplus P_1,$\\$  \bar{T}\equivd S\uplus P-T\equivd Y_{2}\uplus 
Z_{2}\uplus J\uplus R_2\uplus P_2,$ it is clear that $(T-(Y_{1}\uplus E))'\uplus (\bar{T}-(Y_{2}\uplus J))"$ does not intersect this column base.
%To see the latter fact, note that since $B_1$ is a base for $\V^2_B$
% so would $(B_1-R_2')\uplus R_2"= (Y_{1}"\uplus Y_{2}'\uplus R_1')\uplus R_2"$ be.
%We will now show that Lemma \ref{lem:generalitytomaxdistancenew} can
%applied to infer the rigidity of $\{\V^1_{AB}, \V^2_{B}\}.$
\\Observe that $(T-(Y_{1}\uplus E))'\uplus (\bar{T}-(Y_{2}\uplus J))"\uplus Y_{1}"\uplus Y_{2}'\uplus R_2'\uplus R_1"
\supseteq Z_{1}'\uplus R_1'\uplus R_2'\uplus Z_{2}"\uplus R_1"\uplus R_2"\uplus Y_{1}"\uplus Y_{2}'=B.$
\\
Therefore  we have disjoint column bases
$(T-(Y_{1}\uplus E))'\uplus (\bar{T}-(Y_{2}\uplus J))"$
 and
$Y_{1}"\uplus Y_{2}'\uplus R_2'\uplus R_1"$ of $\V^1_{AB}, \V^2_{B}$
 respectively, which together cover $B.$
 By part 3(b) of Lemma \ref{lem:generalitytomaxdistancenew}, it follows that $\{\V^1_{AB}, \V^2_{B}\}$ is rigid. 
\\{\it Proof of Claim ends}
%
%===========================================================================
\\({\it Proof of Theorem \ref{thm:purslow2} continued.})\\
1. 
%$\N_W$ 
%is rigid  iff $\{\V^1_{AX}, \V^2_X\}= \{\V^1_{W'W"S'S"},\V_{S'S"}\}$ is rigid.
Since $\N_P$ is given to be rigid, $\{\V^1_{AX}, \V^2_X\}$ must be rigid.
Therefore, we must have the full sum property satisfied,
 i.e., that $r(\V^1_{AX}+ \V^2_X)\circ X= r((\V^1_{P'P"S'S"}+\V_{S'S"})\circ S'S")= 
2|S|.$ Now $\V_{S'S"}\circ E'Y_{1}' Y_{2}"J"=\0_{E'Y_{1}' Y_{2}"J"}.
$ It follows that $\V^1_{AX}\circ E'Y_{1}' Y_{2}"J"=
  \F_{E'Y_{1}'}\oplus \F_{Y_{2}"J"},$ \\i.e., that 
the sets  of columns $E'\uplus Y'_{1}$  and $J"\uplus Y"_{2}$ 
of $\V^1_{P'P"S'S"}$ 
are independent.

2(a) We will first show that $\{\V^1_{AB},\V^2_{B}\}$ is rigid iff $\{\V^1_{AX},\V^2_{X}\}$ is rigid.
\\
We are given that $Y'_{1}\uplus Z'_{1}\uplus E'$ is part of a column base of  $\V^1_{P'P"S'S"}\circ S'P'$  and $Y'_{2}\uplus Z'_{2}\uplus J'$ can be included in a column cobase of
$\V^1_{P'P"S'S"}\circ S'P'.$ We saw above that this means that there exists a
 column base $T'$ of $\V^1_{P'P"S'S"}\circ S'P'$  that contains $Y'_{1}\uplus Z'_{1}\uplus E'$ but does not intersect $Y'_{2}\uplus Z'_{2}\uplus J'$
and $T'\uplus (S\uplus P -T)"$ is a column base of $\V^1_{P'P"S'S"}.$ 
This means $Y'_{1}\uplus Z'_{1}\uplus E'\uplus Y_{2}"\uplus Z_{2}"\uplus J"$ 
is independent in $\V^1_{P'P"S'S"}.$ Therefore $\V^1_{AX}\circ (X_2-B_2)= \V^1_{P'P"S'S"}\circ E'Y_{1}'Y_{2}"J"=\F_{E'Y_{1}'Y_{2}"J"}= \F_{X_2-B_2}.$

On the other hand, the column base $T'\uplus (S\uplus P -T)"$ of 
$\V^1_{P'P"S'S"}$ does not contain
$ X_1-B_1$\\$=Z"_{1}\uplus E"\uplus  Z_{2}'\uplus J'.$
We therefore have,
$\V^1_{AX}\times (X_1-B_1)=\V^1_{P'P"S'S"}\times E"Z_{1}"Z_{2}'J'=\0_{E"Z_{1}"Z_{2}'J'}=\0_{X_1-B_1}.$

Now we have, 
$\V^2_X\equivd \F_{X_1-B_1}\oplus \0_{X_2-B_2}\oplus \V^2_B,$
$\V^1_{AX}\times (X_1-B_1)=\0_{X_1-B_1},$ 
$\V^1_{AX}\circ (X_2-B_2)=\F_{X_2-B_2}$
and  $\V^1_{AB}\equivd \V^1_{AX}\circ A B_1 X_2 \times AB.$
We can therefore apply Theorem \ref{thm:maxdistancenew2} (taking $Z\equivd X_1-B_1, Y\equivd X_2-B_2$)
and conclude that $\{\V^1_{AB},\V^2_{B}\}$ is rigid iff $\{\V^1_{AX},\V^2_{X}\}$ is rigid.
\\We have proved above the  claim that $\{\V^1_{AB},\V^2_{B}\}$ is rigid .
Therefore, $\{\V^1_{AX},\V^2_{X}\}$ is rigid.
 
3. We have seen that when the conditions of part 2(a) are satisfied, there exists a column  base $T'$ of $\V^1_{P'P"S'S"}\circ S'P'$ that
contains $E'\uplus Y'_{1}\uplus Z'_{1}$ and that does not intersect $J'\uplus Y'_{2}\uplus Z'_{2}.$ 
%Further $T'\uplus ((S\uplus P)-T)"$ is a column base of 
% $\V^1_{P'P"S'S"}.$
\\
We have $T\equivd Y_{1}\uplus 
Z_{1}\uplus E\uplus R_1\uplus P_1,$ where $P_1\subseteq P$
and we have $S= Y_{1}\uplus 
Z_{1}\uplus E \uplus R\uplus  
Y_{2}\uplus Z_{2}\uplus J,$
$\bar{T}\equivd S\uplus P-T= Y_{2}\uplus 
Z_{2}\uplus J\uplus R_2\uplus P_2,$ where $P_2\equivd P-P_1,$ $R_2\equivd R-R_1.$
We have seen that $b_1\equivd T'\uplus \bar{T}"$ is a column base for 
$\V^1_{AX}\equivd  \V^1_{P'P"S'S"}.$
\\
We note that 
$\V^2_{X}=\V_{S'S"}=\F_{E"Z_{1}" Z_{2}'J'}\oplus \0_{E'Y_{1}' Y_{2}"J"}\oplus \V^2_B,$
where $\V^2_B$ has 
column base $ Y_{1}"\uplus Y_{2}'\uplus R_2'\uplus R_1".$
\\
Therefore $b_2\equivd E"\uplus Z_{1}" \uplus Z_{2}'\uplus J' \uplus Y_{1}"\uplus Y_{2}'\uplus R_2'\uplus R_1"$ is a column base for $\V^2_{X}.$
\\
It can be seen that $b_1,b_2$ are disjoint and together cover $X\equivd S'\uplus S".$
\\
By part 2(c) of Theorem \ref{thm:maxdistancenew2},
$b_1\cap A\equivd 
(T'\uplus \bar{T}")\cap A=(T\cap P)'\uplus (P-T)",$
is a column base of $\V^1_{AX}\lrar \V^2_{X}.$
Taking ${P_1'}\equivd (T\cap P)'= b_1\cap P',{P_2"}\equivd (P-T)"=b_1\cap P",$
 we see that $\V^1_{AX}\lrar \A^2_{X}$ has a hybrid representation of the form 
\begin{align}
\label{eqn:hybridconstraints3}
\ppmatrix{i_{P_1"}\\v_{P_2'}}=\ppmatrix{g_{11}&h_{12}\\h_{21}&r_{22}}\ppmatrix{v_{P_1'}\\i_{P_2"}}+\ppmatrix{s_{P_1"}\\s_{P_2'}}.
\end{align}

\end{proof}

\section{Port transformation for general multiports}
\label{subsec:porttransformation}
The Implicit Inversion Theorem (Theorem \ref{thm:IIT}) has a special feature which is suited to electrical
network applications, viz., the fact that in the equation
$\Vsp\lrar \K_S= \K_P, $ because $\Vsp$ is a vector space, we can
relate $\K_S, \K_P, $ even if they are very general. In electrical networks,
this permits us to speak port characteristics for multiports with very general
device characteristics because the topological constraints are linear.
We present below,  a discussion on  `port transformation' for general
 multiports  with arbitrary device characteristic.

%\begin{remark}
%In this section, we use the conventional definition of multiport behaviour,
% with a change of sign for current variables at the ports.
%\end{remark}
The multiport behaviour ${\K}_{P'P"}$ of $\N_P$ on graph $\G_{SP},$ with an arbitrary collection
of vectors $\K_{S'S"}$ as its device characteristic, is given by
${\K}_{P'P"}=[(\V^v({\G_{SP}}))_{S'P'}\oplus (\V^i({\G_{SP}}))_{S"P"}]\lrar {\K}_{S'S"}.$
A natural question is
 `given ${\K}_{P'P"},$ what can we infer about
${\K}_{S'S"}?$'  \\ 
The answer to a more general form of this question is immediate from Theorem \ref{thm:IIT}. 
We state this as a corollary to that theorem.
(Note that the roles of sets $S,P$ in Theorem \ref{thm:IIT} 
and the corollary below, are interchanged.) 
\begin{corollary}
\label{cor:porttransformation}
1. Let $\Vsp$ be a vector space on $S\uplus P$ and $\K_S$ be an arbitrary 
collection of vectors on $S.$
Let $\Vsp\lrar \K_S= \K_P.$ Then $\Vsp\lrar \K_P= (\K_S+\Vsp \times S)\cap \Vsp\circ S=  (\K_S\cap \Vsp \circ S)+\Vsp\times S.$
\\
2. Let $\V_{S'P'S"P"}\equivd (\V^v({\G_{SP}}))_{S'P'}\oplus (\V^i({\G_{SP}}))_{S"P"} $ and let 
 ${\K}_{P'P"}=\V_{S'P'S"P"}\lrar {\K}_{S'S"}.$
Then\\ $\V_{S'P'S"P"}\lrar {\K}_{P'P"}=[{\K}_{S'S"}\cap ((\V^v({\G_{SP}}\circ S))_{S'}\oplus (\V^i({\G_{SP}}\times S))_{S"})]+[(\V^v({\G_{SP}}\times S))_{S'}\oplus (\V^i({\G_{SP}}\circ S))_{S"}].$
\end{corollary}
\begin{proof}
1. Let $\hat{\K}_S\equivd (\K_S+\Vsp \times S)\cap \Vsp\circ S.$
Since $ \Vsp\circ S$ is a vector space and $\Vsp \times S\subseteq \Vsp\circ S,$ 
 a vector $f_S+g_S, g_S\in \Vsp \times S,$ belongs to $\Vsp\circ S$ iff 
$f_S\in \Vsp\circ S.$ Therefore, 
it is clear that \\$(\K_S+\Vsp \times S)\cap \Vsp\circ S=  (\K_S\cap \Vsp \circ S)+\Vsp\times S.$

We note that for any collection of vectors $\K'_S,$
we must have\\
 $\Vsp\lrar \K'_S= \Vsp\lrar (\K'_S+\Vsp \times S)=
\Vsp\lrar (\K'_S\cap \Vsp \circ S).$
Therefore, $\Vsp\lrar \K_S=\Vsp\lrar \hat{\K}_S= \K_P.$\\ 
Since $\hat{\K}_S\subseteq \Vsp\circ S$ and $\hat{\K}_S\supseteq \hat{\K}_S+\Vsp \times S,$ we know by Theorem \ref{thm:IIT}, that 
$\Vsp\lrar(\Vsp\lrar \hat{\K}_S)= \hat{\K}_S.$
Since $\Vsp\lrar \hat{\K}_S=\K_P,$ this proves the result.
\\
2. We have $\V_{S'P'S"P"}\circ S'S"= (\V^v({\G_{SP}})\circ S)_{S'}\oplus (\V^i({\G_{SP}})\circ S)_{S"}= (\V^v({\G_{SP}}\circ S))_{S'}\oplus (\V^i({\G_{SP}}\times S))_{S"}$
and 
$\V_{S'P'S"P"}\times S'S"= (\V^v({\G_{SP}})\times S)_{S'}\oplus (\V^i({\G_{SP}})\times S)_{S"}= (\V^v({\G_{SP}}\times S))_{S'}\oplus (\V^i({\G_{SP}}\circ S))_{S"},$
using Lemma \ref{lem:minorgraphvectorspace}.
The result follows using the previous part, taking $P\equivd P'\uplus P".
\ S\equivd S'\uplus S".$
\end{proof}
It is immediate from Corollary \ref{cor:porttransformation}, that 
 $\Vsp\lrar \K_{S}$ and $\Vsq\lrar \K_{S}$ provide us with the same 
information about $\K_S,$ if $\Vsp\times S= \Vsq\times S$
and $\Vsp\circ S= \Vsq\circ S.$
 This is so because,  if the latter conditions are satisfied, then we have $ \Vsp\lrar (\Vsp \lrar \K_S)=\Vsq\lrar (\Vsq \lrar \K_S)=(\K_S\cap \Vsp \circ S)+\Vsp\times S.$
\\
This fact is useful to note for a generalized multiport
and leads to the following definition.
\begin{definition}
\label{def:internalmodel}
Let $\N^g_P\equivd (\V_{S'P'S"P"},\K_{S'S"}),$  be a generalized multiport.
The collection of vectors\\
%$\V_{S'P'S"P"}\lrar (\V_{S'P'S"P"}\lrar \K_{S'S"})
$(\K_{S'S"}\cap\V_{S'P'S"P"}\circ S'S")+ \V_{S'P'S"P"}\times S'S",$ is  the \nw{internal model} of $\N^g_P,$
available from  its ports.
\end{definition}
It is natural to regard two generalized multiports which have the same 
internal model available at their ports, as port transformations of each 
other. Using Corollary \ref{cor:porttransformation}, this  leads 
to the following definition.
\begin{definition}
\label{def:porttransformation}
%Let $S,P,Q,$ be pairwise disjoint.
We say $\Vsp,\Vsq$ are \nw{port transformations} of each other with respect 
to $S,$ \\iff 
$\Vsp\times S= \Vsq\times S$
and $\Vsp\circ S= \Vsq\circ S.$\\
We say $\Gsp,\Gsq$ are port transformations of each other with respect
to $S,$ \\iff
$\Gsp\times S= \Gsq\times S$
and $\Gsp\circ S= \Gsq\circ S.$\\
Let $\N^g_P\equivd (\V_{S'P'S"P"},\K_{S'S"}),$  be a generalized multiport.
A generalized multiport\\
$\N^g_Q\equivd (\V_{S'Q'S"Q"},\K_{S'S"}),$ is a port transformation  of $\N^g_P$ iff $\V_{S'Q'S"Q"}\times Q'Q"=\V_{S'P'S"P"}\times P'P"$ and $\V_{S'Q'S"Q"}\circ Q'Q"=\V_{S'P'S"P"}\circ P'P".$\\
A multiport $\N_Q\equivd (\Gsp, \K_{S'S"})$ is a port transformation 
of the multiport $\N_P\equivd (\Gsq, \K_{S'S"})$ \\ iff $\Gsp\times S= \Gsq\times S$
and $\Gsp\circ S= \Gsq\circ S.$
\end{definition}
%It is clear that if two generalized multiports are
%port transformations of each other, then they have same internal
% model available at the ports.
The next lemma, which follows from Lemma \ref{lem:minorgraphvectorspace}, emphasizes that the Definition \ref{def:porttransformation}
 is consistently worded. 
\begin{lemma}
Let $\Gsp,\Gsq$ be graphs on $S\uplus P,$ $S\uplus Q,$
respectively.\\
Let $\V_{S'P'S"P"}\equivd (\V^v({\G_{SP}})_{S'P'}\oplus (\V^i({\G_{SP}})_{S"P"}$ and let 
$\V_{S'Q'S"Q"}\equivd (\V^v({\G_{SQ}})_{S'Q'}\oplus (\V^i({\G_{SQ}})_{S"Q"}.$ 
Then $\Gsq$ is a port transformation of $\Gsp$ with respect to $S,$ iff
$\V_{S'Q'S"Q"}$ is a port transformation of $\V_{S'P'S"P"}$ with respect to $S'S".$
\end{lemma}
The next lemma speaks of a useful way of relating port transformations through
a suitable matched composition.
\begin{lemma}
\label{lem:porttransformation}
Let $S,P,Q,$ be pairwise disjoint.
The  spaces $\Vsp,\Vsq$ are port transformations of each other with respect 
to $S,$ iff
there exists $\Vpq$ such that $\Vsp\lrar \Vpq = \Vsq$ and 
such that \\$\Vsp\circ P= \Vpq\circ P, \Vsp\times P= \Vpq\times P.$
\end{lemma}
\begin{proof}
(if) It is easily seen
that $\V_{AB}\lrar (\V_{BC}\lrar \V_C)=(\V_{AB}\lrar \V_{BC})\lrar \V_C,$
when $A,B,C,$ are pairwise disjoint.
It is also easy to see that $\V_{AB}\lrar (\V_{AB}\times B)=\V_{AB}\times A$ 
and $\V_{AB}\lrar (\V_{AB}\circ B)=\V_{AB}\circ A.$ 
\\ Therefore,
  $\Vsq\times S\equivd \Vsq\lrar \0_Q=(\Vsp\lrar \Vpq)\lrar \0_Q = \Vsp\lrar (\Vpq\lrar \0_Q)=  \Vsp\lrar (\Vpq\times P)$\\$=\Vsp\lrar (\Vsp\times P)= \Vsp\times S.$
\\
By replacing $\0_Q$ in the above argument by $\F_Q,$ we see that 
$\Vsq\circ S=\Vsp\circ S.$
\\
(only if) We define $\Vpq\equivd \Vsp\lrar \Vsq.$
Since $\Vsp\lrar \Vsq=\Vpq$ and $\Vsp\times S= \Vsq\times S,
\ \Vsp\circ S= \Vsq\circ S,$
 by Theorem \ref{thm:IITlinear}, we have $\Vsp\lrar\Vpq=\Vsq.$
Next, starting from $\Vsp\lrar \Vsq=\Vpq,$
by using the argument of the previous part, but interchanging the roles 
of $S,P$ and $\Vpq,\Vsq,$ 
we get\\ $\Vsp\circ P= \Vpq\circ P, \Vsp\times P= \Vpq\times P.$ 
\end{proof}
\begin{remark}
We have taken $S,P,Q,$ to be pairwise disjoint in Lemma \ref{lem:porttransformation}. But one of the convenient ways of changing the set of ports 
is to define $\Vsq\equivd \Vsp\lrar(\0_{P_2}\oplus \F_{P_3}), \ 
 P\equivd P_1\uplus P_2\uplus P_3,$ 
taking care 
to keep $\Vsp\circ S= \Vpq\circ S, \Vsp\times S= \Vpq\times S.$
 In this case, to apply Lemma \ref{lem:porttransformation},
 we could define $\tilde{P}$ to be a disjoint copy of $P,$ disjoint also 
from $S,$ take 
$\V_{S\tilde{P}}\equiv (\Vsp)_{S\tilde{P}}$ and 
$\V_{SQ}\equivd \V_{S\tilde{P}}\lrar(\0_{\tilde{P}_2}\oplus \F_{\tilde{P}_3}), \ 
 \tilde{P}\equivd \tilde{P}_1\uplus \tilde{P}_2\uplus \tilde{P}_3,\ 
Q\equivd \tilde{P}_1.$

\end{remark}
%\begin{definition}
%Let $\N^g_P\equivd (\V_{S'P'S"P"},\K_{S'S"}),$  be a generalized multiport.
%A generalized multiport\\
%$\N^g_Q\equivd (\V_{S'Q'S"Q"},\K_{S'S"}),$ is a port transformation  of $\N^g_Q$ iff $\V_{S'Q'S"Q"}\times Q'Q"=\V_{S'P'S"P"}\times P'P"$ and $\V_{S'Q'S"Q"}\circ Q'Q"=\V_{S'P'S"P"}\circ P'P".$\\ 
%The collection of vectors
%$\V_{S'P'S"P"}\lrar (\V_{S'P'S"P"}\lrar \K_{S'S"})
%$(\K_{S'S"}\cap\V_{S'P'S"P"}\circ S'S")+ \V_{S'P'S"P"}\times S'S",$ is  the internal model of $\N^g_P$
%available from  its ports.
%\end{definition}
%Let us call this the internal model of $\N_P$
%available from $\breve{\K}_{P'P"}.$ \\ An immediate second question is\\ 

A natural question next would be `for a given vector space $\Vsp,$ 
what is the port transformation $\Vsq$ with minimum size for the 
set of ports $Q$ and how to construct it efficiently?'

A lower bound is immediate.
\begin{lemma}
\label{lem:portlowerbound}
Let $\Vsp$ be a vector space. Then $|P|\geq r(\Vsp\circ S)-r(\Vsp\times S).$
\end{lemma}
\begin{proof}
We have, by Theorem \ref{thm:dotcrossidentity}, 
$r(\Vsp)=r(\Vsp\circ S)+r(\Vsp\times P)= r(\Vsp\times S)+r(\Vsp\circ P).$
Therefore, $r(\Vsp\circ S)-r(\Vsp\times S)=r(\Vsp\circ P)-r(\Vsp\times P).$ But $|P|\geq r(\Vsp\circ P)$ and $r(\Vsp\times P)\geq 0.$
The result follows.
\end{proof}

We now present a simple algorithm for building the port transformation
with minimum port size, where we actually achieve the lower bound.

\begin{algorithm}
\label{alg:portreduction}
Input: $\V_{TR}$\\
Output: $\V_{T\tilde{R}},$
with $|\tilde{R}|= r(\V_{T\tilde{R}}\circ  T)-r(\V_{T\tilde{R}}\times  T),$
such that\\ $\V_{T\tilde{R}}\times T= \V_{TR}\times T, \V_{T\tilde{R}}\circ  T= \V_{TR}\circ  T.$

Step 1. We build a representative matrix $B$ of $\V_{TR},$ in which\\ $\V_{TR} \circ T,
\V_{TR} \times T,
\V_{TR} \circ R,
\V_{TR} \times R,$ are visible
as shown below.
\begin{align}
\label{eqn:minP}
B= \ppmatrix{B_{1T}&\0_{1R}\\
B_{2T}&B_{2R}\\
\0_{3T}&B_{3R}}.
\end{align}
Here $B_{1T}$ is a representative matrix of
$\V_{TR} \times T,$ $(B_{1T}^T|B_{2T}^T)^T$ is a representative matrix of\\
$\V_{TR} \circ T,$
 $B_{3R}$ is a representative matrix of
$\V_{TR} \times R,$ $(B_{2R}^T|B_{3R}^T)^T$ is a representative matrix of
$\V_{TR} \circ R.$
We note that the rows of $B_{2T},B_{2R}$ are linearly independent
and  their number\\  equals $r(\V_{T{R}}\circ  T)-r(\V_{T{R}}\times  T)=
 r(\V_{T{R}}\circ  R)-r(\V_{T{R}}\times  R).$
\\
Step 2. Let $R=\tilde{R}\uplus R_2,$ where $\tilde{R}$ is a maximal linearly independent
subset of columns of\\ $(B_{2R})$
and let $(B_{2R})= (B_{2\tilde{R}}|B_{2{R_2}}).$\\
Take $\V_{T\tilde{R}}$ to have the representative matrix $\tilde{B}$ shown below.
\begin{align}
\label{eqn:minR2}
\tilde{B}= \ppmatrix{B_{1T}&\0_{1\tilde{R}}\\
B_{2T}&B_{2\tilde{R}}}.
\end{align}
We have\\
$|\tilde{R}|= r(\V_{TR}\circ T)-r(\V_{TR}\times T),$
$\V_{T\tilde{R}}\times T= \V_{TR}\times T, \V_{T\tilde{R}}\circ  T= \V_{TR}\circ  T.$
\end{algorithm}
We give below a simple graph based algorithm for minimising 
the number of ports of a multiport. As can be seen, the algorithm
is linear time on the number of edges in the graph.
We use the term `tree of a graph' to mean a maximal loop free set 
in the graph, even if the graph is not connected.

\begin{algorithm}
\label{alg:graphportminimization}
Input: A graph $\G_{SP}$ on edge set $S\uplus P.$\\
Output: A graph $\G_{S\tilde{P}}$ on edge set $S\uplus \tilde{P},$
such that \\$\V^v(\G_{S\tilde{P}})\times S= \V^v(\G_{S{P}})\times S, \V^v(\G_{S\tilde{P}})\circ S= \V^v(\G_{S{P}})\circ S, $ $|\tilde{P}|= r(\G_{SP}\circ S)-r(\G_{SP}\times S).$\\
Step 1. Build a tree $t_1$ of $\G_{SP}\circ S,$ and grow this into a tree
$t_1\uplus t_2$ of $\G_{SP}.$\\
Step 2. Grow $t_2$ into a tree $\hat{t}_2$ of $\G_{SP}\circ P.$
\\
Step 3. Take $\tilde{P}\equivd \hat{t}_2-{t}_2.$ Build the graph $\G_{S\tilde{P}}\equivd \G_{SP}\circ (S\uplus \hat{t}_2) \times (S\uplus (\hat{t}_2-{t}_2)).$
\\
STOP
\end{algorithm}
{\bf Justification for Algorithm \ref{alg:graphportminimization}
}
\\
We use the following facts from graph theory.\\
Fact 0.  $\V^v(\G), \V^i(\G)$ are determined entirely by 
the set of circuits (cutsets)  and the 
relative orientation of the edges with respect to the  circuits (cutsets) of $\G.$\\
Fact 1. For any graph $\G_{SP},$ if a tree $t$ contains a tree
$t_1$ of $\G_{SP}\circ S,$ then $t-t_1$ is a tree of $\G_{SP}\times P.$\\
Fact 2. If $t_1$ is a tree of $\G_{SP}\circ S$ and $t_2,$ a tree of
$\G_{SP}\times P,$ then $t_1\uplus t_2$ is a tree of $\G_{SP}.$\\
If $t_2$ is contracted in $\G_{SP},$ loops of $\G_{SP}$ within $S,$ continue
to remain as loops, the relative orientation of edges with respect to that of
the loops remains unchanged, and further,  no new loops
are  created in $S.$
Equivalently, $[\V^v(\G_{SP}\times (S\uplus (P-t_2)))]\circ S=(\V^v(\G_{SP}))\circ S.$\\
Fact 3. If $\hat{t}_2$ is a tree of $\G_{SP}\circ P,$ then if  $P-\hat{t}_2$ is deleted
in $\G_{SP},$
cutsets of $\G_{SP}$ within $S,$ continue
to remain as cutsets, the relative orientation of edges with respect to that of  
the cutsets remains unchanged, and further, no new cutsets
are  created in $S.$
Equivalently, $[\V^v(\G_{SP}\circ (S\uplus \hat{t}_2))]\times S= (\V^v(\G_{SP}))\times S.$

We have $\G_{S\tilde{P}}\equivd \G_{SP}\circ (S\uplus \hat{t}_2) \times (S\uplus (\hat{t}_2-t_2))=
\G_{SP}\times(S\uplus (P-t_2))\circ (S\uplus (\hat{t}_2-t_2)).$\\
Therefore, $[\V^v(\G_{S\tilde{P}})]\circ S=[\V^v(\G_{SP}\times(S\uplus (P-t_2))\circ (S\uplus (\hat{t}_2-t_2)))]\circ S= [\V^v(\G_{SP}\times(S\uplus (P-t_2)))]\circ S= [\V^v(\G_{SP})]\circ S$\\
(by Fact 2. above).\\
Next, $[\V^v(\G_{S\tilde{P}})]\times S=[\V^v(\G_{SP}\circ (S\uplus \hat{t}_2) )\times (S\uplus (\hat{t}_2-t_2))]\times S=[\V^v(\G_{SP}\circ (S\uplus \hat{t}_2))]\times S= [\V^v(\G_{SP})]\times S$\\
(by Fact 3. above).
$\ \ \ \ \ \Box$
\\
\begin{remark}
We note that in Algorithm \ref{alg:graphportminimization}, if $\Gsp$ is connected, then $\G_{S\tilde{P}}\circ S$ would always be connected,
but $\G_{S{P}}\circ S$ may not be.
However $[\V^v(\G_{S\tilde{P}})]\circ S$ would always equal $[\V^v(\G_{S{P}})]\circ S.$ 
\end{remark}

%\begin{remark}
%The port minimization idea is essentially intuitive and the result, not surprising.
%In this case ILA is useful primarily for formally justifying an intuitive
%procedure.
%However, the notion
%of decomposition of a network into multiports with minimum number of ports is 
%not quite intuitive. It  has practical utility in analysing networks and is technically useful    
%for creating a theory for dynamical circuits which avoids the state
%and output equations (\cite{HNarayanan1986a,HNarayanan2009,HNPS2013,narayanan2016}).
%\end{remark}

%=========================================
%\begin{remark}
%\item $(\K_{S'P'S"P"}\cap \K_{S'S"})\lrar \K_{P'P"}=
%(\K_{S'P'S"P"}\lrar \K_{P'P"})\cap \K_{S'S"}
%.$
%\item $(\K_{S'P'S"P"}+ \K_{S'S"})\lrar \K_{P'P"}=
%(\K_{S'P'S"P"}\lrar \K_{P'P"})+ \K_{S'S"}
%.$
%
%\end{remark}
%============================================
%\\
%Other applications of Theorem \ref{thm:IITlinear} may be found in
%\cite{HNarayanan1997,HNPS2013, narayanan2016}.
%\subsection{Relation between the behaviour of a linear multiport  and that of its port minimization}
%\label{subsec:relationminport}
%
%=====================================================================
\subsection{Computing port behaviour from the behaviour of port minimized version}
In Section \ref{subsec:dirac}, we reduced the problem of testing rigidity 
of a multiport with gyrators, resistors, controlled and independent sources, 
to that of testing independence of certain columns
of a vector space which is the multiport behaviour of another multiport
with a Dirac structure as its device characterisitc. 
%In Theorem \ref{thm:purslow2}, we need to test independence of certain columns 
%of a vector space which is the multiport behaviour of another multiport. 
In such cases, it would be convenient 
to reduce the problem to that of the independence of a smaller set of columns 
of the behaviour of a port minimized version.
We show how to do this in this section.

We use the notation of Algorithm \ref{alg:graphportminimization}.\\
%The  `minimized ports' $\tilde{P}\equivd \hat{t}_2-t_2.$
Let $\V_{SP}\equivd \V^v(\G_{SP}),$ and let  $t_1, t_2, \hat{t}_2,$  be trees of $\G_{SP}\circ S, \G_{SP}\times P, \G_{SP}\circ P, $ respectively with $t_2\subseteq \hat{t}_2.$
Let the `minimized ports' $\hat{t}_2-t_2$ be denoted by  $\tilde{P}.$

Let $\V_{P'P"}\equivd [(\Vsp)_{S'P'}\oplus (\Vsp^{\perp})_{S"P"}]
\lrar \V_{S'S"}$ and let $\V_{\tilde{P}'\tilde{P}"}$ $\equivd [(\V_{S\tilde{P}})_{S'\tilde{P}'}\oplus (\V_{S\tilde{P}}^{\perp})_{S"\tilde{P}"}
]\lrar \V_{S'S"}.$

We have,
$\V_{S\tilde{P}}=\V^v(\G_{SP} \circ (S\cup \hat{t}_2)\times  (S\cup (\hat{t}_2-t_2)))=\Vsp \circ (S\cup \hat{t}_2)\times  (S\cup (\hat{t}_2-t_2)).$\\
By construction according to Algorithm \ref{alg:graphportminimization}, $\hat{t}_2, t_2$ are  trees of $\G_{SP}\circ P, \G_{SP}\times P,$ respectively. Therefore, by Lemma \ref{lem:minorgraphvectorspace}, $\hat{t}_2,t_2$ are column bases  of
$\V^v(\G_{SP}\circ P)= \Vsp \circ P$ and $\V^v(\G_{SP}\times P)= \Vsp \times P,$ respectively.\\
Again according to Algorithm \ref{alg:graphportminimization}, $\tilde{P}\equivd \hat{t}_2-t_2$ is a column base of $(\V^v(\G_{SP}\circ P))\circ \hat{t}_2\times \tilde{P} =\Vsp \circ \hat{t}_2\times \tilde{P}$\\$=\Vsp\times (P-{t}_2)\circ \tilde{P}$
and $\tilde{P}$ is also a column base of $(\Vsp\times (P-{t}_2)\circ \tilde{P})^{\perp}= \Vsp^{\perp}\circ (P-{t}_2)\times \tilde{P}= \Vsp^{\perp} \times \hat {t}_2\circ \tilde{P}= (\V^i(\G_{SP}\circ P))\times \hat{t}_2\circ \tilde{P}.$
%=\Vsp^{\perp} \times \hat {t}_2\circ \tilde{P}=\Vsp^{\perp}\circ (P-{t}_2)\times \tilde{P}.$

Let $Q^1_{\hat{t}_2P}$ be a fundamental cutset matrix of $\G_{SP}\circ P$
 with respect to $\hat{t}_2.$
Since ${t}_2\subseteq \hat{t}_2,$ we can write
%$Q_{\hat{t}_2P}$ as follows.
\begin{align}
\label{eqn:dotPcrossP1}
Q^1_{\hat{t}_2P}=\ppmatrix{I_{t_2}&0_{t_2(\hat{t}_2-t_2)}&K^1_{t_2(P-\hat{t}_2)}\\
0_{(\hat{t}_2-t_2)t_2}&I_{\hat{t}_2-t_2}&K_{(\hat{t}_2-t_2)(P-\hat{t}_2)}}= \ppmatrix{I_{t_2}&0_{t_2(\tilde{P})}&K^1_{t_2(P-\hat{t}_2)}\\
0_{\tilde{P}t_2}&I_{\tilde{P}}&K_{\tilde{P}(P-\hat{t}_2)}}.
\end{align}
Let $Q^2_{\hat{t}_2P}$ be a fundamental cutset matrix of $\G_{SP}\times  P$
 with respect to $\hat{t}_2.$
Since ${t}_2\subseteq \hat{t}_2,$ we can write
%$Q_{\hat{t}_2P}$ as follows.
\begin{align}
\label{eqn:dotPcrossP21}
Q^2_{\hat{t}_2P}=\ppmatrix{I_{t_2}&K_{t_2(\hat{t}_2-t_2)}&K_{t_2(P-\hat{t}_2)}}
= \ppmatrix{I_{t_2}&K_{t_2(\tilde{P})}&K_{t_2(P-\hat{t}_2)}}
.
\end{align}

We note that $\V^v(\G_{SP}\circ P)= \V^v(\G_{SP})\circ P\supseteq 
\V^v(\G_{SP})\times P= \V^v(\G_{SP}\times P).$
This means that the rows of $Q^2_{\hat{t}_2P}$ belong to the space 
spanned by rows of $Q^1_{\hat{t}_2P}.$
It is also clear that the rows of $Q^2_{\hat{t}_2P}$ together 
with the second set of rows of $Q^1_{\hat{t}_2P}$ are linearly 
independent, and belong to the space
spanned by rows of $Q^1_{\hat{t}_2P}.$ 
Further these rows together are the same in number as the rows 
of $Q^1_{\hat{t}_2P}.$ We conclude that the matrix $Q_{\hat{t}_2P},$ 
shown below,  made up of these rows, is a representative matrix of 
$\Vsp \circ P=\V^v(\Gsp\circ P).$

%Let $Q^1_{\hat{t}_2P}$ be a fundamental cutset matrix of $\G_{SP}\circ P$
% with respect to $\hat{t}_2.$
%Since ${t}_2\subseteq \hat{t}_2,$ we can write 
%$Q_{\hat{t}_2P}$ as follows.
\begin{align}
\label{eqn:dotPcrossP}
Q_{\hat{t}_2P}=\ppmatrix{I_{t_2}&K_{t_2(\hat{t}_2-t_2)}&K_{t_2(P-\hat{t}_2)}\\
0_{(\hat{t}_2-t_2)t_2}&I_{\hat{t}_2-t_2}&K_{(\hat{t}_2-t_2)(P-\hat{t}_2)}}= \ppmatrix{I_{t_2}&K_{t_2(\tilde{P})}&K_{t_2(P-\hat{t}_2)}\\
0_{\tilde{P}t_2}&I_{\tilde{P}}&K_{\tilde{P}(P-\hat{t}_2)}}.
\end{align}
Note that, in this representative matrix of $\Vsp \circ P=\V^v(\Gsp\circ P),$ 
the space 
$\Vsp \times P$ is visible (the first set of rows forms a 
representative matrix of $\Vsp \times P$). 
Also visible is the space $\Vsp\circ P\times (P-{t}_2)$ (the second set of rows restricted to columns $(P-{t}_2), $ forms a 
representative matrix of $\Vsp\circ P\times (P-{t}_2)$)
and the space $\Vsp\circ P\times (P-{t}_2)\circ \tilde{P}$ (the submatrix $I_{\tilde{P}}$ in the second set of rows forms a representative matrix of $\Vsp\circ P \times (P-{t}_2)\circ \tilde{P}$).

We can similarly construct a representative matrix  $B_{(P-{t}_2)P}$ of 
$\Vsp^{\perp} \circ P=\V^i(\Gsp\times P),$
 in which 
the space
$\Vsp^{\perp} \times P$ is visible.
In this matrix, shown below, 
the first set of rows forms a
representative matrix of $\Vsp^{\perp} \times P,$
the second set of rows, restricted to
columns $\hat{{t}}_2,$ forms a
representative matrix of\\ $\Vsp^{\perp}\circ P\times \hat{{t}}_2,$
%and the space $\Vsp^{\perp}\circ P\times (P-{t}_2)\circ \tilde{P}$ (
the submatrix $I_{\tilde{P}}$ in the second set of rows forms a representative matrix of $\Vsp^{\perp}\circ P \times \hat{{t}}_2\circ \tilde{P}= \Vsp^{\perp}\circ (P-t_2)\times \tilde{P} .$
% (the first set of rows form a
%representative matrix of $\Vsp^{\perp} \times P$).
%In this matrix, shown below, also visible is the space $\Vsp^{\perp}\circ P\times (P-{t}_2)$ (the second set of rows, restricted to 
%columns $(P-{t}_2),$  form a
%representative matrix of $\Vsp^{\perp}\circ P\times (P-{t}_2)$)
%and the space $\Vsp^{\perp}\circ P\times (P-{t}_2)\circ \tilde{P}$ (the submatrix $I_{\tilde{P}}$ in the second set of rows forms a representative matrix of $\Vsp^{\perp}\circ P \times (P-{t}_2)\circ \tilde{P}$).
%
%\begin{align}
%\label{eqn:dotPcrossP2}
%B_{(P-{t}_2)P}=\ppmatrix{M_{(P-\hat{t}_2)t_2}&M_{(P-\hat{t}_2)(\hat{t}_2-t_2)}&I_{P-\hat{t}_2}\\
%0&M_{(P-\hat{t}_2)(\hat{t}_2-t_2)}& I_{P-\hat{t}_2}}
%\end{align}
\begin{align}
\label{eqn:dotPcrossP22}
B_{(P-{t}_2)P}=\ppmatrix{
M_{(P-\hat{t}_2)t_2}&M_{(P-\hat{t}_2)(\hat{t}_2-t_2)}&I_{P-\hat{t}_2}\\
M_{(\hat{t}_2-t_2)t_2}&I_{\hat{t}_2-t_2}&0_{(\hat{t}_2-t_2)(P-\hat{t}_2)}
}= \ppmatrix{
M_{(P-\hat{t}_2)t_2}&M_{(P-\hat{t}_2)\tilde{P}}&I_{P-\hat{t}_2}\\
M_{\tilde{P}t_2}&I_{\tilde{P}}&0_{\tilde{P}(P-\hat{t}_2)}
}.
\end{align}
%Note that, in this representative matrix of $\Vsp ^{\perp}\circ P= \V^i(\Gsp\times P),$
% the space 
%$\Vsp^{\perp} \times P= \V^i(\Gsp\circ P),$ is visible (the first set of rows form a
%representative matrix of $\Vsp^{\perp}  \times P$).
%Also visible is the space $\Vsp^{\perp}\circ P\times (P-{t}_2)$ (the second set of rows form a
%representative matrix of $\Vsp^{\perp}\circ P\times (P-{t}_2)$)
%and the space $\Vsp^{\perp}\circ P\times (P-{t}_2)\circ \tilde{P}$ (the submatrix $I_{\tilde{P}}$ in the second set of rows forms a representative matrix of $\Vsp^{\perp}\circ P \times (P-{t}_2)\circ \tilde{P}$).

Let ${\A}_{\tilde{P}'\tilde{P}"}
\equivd \alpha_{\tilde{P}'\tilde{P}"}+{\V}_{\tilde{P}'\tilde{P}"}$
and let ${\A}_{{P}'{P}"}
\equivd \alpha_{{P}'{P}"}+{\V}_{{P}'{P}"}.$
\\
We will show that ${\A}_{{P}'{P}"}$ can be constructed in a simple 
way from ${\A}_{\tilde{P}'\tilde{P}"}.$
We do this by first relating a solution of $\N_P\equivd (\G_{SP},\A_{S'S"})$ to a solution of $\N_{\tilde{P}}\equivd (\G_{S\tilde{P}},\A_{S'S"}).$ 

As before, we denote $\V^v(\G_{SP}), \V^v(\G_{S\tilde{P}}),$ respectively by
$\Vsp, \V_{S\tilde{P}}.$
\\
Let $((f_{S'},\tilde{f}_{\tilde{P}'}),(g_{S"},\tilde{g}_{\tilde{P}"}))$ be a solution of $\N_{\tilde{P}},$ \\i.e., $((f_{S'},\tilde{f}_{\tilde{P}'}),({g}_{S"},\tilde{g}_{\tilde{P}"}))\in [(\V_{S\tilde{P}})_{S'\tilde{P}'}\oplus (\V_{S\tilde{P}}^{\perp})_{S"\tilde{P}"}]\cap {\A}_{S'S"}.$
Then the  corresponding solutions of $\N_P,$ i.e., corresponding vectors in
$[(\V_{SP})_{S'{P}'}\oplus (\Vsp^{\perp})_{S"{P}"}]\cap {\A}_{S'S"}$
have the form\\
$((f_{S'},f_{{P}'}),(g_{S"},g_{{P}"})),$
where $f_{{P}'}\in (\Vsp\circ P)_{P'}, g_{{P}"}\in (\Vsp^{\perp}\circ P)_{P"}.$
\\
Now the vector $f_{{P}'}\in (\Vsp\circ P)_{P'}$ must be a linear combination of the rows of the matrix $Q_{\hat{t}_2P}$ in Equation \ref{eqn:dotPcrossP}
and the vector $\tilde{f}_{\tilde{P}'}\in \Vsp\circ P \times (P-{t}_2)\circ \tilde{P}=\Vsp\circ\hat{t}_2\times \tilde{P}$
is a restriction of a linear combination of the second set of rows of 
$Q_{\hat{t}_2P}.$ Since this set of rows has an identity matrix corresponding to
columns $\tilde{P},$ it follows that 
 $(f_{S'},f_{{P}'})$ can be expanded as 
$(f_{S'},0_{{t'}_2},\tilde{f}_{\tilde{P}'}, f_{\tilde{P}'}^T K_{\tilde{P}(P-\hat{t}_2)})+(0_{S'},h_{{t'}_2},h_{\tilde{P}'},h_{(P-\hat{t}_2)'}),    $
where
 $(h_{{t'}_2},h_{\tilde{P}'},h_{(P-\hat{t}_2)'})\in (\Vsp\times P)_{P'}.$
(The first term in the sum corresponds to the second row of $Q_{\hat{t}_2P}$ 
and the second term corresponds to the first row of $Q_{\hat{t}_2P}$ 
which constitutes a representative matrix of $\Vsp\times P.$)
\\
Next the vector $g_{{P}"}\in (\Vsp^{\perp}\circ P)_{P'}$ must be a linear combination of the rows of the matrix $B_{(P-{t}_2)P}$ in Equation \ref{eqn:dotPcrossP22}
and the vector $\tilde{g}_{\tilde{P}"}\in (\Vsp^{\perp}\circ (P-{t}_2)\times \tilde{P})_{\tilde{P}"},$
is a restriction of a linear combination of the second set of rows of
$B_{(P-{t}_2)P}.$
We note that  this set of rows has an identity matrix corresponding to
columns $\tilde{P}.$ 
Arguing as above, $(g_{S"},\tilde{g}_{\tilde{P}"})$ can be expanded as 
$(g_{S"},g_{\tilde{P}"}^TM_{\tilde{P}t_2},g_{\tilde{P}"}, 0_{(P-\hat{t}_2)"}) +(0_{S"}, k_{{t"}_2}, k_{\tilde{P}"}, k_{(P-\hat{t}_2)"})
,$ where $( k_{{t"}_2},k_{\tilde{P}"},k_{(P-\hat{t}_2)"})\in (\Vsp^{\perp}\times P)_{P"}.$

Suppose ${\V}_{{P}'{P}"},{\V}_{\tilde{P}'\tilde{P}"}
$ are the vector space translates of  ${\A}_{{P}'{P}"},{\A}_{\tilde{P}'\tilde{P}"}
,$ respectively. The relationship between vectors of ${\V}_{{P}'{P}"},{\V}_{\tilde{P}'\tilde{P}"}
,$would be the same as the relationship between vectors in ${\A}_{{P}'{P}"}, {\A}_{\tilde{P}'\tilde{P}"}
,$ This is so because the same role is played by $\Vsp\circ P, \Vsp\times P,
\V_{S\tilde{P}}\circ\tilde{P}, \V_{S\tilde{P}}\times \tilde{P}$ and their orthogonal complements in both cases.

Let $({Q}_{\tilde{P}'}|{B}_{\tilde{P}"})$ be  a representative
matrix of ${\V}_{\tilde{P}'\tilde{P}"}.$
Then the matrix $({Q}_{{P}'}|{B}_{{P}"}),$ shown below, can be  a representative
matrix of ${\V}_{{P}'{P}"},$
taking the columns in the order $(t_2',\tilde{P}',(P-\hat{t}_2)',t"_2,\tilde{P}", (P-\hat{t}_2)"),$ taking $\tilde{P}\equivd \hat{t}_2-t_2, $
using Equation \ref{eqn:dotPcrossP}, Equation \ref{eqn:dotPcrossP22},
\begin{align}
\label{eqn:portminrelation1}
({Q}_{{P}'}|{B}_{{P}"})=
\ppmatrix{0_{t_2'} &\ \  {Q}_{\tilde{P}'}&\ \ {Q}_{\tilde{P}'}(K_{\tilde{P}(P-\hat{t}_2)})&\vdots\vdots   &\ \ {B}_{\tilde{P}"}(M_{\tilde{P}t_2})&\ \ {B}_{\tilde{P}"}&\ \  0_{1(P-\hat{t}_2)"}\\
I_{t_2'}&\ \    K_{t_2'(\tilde{P}')}&\ \   K_{t_2'(P-\hat{t}_2)'}&\vdots\vdots  &\ \    0_{t"_2}&\ \    0_{\tilde{P}"}&\ \   0_{(P-\hat{t}_2)"}
\\
0_{t_2'}&\ \  0_{(\tilde{P}')}&\ \ 0_{(P-\hat{t}_2)'}&\vdots\vdots   &M_{(P-\hat{t}_2)"t"_2}&\ \  M_{(P-\hat{t}_2)"\tilde{P}"}&\ \ I_{(P-\hat{t}_2)"}
},
\end{align}
where
$M_{\tilde{P}t_2}=-K^T_{t_2\tilde{P}}, 
$  $K_{\tilde{P}(P-\hat{t}_2)}= -M^T_{(P-\hat{t}_2)\tilde {P}}\ ,$
using the orthogonality of $\Vsp\times P, \Vsp^{\perp}\circ P$
and\\ $\Vsp\circ P, \Vsp^{\perp}\times P.$

%$M_{(P-\hat{t}_2)t_2}=-K^T_{t_2(P-\hat{t}_2)},$
%nd 
% $B_{3t_2}=(Q_{2(P-\hat{t}_2)})^T(Q_{3\tilde{P}})^T-(Q_{3(P-\hat{t}_2)})^T.$

\subsection{Testing independence of columns of a multiport behaviour using its port minimization}
\begin{problem}
\label{prob:testindependence}
Let $P_1'$ be a set of columns of a multiport behaviour ${\V}_{P'P"}.$
\begin{enumerate}
\item
Test whether $P_1'$ is an independent set of columns of ${\V}_{P'P"}\circ P'$ and of ${\V}_{P'P"}\times P'.$
\item Test whether $P_1"$ is an independent set of columns of ${\V}_{P'P"}\circ P"$ and of ${\V}_{P'P"}\times P".$
\end{enumerate}
\end{problem}
There are situations where a multiport behaviour ${\V}_{P'P"}
\equivd  [(\V^v(\G_{SP}))_{S'P'}\oplus (\V^i(\G_{SP})_{S"P"}]
\lrar \A_{S'S"}$ has size of $P$ much greater  than that of $S$
and we have to test independence as in Problem \ref{prob:testindependence}
.
We have seen that the minimized version $ \V_{\tilde{P}'\tilde{P}"}$
 of ${\V}_{P'P"}$ (using Algorithm \ref{alg:graphportminimization}) 
 has $|\tilde{P}|\leq |S|.$ 
It is therefore worthwhile inferring the independence of columns $P_1'\subseteq P'$ of ${\V}_{P'P"}\circ P', {\V}_{P'P"}\times P'$ 
by computing the independence of columns $(P_1\cap \tilde{P})'$
of $\V_{\tilde{P}'\tilde{P}"}\circ \tilde{P}', \V_{\tilde{P}'\tilde{P}"}\times \tilde{P}'.$
A similar exercise is worth doing with the $"$ entities. 

Let the columns $P_1'$
of $(\V_{SP})_{S'{P}'}$ be dependent. Then some nonzero vector $k_{P_1'}$ on $P_1'$ is orthogonal 
to the columns $P_1'$ of $(\V_{SP})_{S'{P}'}.$
Suppose $((f_{S'},f_{{P}_1'},f_{({P}-P_1)'}),(g_{S"},g_{{P}_1"},g_{({P}-P_1)"}))$  is a vector in\\ $[(\V_{SP})_{S'{P}'}\oplus (\Vsp^{\perp})_{S"{P}"}]\cap {\A}_{S'S"}.$ 
It is clear that $k_{P_1'}$ 
%if a vector on $P_1'$ is orthogonal 
%to the columns $P_1'$ of $(\V_{SP})_{S'\tilde{P}'},$ it 
will also be 
orthogonal to the vector $f_{{P}_1'}.$
Since ${\V}_{P'P"}
\equivd  ([(\V_{SP})_{S'{P}'}\oplus (\Vsp^{\perp})_{S"\tilde{P}"}]\cap {\A}_{S'S"})\circ P'P",$ it follows that the columns $P_1'$ 
of ${\V}_{P'P"}$ must 
also be dependent. 
We conclude that if a set of columns $P_1'$
of ${\V}_{P'P"}$ is to be independent, it is necessary that these 
columns of $(\V_{SP})_{S'{P}'},$ also be independent. 
\\Let columns $P_1'$ of $(\V_{SP})_{S'{P}'},$ be independent.
By Lemma \ref{lem:minorgraphvectorspace}, $P_1$ is a subset of a tree of 
$\G_{SP}\circ P.$ Therefore, in Algorithm \ref{alg:graphportminimization},
we may take $P_1\subseteq \hat{t}_2.$ (We may additionally take $t_2$ to contain as many edges 
of $P_1$ as possible - this is not necessary, but will save effort later.) From Equation \ref{eqn:portminrelation1}, it is clear 
that a subset $P_1'$ of columns of $\tilde{P}'\cup t_2'=\hat{t}_2'$ is independent 
in the representative matrix $({Q}_{{P}'}|{B}_{{P}"})$ of 
$\V_{P'P"}$ iff $P_1'\cap \tilde{P}'$ is independent in the representative 
matrix $({Q}_{\tilde{P}'}|{B}_{\tilde{P}"})$ of $\V_{\tilde{P'}\tilde{P"}}.$
We note that $P_1'$ is an independent subset of columns of ${\V}_{P'P"}$ iff it is an independent 
subset of columns of ${\V}_{P'P"}\circ P'.$
\\
Next we can derive a representative matrix for ${\V}_{P'P"}\cap 0_{P"}$ 
from Equation \ref{eqn:portminrelation1}, as shown below.

%========================================================

\begin{align}
\label{eqn:portminrelationtimes}
(\hat{Q}_{{P}'}|0_{{P}"})=
\ppmatrix{0_{1t_2'} &\ \  \hat{Q}_{\tilde{P}'}&\ \ \hat{Q}_{\tilde{P}'}(K_{\tilde{P}'(P-\hat{t}_2)'})&\vdots\vdots   &\ \ 0_{2t"_2}&\ \ 0_{\tilde{P}"}&\ \  0_{1(P-\hat{t}_2)"}\\
I_{t_2'}&\ \    K_{t_2'(\tilde{P}')}&\ \   K_{t_2'(P-\hat{t}_2)'}&\vdots\vdots  &\ \    0_{2t"_2}&\ \    0_{\tilde{P}"}&\ \   0_{1(P-\hat{t}_2)"}
},
\end{align}
%============================================================
%\begin{align}
%\label{eqn:portminrelationtimes}
%(\hat{Q}_{{P}'}|0_{{P}"})=
%\ppmatrix{0_{1t_2'} &\ \  \hat{Q}_{\tilde{P}'}&\ \ \hat{Q}_{\tilde{P}'}(Q_{2(P-\hat{t}_2)'})&\vdots\vdots   &\ \ 0_{2t"_2}&\ \ 0_{\tilde{P}"}&\ \  0_{1(P-\hat{t}_2)"}\\
%I_{t_2'}&\ \    Q_{3(\tilde{P}')}&\ \   Q_{3(P-\hat{t}_2)'}&\vdots\vdots  &\ \    0_{2t"_2}&\ \    0_{\tilde{P}"}&\ \   0_{1(P-\hat{t}_2)"}
%},
%\end{align}
where the first set of rows is a suitable linear combination of the 
first set of rows of $({Q}_{{P}'}|{B}_{{P}"})$ in Equation \ref{eqn:portminrelation1}.
It is clear once again 
that a subset $P_1'$ of columns of $\tilde{P}'\cup t_2'=\hat{t}_2'$ is independent
in the representative matrix $(\hat{Q}_{{P}'}|0_{{P}"})$ of
${\V}_{\tilde{P'}\tilde{P"}}\cap 0_{P"}$ iff $P_1'\cap \tilde{P}'$ is independent in the representative
matrix $(\hat{Q}_{\tilde{P}'}|0_{\tilde{P}"})$ of ${\V}_{\tilde{P'}\tilde{P"}}
\cap 0_{P"}.$
It follows that a subset $P_1'$ of columns of $\tilde{P}'\cup t_2'=\hat{t}_2'$ is independent
in the representative matrix $(\hat{Q}_{{P}'})$ of
${\V}_{P'P"}\times {P'}$ iff $P_1'\cap \tilde{P}'$ is independent in the representative
matrix $(\hat{Q}_{\tilde{P}'})$ of ${\V}_{\tilde{P'}\tilde{P"}}\times \tilde{P}'
.$
% We note that $P_1'$ is an independent subset of columns of 
%$\V_{P'P"}\cap 0_{P"}$ iff it is an independent subset of columns of
%$\V_{P'P"}\times {P'}$

The argument for the independence of columns $P_1"$ is similar.
If a set of columns $P_1"$
of ${\V}_{P'P"}$ is to be independent, it is necessary that these
columns of $(\V_{SP}^{\perp})_{S"\tilde{P}"},$ also be independent.
Let columns $P_1"$ of $(\V_{SP})^{\perp}_{S"\tilde{P}"},$ be independent.
By Lemma \ref{lem:minorgraphvectorspace}, $P_1$ is a subset of a cotree of
$\G_{SP}\times P,$ which may be taken to be $P-t_2.$ Now $P-\hat{t}_2\subseteq P-t_2.$ 
We may take $P-\hat{t}_2$  to contain as many  edges of $P_1$ as possible 
so that $(P-t_2)-(P-\hat{t}_2)=\hat{t}_2-t_2=\tilde{P}$ contains as few 
edges of $P_1$ as possible (this is not necessary, but will save effort later). 
From Equation \ref{eqn:portminrelation1}, it is clear
that a subset $P_1"$ of columns of $\tilde{P}"\cup (P-\hat{t}_2)"=(P-{t}_2)"$ is independent
in the representative matrix $({Q}_{{P}'}|{B}_{{P}"})$ of
$\V_{P'P"}$ iff $P_1"\cap \tilde{P}"$ is independent in the representative
matrix $({Q}_{\tilde{P}'}|{B}_{\tilde{P}"})$ of $\V_{\tilde{P'}\tilde{P"}}.$
The proof that a subset $P_1"$ of columns of ${\V}_{P'P"}\times P"$ 
is independent iff $P_1"\cap \tilde{P}"$ is an independent subset of columns 
of ${\V}_{P'P"}\times P"$ is similar to the $P_1'$ case for 
${\V}_{P'P"}\times P'.$
\begin{remark}
${\V}_{P'P"}\times P', {\V}_{P'P"}\times P"$ are respectively the open circuit and short circuit port behaviours of $\N_P.$ Even if $\N_P$ is rigid,
they cannot be computed directly using conventional circuit simulators, by open circuiting or short circuiting the ports and computing the solution of the resulting network. This is because, 
in general this network has non unique solution so cannot be handled 
by conventional circuit simulators.
\end{remark}

We summarize the above discussion in the form of an algorithm 
for testing independence of subsets of columns of 
 ${\V}_{P'P"}\circ P',{\V}_{P'P"}\times P',{\V}_{P'P"}\circ P",{\V}_{P'P"}\times P".$
\begin{algorithm}
\label{alg:testindependence}
Input: A linear multiport $\N_P\equivd (\G_{SP},\V_{S'S"}),$ and a subset $P_1\subseteq P.$\\
Output: A statement whether $P_1'$ is independent in ${\V}_{P'P"}\circ P',{\V}_{P'P"}\times P',$\\ and whether $P_1"$ is independent in ${\V}_{P'P"}\circ P",{\V}_{P'P"}\times P",$ where\\ ${\V}_{P'P"}\equivd \V^v(\G_{SP})_{S'P'}\oplus \V^i(\G_{SP})_{S"P"}\lrar \V_{S'S"}.$
\begin{itemize}
\item ($P_1'$) Step 1. Check if $P_1$ is loop free in $\G_{SP}.$ If not,
declare that $P_1'$ is \\not independent in ${\V}_{P'P"}\circ P',$
and not independent in ${\V}_{P'P"}\times P'.$
\\
If yes, construct minimal $\N_{\tilde{P}}$ as in Algorithm \ref{alg:graphportminimization}, taking $\hat{t}_2\supseteq P_1, \tilde{P}\equivd \hat{t}_2-{t}_2$ and taking $t_2\cap P_1$ as large as possible.\\
Step 2. Compute a representative matrix $\V_{\tilde{P}'\tilde{P}"} $ as in Algorithm \ref{alg:hybridgraphnew} and, using it,  \\a representative matrix for $\V_{\tilde{P}'\tilde{P}"} \cap \0_{P"}.$ 
\\
Step 3. Check if the subset of columns $(P_1\cap \tilde{P})'$ is independent in the above two \\representative matrices. If it is independent (dependent) declare 
$P_1'$ to be an independent (dependent) set of columns in the corresponding vector space,
but replacing $\V_{\tilde{P}'\tilde{P}"} \cap \0_{P"}$
by ${\V}_{P'P"}\times P'.$
\item ($P_1"$)
 Step 1. Check if $P_1$ is cutset free in $\G_{SP},$ i.e., whether $(S\cup P)-P_1$ contains a tree of $\G_{SP}.$ If not,
declare that $P_1"$ is \\not independent in ${\V}_{P'P"}\circ P",$
and not independent in ${\V}_{P'P"}\times P".$
\\
If yes, construct minimal $\N_{\tilde{P}}$ as in Algorithm \ref{alg:graphportminimization}, taking $(P-{t}_2)\supseteq P_1$ and taking\\ $(P-\hat{t}_2)\cap P_1$ as large as possible.\\
Step 2. Compute a representative matrix $\V_{\tilde{P}'\tilde{P}"} $ as in Algorithm \ref{alg:hybridgraphnew} and, using it,  \\a representative matrix for $\V_{\tilde{P}'\tilde{P}"} \cap \0_{P'}.$    
\\
Step 3. Check if the subset of columns $(P_1\cap \tilde{P})"$ is independent in the above two \\representative matrices. If it is independent (dependent) declare
$P_1"$ to be an independent (dependent)  set of columns in the corresponding vector space,
but replacing $\V_{\tilde{P}'\tilde{P}"} \cap \0_{P'}$
by ${\V}_{P'P"}\times P".$
\end{itemize}
STOP
\end{algorithm}
\begin{remark}
In Algorithm \ref{alg:testindependence}, while testing for independence 
of $P_1',$ we have taken $P_1\cap t_2$ as large as possible, so that 
$P_1\cap (\hat{t}_2-t_2)$ is as small as possible.
This can be done as follows.
Let $t_1$ be a tree of $\Gsp\circ S.$ Grow this into a tree $t$ of $\Gsp$
by using as many edges of $P_1$ as possible. Take $t_2\equivd t-t_1.$
Grow this into a tree $\hat{t}_2$ of $\Gsp\circ P,$
using edges of $P_1$ to begin with. Since $P_1$ is loop free in $\Gsp,$ 
we must have $P_1\subseteq \hat{t}_2.$
Further $|P_1\cap (\hat{t}_2-t_2)|$ is a minimum.

 While testing for independence
 of $P_1",$ we have taken $P_1\cap (P-\hat{t}_2)$ as large as possible, so that\\
 $P_1\cap ((P-t_2)-(P-\hat{t}_2))=P_1\cap (\hat{t}_2-t_2)$ is as small as possible.
This can be done as follows. 
Grow a tree $\hat{t}_2$ of $\Gsp\circ P$ that contains as few edges of $P_1$
as possible. Let $t_1$ be a tree of $\Gsp\circ S.$ Grow this into a tree $t$ of $\Gsp$ using edges of $\hat{t}_2-P_1$ first. 
Let $t_2\equivd t\cap \hat{t}_2.$
Since $P_1$ is cutset free in $\Gsp,$
we must have $P_1\subseteq P-{t}_2.$
Further $|P_1\cap (\hat{t}_2-t_2)|$ is a minimum.

\end{remark}
\bibliographystyle{elsarticle1-num}
 \bibliography{references}

\end{document}

\bibliographystyle{elsarticle1-num}
\bibliography{references}

\end{document}